\documentclass[12pt]{article}

\usepackage{libertine}
\usepackage[T1]{fontenc}

\usepackage{etex}
\usepackage{amsmath,amssymb,amsfonts,amsthm}
\usepackage{ifpdf}
\usepackage{enumitem}
\usepackage{leftidx}

\usepackage{etoolbox}
\usepackage{fullpage}
\usepackage{longtable}
\usepackage{pdflscape}
\usepackage{multicol}
\usepackage[all]{xy}
\input xy
\xyoption{all}

\usepackage[backend=bibtex8,style=alphabetic,doi=false,isbn=false,url=false,minnames=6,maxnames=6,maxalphanames=6]{biblatex}
\renewbibmacro{in:}{%
  \ifentrytype{article}{}{\printtext{\bibstring{in}\intitlepunct}}}
\renewbibmacro*{volume+number+eid}{%
  \printtext{vol.}
  \printfield{volume}%
  \setunit*{\addnbspace}
  \printfield{number}%
  \setunit{\addcomma\space}%
  \printfield{eid}}
\DeclareFieldFormat[article]{number}{\mkbibparens{#1}}

\usepackage{xifthen}  
\usepackage{breqn}
\usepackage{yfonts}
\usepackage{afterpage}

\usepackage{xcolor}
\definecolor{dark-red}{rgb}{0.7,0.25,0.25}
\definecolor{dark-blue}{rgb}{0.15,0.15,0.55}
\definecolor{medium-blue}{rgb}{0,0,0.65}
\definecolor{DarkGreen}{RGB}{0,150,0}

\ifpdf
\usepackage[pdftex,plainpages=false,hypertexnames=false,pdfpagelabels,breaklinks]{hyperref}
\else
\usepackage[dvips,plainpages=false,hypertexnames=false,breaklinks]{hyperref}
\fi
\hypersetup{
   colorlinks, linkcolor={purple},
   citecolor={medium-blue}, urlcolor={medium-blue}
}

\usepackage{tikz}
\usetikzlibrary{calc}
\usetikzlibrary{shapes}
\usetikzlibrary{backgrounds}
\usetikzlibrary{decorations.pathreplacing}
\usepackage{tikz-qtree}

\tikzstyle{shaded}=[fill=red!10!blue!20!gray!30!white]
\tikzstyle{unshaded}=[fill=white]
\tikzstyle{empty box}=[circle, draw, thick, fill=white, opaque, inner sep=2mm]
\tikzstyle{annular}=[scale=.7, inner sep=1mm, baseline]
\tikzstyle{rectangular}=[scale=.75, inner sep=1mm, baseline=-.1cm]

\usepackage{tcolorbox}
\tcbuselibrary{breakable}
\tcbuselibrary{skins}

\newcommand{\doi}[1]{\href{https://dx.doi.org/#1}{{\small DOI:#1}}}

\newcommand{\googlebooks}[1]{(preview at \href{https://books.google.com/books?id=#1}{google books})}

\newcommand{\numdam}[1]{}

\theoremstyle{plain}
\newtheorem{prop}{Proposition}[section]

\newtheorem{thm}[prop]{Theorem}
\newtheorem{thmalpha}{Theorem}

\newtheorem{lem}[prop]{Lemma}
\newtheorem{cor}[prop]{Corollary}
\newtheorem{fact}[prop]{Fact}
\newtheorem{facts}[prop]{Facts}
\newtheorem*{cor*}{Corollary}
\newtheorem*{thm*}{Theorem}
\newtheorem{question}{Question}
\numberwithin{equation}{section}

\theoremstyle{remark}
\newtheorem{example}[prop]{Example}
\newtheorem*{exc}{Exercise}
\newtheorem{remark}[prop]{Remark}           
\newtheorem*{rem*}{Remark}               
\newtheorem*{example*}{Example}                

\theoremstyle{definition}
\newtheorem{defn}[prop]{Definition}         

\newtheorem*{defn*}{Definition}             

\newenvironment{mycolorbox}[1][]%
  {\if\detokenize{#1}\relax\relax%
      \begin{tcolorbox}%
    \else%
      \begin{tcolorbox}[#1]%
    \fi%
  \vspace{-3mm}%
  \parskip=0.5\baselineskip \advance\parskip by 0pt plus 2pt%
  \parindent=0pt%
}
  {\end{tcolorbox}}

\makeatletter
\@ifpackagelater{tcolorbox}{2015/01/01}%
  {%
    \newenvironment{boxedexample}{\begin{mycolorbox}[breakable,notitle,boxrule=1pt,colback=blue!5,colframe=blue!20,enhanced jigsaw]}{\end{mycolorbox}}
  }
  {%
    \newenvironment{boxedexample}{\begin{mycolorbox}[breakable,notitle,boxrule=1pt,colback=blue!5,colframe=blue!20]}{\end{mycolorbox}}
  }%
\makeatother
\newenvironment{boxedexample*}{\begin{mycolorbox}[notitle,boxrule=1pt,colback=blue!5,colframe=blue!20]}{\end{mycolorbox}}

\theoremstyle{plain}

\newcounter{comment}
\newcommand{\noop}[1]{}

\def\clap#1{\hbox to 0pt{\hss#1\hss}}

\newcommand{\Integer}{\mathbb Z}

\def\semicolon{;}
\def\applytolist#1{
    \expandafter\def\csname multi#1\endcsname##1{
        \def\multiack{##1}\ifx\multiack\semicolon
            \def\next{\relax}
        \else
            \csname #1\endcsname{##1}
            \def\next{\csname multi#1\endcsname}
        \fi
        \next}
    \csname multi#1\endcsname}

\def\calc#1{\expandafter\def\csname c#1\endcsname{{\mathcal #1}}}
\applytolist{calc}QWERTYUIOPLKJHGFDSAZXCVBNM;
\def\bbc#1{\expandafter\def\csname bb#1\endcsname{{\mathbb #1}}}
\applytolist{bbc}QWERTYUIOPLKJHGFDSAZXCVBNM;
\def\bfc#1{\expandafter\def\csname bf#1\endcsname{{\mathbf #1}}}
\applytolist{bfc}QWERTYUIOPLKJHGFDSAZXCVBNM;

\DeclareMathOperator{\depth}{depth}
\DeclareMathOperator{\nbhd}{nbhd}

\DeclareMathOperator{\Tr}{Tr}

\DeclareMathOperator{\FPdim}{FPdim}

\newcommand{\set}[2]{\left\{#1\middle|#2\right\}}
\newcommand{\jw}[1]{f^{(#1)}}
\newcommand{\twoone}{{\rm II}$_1$}

\renewcommand{\imath}{\mathfrak{i}}
\renewcommand{\jmath}{\mathfrak{j}}

\newcommand{\iso}{\cong}

\newcommand{\isoto}{\overset{\iso}{\to}}

\makeatletter
\newcommand{\hashdef}[2]{\@namedef{#1}{#2}}
\newcommand{\hashlookup}[1]{\@nameuse{#1}}
\makeatother

\newcommand{\pathtographs}{diagrams/graphs/}

\hashdef{bwd1duals1}{FF1481245D7F1881}%
\hashdef{bwd1v1duals1v1}{7BFEC24E607C77CA}%
\hashdef{bwd1v1p1duals1v1x2}{7CC0C785EDD0C3AA}%
\hashdef{bwd1v1p1duals1v2x1}{FE6BDA913A50F4D9}%
\hashdef{bwd1v1p1p1duals1v1x3x2}{30F84974997CC5A4}%
\hashdef{bwd1v1p1p1duals1v2x1x3}{99E8BDA8F0BBDAF7}%
\hashdef{bwd1v1p1p1p1duals1v2x1x4x3}{19385ECD51785D9D}%
\hashdef{bwd1v1p1p1p1v1x0x0x0p1x0x0x0p0x1x0x0p0x1x0x0p0x0x1x0p0x0x1x0p0x0x0x1p0x0x0x1v1x1x0x0x0x0x0x0p1x1x0x0x0x0x0x0p0x0x1x1x0x0x0x0p0x0x1x1x0x0x0x0p0x0x0x0x1x0x0x0p0x0x0x0x0x1x0x0p0x0x0x0x0x0x1x0p0x0x0x0x0x0x0x1p0x0x0x0x0x0x0x1duals1v1x2x3x4v2x1x4x3x6x5x8x7x9}{4EE02AFCB9D3F85C}%
\hashdef{bwd1v1p1p1v1x1x0duals1v1x2x3}{53D091157B8A128E}%
\hashdef{bwd1v1p1v0x1duals1v1x2}{2FEE77BB9631D4C9}%
\hashdef{bwd1v1p1v1x0duals1v1x2}{3686B3360A5B8485}%
\hashdef{bwd1v1p1v1x0p0x1p0x1vduals1v1x2v}{5985AFFC2B3FD1AB}%
\hashdef{bwd1v1p1v1x0p0x1v1x1v1v1p1duals1v1x2v1v1x2}{C88396BCD34DCC17}%
\hashdef{bwd1v1p1v1x0p0x1vduals1v1x2v}{312A8FA7509D7A0C}%
\hashdef{bwd1v1p1v1x0p1x0p0x1vduals1v1x2v}{3AD1A0CD36BCCCCB}%
\hashdef{bwd1v1p1v1x0p1x0v1x0p0x1v1x1v1duals1v1x2v1x2v1}{3255351FDAC60732}%
\hashdef{bwd1v1p1v1x0p1x0v1x0p1x0v1x1duals1v1x2v1x2}{8F007BF50C3E09AC}%
\hashdef{bwd1v1p1v1x0p1x0v1x0v1p1v1x0p1x0v1x1duals1v1x2v1v1x2}{21FA8EFC6770814D}%
\hashdef{bwd1v1p1v1x0p1x0v1x0v1p1v1x0v1p1duals1v1x2v1v1}{D341886B743822B2}%
\hashdef{bwd1v1p1v1x0p1x0vduals1v1x2v}{64427DEF67654B1A}%
\hashdef{bwd1v1p1v1x0p1x1duals1v1x2}{FE060C2815A2CF0E}%
\hashdef{bwd1v1p1v1x1v1duals1v1x2v1}{434A46414835DC7C}%
\hashdef{bwd1v1v1p1p1duals1v1}{DDCAE11B9F0B097A}%
\hashdef{bwd1v1v1p1p1p1duals1v1}{CDC7BD640A950B1A}%
\hashdef{bwd1v1v1p1p1v0x0x1p0x0x1duals1v1v1x2}{97C6253CFBF25860}%
\hashdef{bwd1v1v1p1p1v0x0x1p0x0x1duals1v1v2x1}{B6A39D4925C39255}%
\hashdef{bwd1v1v1p1p1v0x0x1p0x0x1v1x0p0x1v1x0p0x1duals1v1v1x2v2x1}{60C3B0DF16C9B488}%
\hashdef{bwd1v1v1p1p1v0x0x1p0x0x1v1x0p0x1v1x0p0x1duals1v1v2x1v1x2}{20BF07A7BCE3A9D1}%
\hashdef{bwd1v1v1p1p1v0x0x1p0x0x1v1x0p1x0duals1v1v1x2}{11C2D4F6781840DA}%
\hashdef{bwd1v1v1p1p1v0x0x1vduals1v1v1}{0A4947B2A4102EA5}%
\hashdef{bwd1v1v1p1p1v0x1x0p0x0x1p0x0x1duals1v1v1x3x2}{C9FBD000FBAA6442}%
\hashdef{bwd1v1v1p1p1v0x1x0p0x0x1vduals1v1v1x2}{E541550C227DE69B}%
\hashdef{bwd1v1v1p1p1v0x1x0p0x1x0p0x0x1duals1v1v1x2x3}{4BF2BBB8875C837E}%
\hashdef{bwd1v1v1p1p1v1x0x0p0x0x1p0x0x1duals1v1v1x2x3}{3F7E124E6DE03EBB}%
\hashdef{bwd1v1v1p1p1v1x0x0p0x0x1p0x0x1duals1v1v1x3x2}{9D2D9A9B9F96C4FC}%
\hashdef{bwd1v1v1p1p1v1x0x0p0x0x1vduals1v1v2x1}{FF0EF152A2971AE0}%
\hashdef{bwd1v1v1p1p1v1x0x0p0x1x0duals1v1v2x1}{63FDB4A04F409318}%
\hashdef{bwd1v1v1p1p1v1x0x0p0x1x0p0x0x1duals1v1v1x3x2}{82870A71444C2161}%
\hashdef{bwd1v1v1p1p1v1x0x0p0x1x0p0x0x1vduals1v1v1x2x3}{23D2117D092CA07A}%
\hashdef{bwd1v1v1p1p1v1x0x0p0x1x0p0x0x1vduals1v1v1x3x2}{8EE51B69D19C73E8}%
\hashdef{bwd1v1v1p1p1v1x0x0p1x0x0p0x1x0p0x1x0p0x0x1p0x0x1v0x0x1x0x0x0duals1v1v3x5x1x6x2x4}{E7BF97A34F83A5AE}%
\hashdef{bwd1v1v1p1p1v1x0x0p1x0x0p0x1x0p0x1x0p0x0x1p0x0x1v0x0x1x0x0x0p0x0x0x0x0x1duals1v1v3x5x1x6x2x4}{FB27D63F64600520}%
\hashdef{bwd1v1v1p1p1v1x0x0p1x0x0p0x1x0p0x1x0p0x0x1p0x0x1v0x1x0x0x0x0p0x0x1x0x0x0duals1v1v3x5x1x6x2x4}{475351BBFD75F85E}%
\hashdef{bwd1v1v1p1v0x1p1x0p0x1v0x1x0p0x0x1p1x0x0p0x1x0v1x0x0x0p0x1x0x0p0x0x1x0v0x0x1p1x0x0p0x1x0vduals1v1v1x3x2v1x3x2v}{FD2DD3E8441DDD52}%
\hashdef{bwd1v1v1p1v0x1p1x0p0x1v1x0x0p0x0x1p0x1x0p0x1x0v0x0x1x0p1x0x0x0p0x1x0x0v0x1x0p0x0x1p1x0x0vduals1v1v1x3x2v1x3x2v}{B3316EAA8C579D08}%
\hashdef{bwd1v1v1p1v0x1p1x0p0x1v1x0x0p0x1x0v0x1duals1v1v1x3x2v1}{57514D6B84AEC10D}%
\hashdef{bwd1v1v1p1v0x1p1x0p0x1v1x0x0p1x0x0p0x1x0v0x1x0p0x0x1p0x0x1v0x0x1duals1v1v1x3x2v3x2x1}{258B7D60F7E710F1}%
\hashdef{bwd1v1v1p1v0x1p1x0p0x1v1x0x0p1x0x0p0x1x0v0x1x0p0x0x1p0x1x0p0x0x1v0x1x0x0p0x1x0x0p0x0x1x0v0x1x0duals1v1v1x3x2v2x1x3x4v1}{79DE60F31A12A971}%
\hashdef{bwd1v1v1p1v0x1p1x0p0x1v1x0x0p1x0x0p0x1x0v0x1x1duals1v1v1x3x2v1}{FCD08E5EE74A3D9F}%
\hashdef{bwd1v1v1p1v1x0p0x1p0x1v1x0x0p0x1x0p0x1x0v1x0x0p1x0x0p0x1x0p0x1x0v0x1x0x0p0x0x1x0p0x1x0x0v0x0x1duals1v1v3x2x1v1x4x3x2v1}{2D2476C3E49DC4D3}%
\hashdef{bwd1v1v1p1v1x0p0x1p0x1v1x0x0p0x1x0p0x1x0v1x0x0p1x0x0p0x1x0v0x1x0duals1v1v3x2x1v1x3x2}{167BB813F9AD2F17}%
\hashdef{bwd1v1v1p1v1x0p0x1p0x1v1x0x0p0x1x0p0x1x0v1x1x0duals1v1v3x2x1v1}{71A4B3AEF8E1B944}%
\hashdef{bwd1v1v1p1v1x0p0x1p0x1v1x0x0p0x1x0p1x0x0p0x0x1v0x1x0x1p0x0x1x0p0x0x1x0duals1v1v3x2x1v1x2x3}{FCD674FC67CE10ED}%
\hashdef{bwd1v1v1p1v1x0p0x1p0x1v1x0x0p0x1x0p1x0x0p0x0x1v0x1x0x1p0x0x1x0p0x0x1x0duals1v1v3x2x1v1x3x2}{563669E4B19DC13A}%
\hashdef{bwd1v1v1p1v1x0p0x1p0x1v1x0x0p0x1x0v1x0duals1v1v3x2x1v1}{895FB7B1B3AA4660}%
\hashdef{bwd1v1v1p1v1x0p0x1p0x1v1x0x0p1x0x0p0x1x0p0x0x1v0x1x0x0p0x1x0x0p0x0x1x1duals1v1v3x2x1v1x2x3}{1181047C9B6DB69D}%
\hashdef{bwd1v1v1p1v1x0p0x1p0x1v1x0x0p1x0x0p0x1x0p0x0x1v0x1x0x0p0x1x0x0p0x0x1x1duals1v1v3x2x1v2x1x3}{03B8E46C121FDBAB}%
\hashdef{bwd1v1v1p1v1x0p0x1p0x1v1x1x0p0x0x1duals1v1v3x2x1}{EE6D0ECBE6A8121F}%
\hashdef{bwd1v1v1p1v1x0p0x1p0x1vduals1v1v2x1x3}{A6F28E7DF09EB8A0}%
\hashdef{bwd1v1v1p1v1x0p0x1p1x0p0x1v1x0x0x1duals1v1v1x3x2x4}{963B66E116DE3C61}%
\hashdef{bwd1v1v1p1v1x0p0x1p1x0v0x0x1p0x1x0p0x1x0p1x0x0v0x0x1x0p0x1x0x0p1x0x0x0p0x0x0x1p1x0x0x0v0x0x0x0x1p1x0x0x0x0p0x0x0x1x0p0x1x0x0x0p0x0x0x1x0vduals1v1v1x3x2v1x3x2x5x4v}{9E26CDD3DD31379C}%
\hashdef{bwd1v1v1p1v1x0p0x1p1x0v0x0x1p0x1x0p1x0x0p0x1x0v0x0x0x1p0x1x0x0p1x0x0x0p0x0x1x0p1x0x0x0v1x0x0x0x0p0x0x0x0x1p0x1x0x0x0p0x0x0x1x0p0x0x0x1x0vduals1v1v1x3x2v1x3x2x5x4v}{52206C75997D1985}%
\hashdef{bwd1v1v1p1v1x0p1x0p0x1p0x1duals1v1v3x2x1x4}{E31BB4AED7B91217}%
\hashdef{bwd1v1v1p1v1x0p1x0p0x1p0x1v0x1x1x0duals1v1v1x3x2x4}{CC3B89144EA94D74}%
\hashdef{bwd1v1v1p1v1x0p1x0p0x1p0x1v0x1x1x0duals1v1v4x2x3x1}{608002965ADF9D01}%
\hashdef{bwd1v1v1p1v1x1duals1v1v1}{09E06155B62BC9DF}%
\hashdef{bwd1v1v1p1v1x1p0x1duals1v1v1x2}{3B336554507AABE9}%
\hashdef{bwd1v1v1p1v1x1v1duals1v1v1}{AB68D3EA68C1403D}%
\hashdef{bwd1v1v1p1v1x1v1v1duals1v1v1v1}{A207DA908F977FE3}%
\hashdef{bwd1v1v1v1p1p1duals1v1v1x2x3}{5B4884EAB80FDD81}%
\hashdef{bwd1v1v1v1p1p1duals1v1v1x3x2}{2892A7FE6F263C97}%
\hashdef{bwd1v1v1v1p1p1v0x1x0p0x0x1v1x0p0x1duals1v1v1x3x2v1x2}{9CE04C5F2FFCB6E7}%
\hashdef{bwd1v1v1v1p1p1v0x1x0p0x0x1v1x1duals1v1v1x3x2v1}{6CC9743BD85C506E}%
\hashdef{bwd1v1v1v1p1p1v0x1x0p0x0x1vduals1v1v1x2x3v}{33AA7E8D30B840B2}%
\hashdef{bwd1v1v1v1p1p1v0x1x0p0x0x1vduals1v1v1x3x2v}{EC2F89463AA53A09}%
\hashdef{bwd1v1v1v1p1p1v0x1x0p0x1x0duals1v1v1x2x3}{4DB6033FAC53C0D3}%
\hashdef{bwd1v1v1v1p1p1v0x1x0p0x1x0p0x0x1duals1v1v1x2x3}{6BE168111F2B4933}%
\hashdef{bwd1v1v1v1p1p1v0x1x0p0x1x0p0x0x1v0x0x1duals1v1v1x2x3v1}{A17D8E417394CB5B}%
\hashdef{bwd1v1v1v1p1p1v0x1x0p0x1x0v1x0p0x1duals1v1v1x2x3v1x2}{1FBAC673549AC212}%
\hashdef{bwd1v1v1v1p1p1v0x1x0p0x1x0v1x0v1duals1v1v1x2x3v1}{DE80EDB1BF395FAE}%
\hashdef{bwd1v1v1v1p1p1v0x1x0p0x1x0v1x0v1p1duals1v1v1x2x3v1}{1EC090143A8AB5F5}%
\hashdef{bwd1v1v1v1p1p1v0x1x0p0x1x0v1x0v1p1v1x0duals1v1v1x2x3v1v1}{DC1BA223CF5738F4}%
\hashdef{bwd1v1v1v1p1p1v0x1x0p0x1x0v1x0v1p1v1x0v1duals1v1v1x2x3v1v1}{11D58FA24CF75CE0}%
\hashdef{bwd1v1v1v1p1p1v0x1x0p0x1x0v1x0v1p1v1x0v1p1v1x0v1p1v1x0v1p1v1x0duals1v1v1x2x3v1v1v1v1v1}{F1433B56890E9DA5}%
\hashdef{bwd1v1v1v1p1p1v0x1x0vduals1v1v1x2x3v}{6D8E6113F23A0457}%
\hashdef{bwd1v1v1v1p1p1v1x0x0p0x0x1v1x0p0x1duals1v1v1x2x3v2x1}{478246CCA7EDD412}%
\hashdef{bwd1v1v1v1p1p1v1x0x0p0x0x1v1x0p1x0p0x1v0x0x1v1duals1v1v1x2x3v1x3x2v1}{62487ECF3509D91F}%
\hashdef{bwd1v1v1v1p1p1v1x0x0p0x0x1v1x0p1x0p0x1v1x0x0p0x0x1v0x1p1x0p1x0p0x1v0x1x0x0p0x0x0x1v0x1p1x0p1x0p0x1v0x1x0x0p0x0x0x1v0x1p1x0p1x0p0x1v0x1x0x0p0x0x0x1v0x1p1x0p1x0p0x1duals1v1v1x2x3v1x3x2v1x2x4x3v1x2x4x3v1x2x4x3v1x2x4x3}{8DC137BDB8ACC9AE}%
\hashdef{bwd1v1v1v1p1p1v1x0x0p0x0x1v1x0p1x0p0x1v1x0x0p0x0x1v1x0p0x1p0x1duals1v1v1x2x3v1x3x2v3x2x1}{7782DBD7EB8E4835}%
\hashdef{bwd1v1v1v1p1p1v1x0x0p0x0x1v1x0p1x0p0x1v1x0x0p0x0x1v1x0p1x0p0x1p0x1v0x0x1x0v1duals1v1v1x2x3v1x3x2v3x2x1x4v1}{0B326DC098379970}%
\hashdef{bwd1v1v1v1p1p1v1x0x0p0x0x1v1x0p1x0p0x1v1x0x0p0x0x1v1x1duals1v1v1x2x3v1x3x2v1}{2D2D055DD6344773}%
\hashdef{bwd1v1v1v1p1p1v1x0x0p0x0x1v1x1duals1v1v1x2x3v1}{538BA85F0C7DAC73}%
\hashdef{bwd1v1v1v1p1p1v1x0x0p0x0x1vduals1v1v1x2x3v}{8BAF581C2C6BC82A}%
\hashdef{bwd1v1v1v1p1p1v1x0x0p0x1x0p0x0x1v1x0x0p0x1x0duals1v1v1x2x3v2x1}{78CF6805F77E26EC}%
\hashdef{bwd1v1v1v1p1p1v1x0x0p0x1x0p0x0x1v1x0x0p0x1x0p0x0x1duals1v1v1x2x3v1x2x3}{772136FAF425B059}%
\hashdef{bwd1v1v1v1p1p1v1x0x0p0x1x0p0x0x1v1x0x0p0x1x0p0x0x1duals1v1v1x2x3v1x3x2}{B5425B05363851E6}%
\hashdef{bwd1v1v1v1p1p1v1x0x0p0x1x0p0x0x1v1x0x0p0x1x0p0x0x1duals1v1v1x2x3v2x1x3}{3C206F14234AE274}%
\hashdef{bwd1v1v1v1p1p1v1x0x0p0x1x0p0x0x1v1x0x0p0x1x0p0x0x1duals1v1v1x3x2v1x3x2}{955DD8674F38DB2F}%
\hashdef{bwd1v1v1v1p1p1v1x0x0p0x1x0p0x0x1v1x0x0p0x1x0p0x0x1duals1v1v2x1x3v2x1x3}{24CD197B973929C1}%
\hashdef{bwd1v1v1v1p1p1v1x0x0p0x1x0p0x0x1vduals1v1v1x2x3v}{F09CEB260E1920DC}%
\hashdef{bwd1v1v1v1p1p1v1x0x0p0x1x0p0x0x1vduals1v1v1x3x2v}{7E3E0C0D9C0612F8}%
\hashdef{bwd1v1v1v1p1p1v1x0x0p0x1x0p0x1x0v1x0x0duals1v1v1x2x3v1}{DB2C121504201C9D}%
\hashdef{bwd1v1v1v1p1p1v1x0x0p0x1x0v1x0p0x1duals1v1v1x2x3v2x1}{20FE18846E5E4AAA}%
\hashdef{bwd1v1v1v1p1p1v1x0x0p1x0x0duals1v1v1x2x3}{A52C066C3571F824}%
\hashdef{bwd1v1v1v1p1p1v1x0x0p1x0x0duals1v1v1x3x2}{9A987C46764C090F}%
\hashdef{bwd1v1v1v1p1p1v1x0x0p1x0x0p0x0x1duals1v1v1x2x3}{13E48D9C41F19344}%
\hashdef{bwd1v1v1v1p1p1v1x0x0p1x0x0v1x0p0x1duals1v1v1x3x2v2x1}{C050735717796088}%
\hashdef{bwd1v1v1v1p1p1v1x0x0v1duals1v1v1x2x3v1}{C4B6B3927CB75C0E}%
\hashdef{bwd1v1v1v1p1p1v1x0x0vduals1v1v1x2x3v}{BBF642A92F2B5D50}%
\hashdef{bwd1v1v1v1p1p1v1x0x0vduals1v1v1x3x2v}{6581A6365CB20977}%
\hashdef{bwd1v1v1v1p1v0x1p0x1v1x0p0x1v0x1p1x0p1x0p0x1v1x0x0x0p0x0x1x0v0x1p1x0p1x0p0x1v1x0x0x0p0x0x1x0duals1v1v1x2v1x2v2x1v1x2}{5F9C79632BF17740}%
\hashdef{bwd1v1v1v1p1v0x1p0x1v1x0p0x1v0x1p1x0p1x0p0x1v1x0x0x0p0x0x1x0v0x1p1x0p1x0p0x1v1x0x0x0p0x0x1x0vduals1v1v1x2v1x2v2x1v1x2}{53E312AA13602AD0}%
\hashdef{bwd1v1v1v1p1v1x0p0x1p0x1v0x0x1p1x0x0p1x0x0p0x1x0v0x0x1x0p1x0x0x0p0x0x0x1vduals1v1v1x2v1x2x4x3v}{C5BA10573E9BDAF9}%
\hashdef{bwd1v1v1v1p1v1x0p0x1p0x1v1x0x0p1x0x0p0x1x0p0x0x1v1x0x0x0p0x0x1x0p0x1x0x0p0x0x0x1vduals1v1v1x2v1x3x2x4v}{C073369AFA6CB3A3}%
\hashdef{bwd1v1v1v1p1v1x0p0x1p1x0v0x0x1p0x1x0p0x1x0p1x0x0v0x0x0x1p0x0x1x0p1x0x0x0p0x1x0x0v1x0x0x0p0x0x1x0p0x1x0x0p0x0x0x1p1x0x0x0vduals1v1v1x2v1x2x4x3v1x3x2x5x4}{A27BF3A2730E7BA5}%
\hashdef{bwd1v1v1v1p1v1x0p0x1p1x0v0x0x1p0x1x0p1x0x0vduals1v1v1x2v1x3x2}{C42FB71E31F42898}%
\hashdef{bwd1v1v1v1p1v1x0p0x1p1x0v1x0x0p0x1x0p0x0x1p0x1x0v1x0x1x0p0x0x0x1p0x0x0x1duals1v1v1x2v4x2x3x1}{F360B25E6AC3CA91}%
\hashdef{bwd1v1v1v1p1v1x0p0x1p1x0v1x0x0p0x1x0v0x1duals1v1v1x2v2x1}{80AC63AFCD4727BD}%
\hashdef{bwd1v1v1v1p1v1x0p0x1p1x0v1x0x0p1x0x0p0x1x0v1x0x0p0x0x1p0x0x1v0x0x1duals1v1v1x2v1x3x2v1}{55A9A44CF89D088F}%
\hashdef{bwd1v1v1v1p1v1x0p0x1p1x0v1x0x0p1x0x0p0x1x0v1x0x0p0x0x1p1x0x0p0x0x1v0x0x1x0p0x0x0x1p0x0x0x1v0x0x1duals1v1v1x2v1x3x2v3x2x1}{489D1B13DBBEE401}%
\hashdef{bwd1v1v1v1p1v1x0p0x1p1x0v1x0x0p1x0x0p0x1x0v1x0x1duals1v1v1x2v1x3x2}{6DA5C92B7130F582}%
\hashdef{bwd1v1v1v1p1v1x0p0x1p1x0v1x1x0p0x0x1duals1v1v1x2v1x2}{A9C6A21B7F716968}%
\hashdef{bwd1v1v1v1p1v1x0p0x1p1x0v1x1x0v1duals1v1v1x2v1}{C34D58AE6D8924A1}%
\hashdef{bwd1v1v1v1p1v1x0p0x1v0x1p1x0p0x1p0x1v0x1x0x0p0x1x0x0v1x0p0x1duals1v1v1x2v1x4x3x2v2x1}{6C9E115B042D79E2}%
\hashdef{bwd1v1v1v1p1v1x0p0x1v0x1p1x0p0x1p0x1v0x1x0x0p0x1x0x0v1x0p0x1duals1v1v1x2v4x3x2x1v2x1}{C5191393E9308458}%
\hashdef{bwd1v1v1v1p1v1x0p0x1v0x1p1x0p0x1v1x0x0p1x0x0p0x1x0v1x0x0p0x0x1p1x0x0p0x0x1p0x1x0p0x0x1v0x0x0x1x0x0p1x0x0x0x0x0v1x0p1x0p0x1duals1v1v1x2v1x3x2v1x2x4x3x6x5v1x3x2}{03CE4D3A57A89E05}%
\hashdef{bwd1v1v1v1p1v1x0p0x1v0x1p1x0p0x1v1x0x0p1x0x0p0x1x0v1x0x0p0x0x1p1x0x0p0x0x1p0x1x0p0x0x1v0x0x0x1x0x0v1duals1v1v1x2v1x3x2v1x2x4x3x6x5v1}{84F37F0C29EB1760}%
\hashdef{bwd1v1v1v1p1v1x0p0x1v0x1p1x0p0x1v1x0x0p1x0x0p0x1x0v1x0x0p0x0x1p1x0x0p0x0x1p0x1x0p0x0x1v1x0x0x0x0x0p0x0x0x1x0x0v1x1duals1v1v1x2v1x3x2v1x2x4x3x6x5v1}{4812856A9E62037E}%
\hashdef{bwd1v1v1v1p1v1x0p0x1v0x1p1x0p0x1v1x0x0p1x0x0p0x1x0v1x0x0p0x1x0p0x0x1p0x0x1p0x0x1duals1v1v1x2v1x3x2v3x5x1x4x2}{46C2AF73BA56DE6D}%
\hashdef{bwd1v1v1v1p1v1x0p0x1v0x1p1x0p0x1v1x0x0p1x0x0p0x1x0v1x0x0p0x1x0p1x0x0p0x0x1p0x0x1p0x0x1v0x0x1x0x0x0p0x0x0x0x0x1v1x0p1x0p0x1p0x1v0x1x0x0p0x0x0x1v1x0p0x1p0x1duals1v1v1x2v1x3x2v6x4x3x2x5x1v4x2x3x1v3x2x1}{7B3C964E102C288C}%
\hashdef{bwd1v1v1v1p1v1x0p0x1v0x1p1x0p0x1v1x0x0p1x0x0p0x1x0v1x0x0p1x0x0p0x0x1p0x1x0p0x0x1p0x0x1v1x0x0x0x0x0p0x0x0x0x1x0v1x0p1x0p0x1p0x1v0x0x0x1v1duals1v1v1x2v1x3x2v1x5x4x3x2x6v4x2x3x1v1}{048667EBAB4BAD36}%
\hashdef{bwd1v1v1v1p1v1x0p0x1v0x1p1x0p0x1v1x0x0p1x0x0p0x1x0v1x0x1p0x1x0p0x0x1duals1v1v1x2v1x3x2v1x3x2}{1C89DD8750BFD72E}%
\hashdef{bwd1v1v1v1p1v1x0p0x1v0x1p1x0p0x1v1x0x0p1x0x0p0x1x0v1x0x1p0x1x0p0x1x0p0x0x1v0x0x0x1v1duals1v1v1x2v1x3x2v1x2x4x3v1}{6243B44611B825F7}%
\hashdef{bwd1v1v1v1p1v1x0p0x1v0x1p1x0p1x0p0x1v0x0x0x1p0x0x1x0vduals1v1v1x2v1x2x4x3v}{9D95BD0008DC180C}%
\hashdef{bwd1v1v1v1p1v1x0p0x1v0x1p1x0p1x0p0x1v1x0x0x0p0x1x0x0p0x0x1x0p0x0x0x1v0x1x0x0p1x0x0x0p1x0x0x0p0x1x0x0p0x0x1x0p0x0x0x1v0x0x0x1x0x0p0x1x0x0x0x0p0x0x0x0x0x1p0x0x0x0x1x0v1x0x0x0p0x1x0x0p1x0x0x0p0x1x0x0p0x0x1x0p0x1x0x0p1x0x0x0p0x0x0x1p0x0x1x0p0x0x0x1v0x0x0x0x1x0x0x1x0x0v1duals1v1v1x2v2x1x3x4v6x2x5x4x3x1v2x1x4x3x6x5x8x7x10x9v1}{CC231A12D7303E70}%
\hashdef{bwd1v1v1v1p1v1x0p0x1v0x1p1x0p1x1v0x1x0vduals1v1v1x2v1x2x3v}{71F5BCB269BFC724}%
\hashdef{bwd1v1v1v1p1v1x0p0x1v0x1p1x0p1x1v1x0x0vduals1v1v1x2v1x2x3v}{C7766D3E2CB7E68E}%
\hashdef{bwd1v1v1v1p1v1x0p0x1v1x0p0x1duals1v1v1x2v2x1}{25781EDEA61320E6}%
\hashdef{bwd1v1v1v1p1v1x0p0x1v1x0p0x1p0x1p0x1v1x0x0x0p1x0x0x0p0x0x0x1v1x0x0p0x1x0p0x0x1p0x1x0p0x0x1v0x0x1x0x0p0x0x0x1x0v1x0p0x1p0x1duals1v1v1x2v2x1x3x4v1x2x3x5x4v3x2x1}{1FE5F5D9B84E3521}%
\hashdef{bwd1v1v1v1p1v1x0p0x1v1x0p0x1p0x1p0x1v1x0x0x0p1x0x0x0p0x0x1x0v1x0x0p0x1x1duals1v1v1x2v4x2x3x1v1x2}{2901033CA85E5FDE}%
\hashdef{bwd1v1v1v1p1v1x0p0x1v1x0p0x1p0x1p0x1v1x0x0x0p1x0x0x0p0x1x0x0v1x0x0p0x1x0p0x0x1p0x1x0p0x0x1v0x0x0x1x0v1duals1v1v1x2v4x2x3x1v1x2x3x5x4v1}{CD5E7F58247377E6}%
\hashdef{bwd1v1v1v1p1v1x0p0x1v1x0p0x1p0x1p0x1v1x0x0x0p1x0x0x0p0x1x0x0v1x0x0p0x1x0p0x1x0p0x0x1duals1v1v1x2v3x2x1x4v1x2x4x3}{F47CE7B04FBA46E9}%
\hashdef{bwd1v1v1v1p1v1x0p0x1v1x0p0x1p0x1p0x1v1x0x0x0p1x0x0x0v1x0p0x1duals1v1v1x2v3x2x1x4v1x2}{4EF36C69FD042ABD}%
\hashdef{bwd1v1v1v1p1v1x0p0x1v1x0p0x1p0x1p0x1v1x0x0x0p1x0x0x0v1x0p0x1duals1v1v1x2v4x3x2x1v1x2}{6310A2CDA4EE6602}%
\hashdef{bwd1v1v1v1p1v1x0p0x1v1x0p0x1p0x1p1x0v1x0x0x0p0x1x0x0p0x0x1x0p0x0x0x1v0x0x1x0p0x0x0x1p0x0x1x0p1x0x0x0p0x0x0x1p0x1x0x0v0x0x0x1x0x0p0x1x0x0x0x0p1x0x0x0x0x0p0x0x0x0x0x1v0x0x0x1p1x0x0x0p0x0x1x0p0x1x0x0p0x1x0x0p1x0x0x0p0x0x0x1p0x0x1x0duals1v1v1x2v1x2x4x3v1x2x4x3x6x5v2x1x4x3x6x5x8x7}{6E1E439C357D0C91}%
\hashdef{bwd1v1v1v1p1v1x0p0x1v1x0p0x1p0x1p1x0v1x0x0x0p0x1x0x0p0x0x1x0p0x0x0x1v0x0x1x0p0x0x0x1p0x0x1x0p1x0x0x0p0x0x0x1p0x1x0x0v0x0x0x1x0x0p0x1x0x0x0x0p1x0x0x0x0x0p0x0x0x0x0x1v0x0x0x1p1x0x0x0p0x0x1x0p0x1x0x0p0x1x0x0p1x0x0x0p0x0x0x1p0x0x1x0vduals1v1v1x2v1x2x4x3v1x2x4x3x6x5v2x1x4x3x6x5x8x7}{FED2CA078FA319DB}%
\hashdef{bwd1v1v1v1p1v1x0p0x1v1x0p0x1p0x1p1x0v1x0x0x0p0x1x0x0p0x0x1x0p0x0x0x1v1x0x0x0p0x0x1x0p0x0x0x1p0x0x1x0p1x0x0x0p0x0x0x1p0x1x0x0vduals1v1v1x2v1x2x4x3v1x2x3x5x4x7x6}{8A056844B167BE6E}%
\hashdef{bwd1v1v1v1p1v1x0p0x1v1x0p0x1p0x1v1x0x0v1duals1v1v1x2v3x2x1v1}{AB1FBD3AA81D6976}%
\hashdef{bwd1v1v1v1p1v1x0p0x1v1x0p0x1p1x0p0x1v0x1x0x0p1x0x0x0p0x0x1x0p0x0x0x1v0x1x0x0p1x0x0x0p0x1x0x0p0x0x0x1p0x0x1x0p1x0x0x0v0x0x1x0x0x0p0x0x0x1x0x0p0x0x0x0x0x1p0x0x0x0x1x0v0x0x0x1p1x0x0x0p0x1x0x0p1x0x1x0p0x0x0x1p0x1x0x0p0x0x1x0v0x0x0x0x1x0x0p0x0x0x0x1x0x0p0x0x0x0x0x1x0p0x0x0x0x0x1x0v1x0x0x0p0x1x0x0p0x0x1x0p0x0x0x1duals1v1v1x2v2x1x3x4v4x5x3x1x2x6v3x5x1x4x2x7x6v2x1x4x3}{CA0A5FA7EBA865FA}%
\hashdef{bwd1v1v1v1p1v1x0p0x1v1x0p0x1p1x0p0x1v0x1x0x0p1x0x0x0p0x0x1x0p0x0x0x1v0x1x0x0p1x0x0x0p1x0x0x0p0x0x1x0p0x0x0x1p0x0x1x0v0x0x0x0x1x0p0x0x0x0x0x1p1x0x0x0x0x0p0x1x0x0x0x0v0x1x0x0p1x1x0x0p0x0x1x1p0x0x0x1duals1v1v1x2v1x3x2x4v3x2x1x5x4x6v4x3x2x1}{543BD6BC3DB93F04}%
\hashdef{bwd1v1v1v1p1v1x0p0x1v1x0p0x1p1x0p0x1v0x1x0x0p1x0x0x0p0x0x1x0p0x0x0x1v1x0x0x0p1x0x1x0p0x1x0x0p0x1x0x1duals1v1v1x2v2x1x3x4v1x2x3x4}{AF09B0127059B7C5}%
\hashdef{bwd1v1v1v1p1v1x0p0x1v1x0p0x1p1x0p1x0p0x1p0x1duals1v1v1x2v1x3x2x5x4x6}{AB581B4A0143954C}%
\hashdef{bwd1v1v1v1p1v1x0p0x1v1x0p0x1p1x0v1x0x0p0x1x0vduals1v1v1x2v1x3x2v}{4D244FEACAA3BC56}%
\hashdef{bwd1v1v1v1p1v1x0p0x1v1x0p1x0p0x1p0x1v0x0x1x0p1x0x0x0p0x1x0x0p0x0x0x1v1x0x0x0p0x0x1x0p0x1x0x0p0x0x0x1p1x0x0x0p0x0x1x0v0x0x0x0x0x1p0x0x0x1x0x0p0x0x0x0x1x0p0x0x1x0x0x0v0x1x0x0p0x0x0x1p0x1x0x0p1x0x0x0p0x0x1x0p1x0x1x0p0x0x0x1v0x1x0x0x0x0x0p0x1x0x0x0x0x0p0x0x1x0x0x0x0p0x0x1x0x0x0x0v1x0x0x0p0x0x1x0p0x1x0x0p0x0x0x1duals1v1v1x2v1x3x2x4v3x4x1x2x5x6v7x4x5x2x3x6x1v1x2x3x4}{8E3CA58923FE2240}%
\hashdef{bwd1v1v1v1p1v1x0p0x1v1x0p1x0p0x1p0x1v0x0x1x0p1x0x0x0p0x1x0x0p0x0x0x1v1x0x0x0p0x0x1x0p0x1x0x0p0x0x1x0p0x0x0x1p1x0x0x0v0x0x0x0x1x0p0x1x0x0x0x0p0x0x1x0x0x0p0x0x0x0x0x1v0x1x0x0p1x1x0x0p0x0x1x1p0x0x0x1v1x0x0x1duals1v1v1x2v1x3x2x4v3x2x1x5x4x6v4x3x2x1}{42FB9E34DCDAE7E7}%
\hashdef{bwd1v1v1v1p1v1x0p0x1v1x0p1x0p0x1p0x1v0x0x1x0p1x0x0x0p0x1x0x0p0x0x0x1v1x0x0x0p1x0x0x0p0x1x0x0p0x0x1x0p0x0x1x0p0x0x0x1v0x0x0x0x1x0p0x1x0x0x0x0p0x0x0x0x0x1p0x0x1x0x0x0v1x1x0x0p0x0x1x0p0x1x0x0p1x0x0x0p0x0x0x1p0x0x1x0p0x0x0x1v0x0x0x0x1x0x0p0x0x0x0x1x0x0p0x1x0x0x0x0x0p0x1x0x0x0x0x0v1x0x0x0p0x1x0x0p0x0x1x0p0x0x0x1duals1v1v1x2v1x3x2x4v3x2x1x6x5x4v1x3x2x5x4x7x6v1x2x4x3}{91CD1257B791DD1B}%
\hashdef{bwd1v1v1v1p1v1x0p0x1v1x0p1x0p0x1p0x1v0x1x0x0p0x0x1x0p0x0x0x1v1x0x0p0x1x0p0x0x1p1x0x0v1x0x0x0p0x1x0x0p0x0x1x0vduals1v1v1x2v1x3x2x4v1x2x4x3v}{42F2311351B09A4B}%
\hashdef{bwd1v1v1v1p1v1x0p0x1v1x0p1x0p0x1p0x1v0x1x0x0p0x0x1x0p0x0x0x1vduals1v1v1x2v1x3x2x4v}{8F7A536E2CBC59C7}%
\hashdef{bwd1v1v1v1p1v1x0p0x1v1x0p1x0p0x1p0x1v1x0x0x0p0x1x0x0p0x0x1x0p0x0x0x1v1x0x0x0p1x0x0x0p0x1x0x0p0x0x1x0p0x0x1x0p0x0x0x1v0x1x0x0x0x0p0x0x1x0x0x0p0x0x0x0x1x0p0x0x0x0x0x1v1x0x0x0p1x0x0x0p0x1x0x0p0x1x0x0p0x0x1x0p0x0x1x0p0x0x0x1p0x0x0x1duals1v1v1x2v3x2x1x4v6x2x4x3x5x1v4x6x7x1x8x2x3x5}{1A4B9B05391D4DA1}%
\hashdef{bwd1v1v1v1p1v1x0p0x1v1x0p1x0p0x1p0x1v1x0x0x0p0x1x0x0p0x0x1x0p0x0x0x1v1x0x0x0p1x0x0x0p0x1x0x0p0x0x1x0p0x0x1x0p0x0x0x1v0x1x0x0x0x0p0x0x1x0x0x0p0x0x0x0x1x0p0x0x0x0x0x1v1x0x0x0p1x0x0x0p0x1x0x0p0x1x0x0p0x0x1x0p0x0x1x0p0x0x0x1p0x0x0x1vduals1v1v1x2v3x2x1x4v6x2x4x3x5x1v4x6x7x1x8x2x3x5}{F470043299CF4903}%
\hashdef{bwd1v1v1v1p1v1x0p0x1v1x0p1x0p0x1v1x0x0p0x0x1vduals1v1v1x2v1x3x2v}{AD28EA863E33C897}%
\hashdef{bwd1v1v1v1p1v1x0p0x1v1x0p1x1p0x1duals1v1v1x2v1x2x3}{0D7CDA5042CFE6E4}%
\hashdef{bwd1v1v1v1p1v1x0p0x1v1x0p1x1p0x1duals1v1v1x2v3x2x1}{5568092B71463EFE}%
\hashdef{bwd1v1v1v1p1v1x0p0x1v1x0p1x1p0x1v0x0x1duals1v1v1x2v1x2x3}{E498789FD8D4BCB6}%
\hashdef{bwd1v1v1v1p1v1x0p0x1v1x0p1x1p0x1v1x0x0duals1v1v1x2v1x2x3}{F4D1FBF52AF0D075}%
\hashdef{bwd1v1v1v1p1v1x0p0x1v1x0p1x1p0x1v1x0x0vduals1v1v1x2v1x2x3v}{9B696AA91156AF41}%
\hashdef{bwd1v1v1v1p1v1x0p0x1v1x0p1x1v1x0vduals1v1v1x2v1x2v}{AC02A5C4765819C2}%
\hashdef{bwd1v1v1v1p1v1x0p0x1v1x1duals1v1v1x2v1}{F32D9ADB2D06E9C0}%
\hashdef{bwd1v1v1v1p1v1x0p0x1v1x1p0x1duals1v1v1x2v1x2}{1400CBBEE5328720}%
\hashdef{bwd1v1v1v1p1v1x0p0x1v1x1p0x1v0x1duals1v1v1x2v1x2}{AA1E73B3E6E573E7}%
\hashdef{bwd1v1v1v1p1v1x0p0x1v1x1p0x1v0x1p0x1duals1v1v1x2v1x2}{00645CAFF670D7E6}%
\hashdef{bwd1v1v1v1p1v1x0p0x1v1x1p0x1v0x1p0x1v1x0p0x1duals1v1v1x2v1x2v1x2}{1323690FC75575C5}%
\hashdef{bwd1v1v1v1p1v1x0p0x1v1x1p0x1v0x1p0x1v1x0p0x1duals1v1v1x2v1x2v2x1}{9DBA45154BF2F54B}%
\hashdef{bwd1v1v1v1p1v1x0p0x1v1x1p0x1v0x1p0x1v1x0v1duals1v1v1x2v1x2v1}{4E6DC623E6D1C3B8}%
\hashdef{bwd1v1v1v1p1v1x0p0x1v1x1p0x1v0x1p0x1v1x0v1p1duals1v1v1x2v1x2v1}{69148CB42BA9C676}%
\hashdef{bwd1v1v1v1p1v1x0p0x1v1x1p0x1v0x1v1duals1v1v1x2v1x2v1}{8D4DF37334AC95F7}%
\hashdef{bwd1v1v1v1p1v1x0p0x1v1x1p0x1v0x1vduals1v1v1x2v1x2v}{53B3AC1E37A5A66D}%
\hashdef{bwd1v1v1v1p1v1x0p0x1v1x1p0x1v1x0v1duals1v1v1x2v1x2v1}{C402DA852B3308A0}%
\hashdef{bwd1v1v1v1p1v1x0p0x1v1x1v1v1duals1v1v1x2v1v1}{AED81DD6493E51BA}%
\hashdef{bwd1v1v1v1p1v1x0p0x1v1x1v1v1p1duals1v1v1x2v1v1x2}{27DE304D71B89583}%
\hashdef{bwd1v1v1v1p1v1x0p0x1v1x1v1v1p1duals1v1v1x2v1v2x1}{D200C8FFEDF72A12}%
\hashdef{bwd1v1v1v1p1v1x0p0x1v1x1v1v1p1v1x0duals1v1v1x2v1v1x2}{CD44274ACA442B93}%
\hashdef{bwd1v1v1v1p1v1x0p0x1v1x1v1v1p1v1x0vduals1v1v1x2v1v1x2v}{C291A6202AA45128}%
\hashdef{bwd1v1v1v1p1v1x0p0x1vduals1v1v1x2v}{C56EA1B65147610F}%
\hashdef{bwd1v1v1v1p1v1x0p1x0duals1v1v1x2}{E26689FBED470E90}%
\hashdef{bwd1v1v1v1p1v1x0p1x0p0x1v0x0x1p0x0x1p1x0x0vduals1v1v1x2v1x3x2}{9497E183336AC573}%
\hashdef{bwd1v1v1v1p1v1x0p1x0p0x1v0x0x1p1x0x0p0x1x0p0x0x1p0x0x1v0x1x0x0x0p0x0x1x0x0v1x0p1x0p0x1p0x1duals1v1v1x2v3x5x1x4x2v4x2x3x1}{48B981BE1E2C12AF}%
\hashdef{bwd1v1v1v1p1v1x0p1x0p0x1v0x1x0p0x0x1p0x0x1p1x0x0v1x0x0x0p0x0x1x0p0x0x0x1p1x0x0x0v0x0x1x0p0x1x0x0p1x0x0x0vduals1v1v1x2v1x2x4x3v1x3x2}{D0525BA6F678E630}%
\hashdef{bwd1v1v1v1p1v1x0p1x0p0x1v1x0x0p0x0x1p0x0x1p0x1x0p0x0x1v1x0x0x0x0p0x0x0x1x0v1x1duals1v1v1x2v3x2x1x5x4v1}{1A08363F71AA6194}%
\hashdef{bwd1v1v1v1p1v1x0p1x0p0x1v1x0x0p0x0x1p0x1x0p0x0x1duals1v1v1x2v2x1x4x3}{65A75F44A1BA3434}%
\hashdef{bwd1v1v1v1p1v1x0p1x0p0x1v1x0x0p0x0x1p0x1x0p0x0x1v0x0x0x1p0x0x1x0p1x0x0x0vduals1v1v1x2v1x2x4x3v}{20838CC741CD2002}%
\hashdef{bwd1v1v1v1p1v1x0p1x0p0x1v1x0x0p0x0x1p0x1x0p0x0x1v1x0x1x0p0x0x0x1p0x0x0x1duals1v1v1x2v4x2x3x1}{9CFA35AA7846855D}%
\hashdef{bwd1v1v1v1p1v1x0p1x0p0x1v1x0x0p0x0x1p1x0x0v0x1x0p0x1x0p0x0x1p0x0x1v1x0x0x0p1x0x0x0p0x0x1x0v0x1x0duals1v1v1x2v2x1x3v1x3x2}{34D69185C5CA6E3C}%
\hashdef{bwd1v1v1v1p1v1x0p1x0p0x1v1x0x0p0x0x1p1x0x0v0x1x0p0x1x0p0x0x1v1x0x0duals1v1v1x2v2x1x3v1}{CE720A02CC6A0FA1}%
\hashdef{bwd1v1v1v1p1v1x0p1x0p0x1v1x0x0p0x0x1v0x1duals1v1v1x2v2x1}{5826C0EBCDC1B2B9}%
\hashdef{bwd1v1v1v1p1v1x0p1x0p0x1v1x0x0p0x1x0p0x0x1p0x0x1v1x0x0x0p1x0x0x0p0x0x1x0p0x1x0x0vduals1v1v1x2v1x3x2x4v}{4D1E9C6B9D772C01}%
\hashdef{bwd1v1v1v1p1v1x0p1x0p0x1v1x0x0p0x1x0p1x0x0p0x0x1v0x1x0x1duals1v1v1x2v1x2x4x3}{51F711548E4FBEC9}%
\hashdef{bwd1v1v1v1p1v1x0p1x0p0x1v1x0x0p1x0x0p0x0x1p0x1x0p0x0x1v0x0x1x0x0p1x0x0x0x0v1x0p1x0p0x1duals1v1v1x2v1x3x2x5x4v1x3x2}{DEC6F70C1A00BFCE}%
\hashdef{bwd1v1v1v1p1v1x0p1x0p0x1v1x0x0p1x0x0p0x0x1p0x1x0p0x0x1v0x0x1x0x0v1duals1v1v1x2v1x3x2x5x4v1}{EFE6C0AEB8462ACD}%
\hashdef{bwd1v1v1v1p1v1x0p1x0p0x1v1x0x0p1x0x0p0x0x1v1x0x1duals1v1v1x2v1x3x2}{0CF723E5FCBA2B46}%
\hashdef{bwd1v1v1v1p1v1x0p1x0p0x1v1x0x1p0x1x0duals1v1v1x2v1x2}{8A684E628315AAE6}%
\hashdef{bwd1v1v1v1p1v1x0p1x0p1x0duals1v1v1x2}{F7ECC43417683405}%
\hashdef{bwd1v1v1v1p1v1x0p1x0p1x0v0x0x1v1p1duals1v1v1x2v1}{7EF35C71535779F6}%
\hashdef{bwd1v1v1v1p1v1x0p1x0p1x0v1x0x0v1duals1v1v1x2v1}{64F46BE038B1C68F}%
\hashdef{bwd1v1v1v1p1v1x0p1x0p1x0v1x0x0v1p1duals1v1v1x2v1}{07920BB6AB5058F6}%
\hashdef{bwd1v1v1v1p1v1x0p1x0v0x1p0x1p1x0v1x0x0vduals1v1v1x2v1x2x3v}{45145EF7FC1BE70C}%
\hashdef{bwd1v1v1v1p1v1x0p1x0v0x1p0x1v1x0p0x1duals1v1v1x2v1x2}{FBF2CA1EBA81F82D}%
\hashdef{bwd1v1v1v1p1v1x0p1x0v0x1p0x1v1x0p0x1duals1v1v1x2v2x1}{0AA6722D9F44BDF3}%
\hashdef{bwd1v1v1v1p1v1x0p1x0v0x1p0x1v1x0p0x1p0x1duals1v1v1x2v1x2}{9F6BB0E8725AC36D}%
\hashdef{bwd1v1v1v1p1v1x0p1x0v0x1p0x1v1x0p0x1p0x1v0x0x1duals1v1v1x2v1x2v1}{D140DE49F480E9DE}%
\hashdef{bwd1v1v1v1p1v1x0p1x0v0x1p0x1v1x0p0x1p0x1v0x0x1v1duals1v1v1x2v1x2v1}{9A2EA7A42D994D21}%
\hashdef{bwd1v1v1v1p1v1x0p1x0v0x1p0x1v1x0p0x1p0x1v0x0x1v1p1duals1v1v1x2v1x2v1}{12C6C01A9827F394}%
\hashdef{bwd1v1v1v1p1v1x0p1x0v0x1v1duals1v1v1x2v1}{D2E199C236550C3B}%
\hashdef{bwd1v1v1v1p1v1x0p1x0v0x1v1p1p1duals1v1v1x2v1}{B57F3DB1977EC9DF}%
\hashdef{bwd1v1v1v1p1v1x0p1x0v0x1v1p1p1v1x0x0duals1v1v1x2v1v1}{E308EBE25CDB6002}%
\hashdef{bwd1v1v1v1p1v1x0p1x0v0x1v1p1p1v1x0x0p0x1x0v0x1duals1v1v1x2v1v1x2}{0D49FC109DB8F3E3}%
\hashdef{bwd1v1v1v1p1v1x0p1x0v0x1v1p1p1v1x0x0v1duals1v1v1x2v1v1}{F3A2716CB3D97A38}%
\hashdef{bwd1v1v1v1p1v1x0p1x0v0x1v1p1p1v1x0x0v1p1duals1v1v1x2v1v1}{41AF4DABB9D37100}%
\hashdef{bwd1v1v1v1p1v1x0p1x0v0x1v1p1p1v1x0x0v1p1v1x0duals1v1v1x2v1v1v1}{877D5FC7A0C634BF}%
\hashdef{bwd1v1v1v1p1v1x0p1x0v0x1v1p1p1v1x0x0v1p1v1x0v1duals1v1v1x2v1v1v1}{D2BC16E4B835B8DB}%
\hashdef{bwd1v1v1v1p1v1x0p1x0v0x1v1p1p1v1x0x0v1p1v1x0v1p1duals1v1v1x2v1v1v1}{05B70A60D5917491}%
\hashdef{bwd1v1v1v1p1v1x0p1x0v1x0p0x1duals1v1v1x2v1x2}{4B2F4D160292120F}%
\hashdef{bwd1v1v1v1p1v1x0p1x0v1x0p0x1p0x1duals1v1v1x2v1x2x3}{E10E5DBB585F53D4}%
\hashdef{bwd1v1v1v1p1v1x0p1x0v1x0p0x1p0x1duals1v1v1x2v1x3x2}{4B192DD98895C014}%
\hashdef{bwd1v1v1v1p1v1x0p1x0v1x0p0x1p0x1v0x1x0duals1v1v1x2v1x2x3}{7BC28FB51B0FBB85}%
\hashdef{bwd1v1v1v1p1v1x0p1x0v1x0p0x1p0x1v0x1x0p0x0x1v1x0p0x1duals1v1v1x2v1x2x3v2x1}{4A67780F80ACCA48}%
\hashdef{bwd1v1v1v1p1v1x0p1x0v1x0p0x1p0x1v0x1x0p0x0x1v1x0p0x1duals1v1v1x2v1x3x2v1x2}{3ED308753F0571E5}%
\hashdef{bwd1v1v1v1p1v1x0p1x0v1x0p0x1p0x1v0x1x0p0x0x1v1x0p1x0p0x1v0x0x1v1duals1v1v1x2v1x2x3v1x3x2v1}{08A026A33D39330F}%
\hashdef{bwd1v1v1v1p1v1x0p1x0v1x0p0x1p0x1v0x1x0p0x0x1v1x0p1x0p0x1v1x0x0p0x0x1v1x0p0x1p0x1duals1v1v1x2v1x2x3v1x3x2v3x2x1}{E3FD9E6703911C8F}%
\hashdef{bwd1v1v1v1p1v1x0p1x0v1x0p0x1p0x1v0x1x0p0x0x1v1x1duals1v1v1x2v1x2x3v1}{F00C11F26AA7B5FA}%
\hashdef{bwd1v1v1v1p1v1x0p1x0v1x0p0x1p0x1v0x1x0p0x0x1v1x1duals1v1v1x2v1x3x2v1}{B66A496BF40107AF}%
\hashdef{bwd1v1v1v1p1v1x0p1x0v1x0p0x1p0x1v0x1x0p0x1x0duals1v1v1x2v1x2x3}{A51876349EA7F4D1}%
\hashdef{bwd1v1v1v1p1v1x0p1x0v1x0p0x1p0x1v1x1x0duals1v1v1x2v1x2x3}{9C0E50FA37747434}%
\hashdef{bwd1v1v1v1p1v1x0p1x0v1x0p0x1p1x0p0x1v0x1x0x0vduals1v1v1x2v1x2x3x4v}{48B66AB8E72DF055}%
\hashdef{bwd1v1v1v1p1v1x0p1x0v1x0p0x1v0x1p1x0p0x1v1x0x0p0x1x0v0x1p1x0p0x1vduals1v1v1x2v1x2v2x1v}{2AA52F211FE28510}%
\hashdef{bwd1v1v1v1p1v1x0p1x0v1x0p0x1v0x1p1x0p0x1vduals1v1v1x2v1x2v}{C4191293A587ED11}%
\hashdef{bwd1v1v1v1p1v1x0p1x0v1x0p0x1v1x0p0x1p1x0p0x1v0x0x0x1p0x0x1x0v1x0p0x1p1x0p0x1v0x1x0x0p0x0x1x0duals1v1v1x2v1x2v2x1v1x2}{B8691E3193E94DC4}%
\hashdef{bwd1v1v1v1p1v1x0p1x0v1x0p0x1v1x0p0x1p1x0p0x1v0x0x0x1p0x0x1x0v1x0p0x1p1x0p0x1v0x1x0x0p0x0x1x0vduals1v1v1x2v1x2v2x1v1x2}{C9725F2E553780B0}%
\hashdef{bwd1v1v1v1p1v1x0p1x0v1x0p0x1v1x0p0x1p1x0p0x1v1x0x0x0p0x0x0x1p0x0x1x0vduals1v1v1x2v1x2v1x3x2}{49A0D66957608D49}%
\hashdef{bwd1v1v1v1p1v1x0p1x0v1x0p0x1v1x0p0x1p1x0p0x1v1x0x0x0p0x1x0x0v1x0p0x1p1x0p0x1v1x1x0x0v1duals1v1v1x2v1x2v2x1v1}{0CF97513F50B79D2}%
\hashdef{bwd1v1v1v1p1v1x0p1x0v1x0p0x1v1x0p0x1p1x0p0x1v1x0x0x0p0x1x0x0v1x0p1x0p0x1p0x1v1x0x0x0p1x0x0x0p0x0x1x0p0x0x1x0v0x1x0x0p1x0x0x0p0x0x1x0p0x0x0x1duals1v1v1x2v1x2v2x1v2x1x4x3}{860FC84683565D3A}%
\hashdef{bwd1v1v1v1p1v1x0p1x0v1x0p0x1v1x0p0x1p1x0p0x1v1x1x0x0p0x0x1x0p0x0x0x1duals1v1v1x2v1x2v1x2x3}{238AB8CDC2B1F01E}%
\hashdef{bwd1v1v1v1p1v1x0p1x0v1x0p0x1v1x0p0x1vduals1v1v1x2v1x2v}{8374DD24A76566AC}%
\hashdef{bwd1v1v1v1p1v1x0p1x0v1x0p0x1v1x0p1x0p0x1p0x1v1x0x0x0p0x0x1x0v1x0p0x1p1x0p0x1v1x0x0x0p1x0x0x0p0x1x0x0p0x1x0x0v1x0x0x0p0x1x0x0p0x0x1x0p0x0x0x1duals1v1v1x2v1x2v2x1v1x2x4x3}{B3BD2B527B389090}%
\hashdef{bwd1v1v1v1p1v1x0p1x0v1x0p0x1v1x0p1x0p0x1p0x1v1x0x0x0p0x0x1x0v1x0p1x0p0x1p0x1v0x1x0x1duals1v1v1x2v1x2v2x1v1}{E3AC1DA985BA9F47}%
\hashdef{bwd1v1v1v1p1v1x0p1x0v1x0p0x1v1x0p1x0p0x1p0x1v1x0x0x0p0x0x1x0v1x0p1x0p0x1p0x1v0x1x0x1v1v1duals1v1v1x2v1x2v2x1v1v1}{99A2A52DDE597394}%
\hashdef{bwd1v1v1v1p1v1x0p1x0v1x0p0x1v1x0p1x0p0x1p0x1v1x0x0x0p0x0x1x0v1x0p1x0p0x1p0x1v1x0x0x0p1x0x0x0p0x0x1x0p0x0x1x0v1x0x0x0p0x1x0x0p0x0x1x0p0x0x0x1duals1v1v1x2v1x2v2x1v1x2x3x4}{D1B8FE18D2457567}%
\hashdef{bwd1v1v1v1p1v1x0p1x0v1x0p0x1v1x1duals1v1v1x2v1x2}{E240B9F9C7214ACB}%
\hashdef{bwd1v1v1v1p1v1x0p1x0v1x0p0x1v1x1v1duals1v1v1x2v1x2v1}{7C8CB016D0D84AF9}%
\hashdef{bwd1v1v1v1p1v1x0p1x0v1x0p0x1v1x1v1v1duals1v1v1x2v1x2v1}{8BB24C2CA6088C35}%
\hashdef{bwd1v1v1v1p1v1x0p1x0v1x0p0x1v1x1v1v1vduals1v1v1x2v1x2v1v}{62F2710C54A026B9}%
\hashdef{bwd1v1v1v1p1v1x0p1x0v1x0p1x0p0x1p0x1duals1v1v1x2v1x2x3x4}{563AA496A66598BC}%
\hashdef{bwd1v1v1v1p1v1x0p1x0v1x0p1x0p0x1p0x1duals1v1v1x2v1x2x4x3}{8DB6DD3051B2B4C8}%
\hashdef{bwd1v1v1v1p1v1x0p1x0v1x0p1x0p0x1p0x1duals1v1v1x2v2x1x4x3}{5D95300A5094CC0E}%
\hashdef{bwd1v1v1v1p1v1x0p1x0v1x0p1x0p0x1p0x1v0x0x1x0duals1v1v1x2v1x2x3x4}{18A1854AC2071CB2}%
\hashdef{bwd1v1v1v1p1v1x0p1x0v1x0p1x0p0x1p0x1v1x0x0x0duals1v1v1x2v1x2x4x3}{AECB9BC4C25D327D}%
\hashdef{bwd1v1v1v1p1v1x0p1x0v1x0p1x0p0x1p0x1v1x0x0x0vduals1v1v1x2v1x2x3x4v}{CA09169F0CD78CDD}%
\hashdef{bwd1v1v1v1p1v1x0p1x0v1x0p1x0p0x1p0x1v1x0x0x0vduals1v1v1x2v1x2x4x3v}{9E500CC62FF5551D}%
\hashdef{bwd1v1v1v1p1v1x0p1x0v1x0p1x0p0x1v1x0x0vduals1v1v1x2v1x2x3v}{E4EC31C9ACBB84F6}%
\hashdef{bwd1v1v1v1p1v1x0p1x0v1x0v1p1vduals1v1v1x2v1v}{2E736C4277C124A1}%
\hashdef{bwd1v1v1v1p1v1x0p1x0v1x1duals1v1v1x2v1}{23B7DB947F5B62F1}%
\hashdef{bwd1v1v1v1p1v1x0p1x0vduals1v1v1x2v}{440F91D75E2B1C8E}%
\hashdef{bwd1v1v1v1p1v1x1duals1v1v1x2}{826D06E8748D81B2}%
\hashdef{bwd1v1v1v1p1v1x1duals1v1v2x1}{568118641CC1751C}%
\hashdef{bwd1v1v1v1p1v1x1p0x1duals1v1v1x2}{3B5506E39F490773}%
\hashdef{bwd1v1v1v1p1v1x1v1duals1v1v1x2v1}{F79ECEE99FF00EDA}%
\hashdef{bwd1v1v1v1p1v1x1v1duals1v1v2x1v1}{4B95A14AE0E91361}%
\hashdef{bwd1v1v1v1v1p1p1v1x0x0p0x1x0p0x0x1v0x1x0p0x0x1v1x0p0x1duals1v1v1v1x3x2v2x1}{3E9244FAEB33506B}%
\hashdef{bwd1v1v1v1v1p1p1v1x0x0p0x1x0p0x0x1v1x0x0p0x1x0v1x0p0x1duals1v1v1v1x2x3v1x2}{E380D74CB4F0C4DF}%
\hashdef{bwd1v1v1v1v1p1p1v1x0x0p0x1x0p0x0x1v1x0x0p0x1x0v1x0p0x1duals1v1v1v2x1x3v2x1}{ACE47192690819DF}%
\hashdef{bwd1v1v1v1v1p1v1x1p0x1duals1v1v1v1x2}{9B8BF3346973333B}%
\hashdef{bwd1v1v1v1v1v1p1p1v0x1x0p0x0x1v1x1duals1v1v1v1x3x2v1}{BC831A7BF6C49235}%
\hashdef{bwd1v1v1v1v1v1p1p1v0x1x0p0x1x0v1x0p0x1duals1v1v1v1x2x3v1x2}{284B15A6B0A5F4FC}%
\hashdef{bwd1v1v1v1v1v1p1p1v1x0x0p0x0x1v1x1duals1v1v1v1x2x3v1}{BC9F2BEAECE32DCD}%
\hashdef{bwd1v1v1v1v1v1p1p1v1x0x0p1x0x0v1x0p0x1duals1v1v1v1x3x2v2x1}{2936A42E4127100A}%
\hashdef{bwd1v1v1v1v1v1p1v1x0p0x1v1x0p0x1p0x1p1x0v1x0x0x0p0x1x0x0p0x0x1x0p0x0x0x1v1x0x0x0p0x0x1x0p0x0x0x1p0x0x1x0p1x0x0x0p0x0x0x1p0x1x0x0vduals1v1v1v1x2v1x2x4x3v1x2x3x5x4x7x6}{F2CA133DC5DA05C6}%
\hashdef{bwd1v1v1v1v1v1p1v1x0p0x1v1x0p0x1p0x1v1x0x0v1duals1v1v1v1x2v2x1x3v1}{1D8CDC2F9BD90A57}%
\hashdef{bwd1v1v1v1v1v1p1v1x0p0x1v1x0p0x1p0x1v1x0x0v1duals1v1v1v1x2v3x2x1v1}{827D15F829A4D1B0}%
\hashdef{bwd1v1v1v1v1v1p1v1x0p0x1v1x0p1x1p0x1v1x0x0v1v1v1duals1v1v1v1x2v1x2x3v1v1}{88B6A430B05C90BF}%
\hashdef{bwd1v1v1v1v1v1p1v1x0p1x0v0x1v1duals1v1v1v1x2v1}{91C06EDEEE95CDAB}%
\hashdef{bwd1v1v1v1v1v1p1v1x0p1x0v1x0p0x1v1x0p0x1p1x0p0x1v1x0x0x0p0x0x0x1p0x0x1x0vduals1v1v1v1x2v1x2v1x3x2}{AAA37C81D9780E4C}%
\hashdef{bwd1v1v1v1v1v1v1v1p1v1x0p0x1v1x0p0x1duals1v1v1v1v1x2v2x1}{C53E32050C862F37}%
\hashdef{bwd1v1v1v1v1v1v1v1p1v1x0p1x0duals1v1v1v1v1x2}{CA1952BED6C8EB18}%
\hashdef{gbg1v1p1p1p1}{80BFF65B12198C8D}%
\hashdef{gbg1v1p1v1x0p0x1p0x1v}{FCF2F070580832AF}%
\hashdef{gbg1v1p1v2x0}{50223609DAC6F143}%
\hashdef{gbg1v1p2}{01E911332D5B4DEE}%
\hashdef{gbg1v1v1p1p1}{DB74ACEC2CF87833}%
\hashdef{gbg1v1v1p1v1x0p0x1v1x1}{5DA6F094BE5A80F5}%
\hashdef{gbg1v1v1p1v1x0p1x0}{912D1208939B9C14}%
\hashdef{gbg1v1v1p1v1x0v1p1}{695321C09ED266D3}%
\hashdef{gbg1v1v1p1v1x0v1v1p1}{26A357657BC91915}%
\hashdef{gbg1v1v1p1v1x0v1v1v1p1}{AEFE76960DCCD39B}%
\hashdef{gbg1v1v1p1v1x0v1v1v1v1}{905AC42C5B77D39A}%
\hashdef{gbg1v1v1v2}{54E725D6A7E31A4B}%
\hashdef{gbg1v1v2}{AFEEC2F24FDC781F}%
\hashdef{gbg1v2}{B824279AC2FB09A4}%
\hashdef{gbg1v2v1}{C732793B1423A836}%
\hashdef{gbg2}{19AF541964E759D3}%
\hashdef{gbg2v1}{CB43ACA92C0B51FF}%
\hashdef{gbg2v1v2}{CD1FD9AB242209FB}%
\hashdef{gbg2v2}{6C955CC3FC9CBDE3}%
\hashdef{gbg3}{983473220055DA39}%

\newcommand{\bigraph}[1]{{\hspace{-3pt}\begin{array}{c}%
  \raisebox{-2.5pt}{\includegraphics[height=6mm]{\pathtographs \hashlookup{#1}}}%
\end{array}\hspace{-3pt}}}
\newcommand{\smallbigraph}[1]{{\hspace{-3pt}\begin{array}{c}%
  \raisebox{-2.5pt}{\includegraphics[height=3mm]{\pathtographs \hashlookup{#1}}}%
\end{array}\hspace{-3pt}}}

\newcommand{\eset}{\emptyset}

\DeclareMathOperator{\Hom}{Hom}

\newcommand{\Aut}{\operatorname{Aut}}
\newcommand{\Irr}{\operatorname{Irr}}

\newcommand{\tr}[1]{\text{tr}(#1)}

\makeatletter

\def\@testdef #1#2#3{%
  \def\reserved@a{#3}\expandafter \ifx \csname #1@#2\endcsname
 \reserved@a  \else
\typeout{^^Jlabel #2 changed:^^J%
\meaning\reserved@a^^J%
\expandafter\meaning\csname #1@#2\endcsname^^J}%
\@tempswatrue \fi}

\usepackage{titling}
\title{The classification of subfactors with index at most $5 \frac{1}{4}$}
\author{Narjess Afzaly, Kim Morrison, and David Penneys}
\date{\emph{In memory of Uffe Haagerup}}
\begin{document}
\maketitle

\begin{abstract}
Subfactor standard invariants encode quantum symmetries.
The small index subfactor classification program has been a rich source of interesting quantum symmetries.
We give the complete classification of subfactor standard invariants to index $5\frac{1}{4}$, which includes $3+\sqrt{5}$, the first interesting composite index.
\end{abstract}

\section{Introduction}

The classification of small index subfactors is an essential part of the
search for exotic quantum symmetries. 
A quantum symmetry is a non-commutative analogue of the representation
category of a finite group. There is no single best axiomatization:
choices include standard invariants of finite
index subfactors \cite{MR1334479,math.QA/9909027} or fusion categories \cite{MR2183279}. We focus on standard invariants here.

Topological field theories and topological phases of matter have revolutionized our understanding of symmetry in physics: these systems do not have
a group of symmetries in the classical sense, but rather possess
quantum symmetries, described by a higher categorical structure. \cite{MR2443722,1410.4540}

\emph{What, then, do quantum symmetries look like?} The basic examples come either
from finite group theory (possibly with cohomological data) or from quantum
enveloping algebras at roots of unity. Many are also realized from conformal field theories.
While there
are a number of constructions producing new quantum symmetries from old, we
are far from having a good structure theory. We are still at the
phenomenological phase of studying quantum symmetries, and understanding the
range of examples is an essential problem.

We now have several instances of quantum symmetries that do not come from
the basic examples, even allowing these constructions.  Indeed, the strangest 
and least understood of all known quantum symmetries were discovered in exhaustive classifications of subfactors at small index \cite{MR1317352}.

A critical next step in our understanding of quantum symmetries will be developing structure theory. (See, for example, Question \ref{q:extension} below.)
This article lays essential groundwork for this, by completing the classification of small index subfactors beyond the first interesting composite index (that is, a product of smaller allowed indices), 
namely $3+\sqrt{5}$. Initially, it was expected that the classification at index $3+\sqrt{5}$ would be very complicated, with a profusion of examples built by composites and other constructions from basic examples at smaller indices. 
These composite planar algebras were classified by \cite{MR3345186}, contradicting that expectation.

There are relatively few subfactor standard invariants in the range we study, suggesting that the as yet unknown structure theory of quantum symmetries will strongly constrain possible examples. 

\begin{thm*}
There are exactly 15 subfactor standard invariants with index in $(5,5\frac{1}{4}]$, besides the Temperley-Lieb-Jones $A_\infty$ and the reducible $A_\infty^{(1)}$ standard invariants at every index.
(See Theorem \ref{thm:Main} below.)
\end{thm*}

\newpage

\setcounter{tocdepth}{2}
\hypersetup{
   linkcolor={black}
}
\tableofcontents
\hypersetup{
   linkcolor={purple}
}

\newpage

\subsection{Quantum symmetries and the Galois correspondence}
\emph{Subfactors are universal hosts for quantum symmetries.} In this
section we say just enough to explain our slogan, before explaining precisely what a subfactor is in the next.

The first astonishing fact (of several!) about subfactors is Jones' index rigidity theorem \cite{MR0696688}.
A subfactor $A\subset B$ has an index measuring the relative sizes of the factors, and this index is
\emph{quantized}:
$$
[B : A] \in \left\{ 4 \cos^2(\pi/n) \middle| n \geq 3 \right\} \cup [4,\infty].
$$

The Jones index has been understood from the beginning as a non-commutative analogue of the index of a field extension. 
The first sign that this analogy is important comes from a subfactor version of the Galois correspondence for field extensions. Given a finite group $G$, there is an essentially unique action of $G$ on $R$, the hyperfinite \twoone-factor \cite{MR587749}. 
We obtain a subfactor $R \subset R \rtimes G$, whose intermediate subfactors $R \subset P \subset R \rtimes G$ are all of the
form $P = R \rtimes H$ for some subgroup $H \subset G$ \cite{MR0123925}.

This analogy runs even deeper. 
The \emph{standard invariant} of a subfactor is a collection of finite dimensional vector spaces equipped with algebraic
operations. 
It can equivalently be axiomatized via Popa's $\lambda$-lattices \cite{MR1334479}, Ocneanu's paragroups \cite{MR996454}, or Jones' planar algebras \cite{math.QA/9909027}. 
The essential feature is that this standard invariant plays the same
role as the Galois group of a field extension --- it describes the quantum
symmetries of the subfactor.

We can now justify our initial slogan.
The `even half' of a standard invariant is a rigid C$^*$-tensor category, and when the subfactor is `finite depth', this is a fusion
category.
Conversely, given a unitary fusion category $\cC$, there is a hyperfinite subfactor which
we should think of as $R \subset R \rtimes \cC$, whose even half is exactly $\cC$,
and moreover this subfactor is essentially unique \cite{MR1055708}, \cite[Theorem 4.1]{MR3028581}. 

\section{Background}
In this section we introduce subfactors, their standard invariants, and their role as quantum symmetries. Readers for
whom this is familiar can skip ahead to Section \ref{sec:Summary} for the overview of our new results.

\subsection{Subfactors and their standard invariants}
A factor is a von Neumann algebra with trivial centre, and a subfactor is a unital 
inclusion of factors.  
Factors are classified into three types; we will be interested throughout in \twoone-subfactors, which are infinite dimensional and have a tracial state. 
(Most of what we describe below extends to type {\rm III} subfactors, cf. \cite{MR829381,MR1269266,MR1339767}.)

To prove the index restriction, Jones introduced the \emph{basic construction} \cite{MR0696688}.
After taking the GNS completion $L^2(B)$ of $B$ with respect to the tracial state, we have the orthogonal projection $e_A$ with range $L^2(A)$.
The basic construction applied to $A \subset B$ is the new factor $\langle B, e_A\rangle$, containing $B$ as a subfactor.
When $[B : A] < \infty$, we get a new \twoone-subfactor $B \subset \langle B, e_A \rangle$ with the same index.

Iterating the basic construction for $A=A_0\subset  A_1=B$, we obtain the \emph{Jones tower}
$$ 
A_0 \subset A_1 \overset{e_1}{\subset} A_2 \overset{e_2}{\subset} A_3 \overset{e_3}{\subset}
\cdots.
$$
The first sign that something genuinely interesting is happening is that the Jones projections $e_i$ satisfy the \emph {Temperley-Lieb-Jones relations} \cite{MR0696688}:
\begin{enumerate}[label=(\arabic*)]
\item
$e_i=e_i^2=e_i^*$,
\item
$e_i e_j = e_j e_i$ when $|i-j|>1$, and
\item
$e_i e_{i\pm1} e_i = [A_1:A_0]^{-1} e_i$.
\end{enumerate}

From the Jones tower, we extract two towers of finite dimensional centralizer algebras \cite{MR999799}:
$$
\xymatrix@C=5pt@R=2pt{
A_0'\cap A_0  &\subset & A_0'\cap A_1 &\subset & A_0'\cap A_2 &\subset &A_0'\cap A_3   &\subset & \cdots
\\
&&\cup&&\cup&&\cup&&
\\
&&
A_1'\cap A_1  &\subset &  A_1'\cap A_2 &\subset &A_1'\cap A_3   &\subset & \cdots
}
$$
There's much more structure present here than just the $*$-algebra structures and their
inclusions --- in particular there are also the restrictions of the conditional expectations $E_i : A_i\to A_{i-1}$ 
(obtained by restricting the Jones projection $e_i$ on $L^2(A_i)$ to $A_i$)
and the Jones projections $e_j$, all interacting according to intricate algebraic relations.

Our preferred way to axiomatize all this data is as a \emph{subfactor planar algebra}, which we briefly define here. (Recall the alternatives are $\lambda$-lattices \cite{MR1334479} or paragroups \cite{MR996454}.)
These are the main objects of study of this article, and they correspond under Theorem \ref{ref:PopaCorrespondence} below to subfactors.
More detail, and a summary of the important techniques for analyzing a subfactor planar algebra, can be
found in the survey article \cite{MR3166042}.


A \emph{shaded planar algebra} \cite{math.QA/9909027} is a collection of complex vector spaces $\cP_\bullet = (\cP_
{n,\pm})_{n\geq 0}$, together with an action of the operad of shaded planar tangles. A shaded planar tangle consists of a disc with several sub-discs removed, a collection of non-intersecting strings in the complementary region (whose endpoints lie on the boundary circles), with an alternating shading of the regions between the strings, and a marked interval on each boundary circle.
For the careful definition of a shaded planar tangle, see \cite{math.QA/9909027,MR2679382}; we settle for giving an illustrative example below.

Suppose we have a shaded planar tangle $T$. We number the output circle $0$, and number the input circles $1$
 through $r$. Suppose there are $2k_i$ points on the $i$-th circle, and the marked interval of the $i$-th circle is
 either unshaded or shaded according to a sign $\pm_i$ respectively.
Then the structure of a planar algebra assigns to this tangle $T$ a multilinear map
$$\cP(T) : \cP_{k_1,\pm_1}\otimes \cdots \otimes \cP_{k_r,\pm_r}\to \cP_{k_0,\pm_0}.$$
For example,
\newcommand{\ncircle}[5]{
	\draw[thick, #1] #2 circle (#3);
	\node at #2 {#5};
	\node at ($#2+(#4:.15cm)+(#4:#3cm)$) {$\star$};
}
$$
\begin{tikzpicture}[baseline = .5cm]
	\clip (-.6,.6) circle (1.6cm);
	\filldraw[shaded] (.8,1.6) circle (.6cm);
	\filldraw[shaded] (0,.4) arc (0:90:.8cm) -- (-1.2,.8) arc (180:270:.8cm);
	\filldraw[shaded] (-100:.4cm) circle (.25cm);
	\filldraw[shaded] (-1.2,2.2)--(-1.2,1.2)--(-2.2,1.2);
	\filldraw[shaded] (-20:1cm)--(0,0)--(20:1.5cm);
	\filldraw[shaded] (-1.4,0) circle (.3cm);
	\draw[ultra thick] (-.6,.6) circle (1.6);
	\ncircle{unshaded}{(-1.2,1.2)}{.4}{235}{}
	\ncircle{unshaded}{(0,0)}{.4}{135}{}
	\node at (-1.2,-.7) {$\star$};
\end{tikzpicture}
:\cP_{2,+}\otimes \cP_{3,-} \to \cP_{3,+}.
$$

We require that the identity tangle acts as the identity map. We can glue
tangles one inside the other (we require that the distinguished intervals
marked with $\star$ match), and we require that gluing
corresponds to composition of multilinear maps. Tangles which are isotopic
must give the same linear map. (Contrary to the usual situation in quantum
algebra, isotopies may move the boundary, although as each boundary circle has
a marked interval this just means that $2\pi$ rotations act as the identity.)

Moreover, to be a \emph{subfactor planar algebra}, we require that:
\begin{itemize}
\item each $\cP_{n,\pm}$ is finite dimensional,
\item $\cP_\bullet$ is \emph{evaluable}, that is $\cP_{0,\pm}$ is 1-dimensional,
\item each $\cP_{n,\pm}$ has an involution $*$ which is compatible with reflection of tangles,
\item $\cP_\bullet$ is \emph{positive}, in the sense that the sesquilinear
form $\langle x, y\rangle = \tr{y^*x}$ on each $\cP_{n,\pm}$ is positive
definite, where
multiplication of elements in $\cP_{n,\pm}$ is stacking, and
$$
\operatorname{tr} = \begin{tikzpicture}[baseline=-.1cm]
	\ncircle{}{(.1,0)}{1}{180}{}
	\draw (-.4,.2) .. controls ++(90:.9cm) and ++(90:1.1cm) ..  (.8,0) .. controls ++(270:1.1cm) and ++(270:.9cm) .. (-.4,-.2);
	\draw (-.3,.2) .. controls ++(90:.7cm) and ++(90:1cm) ..  (.65,0) .. controls ++(270:1cm) and ++(270:.7cm) .. (-.3,-.2);
	\draw (0,.2) .. controls ++(90:.2cm) and ++(90:.4cm) ..  (.25,0) .. controls ++(270:.4cm) and ++(270:.2cm) .. (0,-.2);
	\ncircle{unshaded}{(-.2,0)}{.3}{180}{}
	\node at (.45,0) {\scriptsize{$\cdots$}};
\end{tikzpicture}
\,:
\cP_{n,\pm}\to \cP_{0,\pm}\cong \bbC,
$$
\item and $\cP_\bullet$ is \emph{spherical}, that is closed diagrams are invariant
under spherical isotopy.
\end{itemize}

We sometimes also talk about \emph{non-spherical} or \emph{non-extremal} subfactor planar algebras, which do not satisfy this last axiom.

\begin{thm}[\cite{MR1334479,math.QA/9909027}]
\label{ref:PopaCorrespondence}
Given a finite index \twoone-subfactor, its standard invariant forms a subfactor planar algebra.
Conversely, given a subfactor planar algebra $\cP_\bullet$, there is a \twoone-subfactor whose standard invariant is $\cP_\bullet$.
\end{thm}

Under this correspondence, the planar algebra is spherical if and only if the subfactor is extremal \cite{MR1887878}. We say a subfactor planar algebra is irreducible if $\dim(\cP_{1,\pm})=1$. 
Irreducible subfactor planar algebras correspond to irreducible subfactors, i.e. those $A \subset B$ where $A' \cap B = \bbC$.

From the Jones tower for $A_0 \subset A_1$, we define the associated planar
algebra $\cP_\bullet$ as follows.
The vectors spaces $\cP_{n,\pm}$ are defined as the two towers of centralizer algebras: $\cP_{n,+} = A_0'\cap A_n$ and $\cP_{n,-} = A_1'\cap A_{n+1}$.
The planar algebra structure is given in \cite{math.QA/9909027} or \cite[Section 2.3]{MR2812459}.
We note that as proven in \cite[Section 2.6]{MR2812459}, the planar algebra structure is completely determined by the
following, where $\delta=[A_1:A_0]^{1/2}$:
\begin{itemize}
\item
Stacking elements in $\cP_{n,\pm}$ is multiplication in the centralizer algebra.
\item
The involution $*$ on $\cP_{n,\pm}$ is the involution on the centralizer algebra.
\item
The $n$-th Jones projection is given by 
$
e_n = 
\delta^{-1}\,
\begin{tikzpicture}[baseline =-.1cm]
	\draw[thick] (-1,-.4) rectangle (.3,.4);
	\draw (-.1,-.4)--(-.1,-.3) arc (180:0:.1cm) -- (.1,-.4);
	\draw (-.1,.4)--(-.1,.3) arc (-180:0:.1cm) -- (.1,.4);
	\draw (-.8,-.4)--(-.8,.4);
	\draw (-.3,-.4)--(-.3,.4);
	\node at (-.525,0) {{\scriptsize{$\cdots$}}};
\end{tikzpicture}
\,\in \cP_{n+1,+} = A_0'\cap A_{n+1}.
$
\item
Adding a string on the right is the inclusion $\cP_{n,+} = A_0'\cap A_n \hookrightarrow A_0'\cap A_{n+1}=\cP_{n+1,+}$.
\item
Adding a string on the left is the inclusion $\cP_{n,-} = A_1'\cap A_{n+1} \hookrightarrow A_0'\cap A_{n+1} = \cP_{n+1,+}$.
\item
Capping on the right is $\delta$ times the restriction of the conditional expectation $A_n \to A_{n-1}$ to $$\cP_{n,+} = A_0'\cap A_n \to A_0'\cap A_{n-1} = \cP_{n-1,+}.$$
\item
Capping on the left is $\delta$ times the restriction of the conditional expectation $A_0'\to A_1'$ to $$\cP_{n,+}= A_0'\cap A_n \to A_1'\cap A_n= \cP_{n-1,-}.$$
\end{itemize}

Since $\cP_{0,\pm}$ is 1-dimensional, the shaded and unshaded closed loops are multiples of the empty diagram. The
spherical axiom ensures they are the same multiple, $\delta = [A_1:A_0]^{1/2}$.
We define the index of a subfactor planar algebra by the quantity $\delta^2$.
As an exercise, the reader can check that the diagrammatic $e_n$ is an idempotent with respect to stacking.

A planar tangle with no input discs and $2n$
boundary points gives a map $\bbC \to \cP_{n,\pm}$, and can
be thought of as an element of $\cP_{n,\pm}$. The span of these elements forms the Temperley-Lieb-Jones planar subalgebra $\cT\cL\cJ(\delta)_\bullet$ present inside any planar algebra $\cP_\bullet$ with index $\delta^2$.
Indeed every index value allowed by Jones'
restriction is realized by a
subfactor whose standard invariant is `trivial', in the sense that it is no
bigger than the Temperley-Lieb-Jones planar algebra \cite{MR1198815}.

From a subfactor planar algebra $\cP_\bullet$, we may define a strict pivotal
2-category $\cC$. An equivalent construction is described in detail in \cite[Section 2.1]{MR3157990}.
As a first step, we define a preliminary 2-category $\hat
{\cC}$. It has two objects, called `+' (or `unshaded') and `-' (or `shaded'),
and the 1-morphisms are natural numbers, even for 1-morphisms that do not
change the shading, and odd for those that do. (We use the $\otimes$ symbol to denote horizontal composition; on composable
1-morphisms we have $n \otimes m = n+m$.) The 2-morphisms are given by $\operatorname{Hom}(n \to m)
= \cP_{\frac{1}{2}(n+m), \pm}$,
with the sign determined by
the source of $n$ and $m$. The structure as a pivotal 2-category is
readily provided by the planar algebra operations.
For example, 
\begin{itemize}
\item
if $f \in \operatorname{Hom}_{+-}(1\to 3)$ and $g\in \operatorname{Hom}_{+-}(3\to 5)$, we have
$
g\circ f = 
\begin{tikzpicture}[baseline=-.1cm]
	\fill[shaded] (0,-1) rectangle (.4,1);
	\fill[shaded] (-.2,.5) rectangle (-.1,1);
	\fill[shaded] (-.1,-.5) rectangle (0,.5);
	\fill[unshaded] (0,-.5) rectangle (.1,.5);
	\fill[unshaded] (.1,.5) rectangle (.2,1);
	\draw (0,1) -- (0,-1);
	\draw (.1,1) -- (.1,-.5);
	\draw (-.1,1) -- (-.1,-.5);
	\draw (.2,1) -- (.2,.5);
	\draw (-.2,1) -- (-.2,.5);
	\ncircle{unshaded}{(0,.5)}{.3}{180}{$g$}
	\ncircle{unshaded}{(0,-.5)}{.3}{180}{$f$}
\end{tikzpicture}
$\,,
\item
if $f \in \operatorname{Hom}_{+-}(1\to 3)$ and $g\in \operatorname{Hom}_{--}(4\to 2)$, we have
$
f\otimes g =
\begin{tikzpicture}[baseline=-.1cm]
	\fill[shaded] (-.5,-.5) rectangle (1,.5);
	\fill[shaded] (-.6,0) rectangle (-.5,.5);
	\fill[unshaded] (-.5,0) rectangle (-.4,.5);
	\fill[unshaded] (.45,.5) rectangle (.55,0);
	\fill[unshaded] (.35,0) rectangle (.45,-.5);
	\fill[unshaded] (.55,0) rectangle (.65,-.5);
	\draw (-.5,-.5) -- (-.5,.5);
	\draw (.45,-.5) -- (.45,.5);
	\draw (.55,-.5) -- (.55,.5);
	\draw (.35,0) -- (.35,-.5);
	\draw (.65,0) -- (.65,-.5);
	\draw (-.4,0) -- (-.4,.5);
	\draw (-.6,0) -- (-.6,.5);
	\ncircle{unshaded}{(.5,0)}{.3}{180}{$g$}
	\ncircle{unshaded}{(-.5,0)}{.3}{180}{$f$}
\end{tikzpicture}
$\,, and
\item
evaluation is given by the cap 
$
\operatorname{ev}:{\sb{+}1}_-\otimes  {\sb{-}1}_+\to 0 
= 
\tikz[baseline=.1cm]{
\draw[dashed] (0,0) rectangle (1,.6); \draw[shaded] (0.2,0) arc (180:0:0.3cm); 
}$\,,
while coevaluation is the cup $\operatorname{coev}: 0 \to {\sb{+}1}_-\otimes  {\sb{-}1}_+$, and similarly for the other shading.
\end{itemize}
Finally, we declare $\cC$ to be the idempotent completion of $\hat{\cC}$.

\begin{remark}
When the subfactor planar algebra comes from a finite index subfactor $A \subset B$, this 2-category is a purely algebraic model of
the 2-category of $A-A$, $A-B$, $B-A$, and $B-B$ $L^2$-bimodules (or at least, those generated by the bimodule ${}_A
L^2(B)_B$) \cite{MR2501843}.
\end{remark}

\begin{defn}
A subfactor planar algebra is \emph{finite depth} if $\cC$ has only finitely many isomorphism classes of 1-morphisms.
\end{defn}

\begin{defn}
The
\emph{even part} of a subfactor planar algebra is the $\otimes$-category
obtained
as the endomorphisms of the unshaded object, in the associated 2-category.
\end{defn}

A critical invariant of a subfactor planar algebra is its
\emph{supertransitivity} \cite{MR2972458}.
\begin{defn}
\label{defn:supertransitivity}
The supertransitivity of a subfactor planar algebra $\cP_\bullet$ is the least
integer $k$ such that $\dim
\cP_{k+1,\pm} > \dim \cT\cL\cJ_{k+1,\pm}$.
\end{defn}

Every subfactor planar algebra can be seen as a representation of
the annular Temperley-Lieb-Jones algebra (this is spanned by the planar tangles
with one input disc), and decomposed into a direct sum of
irreducible representations. These have been described in \cite
{MR1659204,MR1929335}.

Capturing slightly less information than the full decomposition into
irreducibles, we can look at the \emph{low weight spaces $\cW_{n,\pm}
\subset \cP_{n,\pm}$} of the planar
algebra: those vectors which are annihilated by capping any two adjacent strings:

\tikzstyle{shaded}=[fill=red!10!blue!20!gray!40!white]
\tikzstyle{empty box}=[circle, draw, thick, fill=white, opaque, inner sep=2mm]
\tikzstyle{annular}=[scale=.6, inner sep=1mm, baseline]
\begin{equation*}
\begin{tikzpicture}[annular]
	\clip (0,0) circle (2cm);
	\draw[] (0,0)--(-68:4cm) arc (-68:-22:4cm) --(0,0); 
	\draw[] (0,0)--(-112:4cm) arc (-112:-252:4cm) --(0,0); 
	\draw[, fill=white] (0,0) .. controls ++(215:2cm) and ++(145:2cm) .. (0,0);
	\draw[] (0,0)--(22:4cm) arc (22:68:4cm) --(0,0); 
	\draw[ultra thick] (0,0) circle (2cm);
	\node at (0,0)  [empty box] (T) {$R$};
	\node at (0:1.5cm) {$\cdot$};
	\node at (12:1.5cm) {$\cdot$};
	\node at (-12:1.5cm) {$\cdot$};
\end{tikzpicture}
 = 0.
 \end{equation*}

\begin{defn}
\label{defn:annular-multiplicities}
The \emph {annular multiplicity sequence} of the planar algebra is the sequence $(\dim \cW_{n,\pm})_{n\geq 0}$.
\end{defn}

The
Fourier transform (1-click rotation) tangle gives a vector space isomorphism $\cW_{n,+} \iso \cW_{n,-}$.
This sequence can be
computed \cite[p.1]{MR2972458}, according to the formula appearing in  \cite[Definition 2.5]{MR2914056}, which is obtained by inverting the generating function for the dimensions of the annular Temperley-Lieb-Jones irreps given in \cite[Corollary 5.4]{MR1929335}.
\begin{equation}
\label{eq:dim-low-weight}
\dim \cW_{n,\pm} = \sum_{r=0}^n (-1)^{r-n} \frac{2n}{n+r} \binom{n+r}{n-r}
\dim P_{r,\pm}.
\end{equation}

For an evaluable $k$-supertransitive
planar algebra, the sequence of annular multiplicities necessarily starts with
$10^k$. When we say below that
a standard invariant has annular multiplicities $*ab$, we mean that while we may
not know the supertransitivity $k$ yet, the annular multiplicity sequence
begins
with $10^kab$.

\subsection{Towards classification}
We now turn towards classifying subfactor standard
invariants. Classical results \cite{MR996454,MR999799}  give us an ADE
classification when the
index is
less than 4. To explain how these Coxeter-Dynkin diagrams arise, we introduce the
\emph{principal graph} $\Gamma(\cP_\bullet)$ of a subfactor planar algebra $\cP_\bullet$.

Recall in the associated pivotal 2-category $\cC$ we have a generating
1-morphism $1:+ \to -$. (When we start with a
subfactor $A \subset B$ this is the bimodule ${}_A L^2(B)_B$.)
We denote it as $X$, and its dual as $X^*$.

\begin{defn}
The vertices of $\Gamma(\cP_\bullet)$ are the isomorphisms classes of simple
1-morphisms in $\cC$. 
If vertices $Y$ and $Z$ have the same shadings on their sources, and $Y$'s target is unshaded while
$Z$'s
target is shaded, the number of edges between $Y$ and $Z$ is $\dim \operatorname{Hom}_\cC(Y \otimes X \to Z) = \dim
\operatorname{Hom}_\cC(Z \otimes X^* \to Y)$.
\end{defn}

The vertices of the principal graph come in 4 types,
according to the shadings of their sources and targets. The principal graph
has two components, according to the shadings of the sources, and each
component is bipartite, according to the shadings of the targets.

Each component of the graph is pointed, with the basepoints being the identity 1-morphisms.
We say the \emph{depth} of a vertex is its distance to the basepoint in that component.

For our purposes, we nearly always consider principal graphs equipped also with the involution recording the duals of
simple 1-morphisms (contragredients of bimodules in the subfactor setting%
\footnote{
Recall that the vertices of $\Gamma_\pm$ correspond to 4 different flavors of $L^2$-bimodules generated by $L^2(B)$: the $A-A$, $A-B$, $B-A$, and $B-B$.
Each bimodule has a contragredient, or \emph{dual}, which is the complex conjugate Hilbert space with the conjugate
action.
For example, given $\sb{A}Q_B$, the dual $\sb{B}\overline{Q}_A = \set{\overline{\xi}}{\xi\in Q}$ with action given by $b\cdot \overline{\xi} \cdot a = \overline{a^*\xi b^*}$.
The dual of an $A-A$ bimodule is again an $A-A$ bimodule and similarly for $B-B$ bimodules, but the dual of an $A-B$ bimodule is a $B-A$ bimodule.	
}). It is easy to see that duality
preserves
depth on the principal graph. We indicate duality of even depth vertices
on the graphs using red marks: a small red tag above a vertex indicates that
it is self-dual, while a red line joining two vertices indicates they are duals of each other. In odd depths, we use the
convention that the vertices at a given depth on one component of the principal graph are dual to the vertices at the
same depth on the other component, in the order that they appear on the page.
As an example, the principal graph of the Haagerup subfactor \cite{MR1686551}, along with its dual data, is given by:
$$
\cH=(\cH_+,\cH_-)=
\left( 
\bigraph{bwd1v1v1v1p1v1x0p0x1v1x0p0x1duals1v1v1x2v2x1}
,
\bigraph{bwd1v1v1v1p1v1x0p1x0duals1v1v1x2}
\right).
$$

We can compute $\dim \cP_{n,\pm}$ as the number of loops of length $2n$ beginning at the basepoint of the $\pm$ component.
This means that the supertransitivity can be read
off the principal graph: it is the greatest integer $k$ such that the
principal graph is the same as $A_\infty$ up to depth $k$.
In
what follows, we will often consider families of potential principal graphs
which differ only in their supertransitivity.
\begin{defn}
A \emph{translation  by $2t$} of graph pair is the new graph pair obtained by increasing the supertransitivity by $2t$.
\end{defn}
(It's essential we only translate by an even amount, to respect the bipartite structure.)
\begin{defn}
An \emph{extension} of a graph pair $\Gamma$ with depth $k$ is another graph pair $\Gamma'$ with depth $k' > k$, such that the truncation of $\Gamma'$ to depth $k$ (that is, deleting all vertices above depth $k$) recovers $\Gamma$.
\end{defn}

In what follows, we will frequently talk about families of principal graphs, which come in two types, vines and weeds.
A \emph{vine} is a finite graph pair for which we will consider the family of translations by $2t$ for all $t\geq 0$. A
\emph{weed} is a finite graph pair for which we will consider the family of arbitrary extensions of  translations by
$2t$ for all $t \geq 0$. 

The principal graph is a finite graph if and only if the standard invariant is finite depth. The standard invariant is irreducible if and only if there is exactly one edge between depths 0 and 1 in the principal graph.

\begin{remark}
Given a finite depth planar algebra $\cP_\bullet$, it is relatively easy
to see that the
index can be recovered as $\lambda(\Gamma(\cP_\bullet))^2$, the square of the
 \emph{graph norm}. (This was first established in \cite{MR934296}.) The graph norm is the largest eigenvalue of the adjacency
 matrix of the
 principal graph.

When $\cP_\bullet$ is infinite depth, $\Gamma(\cP_\bullet)$ has bounded degree, so the adjacency matrix defines a bounded operator on the infinite dimensional Hilbert space given by $\ell^2$ of the vertices. As in the finite case, the graph norm is the norm of the adjacency matrix.
In this case, we only have the inequality
$\lambda(\Gamma(\cP_\bullet))^2 \leq \delta^2$
\cite{MR1278111}.
\end{remark}

\begin{fact}
If we can enumerate all possible graph pairs with norm at most $\delta$, these must include all the principal graphs of
subfactors of index at most $\delta^2$.
\end{fact}

This gives us a very powerful tool --- graph norms are increasing under graph inclusion (strictly increasing for finite graphs),
and so we obtain easy lower bounds for the index of a subfactor with a given principal graph.

When the index is at least 4, the principal graph of the Temperley-Lieb-Jones planar
algebra $\cT\cL\cJ_\bullet$ is two copies of the graph $A_\infty$ (with graph norm 2). When the index is $4
\cos^2(\pi / n)$ for some $n \geq 3$, the principal graph of $\cT\cL\cJ_\bullet$ is two copies of the Coxeter-Dynkin
diagram $A_{n-1}$.

The mere fact that the only bipartite graphs with graph norm less than 2 are the ADE
Coxeter-Dynkin diagrams gives us the start of the classification. It is relatively straightforward
to see that both components of the principal graph must be the same Coxeter-Dynkin diagram.
The full
classification of subfactor planar algebras with index
at most 4 was developed by Jones \cite{MR934296} and Ocneanu \cite{MR996454}, with many of the details provided
by
others
\cite{MR999799,MR1193933,
MR1145672, MR1313457, MR1308617} (see also \cite{MR1929335} for an independent approach using annular tangles). As is well-known by now, there is a unique
subfactor planar algebra
for
each $A_n$
and $D_{2n}$ Coxeter-Dynkin diagram, two distinct subfactor planar algebras for each of $E_6$
and $E_8$, and no subfactor planar algebras for $D_{2n+1}$ or $E_7$.

The classification at index exactly 4 was given by Popa \cite{MR999799,MR1278111,MR1213139};
the principal graphs
are all \emph{affine} Coxeter-Dynkin diagrams.

It is a remarkable fact that once we decide to
ignore subfactors with trivial standard invariant and reducible subfactors, the index is actually
also quantized \emph{above} 4. It is straightforward to see that there is a
gap between $4$ and $\lambda (E_ {10})^2
\sim 4.0264$. This gap is an easy consequence of the following
exercise, which is an excellent introduction to the genre.

\begin{exc}[\cite{MR683990}]
Show that every bipartite graph is either
\begin{enumerate}[label=(\arabic*)]
\item a Coxeter-Dynkin diagram,
\item an affine Coxeter-Dynkin diagram,
\item $A_\infty$, $A_\infty^{(1)}$, or $D_\infty$, or
\item contains one of the following as a subgraph:
\begin{align*}
&
\smallbigraph{gbg3}
&&
\smallbigraph{gbg2v1}
&&
\bigraph{gbg1v1p1p1p1}
\\
&
\bigraph{gbg1v1v1p1p1}
&&
\bigraph{gbg1v1v1p1v1x0v1v1v1p1}
&&
\bigraph{gbg1v1v1p1v1x0v1v1p1}
\\
&
\bigraph{gbg1v1v1p1v1x0v1p1}
&&
\bigraph{gbg1v1v1p1v1x0p1x0}
&&
\bigraph{gbg1v1v1p1v1x0v1v1v1v1}
\end{align*}
\end{enumerate}
Then, by calculating the graph norms of the finitely many exceptions, show that the last has
the lowest graph norm. Thus $\lambda(E_{10})^2$ gives a lower bound on the index of any subfactor
above index 4, leaving aside subfactors with trivial standard invariant and reducible subfactors. \hfill\qed
\end{exc}

\subsection{Some first obstructions}

Enumerating possible graph pairs above index 4 gets difficult quickly.
Over the years, a number of important obstructions to a graph pair being realized as the principal graph of a subfactor have been developed. We recall these briefly in this section. Many of these will be incorporated into the algorithm for enumerating potential principal graphs of subfactors which we describe in Section \ref{sec:combinatorics}.

We saw above that $\dim(\cP_{n,\pm})$ is given by the number of loops of length $2n$ beginning at the basepoint of the $\pm$ component.
We also saw that the Fourier transform (1-click rotation) tangle gives a vector space isomorphism $\cP_{n,+} \iso \cP_{n,-}$, which restricts to a vector space isomorphism $\cW_{n,+}\cong \cW_{n,-}$.
This gives us the following easy constraints.

\begin{fact}[Dimension constraints \cite{MR999799,MR1929335}]
\mbox{}
\begin{enumerate}[label=(\arabic*)]
\item
The numbers of based loops of length $2n$ on either component $\Gamma_\pm$ must agree. 
\item
Both $\Gamma_+$ and $\Gamma_-$ have the same supertransitivity.
\item
The annular multiplicity sequences of $\Gamma_\pm$ must agree by Equation \eqref{eq:dim-low-weight}.
\end{enumerate}
\end{fact}

From the principal graphs $\Gamma=(\Gamma_+,\Gamma_-)$, we can already deduce a lot of information about the associated strict pivotal 2-category $\cC$.
One strong constraint comes from associativity of composition of 1-morphisms.
The Ocneanu 4-partite graph $\cO(\Gamma)$ encodes the same information as
$\Gamma$, but allows us to verify associativity of certain tensor products easily.

\begin{defn}
\label{defn:Ocneanu4Partite}
Suppose $A\subset B$ is a finite index subfactor with standard invariant $\cP_\bullet$ and principal graphs $\Gamma=(\Gamma_+,\Gamma_-)$. 
Taking 2 copies of each of $\Gamma_\pm$, they fit together in the \emph{Ocneanu 4-partite graph} $\cO(\Gamma)$:
\begin{equation*}
\label{eq:4Partite}
\begin{tikzpicture}
	\node (V00) at (0,0) {$V_{00}=\{ A-A \text{ bimodules}\}$};
	\node (V11) at (8,-2) {$V_{11}=\{ B-B \text{ bimodules}\}$};
	\node (V01) at (0,-2) {$V_{01}=\{ A-B \text{ bimodules}\}$};
	\node (V10) at (8,0) {$V_{10}=\{B-A \text{ bimodules}\}$};
	\draw (V00)--(V01) node [right, midway] {$\Gamma_+$};
	\draw (V01)--(V11) node [above, midway] {$\Gamma_-$};
	\draw (V00)--(V10) node [below, midway] {$\Gamma_+$};
	\draw (V10)--(V11) node [left, midway] {$\Gamma_-$};
	\draw (V00)--(V01) node [left, midway] {$-\otimes_A L^2(B)_B$};
	\draw (V01)--(V11) node [below, midway] {$\sb{B}L^2(B)\otimes_A -$};
	\draw (V00)--(V10) node [above, midway] {$\sb{B}L^2(B)\otimes_A -$};
	\draw (V10)--(V11) node [right, midway] {$-\otimes_A L^2(B)_B$};
\end{tikzpicture}
\end{equation*}
We note that the right hand graph is exactly $\Gamma_+$, but the top copy of $\Gamma_+$, while abstractly isomorphic to $\Gamma_+$, has different vertex labels.
We note vertices $\sb{A}P_A$ and $\sb{B}S_A$ are connected by
\begin{equation}
\label{eq:TwistDuals}
\dim(\Hom({\sb{B}L^2(B)\otimes_A P_A}\to {\sb{B}S_A})=\dim(\Hom({\sb{A}\overline{P}\otimes_A L^2(B)_B}\to {\sb{A}\overline{S}_B})
\end{equation}
edges.  Similarly, the right copy of $\Gamma_-$ is exactly $\Gamma_-$, and the
bottom copy is twisted using the dual data. Thus the graph pair $\Gamma=(\Gamma_+,\Gamma_-)$ with dual data is exactly the same data as
$\cO(\Gamma)$. 
\end{defn}

For example, the Ocneanu 4-partite graph of the Haagerup subfactor is given by
$$
\begin{tikzpicture}[baseline]
	\node at (-4.3,1.5) {$\cH_+\,\Bigg\{$};
	\node at (-4.3,-.5) {$\cH_-\,\Bigg\{$};
	\node at (-3,2) {$\sb{A}{\sf{Mod}}_A$};
	\filldraw (-2,2) circle (1mm); 
	\filldraw (0,2) circle (1mm);	
	\filldraw (2,2) circle (1mm); 
	\filldraw (3,2) circle (1mm);
	\filldraw (6,2) circle (1mm);
	\filldraw (7,2) circle (1mm); 
	\node at (-3,1) {$\sb{A}{\sf{Mod}}_B$};
	\filldraw (-1,1) circle (1mm);  
		\draw (-1,1)--(-2,2);  
		\draw (-1,1)--(-2,0);
		\draw (-1,1)--(0,2);  
		\draw (-1,1)--(0,0); 
	\filldraw (1,1) circle (1mm);  
		\draw (1,1)--(0,2);  
		\draw (1,1)--(0,0);
		\draw (1,1)--(2,2);  
		\draw (1,1)--(2,0);
		\draw (1,1)--(3,2);  
		\draw (1,1)--(3,0);  		
	\filldraw (4,1) circle (1mm); 
		\draw (4,1)--(2,2);  
		\draw (4,1)--(2,0);
		\draw (4,1)--(6,2);  
	\filldraw (5,1) circle (1mm); 
		\draw (5,1)--(3,2);  
		\draw (5,1)--(2,0);
		\draw (5,1)--(7,2);	
	\node at (-3,0) {$\sb{B}{\sf{Mod}}_B$};
	\filldraw (-2,0) circle (1mm); 
	\filldraw (0,0) circle (1mm);
	\filldraw (2,0) circle (1mm);
	\filldraw (3,0) circle (1mm); 
	\node at (-3,-1) {$\sb{B}{\sf{Mod}}_A$};
	\filldraw (-1,-1) circle (1mm);  
		\draw (-1,-1)--(-2,-2);  
		\draw (-1,-1)--(-2,0);
		\draw (-1,-1)--(0,-2);  
		\draw (-1,-1)--(0,0); 
	\filldraw (1,-1) circle (1mm);  
		\draw (1,-1)--(0,-2);  
		\draw (1,-1)--(0,0);
		\draw (1,-1)--(2,-2);  
		\draw (1,-1)--(2,0);
		\draw (1,-1)--(3,-2);  
		\draw (1,-1)--(3,0);  		
	\filldraw (4,-1) circle (1mm); 
		\draw (4,-1)--(2,-2);  
		\draw (4,-1)--(2,0);
		\draw (4,-1)--(7,-2);  
	\filldraw (5,-1) circle (1mm);
		\draw (5,-1)--(3,-2);  
		\draw (5,-1)--(2,0);
		\draw (5,-1)--(6,-2);  
	\node at (-3,-2) {$\sb{A}{\sf{Mod}}_A$};	
	\filldraw (-2,-2) circle (1mm); 
	\filldraw (0,-2) circle (1mm);
	\filldraw (2,-2) circle (1mm); 
	\filldraw (3,-2) circle (1mm);
	\filldraw (6,-2) circle (1mm);
	\filldraw (7,-2) circle (1mm); 
\end{tikzpicture}
$$

Since composition of 1-morphisms in $\cC$ is associative, we must have that for every simple $\sb{A}P_A$,
$$
({\sb{B}L^2(B)}\otimes_A P_A) \otimes_A L^2(B)_B \cong {\sb{B}L^2(B)}\otimes_A ({\sb{A}P} \otimes_A L^2(B)_B).
$$ 
There is a similar condition starting with each of the other 3 flavors of bimodules.
We deduce:

\begin{fact}[Associativity constraint \cite{MR996454,MR1642584}]
\label{fact:AssociativityConstraint}
Given two vertices $v$ and $w$ on opposite corners of the Ocneanu 4-partite graph $\cO(\Gamma)$, there are the same number of paths between $v$ and $w$ going either way around.

The two ways of going around the Ocneanu 4-partite graph correspond to:
\begin{enumerate}[label=(\arabic*)]
\item moving to a neighbour on the principal graph, taking dual, moving to a neighbour, and taking dual again, or
\item taking the dual, moving to a neighbour, taking dual, and finally moving to a neighbour again.
\end{enumerate}
\end{fact}

Ocneanu's axiomatization of the standard invariant of a finite depth subfactor used connections on 4-partite graphs.

\begin{defn}
\label{defn:Connection}
A \emph{connection} on $\Gamma$ is a pair $(W,\dim)$ where $\dim$ is a dimension
function on the vertices of $\cO(\Gamma)$ satisfying the Frobenius-Perron
condition, and $W$ is a complex valued function on the loops of length 4 of
$\cO(\Gamma)$ which include a vertex of each
color. The connection is said to be \emph{bi-unitary} if the following two axioms
hold:
\begin{itemize}
\item (Unitarity)
For every $P,R$ on diagonally opposite corners of $\cO(\Gamma)$, the matrix $W(P,-,R,-)$ is unitary, i.e.,
$$
\sum_S W(P,Q,R,S) \overline{W(P,Q',R,S)}=\delta_{Q,Q'}
$$
\item (Renormalization) 
For all $P,Q,R,S$, we have
$$
W(P,Q,R,S) = \sqrt{\frac{\dim(Q)\dim(S)}{\dim(P)\dim(R)}} \overline{W(Q,P,S,R)}
$$
\end{itemize}
\end{defn}

\begin{fact}[Existence of connection \cite{MR996454,MR1642584}]
\label{fact:ExistenceOfConnection}
A necessary condition for $\Gamma$ to be the principal graph of a subfactor is that $\cO(\Gamma)$ must have a bi-unitary connection.
\end{fact}

Ocneanu found the first triple point obstruction, which is a simple consequence of the existence of a bi-unitary
connection. This obstruction was used by Haagerup to classify principal graphs to index $3+\sqrt{3}$ \cite{MR1317352}.
In the case of initial triple points, improvements were made subsequently by \cite{MR2972458, MR3198588, MR3311757}.
These obstructions were invaluable to the previous classification to index 5 \cite{MR2914056,MR2902285},
and the new obstruction \cite{MR3311757} is vital to this classification (see Section \ref{sec:BranchFactorInequalities} below).

\begin{fact}[Ocneanu's triple point obstruction \cite{MR1317352}]
\label{fact:TriplePointObstruction}
Let $A\subset B$ be a finite index subfactor with principal graph $\Gamma=(\Gamma_+,\Gamma_-)$.
Suppose we have two 3-valent vertices $v$ on $\Gamma_+$ and $w$ on $\Gamma_-$ at the same depth, and there are exactly 6 paths between $v$ and $w$ on $\cO(\Gamma)$ (3 in each direction around the square).
If there is a dimension preserving bijection $\beta$ between the neighbors of $v$ on $\Gamma_+$ and the neighbors of $w$ on $\Gamma_-$ such that
\begin{itemize}
\item
for every pair of neighbors $v'$ of $v$ and $w'$ of $w$ such that $\beta(v')\neq w'$, there are exactly 2 paths on $\cO(\Gamma)$ from $v'$ to $w'$ (one each way around the square),
\end{itemize}
then $[B:A]\leq 4$.
\end{fact}

A reader who  (understandably) finds that formulation hard to digest may find working through the following exercise helpful.

\begin{exc}[\cite{MR1317352}]
Show that the obvious bijection between the neighbors of the triple points on the graph pair
$$
(\cH_+,\cH_+)=
\left( 
\bigraph{bwd1v1v1v1p1v1x0p0x1v1x0p0x1duals1v1v1x2v2x1}
,
\bigraph{bwd1v1v1v1p1v1x0p0x1v1x0p0x1duals1v1v1x2v2x1}
\right)
$$
satisfies the bulleted condition in Ocneanu's triple point obstruction \ref{fact:TriplePointObstruction}.
Conclude that $(\cH_+, \cH_+)$ is not the principal graph of a subfactor.
\end{exc}

Popa's principal graph stability gives a strong constraint on extensions of graph pairs.
We denote the truncation of $\Gamma_\pm$ to depth $n$ by $\Gamma_\pm(n)$.

\begin{defn}
A graph $\Gamma_\pm$ is called \emph{stable at depth $n$} if $\Gamma_\pm$ does not merge, split, or have multiple edges
between depths $n$ and $n+1$.
We say $\Gamma=(\Gamma_+,\Gamma_-)$ is stable at depth $n$ if both $\Gamma_\pm$ are stable at depth $n$.
\end{defn}

\begin{fact}[Stability constraint \cite{MR1334479,MR3157990}]
\label{Fact:StabilityConstraint}
Suppose $\delta>2$ and $\Gamma_\pm(n)\neq A_{n+1}$.
\begin{enumerate}[label=(\arabic*)]
\item
If the graph $\Gamma=(\Gamma_+,\Gamma_-)$ is stable at depth $n$, then $\Gamma$ is stable for all depths $k\geq n$, and $\Gamma$ is finite.
(This means $\Gamma_\pm\setminus \Gamma_\pm(n)$ must be a disjoint union of finite type $A$ Coxeter-Dynkin diagrams.)
\item
If the graph $\Gamma_+$ is stable at depths $n$ and $n+1$, then $\Gamma=(\Gamma_+,\Gamma_-)$ is stable at depth $n+1$.
\end{enumerate}
\end{fact}

A final easy obstruction comes from duality.

\begin{fact}[Duality constraint {\cite[Lemma 3.6]{MR2914056}}]
\label{fact:DualityConstraint}
Suppose $\Gamma$ has supertransitivity $n-1$  (so that depth $n$ is one past the branch) with $n$ even.
If the graphs $\Gamma_\pm(n)$ are both simply laced, then the number of self-dual vertices on $\Gamma_+$ at depth $n$ is equal to the number of self-dual vertices on $\Gamma_-$ at depth $n$.
\end{fact}

\section{The main theorem}
\label{sec:Summary}

Over the last few years, we've made considerable progress in understanding
small index subfactor standard invariants. Haagerup initiated the classification of subfactors above index 4, leaving  aside reducible subfactors and the
poorly understood non-amenable subfactors with Temperley-Lieb-Jones standard
invariant. Haagerup gave the
classification up to index $3+\sqrt{3}$ \cite{MR1317352}, with components proved in \cite{MR1625762,MR2472028,MR2979509}. Following this, the next major step was the
classification of subfactors up to  index $5=3+\sqrt{4}$.  This was completed in a series of articles
\cite{MR2914056,MR2902285,MR2993924,MR2902286,MR3335120}, with some additional
number theoretic ingredients in \cite{MR2786219}. There is now a survey
paper summarising this work \cite{MR3166042}.  Unfortunately, the techniques
developed there struggle beyond index 5.

For the special case of 1-supertransitive standard invariants, it has nevertheless been
possible to extend the classification further: up to index $3+\sqrt{5}$ in
\cite{MR3254427}, and then 1-supertransitive standard invariants without intermediates
up to index $6\frac{1}{5}$ in \cite{MR3306607}.

In this article we give the complete classification of subfactor standard invariants with index at most $5\frac{1}{4}$.
At every index above 4, we have the Temperley-Lieb-Jones $A_\infty$ standard invariant, as well as the reducible $A_\infty^{(1)}$ standard invariant (see Lemma \ref{lem:reducible} below).

\begin{thmalpha}
\label{thm:Main}
The only subfactor standard invariants in the index range $(5,5\frac{1}{4}]$, besides the $A_\infty$ and $A_\infty^{(1)}$ standard invariants, are
the following standard invariants at index $\sim5.04892$, the largest root of $x^3-6 x^2+5 x-1$:
\begin{itemize}
\item
the unique subfactor planar algebra coming from the irreducible 3-dimensional representation of the quantum group $\mathfrak{su}(2)_5$ \cite{MR936086,MR3254427} and
\item
the unique subfactor planar algebra coming from (either of) the irreducible 3-dimensional representation of the quantum group $\mathfrak{su}(3)_4$ \cite{MR936086,MR3254427},
\end{itemize}
and the following standard invariants at index $3+\sqrt{5}$:
\begin{itemize}
\item
the Bisch-Jones Fuss-Catalan $A_3*A_4$ subfactor planar algebra and its dual \cite{MR1437496},
\item
the 3 finite quotients of the Fuss-Catalan $A_3*A_4$ subfactor planar algebra \cite{BischHaagerup,MR3345186,1308.5723} (the first is the self-dual tensor product; the other two are not self-dual),
\item
the unique 2D2 subfactor planar algebra and its dual \cite{1406.3401},
\item
Izumi's unique symmetrically self-dual $3^{\bbZ/2\bbZ\times\bbZ/2\bbZ}$ subfactor planar algebra \cite{IzumiUnpublished,MR3314808},
\item
Izumi's unique $3^{\bbZ/4\bbZ}$ subfactor planar algebra and its dual \cite{IzumiUnpublished,1308.5197}, and
\item
the unique symmetrically self-dual 4442 subfactor planar algebra \cite{MR3314808,1406.3401}.
\end{itemize}
\end{thmalpha}
(Our arguments supersede the earlier combinatorial arguments required for index at most $3+\sqrt{3}$ and at most $5$,
but still rely on many obstructions and existence results proved by other authors.)

As in previous classification efforts, the problem essentially divides into three parts: 
\begin{enumerate}[label=(\arabic*)]
\item
First, we enumerate all possible principal graph pairs, satisfying certain
combinatorial constraints, with graph norm up to the square root of the index.
\item
Second, we rule out several `difficult' families of such graphs using quite complicated arguments specific to each
family, and we reduce many `easy' familes of such graphs down to finitely many cases using a standard number
theoretic approach.
\item 
Finally, we completely classify all subfactor planar algebras with the given
principal graph in each of the remaining cases.
\end{enumerate}
In general, one expects all parts to suffer when increasing the bound on the
index; the combinatorial enumeration problem will be harder for the computer,
and moreover it will produce more results, leaving a rapidly growing workload
for the humans.

To overcome these problems, this article brings three essential new tools to bear. 
\begin{enumerate}[label=(\arabic*)]
\item
We use a completely new technique for tackling the combinatorial enumeration
problem. Previous methods produced many isomorphic copies of a single graph 
during the enumeration, and removing the corresponding
redundancies in the search tree became unrealistically computationally expensive.  We now use
McKay's approach of `construction by canonical paths' \cite{MR1606516} to
perform this enumeration in a manner which avoids pairwise 
isomorphism checking between graphs.
\item
A recent new result \cite{MR3311757} using quadratic tangles
tightly constrains certain parameters for principal graphs with annular
multiplicity sequence $*11$ (see Definition \ref{defn:annular-multiplicities}
above). Happily this constraint
applies to many of
the new potential principal graphs we find.
\item
Calegari-Guo \cite{1502.00035} developed a new number theoretic tool for
cylinders (weeds which are stable in the sense of \cite{MR1334479,MR3157990})
similar to the uniform treatment of vines afforded by
\cite{MR2786219}.
We use this in one particular case (and anticipate it would be very useful in any future classifications)
and obtain a generalisation which we apply to a certain `periodic' weed.
\end{enumerate}

The complete list of all non-$A_\infty$, non-$A_\infty^{(1)}$ subfactor standard invariants with index in $(4, 5 \frac{1}{4}]$ appears in Appendix \ref{appendix:SubfactorLandscape}.
In Appendix \ref{appendix:MapOfSubfactors}, we give the current `map of subfactors', showing the overlapping ranges of the classification results to date.

\subsection{The future}
\label{sec:future}
In this section, we suggest some open questions to guide our future exploration of quantum symmetries.

\begin{question}
How far can the classification of small index subfactors go?
\end{question}

Perhaps somewhat
surprisingly, there is not a clear `wall' after which classification becomes
too hard.  Index 6 remains the distant goal; there we know many new phenomena
arise (subfactors not classifiable by their standard invariants
\cite{MR2314611,1309.5354}, as well as infinitely many non-isomorphic finite
depth Bisch-Haagerup subfactors \cite{MR1386923} associated to finite quotients of $\bbZ /
2\bbZ * \bbZ /3\bbZ$). A first step would be $5\frac{1}{3}$, the
largest index of any 4-spoke. Perhaps the next goal should be
$(1+\sqrt{2})^2=3+2\sqrt{2}$, which is the minimum index for an extremal,
reducible subfactor \cite[Corollary 4.6]{MR860811}.

The new combinatorial enumerator seems to be able to look some distance above
our present cut-off of $5 \frac{1}{4}$. However we quickly find graphs which
we don't know how to deal with. 
We present the reader with the following weed
which appears above index $5.27$.
$$
\left(
\bigraph{bwd1v1v1v1p1p1v0x1x0p0x1x0v1x0v1p1v1x0v1p1v1x0v1p1v1x0v1p1v1x0duals1v1v1x2x3v1v1v1v1v1}
\bigraph{bwd1v1v1v1p1p1v1x0x0p0x0x1v1x0p1x0p0x1v1x0x0p0x0x1v0x1p1x0p1x0p0x1v0x1x0x0p0x0x0x1v0x1p1x0p1x0p0x1v0x1x0x0p0x0x0x1v0x1p1x0p1x0p0x1v0x1x0x0p0x0x0x1v0x1p1x0p1x0p0x1duals1v1v1x2x3v1x3x2v1x2x4x3v1x2x4x3v1x2x4x3v1x2x4x3}
\right)
$$
There are many possible avenues for attacking this weed --- a quadratic
tangles approach to quadruple points, number theoretic obstructions based on
the repeating unit in the tail, or analyzing possible connections --- but so
far we have had no success. The absence of a doubly one-by-one connection entry
means we can't use the techniques of Section \ref{sec:DoublyOneByOne} below.

\begin{question}
\label{q:Supertransitivity}
Is there a global bound on supertransitivity for standard invariants above index 4?
\end{question}

Leaving aside the subfactors of index less than 4, high supertransitivity is exceedingly rare.
At this time, the record is the extended Haagerup subfactor \cite{MR2979509} with supertransitivity 7, and the second
highest is the 5-supertransitive Asaeda-Haagerup subfactor. The third is the 4-supertransitive 4442 subfactor at index $3+\sqrt{5}$ \cite{MR3314808}, discovered as part of the project leading up to this paper.

As supertransitivity is a subfactor analog of transitivity of group actions, one hopes by analogy that for sufficiently
large $k$ there are only finitely many non-trivial $k$-supertransitive standard invariants. At this point, we have no evidence that this is not the case for $k \geq 5$.

At present we have two methods for proving bounds on supertransitivity for families of graphs.
First, the number-theoretic results of \cite{MR2786219} show that only finitely many translations of a fixed graph can be principal graphs of subfactors, giving an explicit bound on the supertransitivity. These results are further generalized in \cite{1502.00035} and Section \ref{sec:CG} below, allowing us to eliminate all finite depth translated extensions of a certain graph (\ref{case:10}(9) in Theorem
\ref{thm:Enumerate}), by first proving a bound on supertransitivity in Lemma \ref{lem:t<=63}.

Second, quadratic tangles inequalities sometimes give bounds on supertransitivity, and we see this while eliminating the weeds \ref{case:10}(7) and \ref{case:10}(8) from Theorem
\ref{thm:Enumerate}, in Section \ref{sec:10WithDoubleOneByOne}.
One hopes that this technique can be improved by a deeper understanding of quadratic tangles (see Question \ref{q:Quadratic} below).

\begin{question}
Can Liu's results showing there are only finitely many composites of $A_3$ and $A_4$ be generalized?
\end{question}

Following Bisch-Haagerup's classification of the possible principal graphs of composites of $A_3$ and $A_4$ \cite
{BischHaagerup} (see also \cite{1308.5723}), the one might expect to find an infinite family of finite depth
composite subfactor standard invariants whose principal graphs converge to the $A_3*A_4$ Fuss-Catalan \cite{MR1437496}
principal graph, parallel to the situation at index 4, with the $D_{n+2}^{(1)}$ standard invariants `converging' to 
the $D_\infty$
standard invariant. 

Liu's result \cite{MR3345186} came as quite a surprise, now suggesting that quotients of the free product
Fuss-Catalan standard invariants are also rare.

This is analogous to the situation in Question \ref{q:Supertransitivity}, where we observe that high supertransitivity is rare. A planar algebra $\cP_\bullet$ with supertransitivity $k$ looks like its Temperley-Lieb-Jones subalgebra along with certain extra elements which only appear in the space $\cP_{n,\pm}$ for $n > k$. 
The observation that high supertransitivity is rare can be reformulated as the `difficulty' of adding a large generator to the Temperley-Lieb-Jones algebra. 
Similarly, as any composite planar algebra contains (more or less by definition) the corresponding Fuss-Catalan planar algebra, we could define the \emph{Fuss-Catalan supertransitivity} as the first $n$ so $\cP_{n,\pm}$ contains elements beyond Fuss-Catalan. 
Liu's result can then be interpreted as saying that it is hard to add large generators to a Fuss-Catalan planar algebra.

Is this a general phenomenon? Are there number theoretic, or even algebraic geometric, constraints limiting possible quotients of free products?

\begin{question}
\label{q:Quadratic}
How can we effectively develop the theory of quadratic tangles?
\end{question}

Understanding the representation theory of the annular Temperley-Lieb-Jones algebras led to the well-developed theory of annular tangles \cite{MR1659204,MR1929335}.
As annular tangles together with quadratic tangles generate the entire planar operad, analyzing quadratic tangles systematically is an extremely difficult task.
However, we are rewarded with strong constraints and structure theorems which tightly restrict the structure of subfactor planar algebras \cite{MR2972458,MR3198588,MR3311757}.

Certainly there are more constraints which can be obtained by analyzing higher annular consequences.
A systematic treatment of the space of second annular consequences was given in \cite{1308.5197}, but deeper analysis is needed to extract useful information.

Moreover, we lack a good description of the fusion of low weight representations above index 4, along with the possible quotients which occur in subfactor planar algebras.
Again, we may see tight restrictions which lead to quadratic tangles obstructions.

\begin{question}
\label{q:extension}
Is there a good extension theory for fusion categories?
\end{question}

At this time, we have the theories of $G$-graded extensions of fusion categories for finite groups $G$ \cite{MR2677836}, (de)equivariantization (the notion goes back to \cite{MR1308617,MR1321700}; see also \cite[Section 4]{MR2609644} and \cite{MR3059899}), and exact sequences of fusion categories \cite{MR2863377,MR3161401}.
However, there are certain examples arising from the small index subfactor classification program which behave somewhat like extensions but do not fit into these theories.

The 4442 subfactor at index $3+\sqrt{5}$ \cite{MR3314808} was discovered in exhaustive enumeration of principal graphs.
While this subfactor appears as an equivariantization of Izumi's $3^{\bbZ/2\times \bbZ/2}$ subfactor \cite{IzumiUnpublished}, its even half also resembles a `non-graded' extension of $\operatorname{Rep}(A_4)$.
It has the form $\cC\oplus\cM$ where $\cC = \operatorname{Rep}(A_4)$, and $\cM$ is $\cC$ as a module over itself, while the tensor product structure on $\cM$ is stranger.

\begin{question}
Where do the quadratic categories come from?
\end{question}

A quadratic category has a group $G$ of invertible objects, together with one other orbit of simple objects.
At this point, thanks to \cite{MR1832764,MR2837122,MR3167494}, there is a well-developed theory of quadratic fusion categories using endomorphisms of Cuntz algebras for classification and construction.
These categories include the near group categories \cite{MR1997336} and a possible infinite family generalizing the even half of the exotic Haagerup subfactor \cite{MR1686551}.
Recently, \cite{1501.07324} gave a new construction of the exotic Asaeda-Haagerup subfactor, showing it is related to a quadratic category.

The theory of quadratic categories provides a classification (for a fixed group of invertible objects, and orbit structure for the non-invertible objects) in terms of the finitely many solutions of an explicit system of polynomials. (Analogously, one could attempt directly solving the pentagon equations; there, however, there is a large gauge group, and categories correspond to orbits. In the theory of quadratic categories there is no gauge group.) Nevertheless, constructing quadratic categories by solving these equations leaves something to be desired.
The simple structure of these categories led Evans and Gannon to conjecture that they should arise from conformal field theories, and that there is some unknown underlying construction.

On the planar algebra side, we know two possible skein theoretic approaches to
quadratic categories, based on jellyfish relations
\cite{MR3157990,MR3314808,1308.5197} or Yang-Baxter relations
\cite{1507.04794}. Although both appear promising, Cuntz algebra techniques
have been much more successful to date. We are far from fully understanding this situation.

\begin{question}
Can we find further sources of infinite families of examples?
\end{question}

At this point, almost every known quantum symmetry is related to quantum groups or quadratic categories via known constructions.
In fact, the extended Haagerup subfactor \cite{MR2979509} stands alone as the only quantum symmetry not arising in this way!
While it is certainly important to study this example further (e.g. \cite{1404.3955}), in the hopes that maybe it too is related to quantum groups or quadratic categories, what we really want is new infinite families of examples.

One way to find new families may be to continue the search for quantum symmetries which are `small' by some metric besides the index of the standard invariant, for example rank or global dimension.
For example, we have a full classification of rank 2 fusion categories \cite{MR1981895}, rank 3 pivotal categories \cite{1309.4822}, and a partial classification of rank 4 pseudo-unitary categories \cite{MR3229513}.
There are only finitely many modular categories of a given rank \cite{1310.7050}, and they have been completely classified up to rank 5 \cite{1507.05139}. 

Perhaps it would be interesting to look at certain families of graphs, e.g., spoke graphs.
Are there only finitely many higher spoke graphs with high valence and high supertransitivity? There is good number theoretic evidence that beyond certain families, very few (possibly only finitely many) spoke graphs have an index which is a cyclotomic integer.

\begin{question}
What can we say about hyperfinite $A_3*A_4$ subfactors?
\end{question}
To fully understand non-$A_\infty$ irreducible hyperfinite subfactors with index at most $5\frac{1}{4}$, we would still like to know what $A_3*A_4$ subfactors can exist.
We note that there is a unique $A_3*A_3=D_\infty$ subfactor, coming from the fact that this standard invariant is amenable \cite{MR1278111}.
However, the higher Fuss-Catalan standard invariants \cite{MR1437496} are not amenable \cite{MR1644299}, and we conjecture that $A_3 * A_4$ already exhibits the same unclassifiability phenomenon seen for hyperfinite subfactors with $A_3*D_4$ standard invariant \cite{1309.5354}.
As a proof of concept, it would be interesting to approach $A_3*A_5$ hyperfinite subfactors, which can be obtained by composing subfactors associated to groups.


\subsection{Proof of the main theorem}

The proof of the main theorem splits naturally in many parts, most of which are independent of
each other.  We begin with a combinatorial enumeration that shows that every
principal graph of an irreducible subfactor with index at most $5\frac{1}{4}$ must be
represented by one of a certain list of weeds or one of a certain list of
vines.  This appears as Theorem \ref{thm:Enumerate} below.  This calculation is
closely analogous to the calculation performed in \cite{MR2914056}, but with more advanced combinatorial techniques, described in Section \ref{sec:combinatorics}.

\begin{thm}
\label{thm:Enumerate}
The principal graph of any subfactor with index in $(4,5\frac{1}{4}]$ must either be
\begin{enumerate}[label=(\alph*)]
\item $A_\infty$ (in which case the subfactor is non-amenable)
\item \label{case:reducible} reducible, i.e. there are multiple edges between depths 0 and 1,
\item \label{case:1-supertransitive} exactly 1-supertransitive,
\item \label{case:not-simply-laced} not simply laced,
\item \label{case:11} a translated extension of one of the following
`weeds'with annular multiplicities $*11$,
\begin{enumerate}[label=(\arabic*)]
\item \(\left(\bigraph{bwd1v1v1p1v0x1p1x0p0x1v1x0x0p0x0x1p0x1x0p0x1x0v0x0x1x0p1x0x0x0p0x1x0x0v0x1x0p0x0x1p1x0x0vduals1v1v1x3x2v1x3x2v}, \bigraph{bwd1v1v1p1v0x1p1x0p0x1v0x1x0p0x0x1p1x0x0p0x1x0v1x0x0x0p0x1x0x0p0x0x1x0v0x0x1p1x0x0p0x1x0vduals1v1v1x3x2v1x3x2v}\right)\) 
\item \(\left(\bigraph{bwd1v1v1p1v1x0p0x1p1x0v0x0x1p0x1x0p1x0x0p0x1x0v0x0x0x1p0x1x0x0p1x0x0x0p0x0x1x0p1x0x0x0v1x0x0x0x0p0x0x0x0x1p0x1x0x0x0p0x0x0x1x0p0x0x0x1x0vduals1v1v1x3x2v1x3x2x5x4v}, \bigraph{bwd1v1v1p1v1x0p0x1p1x0v0x0x1p0x1x0p0x1x0p1x0x0v0x0x1x0p0x1x0x0p1x0x0x0p0x0x0x1p1x0x0x0v0x0x0x0x1p1x0x0x0x0p0x0x0x1x0p0x1x0x0x0p0x0x0x1x0vduals1v1v1x3x2v1x3x2x5x4v}\right)\) 
\item \(\left(\bigraph{bwd1v1v1v1p1v1x0p1x0p0x1v0x0x1p0x0x1p1x0x0vduals1v1v1x2v1x3x2}, \bigraph{bwd1v1v1v1p1v1x0p0x1p1x0v0x0x1p0x1x0p1x0x0vduals1v1v1x2v1x3x2}\right)\) 
\item \(\left(\bigraph{bwd1v1v1v1p1v1x0p0x1p0x1v0x0x1p1x0x0p1x0x0p0x1x0v0x0x1x0p1x0x0x0p0x0x0x1vduals1v1v1x2v1x2x4x3v}, \bigraph{bwd1v1v1v1p1v1x0p1x0p0x1v1x0x0p0x0x1p0x1x0p0x0x1v0x0x0x1p0x0x1x0p1x0x0x0vduals1v1v1x2v1x2x4x3v}\right)\) 
\item \(\left(\bigraph{bwd1v1v1v1p1v1x0p0x1p0x1v1x0x0p1x0x0p0x1x0p0x0x1v1x0x0x0p0x0x1x0p0x1x0x0p0x0x0x1vduals1v1v1x2v1x3x2x4v}, \bigraph{bwd1v1v1v1p1v1x0p1x0p0x1v1x0x0p0x1x0p0x0x1p0x0x1v1x0x0x0p1x0x0x0p0x0x1x0p0x1x0x0vduals1v1v1x2v1x3x2x4v}\right)\) 
\item \(\left(\bigraph{bwd1v1v1v1p1v1x0p1x0p0x1v0x1x0p0x0x1p0x0x1p1x0x0v1x0x0x0p0x0x1x0p0x0x0x1p1x0x0x0v0x0x1x0p0x1x0x0p1x0x0x0vduals1v1v1x2v1x2x4x3v1x3x2}, \bigraph{bwd1v1v1v1p1v1x0p0x1p1x0v0x0x1p0x1x0p0x1x0p1x0x0v0x0x0x1p0x0x1x0p1x0x0x0p0x1x0x0v1x0x0x0p0x0x1x0p0x1x0x0p0x0x0x1p1x0x0x0vduals1v1v1x2v1x2x4x3v1x3x2x5x4}\right)\) 
\end{enumerate}
\item \label{case:10} a translated extension of one of the following `weeds' with annular multiplicities $*10$,
\begin{enumerate}[label=(\arabic*)]
\item \(\left(\bigraph{bwd1v1v1v1p1v1x0p0x1v1x1p0x1v0x1vduals1v1v1x2v1x2v}, \bigraph{bwd1v1v1v1p1v1x0p1x0v0x1p0x1p1x0v1x0x0vduals1v1v1x2v1x2x3v}\right)\) 
\item \(\left(\bigraph{bwd1v1v1v1p1v1x0p1x0v1x0v1p1vduals1v1v1x2v1v}, \bigraph{bwd1v1v1v1p1v1x0p0x1v1x0p0x1p1x0v1x0x0p0x1x0vduals1v1v1x2v1x3x2v}\right)\) 
\item \(\left(\bigraph{bwd1v1v1v1p1v1x0p0x1v0x1p1x0p1x1v0x1x0vduals1v1v1x2v1x2x3v}, \bigraph{bwd1v1v1v1p1v1x0p1x0v1x0p1x0p0x1p0x1v1x0x0x0vduals1v1v1x2v1x2x4x3v}\right)\) 
\item \(\left(\bigraph{bwd1v1v1v1p1v1x0p0x1v0x1p1x0p1x1v1x0x0vduals1v1v1x2v1x2x3v}, \bigraph{bwd1v1v1v1p1v1x0p1x0v1x0p0x1p1x0p0x1v0x1x0x0vduals1v1v1x2v1x2x3x4v}\right)\) 
\item \(\left(\bigraph{bwd1v1v1v1p1v1x0p0x1v0x1p1x0p1x0p0x1v0x0x0x1p0x0x1x0vduals1v1v1x2v1x2x4x3v}, \bigraph{bwd1v1v1v1p1v1x0p1x0v1x0p0x1v1x0p0x1vduals1v1v1x2v1x2v}\right)\) 
\item \(\left(\bigraph{bwd1v1v1v1p1v1x0p0x1v1x1v1v1p1v1x0vduals1v1v1x2v1v1x2v}, \bigraph{bwd1v1v1v1p1v1x0p1x0v1x0p0x1v1x1v1v1vduals1v1v1x2v1x2v1v}\right)\) 
\item \(\left(\bigraph{bwd1v1v1v1p1v1x0p0x1v1x0p1x0p0x1p0x1v0x1x0x0p0x0x1x0p0x0x0x1vduals1v1v1x2v1x3x2x4v}, \bigraph{bwd1v1v1v1p1v1x0p1x0v1x0p0x1v0x1p1x0p0x1vduals1v1v1x2v1x2v}\right)\) 
\item \(\left(\bigraph{bwd1v1v1v1p1v1x0p1x0v1x0p0x1v1x0p0x1p1x0p0x1v1x0x0x0p0x0x0x1p0x0x1x0vduals1v1v1x2v1x2v1x3x2}, \bigraph{bwd1v1v1v1p1v1x0p0x1v1x0p0x1p0x1p1x0v1x0x0x0p0x1x0x0p0x0x1x0p0x0x0x1v1x0x0x0p0x0x1x0p0x0x0x1p0x0x1x0p1x0x0x0p0x0x0x1p0x1x0x0vduals1v1v1x2v1x2x4x3v1x2x3x5x4x7x6}\right)\) 
\item \(\left(\bigraph{bwd1v1v1v1p1v1x0p1x0v1x0p0x1v1x0p0x1p1x0p0x1v0x0x0x1p0x0x1x0v1x0p0x1p1x0p0x1v0x1x0x0p0x0x1x0vduals1v1v1x2v1x2v2x1v1x2}, \bigraph{bwd1v1v1v1p1v1x0p0x1v1x0p0x1p0x1p1x0v1x0x0x0p0x1x0x0p0x0x1x0p0x0x0x1v0x0x1x0p0x0x0x1p0x0x1x0p1x0x0x0p0x0x0x1p0x1x0x0v0x0x0x1x0x0p0x1x0x0x0x0p1x0x0x0x0x0p0x0x0x0x0x1v0x0x0x1p1x0x0x0p0x0x1x0p0x1x0x0p0x1x0x0p1x0x0x0p0x0x0x1p0x0x1x0vduals1v1v1x2v1x2x4x3v1x2x4x3x6x5v2x1x4x3x6x5x8x7}\right)\) 
\end{enumerate}
\item \label{case:4-spoke} a translated extension of one of the following stable graphs with annular multiplicities $*20$,
\begin{multicols}{2}
\begin{enumerate}[label=(\arabic*)]
\item \(\left(\bigraph{bwd1v1v1p1p1v0x0x1vduals1v1v1}, \bigraph{bwd1v1v1p1p1v0x0x1vduals1v1v1}\right)\) 
\item \(\left(\bigraph{bwd1v1v1p1p1v1x0x0p0x1x0p0x0x1vduals1v1v1x3x2}, \bigraph{bwd1v1v1p1p1v1x0x0p0x1x0p0x0x1vduals1v1v1x3x2}\right)\) 
\item \(\left(\bigraph{bwd1v1v1v1p1p1v1x0x0vduals1v1v1x3x2v}, \bigraph{bwd1v1v1v1p1p1v1x0x0vduals1v1v1x3x2v}\right)\) 
\item \(\left(\bigraph{bwd1v1v1p1p1v0x1x0p0x0x1vduals1v1v1x2}, \bigraph{bwd1v1v1p1p1v0x1x0p0x0x1vduals1v1v1x2}\right)\) 
\item \(\left(\bigraph{bwd1v1v1p1p1v1x0x0p0x1x0p0x0x1vduals1v1v1x2x3}, \bigraph{bwd1v1v1p1p1v1x0x0p0x1x0p0x0x1vduals1v1v1x2x3}\right)\) 
\item \(\left(\bigraph{bwd1v1v1v1p1p1v0x1x0vduals1v1v1x2x3v}, \bigraph{bwd1v1v1v1p1p1v1x0x0vduals1v1v1x2x3v}\right)\) 
\item \(\left(\bigraph{bwd1v1v1p1p1v1x0x0p0x0x1vduals1v1v2x1}, \bigraph{bwd1v1v1p1p1v1x0x0p0x0x1vduals1v1v2x1}\right)\) 
\item \(\left(\bigraph{bwd1v1v1v1p1p1v0x1x0p0x0x1vduals1v1v1x3x2v}, \bigraph{bwd1v1v1v1p1p1v0x1x0p0x0x1vduals1v1v1x3x2v}\right)\) 
\item \(\left(\bigraph{bwd1v1v1v1p1p1v0x1x0p0x0x1vduals1v1v1x2x3v}, \bigraph{bwd1v1v1v1p1p1v1x0x0p0x0x1vduals1v1v1x2x3v}\right)\) 
\item \(\left(\bigraph{bwd1v1v1v1p1p1v1x0x0p0x1x0p0x0x1vduals1v1v1x3x2v}, \bigraph{bwd1v1v1v1p1p1v1x0x0p0x1x0p0x0x1vduals1v1v1x3x2v}\right)\) 
\item \(\left(\bigraph{bwd1v1v1v1p1p1v1x0x0p0x1x0p0x0x1vduals1v1v1x2x3v}, \bigraph{bwd1v1v1v1p1p1v1x0x0p0x1x0p0x0x1vduals1v1v1x2x3v}\right)\) 
\end{enumerate}
\end{multicols}
\item \label{case:vines} or a translation of one of the `vines' listed in
Section \ref{sec:vines}.
\end{enumerate}
\end{thm}

The proof of this appears in Section \ref{sec:Enumerate-proof}, after we describe the underlying algorithm in Section \ref{sec:orderly} and our implementations in Section \ref{sec:Implementations}.
Unfortunately this classification looks at first quite
discouraging compared with the corresponding classification up to index 5.

First, we give the folklore result which deals with everything in case \ref{case:reducible}, by describing all reducible subfactors in the index range $(4,3+2\sqrt{2})$.
\begin{lem}
\label{lem:reducible}
The only reducible subfactor standard invariants with index in $(4,3+2\sqrt{2})$ are the $A_\infty^{(1)}$ standard invariants at every such index. Moreover these are all non-spherical, and can be obtained as perturbations of the spherical index 4 $A_\infty^{(1)}$ standard invariant.
\end{lem}
\begin{proof}
By \cite{MR1887878} (see also \cite[Section 3]{MR3178106}), a non-spherical reducible subfactor standard invariant can be \emph{perturbed} without altering the principal graph. Amongst these perturbations, there is a unique one with minimal index \cite{MR829381,MR1027496}, and this perturbation is necessarily spherical. By \cite[Corollary 4.6]{MR860811}, a spherical reducible subfactor standard invariant with index in $(4,8)$ must have index $(1+2\cos(\pi/n))^2$, for $n\geq 3$. The least such index is $3+2\sqrt{2}$.

Thus a reducible subfactor standard invariant $\cP_\bullet$ with index in $(4,3+2\sqrt{2})$ must be a pertubation of some reducible spherical subfactor standard invariant $\cQ_\bullet$  at index 4. The only such standard invariants are $A_{2n-1}^{(1)}$ and $A_\infty^{(1)}$. The $A_{2n-1}^{(1)}$ standard invariants, being finite depth, admit no non-spherical perturbations \cite[Proposition 10.4]{MR1642584}. Thus
$\cQ_\bullet$ is the unique $A_\infty^{(1)}$ standard invariant, and necessarily $\cP_\bullet$ also has  principal graph $A_\infty^{(1)}$. 
\end{proof}

We next note that case \ref{case:1-supertransitive} from Theorem \ref{thm:Enumerate}, the 1-supertransitive principal graphs, has already been
completely classified for index less than $6$ in \cite{MR3306607}. Case \ref{case:not-simply-laced} from Theorem \ref{thm:Enumerate}, the non-simply laced graphs, can be handled as follows.
\begin{lem}
\label{lem:simply-laced}
There are no subfactors with index in $(4,5.3]$ whose principal graphs are not simply-laced. 
\end{lem}
\begin{proof}
First, we note
\begin{align*}
\left\|\smallbigraph{gbg3}\right\|^2 & = 9, &
\left\|\smallbigraph{gbg2v2}\right\|^2 & = 8, &
\left\|\smallbigraph{gbg2v1v2}\right\|^2 & > 6.56, &
\left\|\bigraph{gbg1v1p2}\right\|^2 & = 6,
\end{align*}
\vspace{-6mm}
\begin{align*}
\left\|\smallbigraph{gbg1v2v1}\right\| ^2 & > 5.82, & 
\left\|\bigraph{gbg1v1p1v2x0}\right\|^2 & > 5.56, &
\left\|\smallbigraph{gbg1v1v1v2}\right\|^2 & > 5.3.
\end{align*}
The only non-simply laced graphs not containing one of these as a subgraph are
$\smallbigraph{gbg2}$, $\smallbigraph{gbg1v2}$,  and
$\smallbigraph{gbg1v1v2}$. The first is at index 4. In the other two cases,
the dimension of the non-trivial object in the even half is not 1 or $\tau =
\frac{1+\sqrt{5}}{2}$, so these graphs cannot be principal graphs of
subfactors by \cite{MR1981895}. More generally, we note that \cite{MR1237144}
shows that the graphs $\beta_{2n+1}$ (that is, the $2n-2$-translate of
$\smallbigraph{gbg1v2}$) are not principal graphs of subfactors by showing
that the fusion rules must have coefficients of the form $\frac{2k+1}{2}$ for
$k\in \bbN_{\geq 0}$.
\end{proof}

The remaining cases \ref{case:11}, \ref{case:10}, and \ref{case:4-spoke} constitute `weeds', i.e. translated extensions
of some finite set of possibilities, and the cases in \ref{case:vines} are `vines', i.e. translations of some finite set
of possibilities.

While there are significantly more vines than in previous classifications, this itself is no problem, as we by
now have a uniform and rather efficient mechanism for dealing with vines.
This is described in Theorem \ref{thm:survivors}, and all the vines from part \ref{case:vines} are
analyzed in Section \ref{sec:vines}.

The $*11$ graphs can be dealt with by the techniques of \cite{MR3311757}, and
we prove the requisite inequalities for all but one of these in Section
\ref{sec:11}.  The remaining $*11$ graph (case \ref{case:11}(5) in Theorem
\ref{thm:Enumerate}) is treated in Section \ref{sec:11WithDoubleOneByOne}, and
requires additional analysis using doubly one-by-one connection entries, since
it has undetermined relative dimensions.

Nearly all the $*10$ graphs have branch factor not equal to $1$, and can be
dealt with by the techniques of \cite{MR2972458,MR2902285}.
We do this in
Section \ref{sec:10}.  Two graphs (namely \ref{case:10}(7) and \ref{case:10}(8) in Theorem
\ref{thm:Enumerate}) require additional arguments in Section
\ref{sec:10WithDoubleOneByOne} using doubly one-by-one connection entries. 
We note that both these exceptional $*10$ weeds require using Morrison's hexagon obstruction \cite{MR3210716} in an essential way.
Of the remaining $*10$ weeds, we get one candidate graph (from the first graph in either \ref{case:10}(3) or \ref{case:10}(4) from Theorem \ref{thm:Enumerate})
which is eliminated using the formal codegree obstruction (see Section
\ref{sec:FormalCodegrees}).

Additionally, there are three exceptional $*10$ graphs, with branch factor
$1$.  The first (namely \ref{case:10}(5) in Theorem \ref{thm:Enumerate}) is
the depth 2 truncation of one of the weeds ruled out in \cite{MR2902285}, and we
give an easy argument based on stability to
rule it out in Theorem
\ref{thm:NoB}. The second (namely \ref{case:10}(6) in Theorem
\ref{thm:Enumerate}), which begins with an initial hexagon, has been called
the `AMP spider', and was ruled out by purely number theoretic methods in
\cite{1502.00035}. The third (namely \ref{case:10}(9) in Theorem
\ref{thm:Enumerate}) appears to be a truncation of an infinite periodic graph,
rather reminiscent of the $A_3 * A_4$ Fuss-Catalan principal graph. It is
isomorphic to
$$
\cA_{(0)} = 
\left(
\bigraph{bwd1v1v1v1p1v0x1p0x1v1x0p0x1v0x1p1x0p1x0p0x1v1x0x0x0p0x0x1x0v0x1p1x0p1x0p0x1v1x0x0x0p0x0x1x0duals1v1v1x2v1x2v2x1v1x2}
\,,\,
\bigraph{bwd1v1v1v1p1v1x0p0x1v1x0p1x0p0x1p0x1v1x0x0x0p0x1x0x0p0x0x1x0p0x0x0x1v1x0x0x0p1x0x0x0p0x1x0x0p0x0x1x0p0x0x1x0p0x0x0x1v0x1x0x0x0x0p0x0x1x0x0x0p0x0x0x0x1x0p0x0x0x0x0x1v1x0x0x0p1x0x0x0p0x1x0x0p0x1x0x0p0x0x1x0p0x0x1x0p0x0x0x1p0x0x0x1duals1v1v1x2v3x2x1x4v6x2x4x3x5x1v4x6x7x1x8x2x3x5}
\right).
$$
We rule out this weed with a several stage approach which occupies Section \ref{sec:remaining-r=1}.
\begin{enumerate}[label=(\arabic*)]
\item First, we find a doubly one-by-one entry in the
connection, which gives us the value of $q$ as a function of the translation
$t$. 
\item We next run the enumerator to see that any extension is either
finite, or begins with a basic building block, called $\cB$. 
\item We then
show all further extensions are either finite, or periodic with repeating unit $\cB$. This
improved method uses the \emph{tail enumerator} described in Section \ref{sec:Implementations},
a new mode of the graph
enumeration program, adapted to work with only with the repeating unit of a
periodic graph, without looking at the `head'. 
\end{enumerate}

In Lemma \ref{lem:no-finite-extensions} below, we are able to use a variation of the argument in \cite{1502.00035} to take
care of all these finite extensions simultaneously. We show the corresponding indices are not
totally real cyclotomic integers, and thus cannot be indices of subfactors.
This is the first time number theory has been used to rule out all finite
extensions of a particular weed! We deal with the infinite graph by showing
that it eventually has bimodule dimensions less than 1, whatever the supertransitivity.

In every case in \ref{case:4-spoke}, both graphs are identical `4-spoke
graphs'.  The existence of a bi-unitary connection on such graph pairs was
determined by \cite{1304.5907}, and it is relatively little work after that to
complete the classification of subfactors for 4-spoke principal graphs up to
index $5\frac{1}{4}$. We do this in  Section \ref{sec:4-spokes}. (Note that
the index of the infinite 4-spoke is $5\frac{1}{3}$, so we are not yet done
with 4-spokes!)

This concludes the proof of the main theorem. \qed

\vspace{5mm}

We note that the first two of the $*11$ weeds and the first three of the $*10$ weeds are also stable, so in
principle, these could be eliminated using the analysis of Calegari-Guo in
\cite{1502.00035}. All extensions of stable weeds must be stable and finite by
\cite{MR1334479,MR3157990}. Thus by \cite{MR1266785,MR2183279} and Theorem 1.1
of \cite{1502.00035}, for any stable weed, there are at most finitely many
translated extensions which could be principal graphs of subfactors. At this
point it seems simpler to use the planar algebraic obstructions, however. In the
future it may be possible to handle stable graphs just as easily as we handle
vines today, by automating the arguments used in Sections 6 and 7 of \cite{1502.00035}.

In the organization of the proof, we see that after the combinatorial
enumeration, all the other arguments are essentially independent. We have
decided to arrange them according to the novelty of the methods. The really
interesting stuff comes first: the new combinatorics of the enumerator in
Section \ref{sec:combinatorics}, our new approach to infinite periodic graphs
in Section \ref{sec:r=1}, Penneys' obstructions to $*11$ graphs in Section
\ref{sec:BranchFactorInequalities}, 
applications of Morrison's hexagon obstruction in Section \ref{sec:10WithDoubleOneByOne}, 
and eliminating one candidate via formal
codegrees in Section \ref{sec:FormalCodegrees}. We defer Section
\ref{sec:4-spokes} on 4-spokes and Section \ref{sec:vines} on vines to the end of
the article; these sections are essentially applications of already well-known
arguments.

\section{Better combinatorics for graph enumeration}
\label{sec:combinatorics}

In this section we describe our new technique for enumerating principal
graphs, based on McKay's method of construction by canonical paths.
The classification statement we prove using this technique has appeared above as Theorem \ref{thm:Enumerate}. Its proof appears below, in Section \ref{sec:Enumerate-proof}.
 
We begin with the precise definition of the objects we enumerate.

\begin{defn}
A \emph{principal graph pair (PGP)} is a tuple $(\Gamma_+, \Gamma_-, \depth, \overline{\,\cdot\,}, n)$ where
\begin{itemize}
\item $\Gamma_\pm$ are finite bipartite graphs,
\item $\depth$ is a graph homomorphism from $\Gamma_\pm \to \bbN$, the graph with one edge between $n$ and $n+1$ for all $n\in \bbN$, such that $\depth^{-1}(0)$ intersects every connected component of $\Gamma_\pm$,
\item $\overline{\,\cdot\,}$ is an involution on the vertices of $\Gamma_+ \sqcup \Gamma_-$, which
\begin{enumerate}[label=(\arabic*)]
\item takes an even depth vertex  on $\Gamma_\pm$ to a vertex at the same depth on $\Gamma_\pm$, and
\item takes an odd depth vertex on $\Gamma_\pm$ to a vertex at the same depth on $\Gamma_\mp$,
\end{enumerate}
\item and $n$, called the `working depth', is a non-negative integer which is either equal to the maximum distance of a vertex from the base vertex, or to the maximum distance plus one,
\end{itemize}
satisfying the following constraints, which are described below:
\begin{itemize}
\item 
the PGP associativity constraint,
\item
the PGP triple point obstruction, and
\item
the PGP duality constraint.
\end{itemize}
\end{defn}

There is an obvious notion of isomorphism of PGPs.

The \emph{index} of a PGP is $\max\{\lambda_+^2, \lambda_-^2\}$, where $\lambda_\pm$ is the Frobenius-Perron
eigenvalue for $\Gamma_\pm$, that is, the largest eigenvalue of the adjacency matrix.

It is important to note here that we are not allowing multiple edges between vertices, so we only see `simply-laced' principal graphs. 
This is not an essential restriction, but it makes the implementations so much simpler that it is worthwhile having to deal with non simply-laced principal graphs separately.

\begin{defn}
If $\nbhd(v)$ denotes the set of neighbours of a vertex $v$, \emph{associativity between vertices $v$ and $w$} is the condition
\begin{equation}
\label{eq:associativity}
\left| \overline{\nbhd(v)} \cap \nbhd(\overline{w}) \right| = \left| \nbhd(\overline{v}) \cap \overline{\nbhd(w)} \right|.
\end{equation}
This condition is symmetric in $v$ and $w$, and trivially satisfied unless $\depth(v)$ and $\depth(w)$ differ by $-2$, $0$, or $2$.
The \emph{PGP associativity constraint} is associativity between all pairs of vertices $v$ and $w$ such that at least one of $v$ or $w$ is at depth $n-2$ or less.\footnote{
Since we may subsequently add vertices at the working depth $n$, we cannot require this 
identity to hold for pairs of vertices both at depth at least $n-1$, because later vertices at depth $n$ may change
either side of the equation. When $v$ is at depth at most $n-2$, on the other hand, the sizes of the sets in the above equation for a vertex $w$, at any depth, will not change when further vertices are introduced.
Nevertheless, there is an
inequality one may impose for pairs of vertices at depth $n-1$, which we address later.
} 
\end{defn}

\begin{defn}
We first define the \emph{combinatorial dimension relation}, on the set of vertices of a PGP $\Gamma$. This is the weakest equivalence relation such that
\begin{itemize}
\item $v \sim w$ if $v = \overline{w}$, and
\item $v \sim w$ whenever $v$ and $w$ have depth at most $n-2$, and there is a bijection $\alpha: \nbhd(v) \to \nbhd(w)$ such that $u \sim \alpha(u)$ for all $u \in \nbhd(v)$.
\end{itemize}
(Clearly if $\Gamma$ is the principal graph of a subfactor, and $v \sim w$, then $\dim(v) = \dim(w)$.)

A PGP $\Gamma$ satisfies the \emph{PGP triple point obstruction} if it satisfies Ocneanu's triple point obstruction \ref{fact:TriplePointObstruction} for all pairs of triple points with depth at most $n-2$, replacing the condition that the bijection $\beta$ preserves dimensions with the condition that  $\beta(u)\sim u$ for all $u$.
\end{defn}

\begin{remark}
The classification to index 5 only used the triple point obstruction for initial triple points \cite{MR2914056}.
Our use of the combinatorial dimension relation allows us to apply the triple point obstruction at non-initial other triple points, which is an essential improvement, without which the algorithm described here would not terminate.
\end{remark}

The \emph{PGP duality constraint} is identical to the duality constraint \ref{fact:DualityConstraint}.

\begin{facts}
\mbox{}
\begin{enumerate}[label=(\arabic*)]
\item
No extension of a PGP can satisfy the associativity constraint \ref{fact:AssociativityConstraint} unless the PGP satisfies the PGP associativity constraint.
\item
No extension of a PGP can satisfy Ocneanu's triple point obstruction \ref{fact:TriplePointObstruction} unless the PGP satisfies the PGP triple point obstruction.
\item
No extension of a PGP can satisfy the duality constraint \ref{fact:DualityConstraint} unless the PGP satisfies the PGP duality constraint.
\end{enumerate}
\end{facts}

\subsection{Generation by canonical construction paths for principal graphs}
\label{sec:orderly}

\newcommand{\result}[1]{\widehat{#1}}

We now review McKay's method of enumeration by construction by canonical paths
\cite{MR1606516}.
Although our description here is a very close parallel of his, we have
slightly specialized his framework, and allowed ourselves the use of slightly
more sophisticated mathematical language (particularly of groupoids). We hope
that our exposition is accessible, and simultaneously explains the general
picture and the particular instance used in this paper. McKay's
running example is triangle free graphs, while our running example will be
principal graph pairs (PGPs). Throughout this section, the general method is
described in the main text, and the specialization to the PGP example is
illustrated in highlighted boxes.

The general setup for McKay's method is that we have a (countable)
groupoid of combinatorial objects $\cO$, and we would like to enumerate the
isomorphism classes.
\begin{boxedexample*}
The groupoid $\cO$ is all PGPs with index bounded by some fixed constant $L$, with PGP isomorphisms.
\end{boxedexample*}

Our initial plan is to come up with a finite collection of seed objects and
generating steps so that a representative of each class can be reached from a
seed in a finite number of steps. This plan ensures that we can enumerate all
isomorphism classes. The problem, of course, is that we may produce many
representatives of the same class.

\begin{boxedexample}
For example, any PGP $\Gamma$ with index at most $L$ may be reached from the trivial PGP $(\eset, \eset, id, 0)$ using finitely many applications of these two generating steps.
\begin{enumerate}[label=(\arabic*)]
\item
We can increase the working depth at which we add vertices.
\item
We may join new vertices at the working depth $n$ to vertices at the previous depth.
The details differ depending on whether $n$ is even or odd.
\begin{enumerate}[label=(\alph*)]
\item
If $n$ is odd, given two sets of vertices $S_+ \subset V(\Gamma_+)$  and $S_- \subset V(\Gamma_-)$ at depth $n-1$, we
may add a new vertex
$v$ to $\Gamma_+$ and a new vertex $\bar{v}$  to $\Gamma_-$, both at depth $n$, joined to the vertices of $S_+$ and
$S_-$ respectively.
\item
If $n$ is even, there are two ways to add vertices to either $\Gamma_+$ or $\Gamma_-$.
\begin{enumerate}[label=(\roman*)]
\item
Given a set of vertices $S \subset V(\Gamma_\pm)$ at depth $n-1$, we may add a new self dual
vertex
$v$  at depth $n$  to $\Gamma_\pm$, joined by an edge to each vertex of $S$.
\item
Given two sets of vertices $S_1,S_2 \subset V(\Gamma_\pm)$ at depth $n-1$, we may add two new vertices
$v$ and $\bar{v}$ at depth $n$ to $\Gamma_\pm$, joined by edges to $S_1$ and $S_2$ respectively.
\end{enumerate}
\end{enumerate}
\end{enumerate}
In fact, there are restrictions on when these steps may be applied (ensuring associativity and staying below the index limit). We postpone discussing these until later in this section.
\end{boxedexample}

In particular, there are three ways in which this process would produce representatives in the same isomorphism class.
\begin{enumerate}[label=(\arabic*)]
\item 
If two generating steps starting from the same object are equivalent under the automorphism group of that object, the results will be isomorphic.
\item
Two inequivalent generating steps applied to the same object can yield isomorphic objects.
\item
Starting with two non-isomorphic objects and applying a generating step can result in isomorphic objects.
\end{enumerate}

\begin{boxedexample}
\begin{enumerate}[label=(\arabic*)]
\item
Adding a vertex in two different locations to the same graph will give isomorphic results, if those locations are equivalent under the automorphism group action.
$$
\begin{tikzpicture}
\tikzset{grow=right,level distance=130pt}
\tikzset{every tree node/.style={draw,fill=white,rectangle,rounded corners,inner sep=2pt}}
\Tree
[.\node{$\!\!\begin{array}{c}\bigraph{bwd1v1p1duals1v1x2}\\\bigraph{bwd1v1p1duals1v1x2}\end{array}\!\!$};
	[.\node{$\!\!\begin{array}{c}\bigraph{bwd1v1p1v1x0duals1v1x2}\\\bigraph{bwd1v1p1v1x0duals1v1x2}\end{array}\!\!$};]
	[.\node{$\!\!\begin{array}{c}\bigraph{bwd1v1p1v0x1duals1v1x2}\\\bigraph{bwd1v1p1v0x1duals1v1x2}\end{array}\!\!$};]
]
\end{tikzpicture}
$$ 

\item
Two inequivalent ways of adding vertices to the same graph pair may yield isomorphic results.
$$
\begin{tikzpicture}
\tikzset{grow=right,level distance=130pt}
\tikzset{every tree node/.style={draw,fill=white,rectangle,rounded corners,inner sep=2pt}}
\Tree
[.\node{$\!\!\begin{array}{c}\bigraph{bwd1v1v1p1p1v1x0x0p1x0x0p0x1x0p0x1x0p0x0x1p0x0x1v0x0x1x0x0x0duals1v1v3x5x1x6x2x4}\\\bigraph{bwd1v1v1p1p1v1x0x0p1x0x0p0x1x0p0x1x0p0x0x1p0x0x1v0x0x1x0x0x0duals1v1v3x5x1x6x2x4}\end{array}\!\!$};
	[.\node{$\!\!\begin{array}{c}\bigraph{bwd1v1v1p1p1v1x0x0p1x0x0p0x1x0p0x1x0p0x0x1p0x0x1v0x1x0x0x0x0p0x0x1x0x0x0duals1v1v3x5x1x6x2x4}\\\bigraph{bwd1v1v1p1p1v1x0x0p1x0x0p0x1x0p0x1x0p0x0x1p0x0x1v0x1x0x0x0x0p0x0x1x0x0x0duals1v1v3x5x1x6x2x4}\end{array}\!\!$};]
	[.\node{$\!\!\begin{array}{c}\bigraph{bwd1v1v1p1p1v1x0x0p1x0x0p0x1x0p0x1x0p0x0x1p0x0x1v0x0x1x0x0x0p0x0x0x0x0x1duals1v1v3x5x1x6x2x4}\\\bigraph{bwd1v1v1p1p1v1x0x0p1x0x0p0x1x0p0x1x0p0x0x1p0x0x1v0x0x1x0x0x0p0x0x0x0x0x1duals1v1v3x5x1x6x2x4}\end{array}\!\!$};]
]
\end{tikzpicture}
$$
These extensions are not equivalent:
in the first case, but not the second, the new vertices are connected to vertices which are dual to vertices at distance 3 from the old vertices at the working depth.

The resulting PGPs are equivalent by permuting the vertices at successive depths in each graph in the lower pair by
$(1), (1), (1), (231), (436512), (12)$.

\item
Adding vertices to two inequivalent graphs pairs may yield isomorphic results.
\begin{align*}
\begin{tikzpicture}
\tikzset{grow=right,level distance=130pt}
\tikzset{every tree node/.style={draw,fill=white,rectangle,rounded corners,inner sep=2pt}}
\Tree
[.\node{$\!\!\begin{array}{c}\bigraph{bwd1v1p1duals1v2x1}\\\bigraph{bwd1v1p1duals1v2x1}\end{array}\!\!$};
	[.\node{$\!\!\begin{array}{c}\bigraph{bwd1v1p1p1duals1v2x1x3}\\\bigraph{bwd1v1p1p1duals1v2x1x3}\end{array}\!\!$};]
]
\end{tikzpicture}
\\
\begin{tikzpicture}
\tikzset{grow=right,level distance=130pt}
\tikzset{every tree node/.style={draw,fill=white,rectangle,rounded corners,inner sep=2pt}}
\Tree
[.\node{$\!\!\begin{array}{c}\smallbigraph{bwd1v1duals1v1}\\\smallbigraph{bwd1v1duals1v1}\end{array}\!\!$};
	[.\node{$\!\!\begin{array}{c}\bigraph{bwd1v1p1p1duals1v1x3x2}\\\bigraph{bwd1v1p1p1duals1v1x3x2}\end{array}\!\!$};]
]
\end{tikzpicture}
\end{align*}

\end{enumerate}
\end{boxedexample}

McKay's method of construction by canonical paths avoids these problems. First,
however, we need to endow our groupoid of objects with the
following  pieces of extra structure.

\begin{defn}
A McKay groupoid is a countable groupoid $\cO$  with the extra structure
$(\ell, U, L, \result{\cdot}, (\cdot)^{-1}, \phi)$ described in conditions
\ref{condition:level}-\ref{condition:choice-function} below.
\end{defn}

\begin{enumerate}[label=(C\arabic*)]
\item
\label{condition:level}
The groupoid $\cO$ should be graded by $\mathbb{N}$; that is, every object $o\in\cO$ has a isomorphism invariant `level', which we write as $\ell(o)$.

\begin{boxedexample*}
For PGPs, we use the following slightly complicated level function. Let $a$ be the working depth of $o$, let $b$ be the number of self-dual vertices in $o$, and let $c$ be the number of pairs of dual vertices in $o$. Then $\ell(o) = a + b + c$.
$$
\ell
\left(\,
\begin{tikzpicture}[scale = .5, baseline=-.6cm]
	\draw[fill=white, rounded corners,inner sep=2pt] (-.5,-3) rectangle (6.5,1);
	\draw[fill] (0,0) circle (0.05);
	\draw (0.,0.) -- (1.,0.);
	\draw[fill] (1.,0.) circle (0.05);
	\draw (1.,0.) -- (2.,0.);
	\draw[fill] (2.,0.) circle (0.05);
	\draw (2.,0.) -- (3.,0.);
	\draw[fill] (3.,0.) circle (0.05);
	\draw (3.,0.) -- (4.,-0.5);
	\draw (3.,0.) -- (4.,0.5);
	\draw[fill] (4.,-0.5) circle (0.05);
	\draw[fill] (4.,0.5) circle (0.05);
	\draw (4.,-0.5) -- (5.,-0.5);
	\draw (4.,0.5) -- (5.,0.5);
	\draw[fill] (5.,-0.5) circle (0.05);
	\draw[fill] (5.,0.5) circle (0.05);
	\draw (5.,-0.5) -- (6.,-0.5);
	\draw (5.,0.5) -- (6.,0.5);
	\draw[fill] (6.,-0.5) circle (0.05);
	\draw[fill] (6.,0.5) circle (0.05);
	\draw[red, thick] (0.,0.) -- +(0,0.333333) ;
	\draw[red, thick] (2.,0.) -- +(0,0.333333) ;
	\draw[red, thick] (4.,-0.5) -- +(0,0.333333) ;
	\draw[red, thick] (4.,0.5) -- +(0,0.333333) ;
	\draw[red, thick] (6.,-0.5) to[out=135,in=-135] (6.,0.5);
	\draw[fill] (0,-2) circle (0.05);
	\draw (0.,-2.) -- (1.,-2.);
	\draw[fill] (1.,-2.) circle (0.05);
	\draw (1.,-2.) -- (2.,-2.);
	\draw[fill] (2.,-2.) circle (0.05);
	\draw (2.,-2.) -- (3.,-2.);
	\draw[fill] (3.,-2.) circle (0.05);
	\draw (3.,-2.) -- (4.,-2.5);
	\draw (3.,-2.) -- (4.,-1.5);
	\draw[fill] (4.,-2.5) circle (0.05);
	\draw[fill] (4.,-1.5) circle (0.05);
	\draw (4.,-2.5) -- (5.,-2.5);
	\draw (4.,-2.5) -- (5.,-1.5);
	\draw[fill] (5.,-2.5) circle (0.05);
	\draw[fill] (5.,-1.5) circle (0.05);
	\draw[red, thick] (0.,-2.) -- +(0,0.333333) ;
	\draw[red, thick] (2.,-2.) -- +(0,0.333333) ;
	\draw[red, thick] (4.,-2.5) -- +(0,0.333333) ;
	\draw[red, thick] (4.,-1.5) -- +(0,0.333333) ;
\end{tikzpicture}
\,\right)
=6+8+5=19
$$
\end{boxedexample*}

\item
\label{condition:upper-objects}
Associated to an object $o \in \cO$ we define a new set, the `upper objects'
$U(o)$ for $o$. 
The elements of $U(o)$ consist of an object in $\cO$ along with certain extra
data. If $u \in U(o)$ is an upper object for $o$, we write $\result{u}$ for
the resulting object in $\cO$. We must have $\ell(\result{u}) > \ell(o)$.
Essentially, the upper object $u$ records the generating step that produces
$\result{u}$ from $o$.

\begin{boxedexample}
We have three types of upper objects, called `increasing the depth' $I$, `adding a self-dual vertex' $S^+(V)$, and `adding a pair of dual vertices' $P^+(V_1, V_2)$.
\begin{itemize}
\item 
There is only one way to increase the depth, and the underlying object $\result{I}$ is just $o$ with the working depth incremented by one.
We denote $\result{I}$ by adding a white shaded vertex at the next depth, and we denote $I$ by drawing a dotted circle around this white vertex.
$$
o=
\begin{tikzpicture}[scale = .5, baseline=-.6cm]
	\draw[fill=white, rounded corners,inner sep=2pt] (-.5,-3) rectangle (2.5,1);
	\draw[fill] (0,0) circle (0.05);
	\draw (0.,0.) -- (1.,0.);
	\draw[fill] (1.,0.) circle (0.05);
	\draw (1.,0.) -- (2.,-0.5);
	\draw (1.,0.) -- (2.,0.5);
	\draw[fill] (2.,-0.5) circle (0.05);
	\draw[fill] (2.,0.5) circle (0.05);
	\draw[red, thick] (0.,0.) -- +(0,0.333333) ;
	\draw[red, thick] (2.,-0.5) -- +(0,0.333333) ;
	\draw[red, thick] (2.,0.5) -- +(0,0.333333) ;
	\draw[fill] (0,-2) circle (0.05);
	\draw (0.,-2.) -- (1.,-2.);
	\draw[fill] (1.,-2.) circle (0.05);
	\draw (1.,-2.) -- (2.,-2.5);
	\draw (1.,-2.) -- (2.,-1.5);
	\draw[fill] (2.,-2.5) circle (0.05);
	\draw[fill] (2.,-1.5) circle (0.05);
	\draw[red, thick] (0.,-2.) -- +(0,0.333333) ;
	\draw[red, thick] (2.,-2.5) -- +(0,0.333333) ;
	\draw[red, thick] (2.,-1.5) -- +(0,0.333333) ;
\end{tikzpicture}
\,;
\qquad
I=
\begin{tikzpicture}[scale = .5, baseline=-.6cm]
	\draw[fill=white, rounded corners,inner sep=2pt] (-.5,-3) rectangle (3.5,1);
	\draw[fill] (0,0) circle (0.05);
	\draw (0.,0.) -- (1.,0.);
	\draw[fill] (1.,0.) circle (0.05);
	\draw (1.,0.) -- (2.,-0.5);
	\draw (1.,0.) -- (2.,0.5);
	\draw[fill] (2.,-0.5) circle (0.05);
	\draw[fill] (2.,0.5) circle (0.05);
	\draw[red, thick] (0.,0.) -- +(0,0.333333) ;
	\draw[red, thick] (2.,-0.5) -- +(0,0.333333) ;
	\draw[red, thick] (2.,0.5) -- +(0,0.333333) ;
	\draw[thick, dotted] (3,0) circle (.35cm);
	\draw[fill=white] (3.,0.) circle (0.1);
	\draw[fill] (0,-2) circle (0.05);
	\draw (0.,-2.) -- (1.,-2.);
	\draw[fill] (1.,-2.) circle (0.05);
	\draw (1.,-2.) -- (2.,-2.5);
	\draw (1.,-2.) -- (2.,-1.5);
	\draw[fill] (2.,-2.5) circle (0.05);
	\draw[fill] (2.,-1.5) circle (0.05);
	\draw[red, thick] (0.,-2.) -- +(0,0.333333) ;
	\draw[red, thick] (2.,-2.5) -- +(0,0.333333) ;
	\draw[red, thick] (2.,-1.5) -- +(0,0.333333) ;
	\draw[thick, dotted] (3,-2) circle (.35cm);
	\draw[fill=white] (3.,-2.) circle (0.1);
\end{tikzpicture}
\,;
\qquad
\result{I}=
\begin{tikzpicture}[scale = .5, baseline=-.6cm]
	\draw[fill=white, rounded corners,inner sep=2pt] (-.5,-3) rectangle (3.5,1);
	\draw[fill] (0,0) circle (0.05);
	\draw (0.,0.) -- (1.,0.);
	\draw[fill] (1.,0.) circle (0.05);
	\draw (1.,0.) -- (2.,-0.5);
	\draw (1.,0.) -- (2.,0.5);
	\draw[fill] (2.,-0.5) circle (0.05);
	\draw[fill] (2.,0.5) circle (0.05);
	\draw[red, thick] (0.,0.) -- +(0,0.333333) ;
	\draw[red, thick] (2.,-0.5) -- +(0,0.333333) ;
	\draw[red, thick] (2.,0.5) -- +(0,0.333333) ;
	\draw[fill=white] (3.,0.) circle (0.1);
	\draw[fill] (0,-2) circle (0.05);
	\draw (0.,-2.) -- (1.,-2.);
	\draw[fill] (1.,-2.) circle (0.05);
	\draw (1.,-2.) -- (2.,-2.5);
	\draw (1.,-2.) -- (2.,-1.5);
	\draw[fill] (2.,-2.5) circle (0.05);
	\draw[fill] (2.,-1.5) circle (0.05);
	\draw[red, thick] (0.,-2.) -- +(0,0.333333) ;
	\draw[red, thick] (2.,-2.5) -- +(0,0.333333) ;
	\draw[red, thick] (2.,-1.5) -- +(0,0.333333) ;
	\draw[fill=white] (3.,-2.) circle (0.1);
\end{tikzpicture}
$$
\item 
When we add a self-dual vertex, $S^+(V) \in U(o)$, $V$ denotes a subset of the vertices of $o$ at depth $n-1$, all on one graph. 
The underlying object $\result{S^+(V)}$ is the PGP obtained by adding a new self-dual vertex at depth $n$, connected by an edge to each vertex in $V$.
$$
o=
\begin{tikzpicture}[scale = .5, baseline=-.3cm]
	\draw[fill=white, rounded corners,inner sep=2pt] (-.5,-2) rectangle (2.5,1);
	\draw[fill] (0,0) circle (0.05);
	\draw (0.,0.) -- (1.,0.);
	\draw[fill] (1.,0.) circle (0.05);
	\draw (1.,0.) -- (2.,-0.5);
	\draw (1.,0.) -- (2.,0.5);
	\draw[fill] (2.,-0.5) circle (0.05);
	\draw[fill] (2.,0.5) circle (0.05);
	\draw[red, thick] (0.,0.) -- +(0,0.333333) ;
	\draw[red, thick] (2.,-0.5) -- +(0,0.333333) ;
	\draw[red, thick] (2.,0.5) -- +(0,0.333333) ;
	\draw[fill] (0,-1.5) circle (0.05);
	\draw (0.,-1.5) -- (1.,-1.5);
	\draw[fill] (1.,-1.5) circle (0.05);
	\draw (1.,-1.5) -- (2.,-1.5);
	\draw[fill] (2.,-1.5) circle (0.05);
	\draw[red, thick] (0.,-1.5) -- +(0,0.333333) ;
	\draw[red, thick] (2.,-1.5) -- +(0,0.333333) ;
\end{tikzpicture}
\,;\qquad
S^+(V)=
\begin{tikzpicture}[scale = .5, baseline=-.3cm]
	\draw[fill=white, rounded corners,inner sep=2pt] (-.5,-2) rectangle (2.5,1);
	\draw[fill] (0,0) circle (0.05);
	\draw (0.,0.) -- (1.,0.);
	\draw[fill] (1.,0.) circle (0.05);
	\draw (1.,0.) -- (2.,-0.5);
	\draw (1.,0.) -- (2.,0.5);
	\draw[fill] (2.,-0.5) circle (0.05);
	\draw[fill] (2.,0.5) circle (0.05);
	\draw[red, thick] (0.,0.) -- +(0,0.333333) ;
	\draw[red, thick] (2.,-0.5) -- +(0,0.333333) ;
	\draw[red, thick] (2.,0.5) -- +(0,0.333333) ;
	\draw[fill] (0,-1.5) circle (0.05);
	\draw (0.,-1.5) -- (1.,-1.5);
	\draw[fill] (1.,-1.5) circle (0.05);
	\draw (1.,-1.5) -- (2.,-1.5);
	\draw[fill] (2.,-1.5) circle (0.05);
	\draw[red, thick] (0.,-1.5) -- +(0,0.333333) ;
	\draw[red, thick] (2.,-1.5) -- +(0,0.333333) ;
	\draw[dotted, thick] (1,-1.5) circle (.3cm);
	\node at (1.,-.9) {\scriptsize{$V$}};
\end{tikzpicture}
\,;\qquad
\result{S^+(V)}=
\begin{tikzpicture}[scale = .5, baseline=-.6cm]
	\draw[fill=white, rounded corners,inner sep=2pt] (-.5,-3) rectangle (2.5,1);
	\draw[fill] (0,0) circle (0.05);
	\draw (0.,0.) -- (1.,0.);
	\draw[fill] (1.,0.) circle (0.05);
	\draw (1.,0.) -- (2.,-0.5);
	\draw (1.,0.) -- (2.,0.5);
	\draw[fill] (2.,-0.5) circle (0.05);
	\draw[fill] (2.,0.5) circle (0.05);
	\draw[red, thick] (0.,0.) -- +(0,0.333333) ;
	\draw[red, thick] (2.,-0.5) -- +(0,0.333333) ;
	\draw[red, thick] (2.,0.5) -- +(0,0.333333) ;
	\draw[fill] (0,-2) circle (0.05);
	\draw (0.,-2.) -- (1.,-2.);
	\draw[fill] (1.,-2.) circle (0.05);
	\draw (1.,-2.) -- (2.,-2.5);
	\draw (1.,-2.) -- (2.,-1.5);
	\draw[fill] (2.,-2.5) circle (0.05);
	\draw[fill] (2.,-1.5) circle (0.05);
	\draw[red, thick] (0.,-2.) -- +(0,0.333333) ;
	\draw[red, thick] (2.,-2.5) -- +(0,0.333333) ;
	\draw[red, thick] (2.,-1.5) -- +(0,0.333333) ;
\end{tikzpicture}
$$
\item 
When we add a pair of dual vertices $P^+(V_1, V_2) \in U(o)$, the sets $V_1$ and $V_2$ denote two collections of vertices of $o$ at depth $n-1$. 
When the working depth $n$ is even, both must be collections of vertices of the same graph, and we do not distinguish between $P^+(V_1, V_2)$ and $P^+(V_2, V_1)$.
$$
\hspace{-1.2cm}
o=
\begin{tikzpicture}[scale = .5, baseline=-.6cm]
	\draw[fill=white, rounded corners,inner sep=2pt] (-.5,-3) rectangle (5.5,1);
	\draw[fill] (0,0) circle (0.05);
	\draw (0.,0.) -- (1.,0.);
	\draw[fill] (1.,0.) circle (0.05);
	\draw (1.,0.) -- (2.,0.);
	\draw[fill] (2.,0.) circle (0.05);
	\draw (2.,0.) -- (3.,0.);
	\draw[fill] (3.,0.) circle (0.05);
	\draw (3.,0.) -- (4.,-0.5);
	\draw (3.,0.) -- (4.,0.5);
	\draw[fill] (4.,-0.5) circle (0.05);
	\draw[fill] (4.,0.5) circle (0.05);
	\draw (4.,-0.5) -- (5.,-0.5);
	\draw (4.,0.5) -- (5.,0.5);
	\draw[fill] (5.,-0.5) circle (0.05);
	\draw[fill] (5.,0.5) circle (0.05);
	\draw[red, thick] (0.,0.) -- +(0,0.333333) ;
	\draw[red, thick] (2.,0.) -- +(0,0.333333) ;
	\draw[red, thick] (4.,-0.5) -- +(0,0.333333) ;
	\draw[red, thick] (4.,0.5) -- +(0,0.333333) ;
	\draw[fill] (0,-2) circle (0.05);
	\draw (0.,-2.) -- (1.,-2.);
	\draw[fill] (1.,-2.) circle (0.05);
	\draw (1.,-2.) -- (2.,-2.);
	\draw[fill] (2.,-2.) circle (0.05);
	\draw (2.,-2.) -- (3.,-2.);
	\draw[fill] (3.,-2.) circle (0.05);
	\draw (3.,-2.) -- (4.,-2.5);
	\draw (3.,-2.) -- (4.,-1.5);
	\draw[fill] (4.,-2.5) circle (0.05);
	\draw[fill] (4.,-1.5) circle (0.05);
	\draw (4.,-2.5) -- (5.,-2.5);
	\draw (4.,-2.5) -- (5.,-1.5);
	\draw[fill] (5.,-2.5) circle (0.05);
	\draw[fill] (5.,-1.5) circle (0.05);
	\draw[red, thick] (0.,-2.) -- +(0,0.333333) ;
	\draw[red, thick] (2.,-2.) -- +(0,0.333333) ;
	\draw[red, thick] (4.,-2.5) -- +(0,0.333333) ;
	\draw[red, thick] (4.,-1.5) -- +(0,0.333333) ;
\end{tikzpicture}
\,;\,\,
P^+(V_1,V_2)=
\begin{tikzpicture}[scale = .5, baseline=-.6cm]
	\draw[fill=white, rounded corners,inner sep=2pt] (-.5,-3) rectangle (6.1,1);
	\draw[fill] (0,0) circle (0.05);
	\draw (0.,0.) -- (1.,0.);
	\draw[fill] (1.,0.) circle (0.05);
	\draw (1.,0.) -- (2.,0.);
	\draw[fill] (2.,0.) circle (0.05);
	\draw (2.,0.) -- (3.,0.);
	\draw[fill] (3.,0.) circle (0.05);
	\draw (3.,0.) -- (4.,-0.5);
	\draw (3.,0.) -- (4.,0.5);
	\draw[fill] (4.,-0.5) circle (0.05);
	\draw[fill] (4.,0.5) circle (0.05);
	\draw (4.,-0.5) -- (5.,-0.5);
	\draw (4.,0.5) -- (5.,0.5);
	\draw[fill] (5.,-0.5) circle (0.05);
	\draw[fill] (5.,0.5) circle (0.05);
	\draw[red, thick] (0.,0.) -- +(0,0.333333) ;
	\draw[red, thick] (2.,0.) -- +(0,0.333333) ;
	\draw[red, thick] (4.,-0.5) -- +(0,0.333333) ;
	\draw[red, thick] (4.,0.5) -- +(0,0.333333) ;
	\draw[dotted, thick] (5,.5) circle (.3cm);
	\node at (5.7,.5) {\scriptsize{$V_2$}};
	\draw[dotted, thick] (5,-.5) circle (.3cm);
	\node at (5.7,-.5) {\scriptsize{$V_1$}};
	\draw[fill] (0,-2) circle (0.05);
	\draw (0.,-2.) -- (1.,-2.);
	\draw[fill] (1.,-2.) circle (0.05);
	\draw (1.,-2.) -- (2.,-2.);
	\draw[fill] (2.,-2.) circle (0.05);
	\draw (2.,-2.) -- (3.,-2.);
	\draw[fill] (3.,-2.) circle (0.05);
	\draw (3.,-2.) -- (4.,-2.5);
	\draw (3.,-2.) -- (4.,-1.5);
	\draw[fill] (4.,-2.5) circle (0.05);
	\draw[fill] (4.,-1.5) circle (0.05);
	\draw (4.,-2.5) -- (5.,-2.5);
	\draw (4.,-2.5) -- (5.,-1.5);
	\draw[fill] (5.,-2.5) circle (0.05);
	\draw[fill] (5.,-1.5) circle (0.05);
	\draw[red, thick] (0.,-2.) -- +(0,0.333333) ;
	\draw[red, thick] (2.,-2.) -- +(0,0.333333) ;
	\draw[red, thick] (4.,-2.5) -- +(0,0.333333) ;
	\draw[red, thick] (4.,-1.5) -- +(0,0.333333) ;
\end{tikzpicture}
\,;\,\,
\result{P^+(V_1,V_2)}=
\begin{tikzpicture}[scale = .5, baseline=-.6cm]
	\draw[fill=white, rounded corners,inner sep=2pt] (-.5,-3) rectangle (6.5,1);
	\draw[fill] (0,0) circle (0.05);
	\draw (0.,0.) -- (1.,0.);
	\draw[fill] (1.,0.) circle (0.05);
	\draw (1.,0.) -- (2.,0.);
	\draw[fill] (2.,0.) circle (0.05);
	\draw (2.,0.) -- (3.,0.);
	\draw[fill] (3.,0.) circle (0.05);
	\draw (3.,0.) -- (4.,-0.5);
	\draw (3.,0.) -- (4.,0.5);
	\draw[fill] (4.,-0.5) circle (0.05);
	\draw[fill] (4.,0.5) circle (0.05);
	\draw (4.,-0.5) -- (5.,-0.5);
	\draw (4.,0.5) -- (5.,0.5);
	\draw[fill] (5.,-0.5) circle (0.05);
	\draw[fill] (5.,0.5) circle (0.05);
	\draw (5.,-0.5) -- (6.,-0.5);
	\draw (5.,0.5) -- (6.,0.5);
	\draw[fill] (6.,-0.5) circle (0.05);
	\draw[fill] (6.,0.5) circle (0.05);
	\draw[red, thick] (0.,0.) -- +(0,0.333333) ;
	\draw[red, thick] (2.,0.) -- +(0,0.333333) ;
	\draw[red, thick] (4.,-0.5) -- +(0,0.333333) ;
	\draw[red, thick] (4.,0.5) -- +(0,0.333333) ;
	\draw[red, thick] (6.,-0.5) to[out=135,in=-135] (6.,0.5);
	\draw[fill] (0,-2) circle (0.05);
	\draw (0.,-2.) -- (1.,-2.);
	\draw[fill] (1.,-2.) circle (0.05);
	\draw (1.,-2.) -- (2.,-2.);
	\draw[fill] (2.,-2.) circle (0.05);
	\draw (2.,-2.) -- (3.,-2.);
	\draw[fill] (3.,-2.) circle (0.05);
	\draw (3.,-2.) -- (4.,-2.5);
	\draw (3.,-2.) -- (4.,-1.5);
	\draw[fill] (4.,-2.5) circle (0.05);
	\draw[fill] (4.,-1.5) circle (0.05);
	\draw (4.,-2.5) -- (5.,-2.5);
	\draw (4.,-2.5) -- (5.,-1.5);
	\draw[fill] (5.,-2.5) circle (0.05);
	\draw[fill] (5.,-1.5) circle (0.05);
	\draw[red, thick] (0.,-2.) -- +(0,0.333333) ;
	\draw[red, thick] (2.,-2.) -- +(0,0.333333) ;
	\draw[red, thick] (4.,-2.5) -- +(0,0.333333) ;
	\draw[red, thick] (4.,-1.5) -- +(0,0.333333) ;
\end{tikzpicture}
$$
When the working depth $n$ is odd, each is a subset of vertices on different graphs. 
$$
o=
\begin{tikzpicture}[scale = .5, baseline=-.6cm]
	\draw[fill=white, rounded corners,inner sep=2pt] (-.5,-3) rectangle (2.5,1);
	\draw[fill] (0,0) circle (0.05);
	\draw (0.,0.) -- (1.,0.);
	\draw[fill] (1.,0.) circle (0.05);
	\draw (1.,0.) -- (2.,-0.5);
	\draw (1.,0.) -- (2.,0.5);
	\draw[fill] (2.,-0.5) circle (0.05);
	\draw[fill] (2.,0.5) circle (0.05);
	\draw[red, thick] (0.,0.) -- +(0,0.333333) ;
	\draw[red, thick] (2.,-0.5) -- +(0,0.333333) ;
	\draw[red, thick] (2.,0.5) -- +(0,0.333333) ;
	\draw[fill] (0,-2) circle (0.05);
	\draw (0.,-2.) -- (1.,-2.);
	\draw[fill] (1.,-2.) circle (0.05);
	\draw (1.,-2.) -- (2.,-2.5);
	\draw (1.,-2.) -- (2.,-1.5);
	\draw[fill] (2.,-2.5) circle (0.05);
	\draw[fill] (2.,-1.5) circle (0.05);
	\draw[red, thick] (0.,-2.) -- +(0,0.333333) ;
	\draw[red, thick] (2.,-2.5) -- +(0,0.333333) ;
	\draw[red, thick] (2.,-1.5) -- +(0,0.333333) ;
\end{tikzpicture}
\,;
\qquad
P^+(V_1,V_2)=
\begin{tikzpicture}[scale = .5, baseline=-.6cm]
	\draw[fill=white, rounded corners,inner sep=2pt] (-.5,-3) rectangle (3.2,1);
	\draw[fill] (0,0) circle (0.05);
	\draw (0.,0.) -- (1.,0.);
	\draw[fill] (1.,0.) circle (0.05);
	\draw (1.,0.) -- (2.,-0.5);
	\draw (1.,0.) -- (2.,0.5);
	\draw[fill] (2.,-0.5) circle (0.05);
	\draw[fill] (2.,0.5) circle (0.05);
	\draw[red, thick] (0.,0.) -- +(0,0.333333) ;
	\draw[red, thick] (2.,-0.5) -- +(0,0.333333) ;
	\draw[red, thick] (2.,0.5) -- +(0,0.333333) ;
	\draw[dotted, thick] (2,-.5) circle (.4cm);
	\node at (2.8,-.5) {\scriptsize{$V_1$}};
	\draw[fill] (0,-2) circle (0.05);
	\draw (0.,-2.) -- (1.,-2.);
	\draw[fill] (1.,-2.) circle (0.05);
	\draw (1.,-2.) -- (2.,-2.5);
	\draw (1.,-2.) -- (2.,-1.5);
	\draw[fill] (2.,-2.5) circle (0.05);
	\draw[fill] (2.,-1.5) circle (0.05);
	\draw[red, thick] (0.,-2.) -- +(0,0.333333) ;
	\draw[red, thick] (2.,-2.5) -- +(0,0.333333) ;
	\draw[red, thick] (2.,-1.5) -- +(0,0.333333) ;
	\draw[dotted, thick] (2,-2.5) circle (.4cm);
	\node at (2.8,-2.5) {\scriptsize{$V_2$}};
\end{tikzpicture}
\,;
\qquad
\result{P^+(V_1,V_2)}=
\begin{tikzpicture}[scale = .5, baseline=-.6cm]
	\draw[fill=white, rounded corners,inner sep=2pt] (-.5,-3) rectangle (3.5,1);
	\draw[fill] (0,0) circle (0.05);
	\draw (0.,0.) -- (1.,0.);
	\draw[fill] (1.,0.) circle (0.05);
	\draw (1.,0.) -- (2.,-0.5);
	\draw (1.,0.) -- (2.,0.5);
	\draw[fill] (2.,-0.5) circle (0.05);
	\draw[fill] (2.,0.5) circle (0.05);
	\draw (2.,-0.5) -- (3.,0.);
	\filldraw (3.,0.) circle (0.05);
	\draw[red, thick] (0.,0.) -- +(0,0.333333) ;
	\draw[red, thick] (2.,-0.5) -- +(0,0.333333) ;
	\draw[red, thick] (2.,0.5) -- +(0,0.333333) ;
	\draw[fill] (0,-2) circle (0.05);
	\draw (0.,-2.) -- (1.,-2.);
	\draw[fill] (1.,-2.) circle (0.05);
	\draw (1.,-2.) -- (2.,-2.5);
	\draw (1.,-2.) -- (2.,-1.5);
	\draw[fill] (2.,-2.5) circle (0.05);
	\draw[fill] (2.,-1.5) circle (0.05);
	\draw (2.,-2.5) -- (3.,-2.);
	\filldraw (3.,-2.) circle (0.05);
	\draw[red, thick] (0.,-2.) -- +(0,0.333333) ;
	\draw[red, thick] (2.,-2.5) -- +(0,0.333333) ;
	\draw[red, thick] (2.,-1.5) -- +(0,0.333333) ;
\end{tikzpicture}
$$
The resulting object $\result{P^+(V_1,V_2)}$ is the PGP obtained by adding a pair of vertices, dual to each other at depth $n$ (either on the same graph, or different graphs, as appropriate), with the first connected by an edge to each vertex in $V_1$, and the second connected to each vertex in $V_2$.
\end{itemize}

For a given PGP $o$:
\begin{itemize}
\item
We include $I \in U(o)$ if and only if there is at least
one vertex on each graph at the current working depth, 
and the associativity condition holds between all pairs of vertices at depth $n-1$. (This ensures that $\result{I}$ is associative.)

\item We include $S^+(V)$ for every subset $V$ such that
$\result{S^+(V)}$ would have index below our cutoff, and the associativity condition holds between all vertices at depth $n-2$ and the new vertex in $\result{S^+(V)}$.

\item Similarly we include 
$P^+(V_1,V_2)$ if it has small enough index and the associativity condition holds between all vertices at depth $n-2$ and both of the new vertices in $\result{P^+(V_1,V_2)}$.
\end{itemize}

Note that associativity between any pair of vertices $v$ and $w$ is checked precisely once.
\end{boxedexample}

\item
\label{condition:lower-objects}
Associated to an object $o \in \cO$, we define a new set, the `lower objects'
$L(o)$ for $o$. 
The elements of $L(o)$ consist of an object in $\cO$ along with certain extra
data. If $l \in L(o)$ is a lower object for $o$, we write $\result{l}$ for the
resulting object in $\cO$. We must have $\ell(\result{l}) < \ell(o)$.
Again, a lower object $l$ essentially records the generating step that
produces $o$ from $\result{l}$.

We insist that $L(o)$ is empty if and only if $\ell(o) = 0$.

\begin{boxedexample}
This is completely parallel to the description of upper objects.
We have three types of lower objects, called `decreasing the depth' $D$, 
$$
o=
\begin{tikzpicture}[scale = .5, baseline=-.6cm]
	\draw[fill=white, rounded corners,inner sep=2pt] (-.5,-3) rectangle (3.5,1);
	\draw[fill] (0,0) circle (0.05);
	\draw (0.,0.) -- (1.,0.);
	\draw[fill] (1.,0.) circle (0.05);
	\draw (1.,0.) -- (2.,-0.5);
	\draw (1.,0.) -- (2.,0.5);
	\draw[fill] (2.,-0.5) circle (0.05);
	\draw[fill] (2.,0.5) circle (0.05);
	\draw[red, thick] (0.,0.) -- +(0,0.333333) ;
	\draw[red, thick] (2.,-0.5) -- +(0,0.333333) ;
	\draw[red, thick] (2.,0.5) -- +(0,0.333333) ;
	\draw[fill=white] (3.,0.) circle (0.1);
	\draw[fill] (0,-2) circle (0.05);
	\draw (0.,-2.) -- (1.,-2.);
	\draw[fill] (1.,-2.) circle (0.05);
	\draw (1.,-2.) -- (2.,-2.5);
	\draw (1.,-2.) -- (2.,-1.5);
	\draw[fill] (2.,-2.5) circle (0.05);
	\draw[fill] (2.,-1.5) circle (0.05);
	\draw[red, thick] (0.,-2.) -- +(0,0.333333) ;
	\draw[red, thick] (2.,-2.5) -- +(0,0.333333) ;
	\draw[red, thick] (2.,-1.5) -- +(0,0.333333) ;
	\draw[fill=white] (3.,-2.) circle (0.1);
\end{tikzpicture}
\,;\qquad
D=
\begin{tikzpicture}[scale = .5, baseline=-.6cm]
	\draw[fill=white, rounded corners,inner sep=2pt] (-.5,-3) rectangle (3.5,1);
	\draw[fill] (0,0) circle (0.05);
	\draw (0.,0.) -- (1.,0.);
	\draw[fill] (1.,0.) circle (0.05);
	\draw (1.,0.) -- (2.,-0.5);
	\draw (1.,0.) -- (2.,0.5);
	\draw[fill] (2.,-0.5) circle (0.05);
	\draw[fill] (2.,0.5) circle (0.05);
	\draw[red, thick] (0.,0.) -- +(0,0.333333) ;
	\draw[red, thick] (2.,-0.5) -- +(0,0.333333) ;
	\draw[red, thick] (2.,0.5) -- +(0,0.333333) ;
	\draw[fill=white] (3.,0.) circle (0.1);
	\draw[dotted, thick] (3,0) circle (.35cm);
	\draw[fill] (0,-2) circle (0.05);
	\draw (0.,-2.) -- (1.,-2.);
	\draw[fill] (1.,-2.) circle (0.05);
	\draw (1.,-2.) -- (2.,-2.5);
	\draw (1.,-2.) -- (2.,-1.5);
	\draw[fill] (2.,-2.5) circle (0.05);
	\draw[fill] (2.,-1.5) circle (0.05);
	\draw[red, thick] (0.,-2.) -- +(0,0.333333) ;
	\draw[red, thick] (2.,-2.5) -- +(0,0.333333) ;
	\draw[red, thick] (2.,-1.5) -- +(0,0.333333) ;
	\draw[fill=white] (3.,-2.) circle (0.1);
	\draw[dotted, thick] (3,-2) circle (.35cm);
\end{tikzpicture}
\,;\qquad
\result{D}=
\begin{tikzpicture}[scale = .5, baseline=-.6cm]
	\draw[fill=white, rounded corners,inner sep=2pt] (-.5,-3) rectangle (2.5,1);
	\draw[fill] (0,0) circle (0.05);
	\draw (0.,0.) -- (1.,0.);
	\draw[fill] (1.,0.) circle (0.05);
	\draw (1.,0.) -- (2.,-0.5);
	\draw (1.,0.) -- (2.,0.5);
	\draw[fill] (2.,-0.5) circle (0.05);
	\draw[fill] (2.,0.5) circle (0.05);
	\draw[red, thick] (0.,0.) -- +(0,0.333333) ;
	\draw[red, thick] (2.,-0.5) -- +(0,0.333333) ;
	\draw[red, thick] (2.,0.5) -- +(0,0.333333) ;
	\draw[fill] (0,-2) circle (0.05);
	\draw (0.,-2.) -- (1.,-2.);
	\draw[fill] (1.,-2.) circle (0.05);
	\draw (1.,-2.) -- (2.,-2.5);
	\draw (1.,-2.) -- (2.,-1.5);
	\draw[fill] (2.,-2.5) circle (0.05);
	\draw[fill] (2.,-1.5) circle (0.05);
	\draw[red, thick] (0.,-2.) -- +(0,0.333333) ;
	\draw[red, thick] (2.,-2.5) -- +(0,0.333333) ;
	\draw[red, thick] (2.,-1.5) -- +(0,0.333333) ;
\end{tikzpicture}
$$
`deleting a self-dual vertex' $S^-(v)$, where $v$ denotes any self-dual vertex at the working depth,
$$
o=
\begin{tikzpicture}[scale = .5, baseline=-.6cm]
	\draw[fill=white, rounded corners,inner sep=2pt] (-.5,-3) rectangle (2.5,1);
	\draw[fill] (0,0) circle (0.05);
	\draw (0.,0.) -- (1.,0.);
	\draw[fill] (1.,0.) circle (0.05);
	\draw (1.,0.) -- (2.,-0.5);
	\draw (1.,0.) -- (2.,0.5);
	\draw[fill] (2.,-0.5) circle (0.05);
	\draw[fill] (2.,0.5) circle (0.05);
	\draw[red, thick] (0.,0.) -- +(0,0.333333) ;
	\draw[red, thick] (2.,-0.5) -- +(0,0.333333) ;
	\draw[red, thick] (2.,0.5) -- +(0,0.333333) ;
	\draw[fill] (0,-2) circle (0.05);
	\draw (0.,-2.) -- (1.,-2.);
	\draw[fill] (1.,-2.) circle (0.05);
	\draw (1.,-2.) -- (2.,-2.5);
	\draw (1.,-2.) -- (2.,-1.5);
	\draw[fill] (2.,-2.5) circle (0.05);
	\draw[fill] (2.,-1.5) circle (0.05);
	\draw[red, thick] (0.,-2.) -- +(0,0.333333) ;
	\draw[red, thick] (2.,-2.5) -- +(0,0.333333) ;
	\draw[red, thick] (2.,-1.5) -- +(0,0.333333) ;
\end{tikzpicture}
\,;\qquad
S^-(v)=
\begin{tikzpicture}[scale = .5, baseline=-.6cm]
	\draw[fill=white, rounded corners,inner sep=2pt] (-.5,-3) rectangle (3,1);
	\draw[fill] (0,0) circle (0.05);
	\draw (0.,0.) -- (1.,0.);
	\draw[fill] (1.,0.) circle (0.05);
	\draw (1.,0.) -- (2.,-0.5);
	\draw (1.,0.) -- (2.,0.5);
	\draw[fill] (2.,-0.5) circle (0.05);
	\draw[fill] (2.,0.5) circle (0.05);
	\draw[red, thick] (0.,0.) -- +(0,0.333333) ;
	\draw[red, thick] (2.,-0.5) -- +(0,0.333333) ;
	\draw[red, thick] (2.,0.5) -- +(0,0.333333) ;
	\draw[fill] (0,-2) circle (0.05);
	\draw (0.,-2.) -- (1.,-2.);
	\draw[fill] (1.,-2.) circle (0.05);
	\draw (1.,-2.) -- (2.,-2.5);
	\draw (1.,-2.) -- (2.,-1.5);
	\draw[fill] (2.,-2.5) circle (0.05);
	\draw[fill] (2.,-1.5) circle (0.05);
	\draw[red, thick] (0.,-2.) -- +(0,0.333333) ;
	\draw[red, thick] (2.,-2.5) -- +(0,0.333333) ;
	\draw[red, thick] (2.,-1.5) -- +(0,0.333333) ;
	\draw[dotted, thick] (2,-2.5) circle (.4cm);
	\node at (2.7,-2.5) {\scriptsize{$v$}};
\end{tikzpicture}
\,;\qquad
\result{S^-(v)}=
\begin{tikzpicture}[scale = .5, baseline=-.3cm]
	\draw[fill=white, rounded corners,inner sep=2pt] (-.5,-2) rectangle (2.5,1);
	\draw[fill] (0,0) circle (0.05);
	\draw (0.,0.) -- (1.,0.);
	\draw[fill] (1.,0.) circle (0.05);
	\draw (1.,0.) -- (2.,-0.5);
	\draw (1.,0.) -- (2.,0.5);
	\draw[fill] (2.,-0.5) circle (0.05);
	\draw[fill] (2.,0.5) circle (0.05);
	\draw[red, thick] (0.,0.) -- +(0,0.333333) ;
	\draw[red, thick] (2.,-0.5) -- +(0,0.333333) ;
	\draw[red, thick] (2.,0.5) -- +(0,0.333333) ;
	\draw[fill] (0,-1.5) circle (0.05);
	\draw (0.,-1.5) -- (1.,-1.5);
	\draw[fill] (1.,-1.5) circle (0.05);
	\draw (1.,-1.5) -- (2.,-1.5);
	\draw[fill] (2.,-1.5) circle (0.05);
	\draw[red, thick] (0.,-1.5) -- +(0,0.333333) ;
	\draw[red, thick] (2.,-1.5) -- +(0,0.333333) ;
\end{tikzpicture}
$$
and `deleting a pair of dual vertices' $P^-(w_1,w_2)$, where $w_1, w_2$ denote a pair of dual vertices at the working depth.
$$
\hspace{-.4cm}
o=
\begin{tikzpicture}[scale = .48, baseline=-.6cm]
	\draw[fill=white, rounded corners,inner sep=2pt] (-.4,-3) rectangle (6.4,1);
	\draw[fill] (0,0) circle (0.05);
	\draw (0.,0.) -- (1.,0.);
	\draw[fill] (1.,0.) circle (0.05);
	\draw (1.,0.) -- (2.,0.);
	\draw[fill] (2.,0.) circle (0.05);
	\draw (2.,0.) -- (3.,0.);
	\draw[fill] (3.,0.) circle (0.05);
	\draw (3.,0.) -- (4.,-0.5);
	\draw (3.,0.) -- (4.,0.5);
	\draw[fill] (4.,-0.5) circle (0.05);
	\draw[fill] (4.,0.5) circle (0.05);
	\draw (4.,-0.5) -- (5.,-0.5);
	\draw (4.,0.5) -- (5.,0.5);
	\draw[fill] (5.,-0.5) circle (0.05);
	\draw[fill] (5.,0.5) circle (0.05);
	\draw (5.,-0.5) -- (6.,-0.5);
	\draw (5.,0.5) -- (6.,0.5);
	\draw[fill] (6.,-0.5) circle (0.05);
	\draw[fill] (6.,0.5) circle (0.05);
	\draw[red, thick] (0.,0.) -- +(0,0.333333) ;
	\draw[red, thick] (2.,0.) -- +(0,0.333333) ;
	\draw[red, thick] (4.,-0.5) -- +(0,0.333333) ;
	\draw[red, thick] (4.,0.5) -- +(0,0.333333) ;
	\draw[red, thick] (6.,-0.5) to[out=135,in=-135] (6.,0.5);
	\draw[fill] (0,-2) circle (0.05);
	\draw (0.,-2.) -- (1.,-2.);
	\draw[fill] (1.,-2.) circle (0.05);
	\draw (1.,-2.) -- (2.,-2.);
	\draw[fill] (2.,-2.) circle (0.05);
	\draw (2.,-2.) -- (3.,-2.);
	\draw[fill] (3.,-2.) circle (0.05);
	\draw (3.,-2.) -- (4.,-2.5);
	\draw (3.,-2.) -- (4.,-1.5);
	\draw[fill] (4.,-2.5) circle (0.05);
	\draw[fill] (4.,-1.5) circle (0.05);
	\draw (4.,-2.5) -- (5.,-2.5);
	\draw (4.,-2.5) -- (5.,-1.5);
	\draw[fill] (5.,-2.5) circle (0.05);
	\draw[fill] (5.,-1.5) circle (0.05);
	\draw[red, thick] (0.,-2.) -- +(0,0.333333) ;
	\draw[red, thick] (2.,-2.) -- +(0,0.333333) ;
	\draw[red, thick] (4.,-2.5) -- +(0,0.333333) ;
	\draw[red, thick] (4.,-1.5) -- +(0,0.333333) ;
\end{tikzpicture}
\,;\,
P^-(w_1,w_2)=
\begin{tikzpicture}[scale = .48, baseline=-.6cm]
	\draw[fill=white, rounded corners,inner sep=2pt] (-.4,-3) rectangle (7.1,1);
	\draw[fill] (0,0) circle (0.05);
	\draw (0.,0.) -- (1.,0.);
	\draw[fill] (1.,0.) circle (0.05);
	\draw (1.,0.) -- (2.,0.);
	\draw[fill] (2.,0.) circle (0.05);
	\draw (2.,0.) -- (3.,0.);
	\draw[fill] (3.,0.) circle (0.05);
	\draw (3.,0.) -- (4.,-0.5);
	\draw (3.,0.) -- (4.,0.5);
	\draw[fill] (4.,-0.5) circle (0.05);
	\draw[fill] (4.,0.5) circle (0.05);
	\draw (4.,-0.5) -- (5.,-0.5);
	\draw (4.,0.5) -- (5.,0.5);
	\draw[fill] (5.,-0.5) circle (0.05);
	\draw[fill] (5.,0.5) circle (0.05);
	\draw (5.,-0.5) -- (6.,-0.5);
	\draw (5.,0.5) -- (6.,0.5);
	\draw[fill] (6.,-0.5) circle (0.05);
	\draw[fill] (6.,0.5) circle (0.05);
	\draw[red, thick] (0.,0.) -- +(0,0.333333) ;
	\draw[red, thick] (2.,0.) -- +(0,0.333333) ;
	\draw[red, thick] (4.,-0.5) -- +(0,0.333333) ;
	\draw[red, thick] (4.,0.5) -- +(0,0.333333) ;
	\draw[red, thick] (6.,-0.5) to[out=135,in=-135] (6.,0.5);
	\draw[dotted, thick] (6,.5) circle (.3cm);
	\node at (6.7,.5) {\scriptsize{$w_2$}};
	\draw[dotted, thick] (6,-.5) circle (.3cm);
	\node at (6.7,-.5) {\scriptsize{$w_1$}};
	\draw[fill] (0,-2) circle (0.05);
	\draw (0.,-2.) -- (1.,-2.);
	\draw[fill] (1.,-2.) circle (0.05);
	\draw (1.,-2.) -- (2.,-2.);
	\draw[fill] (2.,-2.) circle (0.05);
	\draw (2.,-2.) -- (3.,-2.);
	\draw[fill] (3.,-2.) circle (0.05);
	\draw (3.,-2.) -- (4.,-2.5);
	\draw (3.,-2.) -- (4.,-1.5);
	\draw[fill] (4.,-2.5) circle (0.05);
	\draw[fill] (4.,-1.5) circle (0.05);
	\draw (4.,-2.5) -- (5.,-2.5);
	\draw (4.,-2.5) -- (5.,-1.5);
	\draw[fill] (5.,-2.5) circle (0.05);
	\draw[fill] (5.,-1.5) circle (0.05);
	\draw[red, thick] (0.,-2.) -- +(0,0.333333) ;
	\draw[red, thick] (2.,-2.) -- +(0,0.333333) ;
	\draw[red, thick] (4.,-2.5) -- +(0,0.333333) ;
	\draw[red, thick] (4.,-1.5) -- +(0,0.333333) ;
\end{tikzpicture}
;\,
\result{P^-(w_1,w_2)}=
\begin{tikzpicture}[scale = .48, baseline=-.6cm]
	\draw[fill=white, rounded corners,inner sep=2pt] (-.4,-3) rectangle (5.4,1);
	\draw[fill] (0,0) circle (0.05);
	\draw (0.,0.) -- (1.,0.);
	\draw[fill] (1.,0.) circle (0.05);
	\draw (1.,0.) -- (2.,0.);
	\draw[fill] (2.,0.) circle (0.05);
	\draw (2.,0.) -- (3.,0.);
	\draw[fill] (3.,0.) circle (0.05);
	\draw (3.,0.) -- (4.,-0.5);
	\draw (3.,0.) -- (4.,0.5);
	\draw[fill] (4.,-0.5) circle (0.05);
	\draw[fill] (4.,0.5) circle (0.05);
	\draw (4.,-0.5) -- (5.,-0.5);
	\draw (4.,0.5) -- (5.,0.5);
	\draw[fill] (5.,-0.5) circle (0.05);
	\draw[fill] (5.,0.5) circle (0.05);
	\draw[red, thick] (0.,0.) -- +(0,0.333333) ;
	\draw[red, thick] (2.,0.) -- +(0,0.333333) ;
	\draw[red, thick] (4.,-0.5) -- +(0,0.333333) ;
	\draw[red, thick] (4.,0.5) -- +(0,0.333333) ;
	\draw[fill] (0,-2) circle (0.05);
	\draw (0.,-2.) -- (1.,-2.);
	\draw[fill] (1.,-2.) circle (0.05);
	\draw (1.,-2.) -- (2.,-2.);
	\draw[fill] (2.,-2.) circle (0.05);
	\draw (2.,-2.) -- (3.,-2.);
	\draw[fill] (3.,-2.) circle (0.05);
	\draw (3.,-2.) -- (4.,-2.5);
	\draw (3.,-2.) -- (4.,-1.5);
	\draw[fill] (4.,-2.5) circle (0.05);
	\draw[fill] (4.,-1.5) circle (0.05);
	\draw (4.,-2.5) -- (5.,-2.5);
	\draw (4.,-2.5) -- (5.,-1.5);
	\draw[fill] (5.,-2.5) circle (0.05);
	\draw[fill] (5.,-1.5) circle (0.05);
	\draw[red, thick] (0.,-2.) -- +(0,0.333333) ;
	\draw[red, thick] (2.,-2.) -- +(0,0.333333) ;
	\draw[red, thick] (4.,-2.5) -- +(0,0.333333) ;
	\draw[red, thick] (4.,-1.5) -- +(0,0.333333) ;
\end{tikzpicture}
$$
The resulting objects are the PGPs which are obtained by deleting the marked vertices.

We have $D \in L(o)$ if and only if there are no vertices on either graph at
the working depth.  We have $S^-(v) \in L(o)$ and $P^-(w_1,w_2) \in L(o)$ for
all valid choices of $v$ or $w_1,w_2$.
\end{boxedexample}

\item
\label{condition:equivariance}
The groupoid $\cO$ acts on the bundle of sets $U$.  That is, for each
$o$, $U(o)$ carries an action of the automorphism group $\Aut(o)$, and
moreover we have a coherent family of bijections $U(o) \iso U(o')$ for each
isomorphism $o \iso o'$.   Similarly $\cO$ acts on the bundle of sets $L$.

\begin{boxedexample*}
For example, the following upper objects in $U(o)$ are isomorphic:
$$
o=
\begin{tikzpicture}[scale = .5, baseline=-.6cm]
	\draw[fill=white, rounded corners,inner sep=2pt] (-.5,-3) rectangle (2.5,1);
	\draw[fill] (0,0) circle (0.05);
	\draw (0.,0.) -- (1.,0.);
	\draw[fill] (1.,0.) circle (0.05);
	\draw (1.,0.) -- (2.,-0.5);
	\draw (1.,0.) -- (2.,0.5);
	\draw[fill] (2.,-0.5) circle (0.05);
	\draw[fill] (2.,0.5) circle (0.05);
	\draw[red, thick] (0.,0.) -- +(0,0.333333) ;
	\draw[red, thick] (2.,-0.5) -- +(0,0.333333) ;
	\draw[red, thick] (2.,0.5) -- +(0,0.333333) ;
	\draw[fill] (0,-2) circle (0.05);
	\draw (0.,-2.) -- (1.,-2.);
	\draw[fill] (1.,-2.) circle (0.05);
	\draw (1.,-2.) -- (2.,-2.5);
	\draw (1.,-2.) -- (2.,-1.5);
	\draw[fill] (2.,-2.5) circle (0.05);
	\draw[fill] (2.,-1.5) circle (0.05);
	\draw[red, thick] (0.,-2.) -- +(0,0.333333) ;
	\draw[red, thick] (2.,-2.5) -- +(0,0.333333) ;
	\draw[red, thick] (2.,-1.5) -- +(0,0.333333) ;
\end{tikzpicture}
\,;
\qquad
P^+(V_1,V_2)=
\begin{tikzpicture}[scale = .5, baseline=-.6cm]
	\draw[fill=white, rounded corners,inner sep=2pt] (-.5,-3) rectangle (3.2,1);
	\draw[fill] (0,0) circle (0.05);
	\draw (0.,0.) -- (1.,0.);
	\draw[fill] (1.,0.) circle (0.05);
	\draw (1.,0.) -- (2.,-0.5);
	\draw (1.,0.) -- (2.,0.5);
	\draw[fill] (2.,-0.5) circle (0.05);
	\draw[fill] (2.,0.5) circle (0.05);
	\draw[red, thick] (0.,0.) -- +(0,0.333333) ;
	\draw[red, thick] (2.,-0.5) -- +(0,0.333333) ;
	\draw[red, thick] (2.,0.5) -- +(0,0.333333) ;
	\draw[dotted, thick] (2,-.5) circle (.4cm);
	\node at (2.8,-.5) {\scriptsize{$V_1$}};
	\draw[fill] (0,-2) circle (0.05);
	\draw (0.,-2.) -- (1.,-2.);
	\draw[fill] (1.,-2.) circle (0.05);
	\draw (1.,-2.) -- (2.,-2.5);
	\draw (1.,-2.) -- (2.,-1.5);
	\draw[fill] (2.,-2.5) circle (0.05);
	\draw[fill] (2.,-1.5) circle (0.05);
	\draw[red, thick] (0.,-2.) -- +(0,0.333333) ;
	\draw[red, thick] (2.,-2.5) -- +(0,0.333333) ;
	\draw[red, thick] (2.,-1.5) -- +(0,0.333333) ;
	\draw[dotted, thick] (2,-2.5) circle (.4cm);
	\node at (2.8,-2.5) {\scriptsize{$V_2$}};
\end{tikzpicture}
\cong
\begin{tikzpicture}[scale = .5, baseline=-.6cm]
	\draw[fill=white, rounded corners,inner sep=2pt] (-.5,-3) rectangle (3.2,1);
	\draw[fill] (0,0) circle (0.05);
	\draw (0.,0.) -- (1.,0.);
	\draw[fill] (1.,0.) circle (0.05);
	\draw (1.,0.) -- (2.,-0.5);
	\draw (1.,0.) -- (2.,0.5);
	\draw[fill] (2.,-0.5) circle (0.05);
	\draw[fill] (2.,0.5) circle (0.05);
	\draw[red, thick] (0.,0.) -- +(0,0.333333) ;
	\draw[red, thick] (2.,-0.5) -- +(0,0.333333) ;
	\draw[red, thick] (2.,0.5) -- +(0,0.333333) ;
	\draw[dotted, thick] (2,.5) circle (.4cm);
	\node at (2.8,.5) {\scriptsize{$V_1'$}};
	\draw[fill] (0,-2) circle (0.05);
	\draw (0.,-2.) -- (1.,-2.);
	\draw[fill] (1.,-2.) circle (0.05);
	\draw (1.,-2.) -- (2.,-2.5);
	\draw (1.,-2.) -- (2.,-1.5);
	\draw[fill] (2.,-2.5) circle (0.05);
	\draw[fill] (2.,-1.5) circle (0.05);
	\draw[red, thick] (0.,-2.) -- +(0,0.333333) ;
	\draw[red, thick] (2.,-2.5) -- +(0,0.333333) ;
	\draw[red, thick] (2.,-1.5) -- +(0,0.333333) ;
	\draw[dotted, thick] (2,-1.5) circle (.4cm);
	\node at (2.8,-1.5) {\scriptsize{$V_2'$}};
\end{tikzpicture}
=P^+(V_1',V_2')
$$
The automorphism of $o$ which swaps the vertices $V_1,V_1'$ and the vertices $V_2,V_2'$ induces the above isomorphism.
\end{boxedexample*}

\begin{remark}
\label{rem:isomorphisms-between-upper-objects}
When we write $u \iso u'$, for $u \in U(o)$ and $u' \in U(o')$, we mean an isomorphism $o \iso o'$ carrying $u$ to $u'$ (and similarly for lower objects).
\end{remark}

\item
\label{condition:inverses}
Each upper object $u \in U(o)$ must have an inverse lower object,
denoted $u^{-1}$, in $L(\result{u})$, such that $\result{u^{-1}} = o$. 
We require that when $u_1^{-1} \iso u_2^{-1}$, it is also the case that $u_1 \iso u_2$. 
\begin{boxedexample*}%
For $o$ and $u$ the first two graphs below, we see that $\result{u^{-1}}=o$:
$$
o=
\begin{tikzpicture}[scale = .5, baseline=-.6cm]
	\draw[fill=white, rounded corners,inner sep=2pt] (-.5,-3) rectangle (2.5,1);
	\draw[fill] (0,0) circle (0.05);
	\draw (0.,0.) -- (1.,0.);
	\draw[fill] (1.,0.) circle (0.05);
	\draw (1.,0.) -- (2.,-0.5);
	\draw (1.,0.) -- (2.,0.5);
	\draw[fill] (2.,-0.5) circle (0.05);
	\draw[fill] (2.,0.5) circle (0.05);
	\draw[red, thick] (0.,0.) -- +(0,0.333333) ;
	\draw[red, thick] (2.,-0.5) -- +(0,0.333333) ;
	\draw[red, thick] (2.,0.5) -- +(0,0.333333) ;
	\draw[fill] (0,-2) circle (0.05);
	\draw (0.,-2.) -- (1.,-2.);
	\draw[fill] (1.,-2.) circle (0.05);
	\draw (1.,-2.) -- (2.,-2.5);
	\draw (1.,-2.) -- (2.,-1.5);
	\draw[fill] (2.,-2.5) circle (0.05);
	\draw[fill] (2.,-1.5) circle (0.05);
	\draw[red, thick] (0.,-2.) -- +(0,0.333333) ;
	\draw[red, thick] (2.,-2.5) -- +(0,0.333333) ;
	\draw[red, thick] (2.,-1.5) -- +(0,0.333333) ;
\end{tikzpicture}
\,;
\qquad
u=
\begin{tikzpicture}[scale = .5, baseline=-.6cm]
	\draw[fill=white, rounded corners,inner sep=2pt] (-.5,-3) rectangle (3.2,1);
	\draw[fill] (0,0) circle (0.05);
	\draw (0.,0.) -- (1.,0.);
	\draw[fill] (1.,0.) circle (0.05);
	\draw (1.,0.) -- (2.,-0.5);
	\draw (1.,0.) -- (2.,0.5);
	\draw[fill] (2.,-0.5) circle (0.05);
	\draw[fill] (2.,0.5) circle (0.05);
	\draw[red, thick] (0.,0.) -- +(0,0.333333) ;
	\draw[red, thick] (2.,-0.5) -- +(0,0.333333) ;
	\draw[red, thick] (2.,0.5) -- +(0,0.333333) ;
	\draw[dotted, thick] (2,0) ellipse (.3cm and .85cm);
	\node at (2.8,0) {\scriptsize{$V_1$}};
	\draw[fill] (0,-2) circle (0.05);
	\draw (0.,-2.) -- (1.,-2.);
	\draw[fill] (1.,-2.) circle (0.05);
	\draw (1.,-2.) -- (2.,-2.5);
	\draw (1.,-2.) -- (2.,-1.5);
	\draw[fill] (2.,-2.5) circle (0.05);
	\draw[fill] (2.,-1.5) circle (0.05);
	\draw[red, thick] (0.,-2.) -- +(0,0.333333) ;
	\draw[red, thick] (2.,-2.5) -- +(0,0.333333) ;
	\draw[red, thick] (2.,-1.5) -- +(0,0.333333) ;
	\draw[dotted, thick] (2,-2) ellipse (.3cm and .85cm);
	\node at (2.8,-2.) {\scriptsize{$V_2$}};
\end{tikzpicture}
\,;
\qquad
\result{u}=
\begin{tikzpicture}[scale = .5, baseline=-.6cm]
	\draw[fill=white, rounded corners,inner sep=2pt] (-.5,-3) rectangle (3.5,1);
	\draw[fill] (0,0) circle (0.05);
	\draw (0.,0.) -- (1.,0.);
	\draw[fill] (1.,0.) circle (0.05);
	\draw (1.,0.) -- (2.,-0.5);
	\draw (1.,0.) -- (2.,0.5);
	\draw[fill] (2.,-0.5) circle (0.05);
	\draw[fill] (2.,0.5) circle (0.05);
	\draw (2.,0.5) -- (3.,0.);
	\draw (2.,-0.5) -- (3.,0.);
	\filldraw (3.,0.) circle (0.05);
	\draw[red, thick] (0.,0.) -- +(0,0.333333) ;
	\draw[red, thick] (2.,-0.5) -- +(0,0.333333) ;
	\draw[red, thick] (2.,0.5) -- +(0,0.333333) ;
%
	\draw[fill] (0,-2) circle (0.05);
	\draw (0.,-2.) -- (1.,-2.);
	\draw[fill] (1.,-2.) circle (0.05);
	\draw (1.,-2.) -- (2.,-2.5);
	\draw (1.,-2.) -- (2.,-1.5);
	\draw[fill] (2.,-2.5) circle (0.05);
	\draw[fill] (2.,-1.5) circle (0.05);
	\draw (2.,-1.5) -- (3.,-2.);
	\draw (2.,-2.5) -- (3.,-2.);
	\filldraw (3.,-2.) circle (0.05);
	\draw[red, thick] (0.,-2.) -- +(0,0.333333) ;
	\draw[red, thick] (2.,-2.5) -- +(0,0.333333) ;
	\draw[red, thick] (2.,-1.5) -- +(0,0.333333) ;
\end{tikzpicture}
\,;
\qquad
u^{-1}=
\begin{tikzpicture}[scale = .5, baseline=-.6cm]
	\draw[fill=white, rounded corners,inner sep=2pt] (-.5,-3) rectangle (4.2,1);
	\draw[fill] (0,0) circle (0.05);
	\draw (0.,0.) -- (1.,0.);
	\draw[fill] (1.,0.) circle (0.05);
	\draw (1.,0.) -- (2.,-0.5);
	\draw (1.,0.) -- (2.,0.5);
	\draw[fill] (2.,-0.5) circle (0.05);
	\draw[fill] (2.,0.5) circle (0.05);
	\draw (2.,0.5) -- (3.,0.);
	\draw (2.,-0.5) -- (3.,0.);
	\filldraw (3.,0.) circle (0.05);
	\draw[red, thick] (0.,0.) -- +(0,0.333333) ;
	\draw[red, thick] (2.,-0.5) -- +(0,0.333333) ;
	\draw[red, thick] (2.,0.5) -- +(0,0.333333) ;
	\draw[dotted, thick] (3,0) circle (.3cm);
	\node at (3.8,0) {\scriptsize{$w_1$}};
	\draw[fill] (0,-2) circle (0.05);
	\draw (0.,-2.) -- (1.,-2.);
	\draw[fill] (1.,-2.) circle (0.05);
	\draw (1.,-2.) -- (2.,-2.5);
	\draw (1.,-2.) -- (2.,-1.5);
	\draw[fill] (2.,-2.5) circle (0.05);
	\draw[fill] (2.,-1.5) circle (0.05);
	\draw (2.,-1.5) -- (3.,-2.);
	\draw (2.,-2.5) -- (3.,-2.);
	\filldraw (3.,-2.) circle (0.05);
	\draw[red, thick] (0.,-2.) -- +(0,0.333333) ;
	\draw[red, thick] (2.,-2.5) -- +(0,0.333333) ;
	\draw[red, thick] (2.,-1.5) -- +(0,0.333333) ;
	\draw[dotted, thick] (3,-2) circle (.3cm);
	\node at (3.8,-2.) {\scriptsize{$w_2$}};
\end{tikzpicture}
$$
An isomorphism $g: \result{u_1} \isoto \result{u_2}$ carrying $u_1^{-1} \in L(\result{u_1})$ to $u_2^{-1} \in L(\result{u_2})$ can be restricted to the vertices which are not being deleted, obtaining an automorphism $g':o \isoto o$. This automorphism then carries $u_1 \in U(o)$ to $u_2 \in U(o)$.
\end{boxedexample*}

\item
\label{condition:choice-function}
Finally, for each $o$ such that $L(o)$ is non-empty, we have a chosen
orbit $\phi(o) \in L(o)/\Aut(o)$ called the canonical reduction orbit of $o$.
These choices must be coherent with respect to the groupoid action, that is
for $g: o \iso o'$, $\phi(g(o)) = g(\phi(o))$.

(The choice of canonical reduction is the critical optimization step for this
algorithm; see below.)

\begin{boxedexample}
When $o$ is a PGP, we have many potential criteria to choose the canonical
reduction orbit in $L(o)$. Here we explain the general framework for such criteria, deferring the actual choice we make to Section \ref{sec:Implementations}. 

If there are no vertices at the working depth,
$L(o)$ is the singleton containing the lower object which decreases the
working depth, and there is no choice to make.  Otherwise, every lower object
deletes either a self-dual vertex or a pair of dual vertices at the working
depth. (Note that if the working depth is odd, there can not be any self-dual
vertices.)

We may make choices such as to prefer deleting vertices from $\Gamma_-$ over $\Gamma_+$, deleting dual pairs of vertices over deleting self-dual vertices, and so on, as long as these choices are invariant under the groupoid (because we are choosing an orbit, not a particular lower object). Our two implementations make slightly different choices here. After
expressing these preferences, there may still be alternatives, and indeed in the simplest case, where we have no such preferences, all the orbits are alternatives! 

The {\tt nauty} \cite{MR3131381} package provides an algorithm for canonically labelling the vertices of a vertex-coloured graph. Since the data of a lower object for a PGP can be encoded as a vertex-coloured graph, preserving automorphisms, it is always possible to use this canonical labelling to make a choice of orbit of lower objects.

The precise details differ in our two implementations; we describe these in Section \ref{sec:Implementations}.
\end{boxedexample}

\end{enumerate}
This concludes the definition of a McKay groupoid (and simultaneously how we see PGPs as an example).

\begin{defn}
Given an upper object $u\in U(o)$, we say it is \emph{genuine} exactly if
$u^{-1}$ is contained in the canonical reduction orbit $\phi(\result{u})$.
\end{defn}
Observe that this property is preserved by the groupoid action: if $g : o \iso
o'$, then $u$ is genuine if and only if $g(u)$ is genuine.

\begin{remark}
Making a clever choice of canonical reduction $\phi$ amongst the automorphism orbits in
$L(o)$ can provide a significant speed-up.
The key fact is that when preparing the upper objects as
above, we may omit any upper object that we know in advance can not possibly
be genuine.

If the choice function simply relies on canonical labellings from {\tt nauty}, it is essentially a black box, and it is not possible to make such predictions.
As such, it would be impossible to prune the list of upper
objects. Moreover, calls to {\tt nauty} can be computationally expensive. We
find that it is possible to specify $\phi$ in a way that drastically reduces
the number of calls needed; this is described in Section \ref{sec:Implementations}.
\end{remark}

\begin{boxedexample}
We consider the following PGP (with $\Gamma_+$ the upper graph, $\Gamma_-$ the lower graph).
$$
o=
\begin{tikzpicture}[scale = .5, baseline=-.6cm]
	\draw[fill=white, rounded corners,inner sep=2pt] (-.5,-3) rectangle (2.5,1);
	\draw[fill] (0,0) circle (0.05);
	\draw (0.,0.) -- (1.,0.);
	\draw[fill] (1.,0.) circle (0.05);
	\draw (1.,0.) -- (2.,-0.5);
	\draw (1.,0.) -- (2.,0.5);
	\draw[fill] (2.,-0.5) circle (0.05);
	\draw[fill] (2.,0.5) circle (0.05);
	\draw[red, thick] (0.,0.) -- +(0,0.333333) ;
	\draw[red, thick] (2.,-0.5) to[out=135,in=-135] (2.,0.5);
	\draw[fill] (0,-2) circle (0.05);
	\draw (0.,-2.) -- (1.,-2.);
	\draw[fill] (1.,-2.) circle (0.05);
	\draw (1.,-2.) -- (2.,-2.5);
	\draw (1.,-2.) -- (2.,-1.5);
	\draw[fill] (2.,-2.5) circle (0.05);
	\draw[fill] (2.,-1.5) circle (0.05);
	\draw[red, thick] (0.,-2.) -- +(0,0.333333) ;
	\draw[red, thick] (2.,-2.5) to[out=135,in=-135] (2.,-1.5);
\end{tikzpicture}
$$
We give an example of an upper object $S^+(V)\in U(o)$ which is not genuine since $S^+(V)^{-1}\notin \phi(\result{S^+(V)})$, for the choice function $\phi$ used in the {\tt Scala} implementation, given in detail in Section \ref{sec:Implementations} below.
$$
S^+(V)=
\begin{tikzpicture}[scale = .5, baseline=-.6cm]
	\draw[fill=white, rounded corners,inner sep=2pt] (-.5,-3) rectangle (2.5,1);
	\draw[fill] (0,0) circle (0.05);
	\draw (0.,0.) -- (1.,0.);
	\draw[fill] (1.,0.) circle (0.05);
	\draw (1.,0.) -- (2.,-0.5);
	\draw (1.,0.) -- (2.,0.5);
	\draw[fill] (2.,-0.5) circle (0.05);
	\draw[fill] (2.,0.5) circle (0.05);
	\draw[red, thick] (0.,0.) -- +(0,0.333333) ;
	\draw[red, thick] (2.,-0.5) to[out=135,in=-135] (2.,0.5);
	\draw[dotted, thick] (1,0) circle (.3cm);
	\node at (1,.6) {\scriptsize{$V$}};
	\draw[fill] (0,-2) circle (0.05);
	\draw (0.,-2.) -- (1.,-2.);
	\draw[fill] (1.,-2.) circle (0.05);
	\draw (1.,-2.) -- (2.,-2.5);
	\draw (1.,-2.) -- (2.,-1.5);
	\draw[fill] (2.,-2.5) circle (0.05);
	\draw[fill] (2.,-1.5) circle (0.05);
	\draw[red, thick] (0.,-2.) -- +(0,0.333333) ;
	\draw[red, thick] (2.,-2.5) to[out=135,in=-135] (2.,-1.5);
\end{tikzpicture}
\,;\,\,
\result{S^+(V)}=
\begin{tikzpicture}[scale = .5, baseline=-.6cm]
	\draw[fill=white, rounded corners,inner sep=2pt] (-.5,-3) rectangle (2.5,1);
	\draw[fill] (0,0) circle (0.05);
	\draw (0.,0.) -- (1.,0.);
	\draw[fill] (1.,0.) circle (0.05);
	\draw (1.,0.) -- (2.,-0.5);
	\draw (1.,0.) -- (2.,0.);
	\draw (1.,0.) -- (2.,0.5);
	\draw[fill] (2.,-0.5) circle (0.05);
	\draw[fill] (2.,0.) circle (0.05);
	\draw[fill] (2.,0.5) circle (0.05);
	\draw[red, thick] (0.,0.) -- +(0,0.166667) ;
	\draw[red, thick] (2.,-0.5) to[out=135,in=-135] (2.,0.);
	\draw[red, thick] (2.,0.5) -- +(0,0.166667) ;
	\draw[fill] (0,-2) circle (0.05);
	\draw (0.,-2.) -- (1.,-2.);
	\draw[fill] (1.,-2.) circle (0.05);
	\draw (1.,-2.) -- (2.,-2.5);
	\draw (1.,-2.) -- (2.,-1.5);
	\draw[fill] (2.,-2.5) circle (0.05);
	\draw[fill] (2.,-1.5) circle (0.05);
	\draw[red, thick] (0.,-2.) -- +(0,0.333333) ;
	\draw[red, thick] (2.,-2.5) to[out=135,in=-135] (2.,-1.5);
\end{tikzpicture}
\,;\,\,
S^+(V)^{-1}=
\begin{tikzpicture}[scale = .5, baseline=-.6cm]
	\draw[fill=white, rounded corners,inner sep=2pt] (-.5,-3) rectangle (3,1);
	\draw[fill] (0,0) circle (0.05);
	\draw (0.,0.) -- (1.,0.);
	\draw[fill] (1.,0.) circle (0.05);
	\draw (1.,0.) -- (2.,-0.5);
	\draw (1.,0.) -- (2.,0.);
	\draw (1.,0.) -- (2.,0.5);
	\draw[fill] (2.,-0.5) circle (0.05);
	\draw[fill] (2.,0.) circle (0.05);
	\draw[fill] (2.,0.5) circle (0.05);
	\draw[red, thick] (0.,0.) -- +(0,0.166667) ;
	\draw[red, thick] (2.,-0.5) to[out=135,in=-135] (2.,0.);
	\draw[red, thick] (2.,0.5) -- +(0,0.166667) ;
	\draw[dotted, thick] (2,.5) circle (.3cm);
	\node at (2.6,.5) {\scriptsize{$v$}};
	\draw[fill] (0,-2) circle (0.05);
	\draw (0.,-2.) -- (1.,-2.);
	\draw[fill] (1.,-2.) circle (0.05);
	\draw (1.,-2.) -- (2.,-2.5);
	\draw (1.,-2.) -- (2.,-1.5);
	\draw[fill] (2.,-2.5) circle (0.05);
	\draw[fill] (2.,-1.5) circle (0.05);
	\draw[red, thick] (0.,-2.) -- +(0,0.333333) ;
	\draw[red, thick] (2.,-2.5) to[out=135,in=-135] (2.,-1.5);
\end{tikzpicture}
$$
Because the choice function $\phi$ prefers to delete vertices from $\Gamma_-$, we have
$$
\phi(\result{S^+(V)})=
\left\{\,
\begin{tikzpicture}[scale = .5, baseline=-.6cm]
	\draw[fill=white, rounded corners,inner sep=2pt] (-.5,-3) rectangle (3.2,1);
	\draw[fill] (0,0) circle (0.05);
	\draw (0.,0.) -- (1.,0.);
	\draw[fill] (1.,0.) circle (0.05);
	\draw (1.,0.) -- (2.,-0.5);
	\draw (1.,0.) -- (2.,0.);
	\draw (1.,0.) -- (2.,0.5);
	\draw[fill] (2.,-0.5) circle (0.05);
	\draw[fill] (2.,0.) circle (0.05);
	\draw[fill] (2.,0.5) circle (0.05);
	\draw[red, thick] (0.,0.) -- +(0,0.166667) ;
	\draw[red, thick] (2.,-0.5) to[out=135,in=-135] (2.,0.);
	\draw[red, thick] (2.,0.5) -- +(0,0.166667) ;
	\draw[fill] (0,-2) circle (0.05);
	\draw (0.,-2.) -- (1.,-2.);
	\draw[fill] (1.,-2.) circle (0.05);
	\draw (1.,-2.) -- (2.,-2.5);
	\draw (1.,-2.) -- (2.,-1.5);
	\draw[fill] (2.,-2.5) circle (0.05);
	\draw[fill] (2.,-1.5) circle (0.05);
	\draw[red, thick] (0.,-2.) -- +(0,0.333333) ;
	\draw[red, thick] (2.,-2.5) to[out=135,in=-135] (2.,-1.5);
	\draw[dotted, thick] (2,-2.5) circle (.3cm);
	\node at (2.8,-2.5) {\scriptsize{$w_1$}};
	\draw[dotted, thick] (2,-1.5) circle (.3cm);
	\node at (2.8,-1.5) {\scriptsize{$w_2$}};
\end{tikzpicture}
\right\}.
$$
\end{boxedexample}

Our strategy now, described in detail in the next section, is to build  a tree of objects in $\cO$, such that the children of an object $o$ are
the resulting objects of representatives of the $\Aut(o)$ orbits of genuine upper objects.

\subsection{Exhaustivity and uniqueness}

We write $K(\cO)$ for the set of isomorphism classes of the groupoid $\cO$, and $K(\cO_1)$ for the set of isomorphisms classes of elements $o\in \cO$ such that $L(o)$ is not empty.

\begin{lem}[cf. {\cite[Lemma 1]{MR1606516}}\footnote{The statement of Lemma 1 in \cite{MR1606516} is slightly incorrect; it should say that for each element $\hat{X} \in f(\check{X})$ there is a $Y \in p(S)$ such that $\hat{X} \in U(Y)$, rather than that all elements of $f(\check{X})$ are in a single such $U(Y)$.}]
There is a unique function $\pi: K(\cO_1) \to K(\cO)$ such that for each
\begin{itemize}
\item $o \in [o] \in K(\cO_1)$,
\item $l \in \phi(o)$, and
\item $o' \in \cO$ and $u \in U(o')$ such that $u^{-1} \iso l$,
\end{itemize}
we have $o' \in \pi([o])$.
\end{lem}
\begin{proof}
We just need to show that for any allowed choices $(o_1,l_1,o'_1,u_1)$ and $(o_2, l_2, o'_2, u_2)$ with $o_1 \iso o_2$, we have $o'_1 \iso o'_2$.

Condition \ref{condition:choice-function} shows $l_1 \iso l_2$. Now $u_1^{-1} \iso u_2^{-1}$, and by Condition \ref{condition:inverses}, $u_1 \iso u_2$, which in particular (recall Remark \ref{rem:isomorphisms-between-upper-objects}) means $o'_1 \iso o'_2$.
\end{proof}

As in \cite[p. 6]{MR1606516}, this is called the parent function.  An ancestor
of an isomorphism class $[o]$ is an $[o']$ such that $\pi^k([o]) = [o']$ for
some $k$.  There is an obvious notion of the children and descendants of
$[o]$.  Because $\pi$ reduces the level by one by Condition
\ref{condition:lower-objects}, it is clear that every isomorphism class is a
descendant of some `progenitor' $[o]$ with $L(o)$ empty.

Given a McKay groupoid $\cO$ satisfying Conditions \ref{condition:level}--\ref{condition:choice-function}, we define a forest $\cF$ of elements. 
Theorems \ref{thm:exhaustive} and \ref{thm:uniqueness} below shows that this forest consists of a single representative of each isomorphism class $[o] \in K(\cO)$.
\begin{defn}
\label{defn:tree}
We first define a tree $\cT_{[r]} \subset \cO$ for any isomorphism class $[r]$. The root is an arbitrarily chosen representative $r$ of $[r]$, and the children of any node $o \in \cT_{[r]}$ are the obtained as follows:
\begin{itemize}
\item for each orbit in $U(o)/\Aut(o)$,
\item pick a representative $u$, and
\item accept $\result{u}$ if $u$ is genuine (i.e. $u^{-1} \in \phi(\result{u})$), rejecting otherwise.
\end{itemize}

The forest $\cF$ is the union of all the trees $\cT_{[r]}$ where $[r]$ varies over all isomorphism classes for which $L(r)$ is empty.
\end{defn}

\begin{remark}
This definition is a close parallel of the procedure ${\tt scan}(X,n)$  \cite[p. 6]{MR1606516}. Our setup is less general, and in particular we use $u^{-1}$ as the canonical choice of $\check{Y} \in f'(\widehat{X})$ in ${\tt scan}(X,n)$.
\end{remark}

\begin{thm}
\label{thm:exhaustive}
For any $r \in \cO$, if $[o]$ is descended from $[r]$, then some $o \in [o]$ appears in $\cT_{[r]}$.
\end{thm}
\begin{proof}
We say a descendant $[o]$ of $[r]$ is in generation $i$ if $\pi^i([o]) = [r]$.
We induct on the generation. The base case is trivial.

Consider $[o]$ descended from $[r]$ in generation $i+1$. Pick $o$, a
representative of $[o]$, some $l \in \phi(o)$, and $u$ such that $l \iso
u^{-1}$. The upper object $u$ is an upper object for some $o'$.
By the definition of $\pi$, $[o'] = \pi([o])$, and $[o']$ is also a descendant
of $[r]$, but in generation at most $i$. By induction, some other $o'' \iso o'$
appears in $\cT_{[r]}$. As $U$ is equivariant by Condition
\ref{condition:equivariance}, there is a $u' \in U(o'')$, with $u' \iso u$. We
see that ${u'}^{-1} \iso l$, and so $\result{u'} \iso o$.

Finally, we need to check that ${u'}^{-1} \in \phi(\result{u'})$, so that we
accept $\result{u'}$. As $\phi$ is also equivariant by Condition
\ref{condition:choice-function}, this follows from $l \in \phi(o)$.
\end{proof}

\begin{thm}
\label{thm:uniqueness}
Given any $[r]$, the elements of $\cT_{[r]}$ are pairwise non-isomorphic.
\end{thm}
\begin{proof}
Again, we induct on the generation. The base case is trivial.

Suppose $o'_1$ and $o'_2$ in generation $i+1$ of $\cT_{[r]}$ are isomorphic.
Because they are in $\cT_{[r]}$, there are $u_1$ and $u_2$ with $o'_i =
\result{u_i}$ and $u_i^{-1} \in \phi(o'_i)$.  Combining the equivariance of
$\phi$ (Condition \ref{condition:choice-function}) and the fact that $\phi(o)$ is
a single orbit, we have that $u_1^{-1} \iso u_2^{-1}$, and hence by Condition
\ref{condition:inverses} $u_1 \iso u_2$.

Let $o_i = \result{u_i^{-1}}$. By construction, $o_i \in \cT_{[r]}$ and $o_1
\iso o_2$, so by the inductive hypothesis $o_1 = o_2$.  We see that $u_i$ was
our chosen representative of $U(o_i)/\Aut(o_i)$ in Definition \ref{defn:tree},
but now only one choice is available, so in fact $u_1 = u_2$, and hence $o'_1
= o'_2$, as desired.
\end{proof}

\begin{cor}
The elements of $\cF$ are pairwise non-isomorphic, because the root of the tree an element appears in is an isomorphism invariant.
\end{cor}

As noted above, the only PGP with no parent is $(\eset, \eset, id, 0)$. In
practice we are very often interested in the descendants of $\left(\scalebox{0.5}{$\bigraph{bwd1duals1}$},\scalebox{0.5}{$\bigraph{bwd1duals1}$}\right)$, as all
principal graphs of irreducible subfactors are of this form. However, in
Section \ref{sec:Implementations}, where we describe the `tail enumerator', we
will be interested in other roots.

\subsection{The implementations}
\label{sec:Implementations}

We have two independent implementations of the algorithm described above for
PGPs. The first was written by Narjess Afzaly, as part of her ANU PhD thesis
work with Brendan McKay, and is implemented in {\tt C}. The second was written
later by Kim Morrison and David Penneys, in {\tt Scala}. The implementations
are independent in the sense that they share no common code, and in fact
neither group read the code of the other implementation. The {\tt C}
implementation is faster, although both programs suffice to do all the
computations required in this paper. To the extent possible (subject to the
constraints described below) all computations have been reproduced in both
implementations and compared.

Both implementations are best run by means of the {\tt Mathematica} wrappers
we've prepared for them. The {\tt Scala} code is used in the {\tt Mathematica}
notebook {\tt enumerator.nb} included with the {\tt arXiv} sources of this
article to give the proof of Theorem \ref{thm:Enumerate} below. The {\tt C} 
code can be run by loading {\tt /development/afzaly-enumerator/Enumeration-setup.m} from the {\tt FusionAtlas} repository in a {\tt Mathematica} session
(after first loading the {\tt FusionAtlas} itself), and then using the command
{\tt ExtendToDepth}.

Our choice function $\phi$ is specified as follows, in descending order of priority:
\begin{itemize}
\item In the {\tt Scala} implementation only, if the PGP has vertices on both the principal and dual principal graphs at the
working depth, then the canonical reduction will be a lower object which
removes vertices from the dual principal graph.

\item If the PGP has both self-dual vertices and pairs of dual vertices at the
working depth, then the canonical reduction will be a lower object which
removes a dual pair.  The canonical reduction only removes a self-dual vertex
if all vertices at the working depth are self-dual.

\item Subject to these constraints, the canonical reduction is a lower object which
removes a set (either a single vertex or a dual pair) of vertices with least
total degree amongst the lower objects.

\item Finally, ties are broken using canonical labellings from {\tt nauty}, according to one of the two strategies described here. The 
{\tt C} implementation uses the first strategy, while the {\tt Scala}
implementation uses the second strategy.

In both, we need to encode PGPs as vertex-coloured graphs, in order to be able to use {\tt nauty}. From the underlying pair $(\Gamma_+, \Gamma_-)$ for $o$, we apply a vertex-colouring according to depth, and additionally add new (depth-preserving) edges between pairs of dual objects. Then the
automorphism group of this graph, which we denote $G(\Gamma_+, \Gamma_-)$ is exactly the automorphism group of the PGP $o$.

\begin{enumerate}[label=(\arabic*)]
\item We observe that all the lower objects we are considering are described by some subset of
vertices (of size 1 or 2) on a fixed graph. Only some of these subsets are allowed according to the choices described
above.  We use {\tt nauty} \cite {MR3131381} to compute a
canonical labelling of the vertices of the graph $G(\Gamma_+, \Gamma_-)$,  as well as the action of the automorphism group on subsets of the appropriate size (in some cases, the {\tt C} implementation shortcuts this calculation, using the action on vertices to quickly deduce the action of subsets of size 2). 
We identify which subset has least canonical labelling, and then choose the orbit of lower objects consisting of the images of this subset under the automorphism action.

\item For each orbit $[l]$ of lower objects satisfying our preferences, we pick a representative $l
\in L(o)$, and construct a single vertex-coloured graph $\cG_l$ encoding $l$, in an $\Aut(o)$ equivariant manner. This
is the same vertex-coloured graph as that described in the first alternative, with the addition of an extra
vertex-colour for the vertices to be deleted by $l$. We then call {\tt nauty} to canonically label
the vertices of $\cG_l$, obtaining a vertex-coloured graph $\cG_{[l]}$ which did not depend on the representative $l$.
We then define a total ordering on vertex-coloured graphs (e.g. dictionary order on a textual representation), and
declare that our chosen orbit is the one with the least $\cG_{[l]}$.
\end{enumerate}
Note the second approach may require several calls to ${\tt nauty}$ when the first requires just one; the {\tt Scala} implementation is
certainly less efficient.
\end{itemize}

\begin{boxedexample}
For simplicity, we consider the following graph (not a graph pair) at an even working depth, whose lower objects consist of deleting vertices at the working depth.
$$
\begin{tikzpicture}
	\draw[fill=white, rounded corners,inner sep=2pt] (-.5,-1.2) rectangle (4.5,1.2);
\draw[fill] (0,0.) circle (0.05);
\draw (0.,0.) -- (1.,0.);
\draw[fill] (1.,0.) circle (0.05);
\draw (1.,0.) -- (2.,-0.375);
\draw (1.,0.) -- (2.,-0.125);
\draw (1.,0.) -- (2.,0.125);
\draw (1.,0.) -- (2.,0.375);
\draw[fill] (2.,-0.375) circle (0.05);
\draw[fill] (2.,-0.125) circle (0.05);
\draw[fill] (2.,0.125) circle (0.05);
\draw[fill] (2.,0.375) circle (0.05);
\draw (2.,-0.375) -- (3.,-0.875);
\draw (2.,-0.375) -- (3.,-0.625);
\draw (2.,-0.125) -- (3.,-0.375);
\draw (2.,-0.125) -- (3.,-0.125);
\draw (2.,0.125) -- (3.,0.125);
\draw (2.,0.125) -- (3.,0.375);
\draw (2.,0.375) -- (3.,0.625);
\draw (2.,0.375) -- (3.,0.875);
\draw[fill] (3.,-0.875) circle (0.05);
\draw[fill] (3.,-0.625) circle (0.05);
\draw[fill] (3.,-0.375) circle (0.05);
\draw[fill] (3.,-0.125) circle (0.05);
\draw[fill] (3.,0.125) circle (0.05);
\draw[fill] (3.,0.375) circle (0.05);
\draw[fill] (3.,0.625) circle (0.05);
\draw[fill] (3.,0.875) circle (0.05);
\draw (3.,-0.875) -- (4.,-1.);
\draw (3.,-0.625) -- (4.,-1.);
\draw (3.,-0.875) -- (4.,-0.75);
\draw (3.,-0.625) -- (4.,-0.75);
\draw (3.,-0.375) -- (4.,-0.5);
\draw (3.,-0.125) -- (4.,-0.5);
\draw (3.,-0.375) -- (4.,-0.25);
\draw (3.,-0.125) -- (4.,-0.25);
\draw (3.,0.125) -- (4.,0.);
\draw (3.,0.375) -- (4.,0.25);
\draw (3.,0.625) -- (4.,0.5);
\draw (3.,0.875) -- (4.,0.75);
\draw (3.,0.875) -- (4.,1.);
\draw[fill] (4.,-1.) circle (0.05);
\draw[fill] (4.,-0.75) circle (0.05);
\draw[fill] (4.,-0.5) circle (0.05);
\draw[fill] (4.,-0.25) circle (0.05);
\draw[fill] (4.,0.) circle (0.05);
\draw[fill] (4.,0.25) circle (0.05);
\draw[fill] (4.,0.5) circle (0.05);
\draw[fill] (4.,0.75) circle (0.05);
\draw[fill] (4.,1.) circle (0.05);
\draw[red, thick] (0.,0.) -- +(0,0.0833333) ;
\draw[red, thick] (2.,-0.375) -- +(0,0.0833333) ;
\draw[red, thick] (2.,-0.125) -- +(0,0.0833333) ;
\draw[red, thick] (2.,0.125) -- +(0,0.0833333) ;
\draw[red, thick] (2.,0.375) -- +(0,0.0833333) ;
\draw[red, thick] (4.,-1.) to[out=135,in=-135] (4.,-0.75);
\draw[red, thick] (4.,-0.5) to[out=135,in=-135] (4.,-0.25);
\draw[red, thick] (4.,0.) to[out=135,in=-135] (4.,0.25);
\draw[red, thick] (4.,0.5) to[out=135,in=-135] (4.,0.75);
\draw[red, thick] (4.,1.) -- +(0,0.0833333) ;
\end{tikzpicture}
$$
There are 4 orbits of lower objects under the $\Aut(o)$ action:
\begin{enumerate}[label=(\arabic*)]
\item deleting one of the first two dual pairs (counting from the bottom), 
\item deleting the third dual pair, 
\item deleting the fourth dual pair, and 
\item deleting the self-dual vertex.
\end{enumerate}
 
Here, the canonical reduction orbit must be deleting a dual pair with least
total degree amongst dual pairs, so it must be deleting either the third or
fourth dual pair. We then call {\tt nauty} to produce a canonical labeling to
break the tie.
\end{boxedexample}


We note the following differences between the implementations. The optimizations made in the {\tt C} implementation are described in detail in Afzaly's Ph.D. thesis \cite{narjess}.
\begin{itemize}
\item
The {\tt C} implementation only enumerates simply-laced principal graphs (as
reflected in the definition of PGPs above), while the {\tt Scala} implementation can
also produce non-simply-laced graphs (requiring the obvious modifications to
the definitions above). Given Lemma \ref{lem:simply-laced}, this is not a
significant difference for the purposes of this paper.

\item The {\tt C} implementation uses certain shortcuts for estimating graph norms, while the {\tt Scala} implementation uses the straightforward heuristic of bounding a graph norm below by $|| A^{n+1} v || / || A^n v ||$ for any chosen $n$ and $v$ (taking $n=10$ and $v$ the vector which is $1$ on every vertex is good enough).

\item
The {\tt C} implementation differs slightly from the description above in that it treats the graphs $(\Gamma_+, \Gamma_-)$ in a PGP as an unordered pair. Thus the output from the {\tt Scala} implementation contains $(\Gamma_+, \Gamma_-)$ and $(\Gamma_-, \Gamma_+)$ separately whenever $\Gamma_+ \neq \Gamma_-$.

\item 
The {\tt Scala}
implementation assumes that there is exactly one vertex on each graph in depth
0, while the {\tt C} implementation allows arbitrarily many. This is a more
significant limitation, as this freedom is essential for the `tail enumerator'
described below, and used in Section \ref{sec:periodicity}.

\item The {\tt C} implementation does not natively implement the triple point obstruction; we filter its output using an implementation of the triple point obstruction written in {\tt Mathematica} (thus the two overall implementations maintain separate codebases, although both triple point obstructions were implemented by the second and third authors).

\item
Finally, the {\tt Scala}
implementation includes one additional inequality, which sometimes rules out a
PGP on the basis of having no possible associative descendants. Let $\nbhd^+(v)$ and $\nbhd^-(v)$ denote the neighbours of a vertex $v$ at the next and the previous depth respectively. 
We can decompose both sides of the associativity constraint in Equation \eqref{eq:associativity} as a sum of two positive terms. For vertices $v$ and $w$ at the same depth (and in particular depth $n-1$) we have
$$ A + B = C + D $$
where
\begin{align*}
A & = \left| \overline{\nbhd^+(v)} \cap \nbhd^+(\overline{w}) \right| &
B & = \left| \overline{\nbhd^-(v)} \cap \nbhd^-(\overline{w}) \right| \\
C & = \left| \nbhd^+(\overline{v}) \cap \overline{\nbhd^+(w)} \right| & 
D & = \left| \nbhd^-(\overline{v}) \cap \overline{\nbhd^-(w)} \right|.
\end{align*}
We are interested in the specific case where $n$ is even, and there are already vertices on $\Gamma_-$ at depth $n$. 
Consider $v$ on $\Gamma_+$ at depth $n-1$, and $w$ on $\Gamma_-$ at depth $n-1$.
We know that in any genuine child the additional vertices at depth $n$ will also be on $\Gamma_-$. 
Thus three out of the four terms above will not change: only $C = \left| \nbhd^+(\overline{v}) \cap \overline{\nbhd^+(w)} \right|$ will increase as we look at genuine children. 
Thus we can discard any graph in this situation for which $A + B - C - D$ is already negative. 
(In fact, this apparently rather specific check saves a huge amount of effort, reducing the total number of graphs considered in our application from $239710$ to $17360$!)
The {\tt C} implementation similarly implements associativity checks as early as possible \cite{narjess}, although the details are different because of the symmetry between $\Gamma_+$ and $\Gamma_-$ there.
\end{itemize}

\paragraph{The tail enumerator.}
We can apply the above algorithm starting with a PGP which does not represent
the entirety of a subfactor principal graph up to some depth, but rather a
`block' $\cB$ of a principal graph containing only the vertices at depths
between some specified limits $a$ and $a+k$. Using this, we can
enumerate all possible blocks $\cC$ such that $\cA\cB\cC$ could be the
principal graph of subfactor, where $\cA\cB\cC$ represents a principal graph
given by some unknown $\cA$ from depth 0 to depth $a$ (with $a$ itself perhaps
also unknown), the fixed block $\cB$ from depth $a$ to depth $a+k$, and the
block $\cC$ from depth $a+k$ onwards. We call this application of the general
algorithm the \emph{`tail enumerator'}. In practice we use it knowing that
$\cA$ is of the form $\cA_{(0)}\cB^m$, for some fixed initial graph
$\cA_{(0)}$ followed by an unknown number of repeats of a certain block $\cB$.

\subsection{Application to graph pairs up to index \texorpdfstring{$5\frac{1}{4}$}{5 1/4}}
\label{sec:Enumerate-proof}

\begin{proof}[Proof of Theorem \ref{thm:Enumerate}]
The algorithm described in Section \ref{sec:orderly}, implemented as described
in Section \ref{sec:Implementations}, and applied to the input $\left(\scalebox{0.5}{$\bigraph{bwd1duals1}$},\scalebox{0.5}{$\bigraph{bwd1duals1}$}\right)$, enumerates all possible principal graphs
of irreducible subfactors up to a given index, by Theorem \ref{thm:exhaustive}. The output is a tree whose nodes constitute a single
representative of every isomorphism class, and is produced via a depth-first
traversal. There is no guarantee that this tree is not infinite. We may
instruct the enumeration algorithm to ignore any specified sub-trees.

A subfactor principal graph needs to satisfy more associativity conditions
than PGPs do --- in particular it must satisfy associativity between pairs of
vertices at the penultimate depth, or between pairs of vertices at the
ultimate depth. As a result, we don't actually need to look at every node of
this tree, but rather only those nodes which have no vertices at the working
depth (i.e. graphs just produced by increasing the depth). These are exactly
the PGPs which satisfy the associativity condition at the penultimate depth
(which is actually $n-2$ for a PGP with no vertices at the working depth $n$),
and we then separately filter out those that satisfy the associativity
condition at the ultimate depth $n-1$.

As a simplification in the primary implementation, we only generate 
simply-laced graphs, requiring case \ref{case:not-simply-laced} above. (The secondary implementation does not
have this restriction.) We ask the algorithm to ignore any subtrees which are
exactly 1-supertransitive (giving case \ref{case:1-supertransitive}), or which are 4-supertransitive
(every 4-supertransitive graph other than $A_\infty$ is a translation of an
exactly 2- or 3-supertransitive graph), or which correspond to extensions of
any of the graphs listed in cases \ref{case:11}, \ref{case:10}, and \ref{case:4-spoke}. Finally, the graphs
listed in case \ref{case:vines} are exactly the output of this program.

For reference, the {\tt Scala} implementation takes about 440s on a 1.7GHz Core i7 processor, and considers 17360 PGPs
in total. (Running only up to index 5 takes 20s, considering 992 PGPs, while running up to index $3+\sqrt{3}$ takes
6s, considering 251 PGPs.)
\end{proof}

\section{Weeds with branch factor \texorpdfstring{$r=1$}{r=1}}
\label{sec:r=1}

Recall from \cite{MR2972458,MR2902285} that for an $n-1$ supertransitive weed starting with an initial triple point, 
$$
\begin{tikzpicture}[baseline=-.1cm]
	\draw[fill] (-2,0) circle (0.05);
	\node at (-2,-.3) {\scriptsize{$0$}};	
	\draw (-2.,0.) -- (-1.,0.);
	\draw[fill] (-1,0) circle (0.05);
	\node at (-1,-.3) {\scriptsize{$1$}};
	\node at (-.5,0) {$\cdots$};
	\draw[fill] (0,0) circle (0.05);
	\node at (0,-.3) {\scriptsize{$n-2$}};
	\draw (0.,0.) -- (1.,0.);
	\draw[fill] (1.,0.) circle (0.05);
	\node at (1,-.3) {\scriptsize{$n-1$}};
	\draw (1.,0.) -- (2.,-0.5);
	\draw (1.,0.) -- (2.,0.5);
	\draw[fill=white] (2.,-0.5) circle (0.05);
	\node at (2.3,.5) {$Q$};
	\draw[fill=white] (2.,0.5) circle (0.05);
	\node at (2.3,-.5) {$P$};
	\node at (2,-.8) {\scriptsize{$n$}};	
\end{tikzpicture}\,,
$$
the \emph{branch factor} $r= \Tr(Q)/\Tr(P)$ is the ratio of the
dimensions of the vertices one past the branch. The branch factor is connected
to the rotational eigenvalue $\omega_A$ of the new low-weight rotational
eigenvector $A$ at depth $n$, which is the new element in the $n$-box space
perpendicular to Temperley-Lieb-Jones. When $r\neq 1$, there are tight restrictions
on the possible dimensions of vertices, which we can leverage to rule out
weeds using branch factor inequalities. Section
\ref{sec:BranchFactorInequalities} is devoted to this task.

However, when $r=1$, this corresponds to $\omega_A=-1$, and branch factor
inequalities no longer help. In Theorem \ref{thm:Enumerate}, we see three
weeds with branch factor $r=1$, namely \ref{case:10}(5), \ref{case:10}(6), and \ref{case:10}(9), and we eliminate each weed with a different
`bespoke' argument, of varying difficulty.

For example, the weed $\cB$ from \cite{MR2914056,MR2902285} was eliminated
using connections together with an intricate graph norm argument. However, as
noted in \cite[Section 5.2.1]{MR3166042}, this weed is easily eliminated by
Popa's principle graph stability \cite{MR1334479,MR3157990}. We now eliminate
the truncation $\cB'$ of $\cB$ by 2. (This is \ref{case:10}(5) from Theorem \ref{thm:Enumerate}.)

\begin{thm}
\label{thm:NoB}
There are no subfactors with principal graph a translated extension of 
$$
\cB'=
\left(\bigraph{bwd1v1v1v1p1v1x0p0x1v0x1p1x0p1x0p0x1v0x0x0x1p0x0x1x0vduals1v1v1x2v1x2x4x3v}, \bigraph{bwd1v1v1v1p1v1x0p1x0v1x0p0x1v1x0p0x1vduals1v1v1x2v1x2v}\right).
$$
\end{thm}

\begin{proof}
This weed is actually stable at the penultimate depth, so any subfactor extension must end with $A_{\text{finite}}$ tails \cite{MR1334479,MR3157990} (this result relies on \cite{MR1356624}).
However \cite[Lemma 4.14]{MR2902285} shows that associativity is never satisfied for an extension with $A_{\text{finite}}$ tails, a contradiction. 
\end{proof}

The next weed, \ref{case:10}(6) from Theorem \ref{thm:Enumerate}, is stable.
By the Stability Constraint \ref{Fact:StabilityConstraint}, any extension could only be realized by graphs in the following two parameter family.
$$
\left(
\begin{tikzpicture}[baseline=-.1cm, scale=.7]
	\draw[fill] (0,0.) circle (0.05);
	\draw (0.,0.) -- (1.,0.);
	\draw[fill] (1.,0.) circle (0.05);
	\draw[dashed] (1.,0.) -- (2.,0.);
	\draw[fill] (2.,0.) circle (0.05);
	\draw (2.,0.) -- (3.,0.);
	\draw[fill] (3.,0.) circle (0.05);
	\draw (3.,0.) -- (4.,-0.5);
	\draw (3.,0.) -- (4.,0.5);
	\draw[fill] (4.,-0.5) circle (0.05);
	\draw[fill] (4.,0.5) circle (0.05);
	\draw (4.,-0.5) -- (5.,-0.5);
	\draw (4.,0.5) -- (5.,0.5);
	\draw[fill] (5.,-0.5) circle (0.05);
	\draw[fill] (5.,0.5) circle (0.05);
	\draw (5.,-0.5) -- (6.,0.);
	\draw (5.,0.5) -- (6.,0.);
	\draw[fill] (6.,0.) circle (0.05);
	\draw (6.,0.) -- (7.,0.);
	\draw[fill] (7.,0.) circle (0.05);
	\draw (7.,0.) -- (8.,-0.5);
	\draw (7.,0.) -- (8.,0.5);
	\draw[fill] (8.,-0.5) circle (0.05);
	\draw[fill] (8.,0.5) circle (0.05);
	\draw[dashed] (8.,-0.5) -- (9.,0.);
	\draw[fill] (9.,0.) circle (0.05);
	\draw (9.,0.) -- (10.,0.);
	\draw[fill] (10.,0.) circle (0.05);
	\draw[red, thick] (0.,0.) -- +(0,0.333333) ;
	\draw[red, thick] (2.,0.) -- +(0,0.333333) ;
	\draw[red, thick] (4.,-0.5) -- +(0,0.333333) ;
	\draw[red, thick] (4.,0.5) -- +(0,0.333333) ;
	\draw[red, thick] (6.,0.) -- +(0,0.333333) ;
	\draw[red, thick] (8.,-0.5) -- +(0,0.333333) ;
	\draw[red, thick] (8.,0.5) -- +(0,0.333333) ;
	\node at (1.5,-1) {$\underbrace{\hspace{2.2cm}}_{a\text{ edges}}$};
	\node at (8.5,-1) {$\underbrace{\hspace{2.2cm}}_{b\text{ edges}}$};
\end{tikzpicture}
\,,\,
\begin{tikzpicture}[baseline=-.1cm, scale=.7]
	\draw[fill] (0,0.) circle (0.05);
	\draw (0.,0.) -- (1.,0.);
	\draw[fill] (1.,0.) circle (0.05);
	\draw[dashed] (1.,0.) -- (2.,0.);
	\draw[fill] (2.,0.) circle (0.05);
	\draw (2.,0.) -- (3.,0.);
	\draw[fill] (3.,0.) circle (0.05);
	\draw (3.,0.) -- (4.,-0.5);
	\draw (3.,0.) -- (4.,0.5);
	\draw[fill] (4.,-0.5) circle (0.05);
	\draw[fill] (4.,0.5) circle (0.05);
	\draw (4.,-0.5) -- (5.,-0.5);
	\draw (4.,-0.5) -- (5.,0.5);
	\draw[fill] (5.,-0.5) circle (0.05);
	\draw[fill] (5.,0.5) circle (0.05);
	\draw (5.,-0.5) -- (6.,-0.5);
	\draw (5.,0.5) -- (6.,0.5);
	\draw[fill] (6.,-0.5) circle (0.05);
	\draw[fill] (6.,0.5) circle (0.05);
	\draw (6.,-0.5) -- (7.,0.);
	\draw (6.,0.5) -- (7.,0.);
	\draw[fill] (7.,0.) circle (0.05);
	\draw (7.,0.) -- (8.,0.);
	\draw[fill] (8.,0.) circle (0.05);
	\draw[dashed] (8.,0.) -- (9.,0.);
	\draw[fill] (9.,0.) circle (0.05);
	\draw (9.,0.) -- (10.,0.);
	\draw[fill] (10.,0.) circle (0.05);
	\draw[red, thick] (0.,0.) -- +(0,0.333333) ;
	\draw[red, thick] (2.,0.) -- +(0,0.333333) ;
	\draw[red, thick] (4.,-0.5) -- +(0,0.333333) ;
	\draw[red, thick] (4.,0.5) -- +(0,0.333333) ;
	\draw[red, thick] (6.,-0.5) -- +(0,0.333333) ;
	\draw[red, thick] (6.,0.5) -- +(0,0.333333) ;
	\draw[red, thick] (8.,0.) -- +(0,0.333333) ;
	\node at (1.5,-1) {$\underbrace{\hspace{2.2cm}}_{a\text{ edges}}$};
	\node at (8.5,-1) {$\underbrace{\hspace{2.2cm}}_{b\text{ edges}}$};
\end{tikzpicture}
\right)
$$
This family has been called the `AMP
spider'.
To eliminate this weed, a new number-theoretic technique was developed by
Calegari-Guo \cite{1502.00035}. This technique is a significant
generalisation of and improvement over
the main result of \cite{MR2786219} which in turn was developed to 
treat uniformly the vines in the
classification of subfactors to index 5 \cite{MR2902286}. As discussed in the
introduction, we hope that it will be possible to implement this technique to
 treat uniformly cylinders (stable weeds) as we treat vines now.
 
\begin{thm}
\label{thm:NoAMP}
There are no subfactor planar algebras whose principal graphs are an instance of the AMP spider. 
\end{thm}
\begin{proof}
The article \cite{1502.00035} shows that any subfactor principal graph which
is a translated extension of the AMP spider must have $b\leq 56$. This reduces
this weed to 56 vines, which we deal with using the algorithms from
\cite{MR2902286}. The only cyclotomic translates are when $(a,b)\in
\{(0,0),(1,1)\}$. This calculation is performed in the {\tt Mathematica} notebook
{\tt{CalegariGuoSmallCases.nb}} bundled with the {\tt{arXiv}} source.

We now eliminate the two remaining cases.
First, $(a,b)=(0,0)$, corresponding to the graph
$$
\bigraph{gbg1v1v1p1v1x0p0x1v1x1}
$$
is not possible, since the branch point does not occur at an odd depth, which is necessary by Ocnenau's triple point obstruction \cite{MR1317352} (see also \cite[Theorem 5.1.11]{MR2972458}).
To eliminate 
$$
\left(
\bigraph{bwd1v1p1v1x0p0x1v1x1v1v1p1duals1v1x2v1v1x2},
\bigraph{bwd1v1p1v1x0p1x0v1x0p0x1v1x1v1duals1v1x2v1x2v1}
\right),
$$
we note that this graph pair has index 5, which is not a composite index.
However, by looking at the dual graph, we see that there is a normalizer, which would give rise to an intermediate subfactor \cite[Proposition 1.7]{MR860811}, a contradiction.
(A planar algebraic proof of this easy fact is available at \cite[Lemma 4.7]{1308.5723}.)
\end{proof}

\subsection{The remaining \texorpdfstring{$r=1$}{r=1} weed}
\label{sec:remaining-r=1}

In the remainder of this section, we eliminate the final weed with branch factor $r=1$, \ref{case:10}(9) from Theorem \ref{thm:Enumerate}.

\begin{thm}
\label{thm:r=1}
No subfactor planar algebra has principal graph a translated extension of
$$
\cA_{(0)} = 
\left(
\bigraph{bwd1v1v1v1p1v0x1p0x1v1x0p0x1v0x1p1x0p1x0p0x1v1x0x0x0p0x0x1x0v0x1p1x0p1x0p0x1v1x0x0x0p0x0x1x0vduals1v1v1x2v1x2v2x1v1x2}
\,,\,
\bigraph{bwd1v1v1v1p1v1x0p0x1v1x0p1x0p0x1p0x1v1x0x0x0p0x1x0x0p0x0x1x0p0x0x0x1v1x0x0x0p1x0x0x0p0x1x0x0p0x0x1x0p0x0x1x0p0x0x0x1v0x1x0x0x0x0p0x0x1x0x0x0p0x0x0x0x1x0p0x0x0x0x0x1v1x0x0x0p1x0x0x0p0x1x0x0p0x1x0x0p0x0x1x0p0x0x1x0p0x0x0x1p0x0x0x1vduals1v1v1x2v3x2x1x4v6x2x4x3x5x1v4x6x7x1x8x2x3x5}
\right).
$$
(Note that $\cA_{(0)}$ is isomorphic to the graph in part \ref{case:10}(9) of Theorem \ref{thm:Enumerate}.)
\end{thm}

We define $\cA_{(2t)}$ to be the translation of $\cA_{(0)}$ by $2t$. 

Theorem \ref{thm:r=1} follows from the following four lemmas.

\begin{lem}
\label{lem:no-finite-extensions}
There are no finite depth subfactor planar algebras with principal graph a translated extension of $\cA_{(0)}$.
\end{lem}
This lemma is a striking application of number theory; while previously we've
used the cyclotomicity of the index for finite depth subfactors to rule out
arbitrary translations of a fixed graph, this lemma is the first case in which
we are able to rule out arbitrary translations and \emph{all finite
extensions}.
Its proof appears in Section \ref{sec:CG}.

We prove the next two lemmas, which both rely on graph enumeration arguments, in Section \ref{sec:periodicity}.

\begin{lem}
\label{lem:A-then-B}
Any subfactor principal graph which is an infinite depth extension of $\cA_{(2t)}$ must be an extension of $\cA_{(2t)}\cB$ where 
$$
\cB=
\left(\,
\begin{tikzpicture}[baseline=0.7cm,scale=.65]
	\filldraw (0,0) circle (1mm);
	\filldraw (0,2) circle (1mm);
	\filldraw (2,0) circle (1mm);
	\filldraw (2,1) circle (1mm);
	\filldraw (2,2) circle (1mm);
	\filldraw (2,3) circle (1mm);%
	\draw (0,0)--(2,1);
	\draw (0,0)--(2,2);
	\draw (0,2)--(2,0);
	\draw (0,2)--(2,3);
	\filldraw (4,0) circle (1mm);
	\filldraw (4,2) circle (1mm);
	\draw (2,0)--(4,0);
	\draw (2,2)--(4,2);
	\filldraw (6,0) circle (1mm);
	\filldraw (6,1) circle (1mm);
	\filldraw (6,2) circle (1mm);
	\filldraw (6,3) circle (1mm);%
	\draw (4,0)--(6,1);
	\draw (4,0)--(6,2);
	\draw (4,2)--(6,0);
	\draw (4,2)--(6,3);
	\filldraw (8,0) circle (1mm);
	\filldraw (8,2) circle (1mm);
	\draw (6,0)--(8,0);
	\draw (6,2)--(8,2);
	\draw[thick,red] (0,0)--(0,0.2);
	\draw[thick,red] (0,2)--(0,2.2);
	\draw[thick,red] (4,0) .. controls ++(135:.5cm) and ++(225:.5cm) .. (4,2);
	\draw[thick,red] (8,0)--(8,0.2);
	\draw[thick,red] (8,2)--(8,2.2);
\end{tikzpicture}
\,,\,
\begin{tikzpicture}[baseline=0.7cm, scale=.65]
	\filldraw (0,-0.5) circle (1mm);
	\filldraw (0,0) circle (1mm);
	\filldraw (0,0.5) circle (1mm);
	\filldraw (0,1) circle (1mm);
	\filldraw (0,1.5) circle (1mm);
	\filldraw (0,2) circle (1mm);
	\filldraw (0,2.5) circle (1mm);
	\filldraw (0,3) circle (1mm);
	\draw (0,0)--(2,0);
	\draw (0,1)--(2,1);
	\draw (0,2)--(2,2);
	\draw (0,3)--(2,3);
	\filldraw (2,0) circle (1mm);
	\filldraw (2,1) circle (1mm);
	\filldraw (2,2) circle (1mm);
	\filldraw (2,3) circle (1mm);
	\draw (2,0)--(4,-0.5);
	\draw (2,0)--(4,0);
	\draw (2,1)--(4,0.5);
	\draw (2,1)--(4,1);
	\draw (2,2)--(4,1.5);
	\draw (2,2)--(4,2);
	\draw (2,3)--(4,2.5);
	\draw (2,3)--(4,3);
	\filldraw (4,-0.5) circle (1mm);
	\filldraw (4,0) circle (1mm);
	\filldraw (4,0.5) circle (1mm);
	\filldraw (4,1) circle (1mm);
	\filldraw (4,1.5) circle (1mm);
	\filldraw (4,2) circle (1mm);
	\filldraw (4,2.5) circle (1mm);
	\filldraw (4,3) circle (1mm);
	\draw (4,0)--(6,0);
	\draw (4,1)--(6,1);
	\draw (4,2)--(6,2);
	\draw (4,3)--(6,3);
	\filldraw (6,0) circle (1mm);
	\filldraw (6,1) circle (1mm);
	\filldraw (6,2) circle (1mm);
	\filldraw (6,3) circle (1mm);
	\draw (6,0)--(8,-0.5);
	\draw (6,0)--(8,0);
	\draw (6,1)--(8,0.5);
	\draw (6,1)--(8,1);
	\draw (6,2)--(8,1.5);
	\draw (6,2)--(8,2);
	\draw (6,3)--(8,2.5);
	\draw (6,3)--(8,3);
	\filldraw (8,-0.5) circle (1mm);
	\filldraw (8,0) circle (1mm);
	\filldraw (8,0.5) circle (1mm);
	\filldraw (8,1) circle (1mm);
	\filldraw (8,1.5) circle (1mm);
	\filldraw (8,2) circle (1mm);
	\filldraw (8,2.5) circle (1mm);
	\filldraw (8,3) circle (1mm);
	\draw[thick,red] (0,-.5) .. controls ++(135:.5cm) and ++(225:.5cm) .. (0,1);
	\draw[thick,red] (0,0) .. controls ++(135:.5cm) and ++(225:.5cm) .. (0,2);
	\draw[thick,red] (0,.5) .. controls ++(135:.5cm) and ++(225:.5cm) .. (0,2.5);
	\draw[thick,red] (0,1.5) .. controls ++(135:.5cm) and ++(225:.5cm) .. (0,3);
	\draw[thick,red] (4,0)--(4,0.2);
	\draw[thick,red] (4,.5)--(4,.7);
	\draw[thick,red] (4,2)--(4,2.2);
	\draw[thick,red] (4,2.5)--(4,2.7);
	\draw[thick,red] (4,-.5) .. controls ++(135:.5cm) and ++(225:.5cm) .. (4,3);
	\draw[thick,red] (4,1) .. controls ++(135:.2cm) and ++(225:.2cm) .. (4,1.5);
	\draw[thick,red] (8,-.5) .. controls ++(135:.5cm) and ++(225:.5cm) .. (8,1);
	\draw[thick,red] (8,0) .. controls ++(135:.5cm) and ++(225:.5cm) .. (8,2);
	\draw[thick,red] (8,.5) .. controls ++(135:.5cm) and ++(225:.5cm) .. (8,2.5);
	\draw[thick,red] (8,1.5) .. controls ++(135:.5cm) and ++(225:.5cm) .. (8,3);
\end{tikzpicture}\;\;
\right)
$$
\end{lem}

\begin{lem}
\label{lem:Bs-forever}
The only possible subfactor principal graph which is an infinite depth extension
of $\cA_{(2t)}\cB$ is $\cA_{(2t)}\cB^\infty$.
\end{lem}

Finally, the remaining possibilities are ruled out by the next lemma, proved
in Section \ref{sec:no-infinite-depth}.

\begin{lem}
\label{lem:NoInfiniteDepth}
There is no subfactor planar algebra with principal graph $\cA_{(2t)}\cB^\infty$.
\end{lem}

\subsection{Relative dimensions and doubly one-by-one connection entries}
\label{sec:DoublyOneByOne}

For some weeds it is possible to determine the dimensions of all the vertices
as functions of $q$ and $t$, where $t$ is the translation. If we can do so, we
call these dimensions the \emph{relative dimensions} of the vertices
\cite[Section 4.1]{MR2902285}.

To compute the relative dimensions, we combine and solve the following three
types of equations.
\begin{enumerate}[label=(\arabic*)]
\item
Since the trivial vertex marked $\star$ always has dimension 1, and the quantum 
dimension of the vertex at depth 1 is $[2]=q+q^{-1}$, we set the dimension of 
the leftmost vertex of our weed equal to $[t+1]=\frac{q^t-q^{-t}}{q-q^{-1}}$.
\item
If two vertices are dual to each other, they must have the same dimension.
\item
For every vertex $v$ for which we know all its neighbors $w$, we have the 
equation $[2]\dim(v)=\sum_{\text{neighbors $w$}}\dim(w)$. 
\end{enumerate}

When we can determine all the relative dimensions, it is possible to use
certain branch factor inequalities to try to rule out the weed, as in Section
\ref{sec:BranchFactorInequalities} below. However, in many cases, we have too
few equations to determine all the relative dimensions. In extremely fortunate
situations, we can calculate an additional unknown dimension by using
connection techniques, namely looking for doubly one-by-one connection
entries.

By Fact \ref{fact:ExistenceOfConnection}, a necessary condition for a PGP $\Gamma$ to be the principal graph of a subfactor is that it must have a bi-unitary connection on the Ocenanu 4-partite graph $\cO(\Gamma)$.
We can form this 4-partite graph for an arbitrary PGP $\Gamma$ parallel to Definition \ref{defn:Ocneanu4Partite}, and associativity of $\Gamma$ directly corresponds to the associativity constraint for $\cO(\Gamma)$.

\begin{remark}
A PGP $\Gamma$ will not have any extension which is the principal graph of a subfactor unless there is a `partial
connection' on $\Gamma$.

A partial connection on a PGP with working depth $n$ is a pair $(W, \dim)$ as in Definition \ref{defn:Connection}, with
$W(P,Q,R,S)$ only defined when at least one of $\{P, R\}$ and at least one of $\{Q, S\}$ have depth at most $n-2$. The
unitarity and renormalization conditions must be satisfied whenever all the relevant entries are defined. (The
restriction on the defined entries ensures that the conditions we impose for a partial
connection on a PGP $\Gamma$ are implied by the conditions for a partial connection on
any extension of $\Gamma$.)
\end{remark}

In certain nice conditions, we can use the existence of a bi-unitary
connection to specify certain unknown relative dimensions.

\begin{defn}
Suppose $\Gamma$ is a PGP with a bi-unitary connection on $\cO(\Gamma)$. 
A
\emph{doubly one-by-one connection entry} is a loop of length 4 $\ell=(P,Q,R,S)$ around $\cO(\Gamma)$
which is the unique such loop passing through the pairs $(P,R)$ and $(Q,S)$.
Such an $\ell$ exists if and only if there are exactly 2 paths of length 2 on $\cO(\Gamma)$ between $P$ and $R$ and exactly 2 paths of length 2 between $Q$ and $S$.
For an example, see the proof of Lemma \ref{lem:doubly-one-by-one} 
below.

If we have a doubly one-by-one connection entry, then the one-by-one matrix
$W(P,-,R,-)$ is unitary by the Unitarity Axiom. Also, by the Renormalization
Axiom, so is the one-by-one matrix $W(Q,-,S,-)$. Hence we must have the
following identity of dimensions:
$$
\dim(P)\dim(R)=\dim(Q)\dim(S).
$$
\end{defn}

The presence of a doubly one-by-one connection entry sometimes allows us
to solve for one relative dimension in terms of three other known relative
dimensions.

We now return to our example at hand.  Looking at the section of $\cA_{(0)}$
between depths 6 and 10, which we will denote $(\cA_+,\cA_-)$, we get $\cB$
with two missing univalent vertices at depth 8:
$$
(\cA_+,\cA_-)
=
\left(\,
\begin{tikzpicture}[baseline=0.7cm,scale=.65]
	\filldraw (0,0) circle (1mm);
	\filldraw (0,2) circle (1mm);
	\filldraw (2,0) circle (1mm);
	\filldraw (2,1) circle (1mm);
	\filldraw (2,2) circle (1mm);
	\filldraw (2,3) circle (1mm);%
	\draw (0,0)--(2,1);
	\draw (0,0)--(2,2);
	\draw (0,2)--(2,0);
	\draw (0,2)--(2,3);
	\filldraw (4,0) circle (1mm);
	\filldraw (4,2) circle (1mm);
	\draw (2,0)--(4,0);
	\draw (2,2)--(4,2);
	\filldraw (6,0) circle (1mm);
	\filldraw (6,1) circle (1mm);
	\filldraw (6,2) circle (1mm);
	\filldraw (6,3) circle (1mm);%
	\draw (4,0)--(6,1);
	\draw (4,0)--(6,2);
	\draw (4,2)--(6,0);
	\draw (4,2)--(6,3);
	\filldraw (8,0) circle (1mm);
	\filldraw (8,2) circle (1mm);
	\draw (6,0)--(8,0);
	\draw (6,2)--(8,2);
	\draw[thick,red] (0,0)--(0,0.2);
	\draw[thick,red] (0,2)--(0,2.2);
	\draw[thick,red] (4,0) .. controls ++(135:.5cm) and ++(225:.5cm) .. (4,2);
	\draw[thick,red] (8,0)--(8,0.2);
	\draw[thick,red] (8,2)--(8,2.2);
\end{tikzpicture}
\,,\,
\begin{tikzpicture}[baseline=0.7cm, scale=.65]
	\filldraw (0,-0.5) circle (1mm);
	\filldraw (0,0) circle (1mm);
	\filldraw (0,0.5) circle (1mm);
	\filldraw (0,1) circle (1mm);
	\filldraw (0,1.5) circle (1mm);
	\filldraw (0,2) circle (1mm);
	\filldraw (0,2.5) circle (1mm);
	\filldraw (0,3) circle (1mm);
	\draw (0,0)--(2,0);
	\draw (0,1)--(2,1);
	\draw (0,2)--(2,2);
	\draw (0,3)--(2,3);
	\filldraw (2,0) circle (1mm);
	\filldraw (2,1) circle (1mm);
	\filldraw (2,2) circle (1mm);
	\filldraw (2,3) circle (1mm);
	\draw (2,0)--(4,-0.5);
	\draw (2,0)--(4,0);
	\draw (2,1)--(4,1);
	\draw (2,2)--(4,1.5);
	\draw (2,2)--(4,2);
	\draw (2,3)--(4,3);
	\filldraw (4,-0.5) circle (1mm);
	\filldraw (4,0) circle (1mm);
	\filldraw (4,1) circle (1mm);
	\filldraw (4,1.5) circle (1mm);
	\filldraw (4,2) circle (1mm);
	\filldraw (4,3) circle (1mm);
	\draw (4,0)--(6,0);
	\draw (4,1)--(6,1);
	\draw (4,2)--(6,2);
	\draw (4,3)--(6,3);
	\filldraw (6,0) circle (1mm);
	\filldraw (6,1) circle (1mm);
	\filldraw (6,2) circle (1mm);
	\filldraw (6,3) circle (1mm);
	\draw (6,0)--(8,-0.5);
	\draw (6,0)--(8,0);
	\draw (6,1)--(8,0.5);
	\draw (6,1)--(8,1);
	\draw (6,2)--(8,1.5);
	\draw (6,2)--(8,2);
	\draw (6,3)--(8,2.5);
	\draw (6,3)--(8,3);
	\filldraw (8,-0.5) circle (1mm);
	\filldraw (8,0) circle (1mm);
	\filldraw (8,0.5) circle (1mm);
	\filldraw (8,1) circle (1mm);
	\filldraw (8,1.5) circle (1mm);
	\filldraw (8,2) circle (1mm);
	\filldraw (8,2.5) circle (1mm);
	\filldraw (8,3) circle (1mm);
	\draw[thick,red] (0,-.5) .. controls ++(135:.5cm) and ++(225:.5cm) .. (0,1);
	\draw[thick,red] (0,0) .. controls ++(135:.5cm) and ++(225:.5cm) .. (0,2);
	\draw[thick,red] (0,.5) .. controls ++(135:.5cm) and ++(225:.5cm) .. (0,2.5);
	\draw[thick,red] (0,1.5) .. controls ++(135:.5cm) and ++(225:.5cm) .. (0,3);
	\draw[thick,red] (4,0)--(4,0.2);
	\draw[thick,red] (4,2)--(4,2.2);
	\draw[thick,red] (4,-.5) .. controls ++(135:.5cm) and ++(225:.5cm) .. (4,3);
	\draw[thick,red] (4,1) .. controls ++(135:.2cm) and ++(225:.2cm) .. (4,1.5);
	\draw[thick,red] (8,-.5) .. controls ++(135:.5cm) and ++(225:.5cm) .. (8,1);
	\draw[thick,red] (8,0) .. controls ++(135:.5cm) and ++(225:.5cm) .. (8,2);
	\draw[thick,red] (8,.5) .. controls ++(135:.5cm) and ++(225:.5cm) .. (8,2.5);
	\draw[thick,red] (8,1.5) .. controls ++(135:.5cm) and ++(225:.5cm) .. (8,3);
\end{tikzpicture}\;\;
\right)
$$
Recall that $\cA_{(2t)}$ is the $2t$-translation of $\cA_{(0)}$.

\begin{lem}
\label{lem:doubly-one-by-one}
There is a doubly one-by-one entry of the connection for any principal graph which is an extension of $\cA_{(2t)}$, and so
\begin{align}
a^2 X_1(q) + X_0(q) + a^{-2} X_1(q^{-1}) = 0
\label{eq:DoubleOneByOne}
\end{align}
where $a=q^{2t}$ and 
\begin{align*}
X_1(q) & = q^{16}-2 q^{14}-q^{12}-2 q^{10}-3 q^8-4 q^6-q^4 \\
X_0(q) & = -2 q^4+8 q^2+12+ 8 q^{-2}- 2 q^{-4}.
\end{align*}
\end{lem}
\begin{proof}
The loop $(V^p_{2t+6,1},V^p_{2t+5,1},V^d_{2t+6,2},V^d_{2t+7,1})$ (in the notation of \cite{MR2902285}: the $p$ and $d$ stand for `principle' and `dual' to specify which graph, and the subscript $i,j$ gives the depth and height) in $\cO(\cA_+,\cA_-)$, outlined in blue below, gives a doubly one-by-one connection entry.
$$
\begin{tikzpicture}[baseline, scale = .95]
	\node at (-2.3,1.5) {$\cA_+\,\Bigg\{$};
	\node at (-2.3,-.5) {$\cA_-\,\Bigg\{$};
%
	\coordinate (aa1) at (0,2);
	\coordinate (aa2) at (1,2);
	\coordinate (aa3) at (6,2);
	\coordinate (aa4) at (7,2);
	\coordinate (aa5) at (12,2);
	\coordinate (aa6) at (13,2);
%
	\coordinate (ab1) at (4,1);
	\coordinate (ab2) at (4.5,1);
	\coordinate (ab3) at (5,1);
	\coordinate (ab4) at (5.5,1);
	\coordinate (ab5) at (10,1);
	\coordinate (ab6) at (10.5,1);
	\coordinate (ab7) at (11,1);
	\coordinate (ab8) at (11.5,1);
%
	\coordinate (bb1) at (0,0);
	\coordinate (bb2) at (.5,0);
	\coordinate (bb3) at (1,0);
	\coordinate (bb4) at (1.5,0);
	\coordinate (bb5) at (2,0);
	\coordinate (bb6) at (2.5,0);
	\coordinate (bb7) at (3,0);
	\coordinate (bb8) at (3.5,0);
	\coordinate (bb9) at (6,0);
	\coordinate (bb10) at (6.5,0);
	\coordinate (bb11) at (7,0);
	\coordinate (bb12) at (7.5,0);
	\coordinate (bb13) at (8,0);
	\coordinate (bb14) at (8.5,0);
	\coordinate (bb15) at (9,0);
	\coordinate (bb16) at (9.5,0);
	\coordinate (bb17) at (12,0);
	\coordinate (bb18) at (12.5,0);
	\coordinate (bb19) at (13,0);
	\coordinate (bb20) at (13.5,0);
	\coordinate (bb21) at (14,0);
	\coordinate (bb22) at (14.5,0);
	\coordinate (bb23) at (15,0);
	\coordinate (bb24) at (15.5,0);
%
	\coordinate (ba1) at (4,-1);
	\coordinate (ba2) at (4.5,-1);
	\coordinate (ba3) at (5,-1);
	\coordinate (ba4) at (5.5,-1);
	\coordinate (ba5) at (10,-1);
	\coordinate (ba6) at (10.5,-1);
	\coordinate (ba7) at (11,-1);
	\coordinate (ba8) at (11.5,-1);
%
	\coordinate (aa1d) at (0,-2);
	\coordinate (aa2d) at (1,-2);
	\coordinate (aa3d) at (6,-2);
	\coordinate (aa4d) at (7,-2);
	\coordinate (aa5d) at (12,-2);
	\coordinate (aa6d) at (13,-2);
%
		\draw (aa2)--(ab1);
		\draw (aa1)--(ab2);
		\draw (aa1)--(ab3);
		\draw (aa2)--(ab4);
		\draw[blue] (aa3)--(ab1);
		\draw (aa4)--(ab3);
		\draw (aa4)--(ab5);
		\draw (aa3)--(ab6);
		\draw (aa3)--(ab7);
		\draw (aa4)--(ab8);
		\draw (aa5)--(ab5);
		\draw (aa6)--(ab7);
%
		\draw (ab1)--(bb6);
		\draw (ab2)--(bb1);
		\draw (ab3)--(bb2);
		\draw (ab4)--(bb5);
		\draw[blue] (ab1)--(bb10);
		\draw (ab1)--(bb16);
		\draw (ab2)--(bb13);
		\draw (ab3)--(bb12);
		\draw (ab3)--(bb14);
		\draw (ab4)--(bb9);
		\draw (ab5)--(bb10);
		\draw (ab6)--(bb13);
		\draw (ab7)--(bb14);
		\draw (ab8)--(bb9);
		\draw (ab5)--(bb20);
		\draw (ab5)--(bb22);
		\draw (ab6)--(bb17);
		\draw (ab6)--(bb23);
		\draw (ab7)--(bb24);
		\draw (ab7)--(bb18);
		\draw (ab8)--(bb19);
		\draw (ab8)--(bb21);
%
		\draw (bb2)--(ba1);
		\draw (bb4)--(ba2);
		\draw (bb6)--(ba3);
		\draw (bb8)--(ba4);
		\draw (bb9)--(ba1);
		\draw (bb10)--(ba1);
		\draw (bb12)--(ba2);
		\draw (bb13)--(ba3);
		\draw (bb14)--(ba3);
		\draw (bb16)--(ba4);
		\draw[blue] (bb10)--(ba5);
		\draw (bb12)--(ba6);
		\draw (bb14)--(ba7);
		\draw (bb16)--(ba8);
		\draw (bb17)--(ba5);
		\draw (bb18)--(ba5);
		\draw (bb19)--(ba6);
		\draw (bb20)--(ba6);
		\draw (bb21)--(ba7);
		\draw (bb22)--(ba7);
		\draw (bb23)--(ba8);
		\draw (bb24)--(ba8);
%
		\draw (ba1)--(aa2d);
		\draw (ba2)--(aa1d);
		\draw (ba3)--(aa1d);
		\draw (ba4)--(aa2d);
		\draw (ba1)--(aa4d);
		\draw (ba3)--(aa3d);
		\draw[blue] (ba5)--(aa3d);
		\draw (ba6)--(aa4d);
		\draw (ba7)--(aa4d);
		\draw (ba8)--(aa3d);
		\draw (ba5)--(aa5d);
		\draw (ba7)--(aa6d);
%
	\node at (-1,2) {$\sb{A}{\sf{Mod}}_A$};
	\filldraw (aa1) circle (.8mm);
	\filldraw (aa2) circle (.8mm);
	\filldraw[blue] (aa3) circle (.8mm);
	\filldraw (aa4) circle (.8mm);
	\filldraw (aa5) circle (.8mm);
	\filldraw (aa6) circle (.8mm);
%
	\node at (-1,1) {$\sb{A}{\sf{Mod}}_B$};
	\filldraw[blue] (ab1) circle (.8mm); 
	\filldraw (ab2) circle (.8mm);
	\filldraw (ab3) circle (.8mm);
	\filldraw (ab4) circle (.8mm);
	\filldraw (ab5) circle (.8mm); 
	\filldraw (ab6) circle (.8mm);
	\filldraw (ab7) circle (.8mm);
	\filldraw (ab8) circle (.8mm);
%
	\node at (-1,0) {$\sb{B}{\sf{Mod}}_B$};
	\filldraw (bb1) circle (.8mm);
	\filldraw (bb2) circle (.8mm);
	\filldraw (bb3) circle (.8mm);
	\filldraw (bb4) circle (.8mm);
	\filldraw (bb5) circle (.8mm);
	\filldraw (bb6) circle (.8mm);
	\filldraw (bb7) circle (.8mm);
	\filldraw (bb8) circle (.8mm);
	\filldraw (bb9) circle (.8mm);
	\filldraw[blue] (bb10) circle (.8mm);
	\filldraw (bb12) circle (.8mm);
	\filldraw (bb13) circle (.8mm);
	\filldraw (bb14) circle (.8mm);
	\filldraw (bb16) circle (.8mm);
	\filldraw (bb17) circle (.8mm);
	\filldraw (bb18) circle (.8mm);
	\filldraw (bb19) circle (.8mm);
	\filldraw (bb20) circle (.8mm);
	\filldraw (bb21) circle (.8mm);
	\filldraw (bb22) circle (.8mm);
	\filldraw (bb23) circle (.8mm);
	\filldraw (bb24) circle (.8mm);
%
	\node at (-1,-1) {$\sb{B}{\sf{Mod}}_A$};
	\filldraw (ba1) circle (.8mm);
	\filldraw (ba2) circle (.8mm);
	\filldraw (ba3) circle (.8mm);
	\filldraw (ba4) circle (.8mm);
	\filldraw[blue] (ba5) circle (.8mm);
	\filldraw (ba6) circle (.8mm);
	\filldraw (ba7) circle (.8mm);
	\filldraw (ba8) circle (.8mm);
%
	\node at (-1,-2) {$\sb{A}{\sf{Mod}}_A$};
	\filldraw (aa1d) circle (.8mm);
	\filldraw (aa2d) circle (.8mm);
	\filldraw[blue] (aa3d) circle (.8mm);
	\filldraw (aa4d) circle (.8mm);
	\filldraw (aa5d) circle (.8mm);
	\filldraw (aa6d) circle (.8mm);
\end{tikzpicture}
$$

Happily, we can compute the dimensions of these vertices in the principal graph $\cA_{(2t)}\cC$, as functions of $q$ and $a=q^{2t}$, irrespective of tail $\cC$.
The formula in the statement is now obtained from the Renormalization Axiom.
\end{proof}

\subsection{A variation on Calegari-Guo}
\label{sec:CG}

In this section we adapt the arguments of \cite{1502.00035} to prove the
following number theoretic result, with a very helpful consequence.

\begin{thm}
\label{thm:number-theory}
Suppose $t \geq 0$, and $q > 1$ satisfies
\begin{align*}
P_t(q) & = q^{2t} X_1(q) + X_0(q) + q^{-2t} X_1(q^{-1}) = 0,
\end{align*}
where $X_0$ and $X_1$ are the Laurent polynomials defined in Lemma \ref{lem:doubly-one-by-one}.
Then the quantity
$(q+q^{-1})^2$ is not a totally real cyclotomic integer.
\end{thm}

\begin{remark}
\label{rem:DefOfqt}
We write  $q_t$ for the largest real root of $P_t$, and $q_\infty$ for the largest real root of $X_1$. 
It is clear that $q_t \to q_\infty$, but we will need careful control of the convergence.
\end{remark}

Before we prove this theorem, we use it to prove Lemma \ref{lem:no-finite-extensions}.

\begin{proof}[Proof of Lemma \ref{lem:no-finite-extensions}]
Suppose $(\Gamma_+,\Gamma_-)$ is a finite translated extension of $\cA_{(2t)}$
which occurs as a subfactor principal graph. By \cite{MR934296}, the index of
this subfactor must be equal to $\|\Gamma_\pm\|^2$. By the uniqueness of the
Frobenius-Perron dimension function on the vertices, together with Lemma
\ref{lem:doubly-one-by-one}, $\|\Gamma_\pm\|^2=(q+q^{-1})^2$ where $q$ is a
solution to Equation \eqref{eq:DoubleOneByOne} which is greater than 1.
Now by
\cite{MR1266785,MR2183279}, the index of a finite depth subfactor must be a
totally real cyclotomic integer. But by Theorem \ref{thm:number-theory},
$(q+q^{-1})^2$ is not a totally real cyclotomic integer, a contradiction.
\end{proof}

The proof of Theorem \ref{thm:number-theory} occupies the remainder of this
section. Some calculations are performed in the Mathematica notebook {\tt
CalegariGuoAdaptation.nb}, and the interested reader should look there for
more details.

\def\tareesidedbox#1{\setbox0=\hbox{$#1$}\dimen0=\wd0 \advance\dimen0 by3pt\rlap{\hbox{\vrule height9pt width.4pt
 depth2pt \kern-.4pt\vrule height9.4pt width\dimen0 depth-9pt\kern-.4pt \vrule height9pt width.4pt depth2pt}}
 \relax \hbox to\dimen0{\hss$#1$\hss}}
\def\ho#1{\tareesidedbox#1}

\def\tareesidedboy#1{\setbox0=\hbox{$#1$}\dimen0=\wd0 \advance\dimen0 by3pt\rlap{\hbox{\vrule height9.8pt width.4pt
 depth2pt \kern-.4pt\vrule height10pt width\dimen0 depth-9.6pt\kern-.4pt \vrule height9.8pt width.4pt depth2pt}}
 \relax \hbox to\dimen0{\hss$#1$\hss}}
\def\hr#1{\tareesidedboy#1}

\newcommand{\house}[1]{\tareesidedbox#1}

Recall Theorem 3.3 of \cite{1502.00035} which we restate here with the parameter $L$ fixed as $L = 3.16826$ (for which $B(L^2) < 0$).
\begin{thm*}
Let $\beta$ be a totally real algebraic integer such that
\begin{enumerate}[label=(\arabic*)]
\item $\beta^2$ is not Galois conjugate to $(\zeta_t + \zeta_t^{-1})^2$ for any $t \leq 52$.
\item The largest conjugate $\house{\beta}$ of $\beta$ is less than $L=3.16826$.
\item At most $M$ conjugates of $\beta^2$ lie outside the interval $[0,4]$.
\end{enumerate}
Then, $\cM(\beta) = \frac{\operatorname{Tr}_{\bbQ(\beta^2)/\bbQ}(\beta^2)}{[\bbQ(\beta^2):\bbQ]} < \frac{14}{5}$ or
$[\bbQ(\beta^2):\bbQ] < \frac{260}{11} \cdot M$.
\end{thm*}

Our plan is to apply this theorem with $\beta_t = (q_t+q_t^{-1})^2 - 2$
(recall $q_t$ was defined in Remark \ref{rem:DefOfqt}). After showing that we
can take $M=1$, this will give us a dichotomy: either $\cM(\beta_t) <
\frac{14}{5}$ or $[\bbQ(\beta_t^2) : \bbQ] \leq 23$. In the first case, there
are not many possibilities for $\beta_t^2$ being a cyclotomic integer,
controlled by later results of \cite{1502.00035}. In the second case, analysis
of the rate of the convergence of $\house{\beta_t}$ as $t \to \infty$ will
give a bound on $t$, after which we have reduced the problem to a finite
number of cases.

\begin{lem}
\label{lem:UniqueRoot}
The polynomial $P_t(q)$ has a unique root $q > 1$.
\end{lem}
\begin{proof}
By Descartes' rule of signs, $P_t(q)$ has either 4, 2, or 0 positive roots.
Easily, $q=1$ is a double root. The polynomial $P_t(q)$ is reciprocal, so there is at most one root $q > 1$. 
Since $P_t(q) \to \infty$ as $q \to \infty$ and $P_t''(1) = -32 \left(3 t^2+24
t+52\right) < 0$, there is exactly one such root.
\end{proof}

\begin{lem}
The largest roots $q_t$ of $P_t(q)$ are always smaller than $q_\infty$.
\end{lem}
\begin{proof}
By definition $X_1(q_\infty) = 0$, so we just calculate $P_t(q_\infty) = X_0(q_\infty) + q_\infty^{-2t} X_1(q_\infty) > 0$. The previous lemma, and the fact that $P_t(q)$ is eventually positive, gives the result.
\end{proof}

\begin{remark}
\label{rem:q-to-beta}
Consider the function $f(q) = ((q+q^{-1})^2 - 2)^2$. This maps points $q$
on the unit circle to $[0,4]$, and points $q > 1$ to $(4,\infty)$. Complex
numbers off the unit circle or real line are sent to points off the real line.

Since $\beta_t \in \mathbb Q(q_t)$, $f$ gives a surjection from the Galois
conjugates of $q_t$ to the Galois conjugates of $\beta_t^2$,
and in particular from the Galois conjugates of $q_t$ greater
than 1 to the Galois conjugates of $\beta_t^2$
greater than 4, and the last two lemmas establish that condition (2) holds and
condition (3) holds with $M = 1$ in Calegari-Guo's theorem.

Moreover, if $q_t$ has Galois conjugates that are neither real nor on the unit 
circle, $\beta_t$ cannot be totally real.
\end{remark}

\begin{lem}
\label{lem:q_n-monotonic}
The largest roots $q_t$ converge monotonically to $q_\infty$.
\end{lem}
\begin{proof}
Recall $P_t(q)$ has a unique root $q_t$ greater than 1. 
It suffices to show $P_t(q_{t-1}) < 0$. 
We first observe that $X_0(q)$ and $X_1(q^{-1})$ are positive for $1 < q < q_\infty$, and then calculate
\begin{align*}
P_t(q_{t-1}) & = q^2 q^{2t - 2} X_1(q_{t-1}) + X_0(q_{t-1}) + q^{-2t} X_1(q_{t-1}^{-1}) \\ 
			 & = - q^2 (X_0(q_{t-1}) + q^{-2t+2} X_1(q_{t-1}^{-1})) + X_0(q_{t-1}) + q^{-2t} X_1(q_{t-1}^{-1}) \\
			 & = (1-q^2) X_0(q_{t-1}) + q^{-2t} (1 - q^4) X_1(q_{t-1}^{-1}) \\
			 & < 0. \qedhere
\end{align*}
\end{proof}

Although we're staying purely in the number theoretic context for now, recall
that $q_t + q_t^{-1}$ is actually the largest $\ell^2$-eigenvalue of the graph
$\cA_{(2t)}\cB^\infty$, which easily gives both of the last two lemmas.

Condition (1) of Calegari-Guo's theorem is trivial as $|\zeta_t^k +
\zeta_t^{-k}|^2 \leq 4$, while $|\beta_t^2| > |\beta_0^2| > 9$ by
monotonicity.

We have now established that Theorem 3.3 of \cite{1502.00035} applies to $\beta_t$.

\begin{lem}
\label{lem:convergence}
We have $| q_\infty - q_t | < \frac{23}{7077} (2.802)^{-t} < \exp(-5-t)$.
\end{lem}
\begin{proof}
By monotonicity, $q_t \in [q_0, q_\infty)$. We now follow the proof of \cite[Lemma 6.3]{1502.00035}. In this interval, we have the inequalities
\begin{align*}
| X_1(q) | & > 7077 | q-q_\infty | \\
| X_0(q) | & < 22 \\
| X_1(q^{-1}) | & < 1.
\end{align*}
As $q_t^{2t} X_1(q_t) + X_0(q_t) + q_t^{-2t} X_1(q_t^{-1}) = 0$, by the triangle inequality we obtain
$$
q_0^{2t} 7077 |q_t-q_{\infty}|
<
q_t^{2t} |X_1(q_t)|
\leq
|X_0(q_t)| + q_t^{-2t} |X_1(q_t^{-1})|
<
23,
$$
giving the result, since $2.802<q_0^2$.
\end{proof}

\begin{lem}
If $\cM(\beta_t) < 14/5$ then $\beta_t$ is not a totally real cyclotomic integer.
\end{lem}
\begin{proof}
By \cite[Proposition 4.3]{1502.00035}, a totally real cyclotomic integer
$\beta$ with $\cM(\beta) < 14/5$ must either be a sum of two roots of unity
(and hence less than 2) or $\house{\beta}$ must be on an explicit finite list.
By Lemma \ref{lem:convergence}, once $t \geq 2$, $q_t$ is within $0.001$ of
$q_\infty$, and so $\beta_t$ is within $0.01$ of $3.17$, and there are no such
numbers on the list.
\end{proof}

\begin{lem}
\label{lem:t<=63}
If $[\bbQ(\beta_t^2) : \bbQ] \leq 23$ and $\beta_t$ is totally real, then $t \leq 66$.
\end{lem}
\begin{proof}
As $q_t$ is not Galois conjugate to $q_\infty$, $X_1(q_t)$ is a non-zero algebraic number, and so
$$\prod_{\sigma \in \operatorname{Gal}(\bbQ(\beta^2) / \bbQ)} | \sigma X_1(q_t) | \geq 1,$$
being a positive integer.

For $\sigma$ non-trivial, $|\sigma q_t| \leq 1$ by Remark \ref{rem:q-to-beta}, so $| \sigma X_1(q_t)| \leq 14$, and so
$$| X_1(q_t) | \geq \frac{1}{14^{22}}.$$

Applying Lemma \ref{lem:convergence}, we see
\begin{align*}
|X_1(q_t)| & = | X_1(q_t) - X_1(q_\infty) | \\
& \leq 14 |q_t^{16} - q_\infty^{16}| \\ 
& =  14 |q_\infty - q_t| |q_\infty^{15} + q_\infty^{14}q_t + \cdots q_t^{15}| \\
& \leq 14 \cdot 16 \cdot q_\infty^{15} \cdot \exp(-5-t).
\end{align*}
This
gives
$$ 14 \cdot 16 \cdot q_\infty^{15} \cdot \exp(-5-t) \geq \frac{1}{14^{22}} $$
and hence the result. 
\end{proof}

Finally we check the cases $t \leq 66$ individually, verifying that the
totally real roots are not cyclotomic integers.
This completes the proof of  Theorem \ref{thm:number-theory}.
\qed

\subsection{The tail enumerator and periodicity}
\label{sec:periodicity}

\begin{proof}[Proof of Lemma \ref{lem:A-then-B}]
By Lemma \ref{lem:doubly-one-by-one}, any subfactor with principal graph which is a translated extension of $\cA_{(2t)}$ must have
$(a=q^{2t}, q)$ satisfying Equation \eqref{eq:DoubleOneByOne}. By Lemma \ref{lem:q_n-monotonic}, $q < q_\infty$, and so the index of this subfactor must be at most $(q_\infty+q_\infty^{-1})^2 < 5.1683$.

We thus run the enumerator described in Section \ref{sec:orderly} on
$\cA_{(0)}$, up to index 5.1683, stopping whenever we reach depth 17. We find
6 vines (which aren't relevant now, as we're only interested in infinite depth
extensions), and 20 graphs at depth 17. All of these 20 graphs truncate back
to depth 14 as $\cA_{(0)}\cB$.
\end{proof}

By running the tail enumerator described in Section \ref{sec:Implementations}
on the basic block $\cB$ for 12 depths, we observe the following.

\begin{lem}
Any extension of the basic block $\cB$ must either
\begin{itemize}
\item start with another basic block,
\item be one of several cylinders, or 
\item terminate before depth 16, and in particular be finite depth.
\end{itemize}
\end{lem}

These computations are performed in the Mathematica notebook {\tt{r1TailEnumerator.nb}} bundled with the {\tt{arXiv}} source.

\begin{cor}
Any principal graph of the form $\cA_{(2t)}\cB\cC$ where $\cB$ is the basic block has the form
\begin{itemize}
\item $\cA_{(2t)}\cB^\infty$,
\item $\cA_{(2t)}\cB^k\cC$, for some $k$, where $\cC$ is a cylindrical tail, or
\item $\cA_{(2t)}\cB^k\cD$, for some $k$, where $\cD$ is a finite tail.
\end{itemize}
\end{cor}

Combined with Lemma \ref{lem:no-finite-extensions} we have proved Lemma \ref{lem:Bs-forever}.
\qed

\subsection{Ruling out the infinite depth graph}
\label{sec:no-infinite-depth}

We now consider the family of infinite depth graphs $\cA_{(2t)}\cB^\infty$.
Suppose that we have a subfactor with principal graph $\cA_{(2t)}\cB^\infty$,
with index $(q+q^{-1})^2$ (which may well be greater than the square of the
graph norm of $\cA_{(2t)}\cB^\infty$ itself).

We introduce $a=q^{2t}$. Happily, the dimension of every bimodule may be
computed as a rational function in $a$ and $q$.

\begin{lem}
We have the following explicit formula for $a^2$:
\begin{align}
a^2 & = \frac{q^{20}-4 q^{18}-6 q^{16}-4 q^{14}+q^{12}-\sqrt{q^{20} \left(q^2-1\right)^2 \left(q^2+1\right)^6 \left(q^8-2 q^6-q^4-2 q^2+1\right)}}{q^{20} \left(q^{12}-2 q^{10}-q^8-2 q^6-3 q^4-4
q^2-1\right)}
\label{eq:r1a}
\end{align}
\end{lem}
\begin{proof}
By Lemma \ref{lem:doubly-one-by-one}, we see that $a$ satisfies Equation
\eqref{eq:DoubleOneByOne}. Applying the quadratic formula, we get two possible
solutions for $a^2$. Only the one in Equation \eqref{eq:r1a} is greater than 1
in the relevant range for $q$, i.e., $1.65<q<q_\infty$, where $1.65$ is less
than the $q$ such that $q+q^{-1}$ is the graph norm of $\cA_{(0)}$, and
$q_\infty$ is as in Section \ref{sec:CG}.
\end{proof}

Let $\alpha_m$ denote the dimension of the two bimodules on the first graph at depth $6+2t+2m$. One easily sees that
these satisfy the recursion $\alpha_{m+1} = ((q+q^{-1})^2 - 3) \alpha_m - \alpha_{m-1}$, and thus
$$
\alpha_m 
= 
c_{+} \left(\frac{1{-}q^2 {+}q^4 {+} \sqrt{1{-}2q^2{-}q^4{-}2q^6{+}q^8}}{2q^2} \right)^m 
+ 
c_{-} \left(\frac{1{-}q^2 {+}q^4 {-} \sqrt{1{-}2q^2{-}q^4{-}2q^6{+}q^8}}{2q^2}\right)^m
$$
for some constants $c_+$ and $c_-$ (depending on $t$).
We can solve for these constants (cf. the notebook {\tt code/r=1.nb}), using
\begin{align*}
\alpha_0 & = \frac{a^2 q^{16}-a^2 q^{14}-a^2 q^{12}-a^2 q^{10}+q^6+q^4+q^2-1}{2 a q^6 \left(q^2-1\right) \left(q^2+1\right)} \\
\intertext{and}
\alpha_1 & = \frac{a^2 q^{20}-2 a^2 q^{18}-a^2 q^{16}-a^2 q^{14}-a^2 q^{10}+q^{10}+q^6+q^4+2 q^2-1}{2 a q^8 \left(q^2-1\right) \left(q^2+1\right)},
\end{align*}
obtaining
\begin{align*}
c_+ & = 
\frac{f_+(a,q)}{8 a q^{10} \left(q^4-1\right) z}
\text{ and }
c_-  = 
\frac{f_-(a,q)}{2 a q^6 \left(q^2-1\right) \left(q^2+1\right) z \left(-q^4+q^2+z-1\right)}
\end{align*}
where
\begin{align*}
z & = 
\sqrt{q^8-2 q^6-q^4-2 q^2+1}
\\
f_+(a,q) & =
z^2 \left(a^2 q^{10} \left(-q^{10}+2 q^8+q^6+q^4+1\right) -q^{10}-q^6-q^4-2 q^2+1\right)
\\&
\quad + z \left(2 a^2 q^{14} \left(q^6-q^4-q^2-1\right) +2 q^4 \left(q^6+q^4+q^2-1\right) \right)
\\&
\quad +a^2 q^{10} \left(q^{18}-4 q^{16}+4 q^{14}-3 q^{12}+2 q^{10}-2 q^8+3 q^6-2 q^4+2 q^2-1\right) 
\\&
\quad +q^{18}-2 q^{16}+2 q^{14}-3 q^{12}+2 q^{10}-2 q^8+3 q^6-4 q^4+4 q^2-1
\\
f_-(a,q) & =
z \left(a^2 q^{10} \left(-q^{10}+2 q^8+q^6+q^4+1\right) -q^{10}-q^6-q^4-2 q^2+1\right)
\\&
\quad + a^2 q^{10} \left(q^{14}-3 q^{12}+2 q^6+q^2-1\right) 
\\&
\quad +q^{14}-q^{12}-2 q^8+3 q^2-1
\end{align*}
Finally, substituting $a$ from Equation \eqref{eq:r1a} and simplifying, we obtain $c_+ = 0$, and note that for $q >
\frac{1}{2}(1+\sqrt{5})$, i.e. index greater than 5, we have
$$
\frac{1-q^2 +q^4 - \sqrt{1-2q^2-q^4-2q^6+q^8}}{2q^2}
<1,
$$
so $\alpha_m$ is eventually less than 1, for large enough $m$. 
Thus there is no subfactor with principal graph $\cA_{(2t)}\cB^\infty$.
This completes the proof of Lemma \ref{lem:NoInfiniteDepth}.
\qed

\section{Ruling out weeds using branch factor inequalities}
\label{sec:BranchFactorInequalities}

The results of \cite{MR2972458, MR3198588, MR3311757} give, for certain graph
pairs $\Gamma=(\Gamma_+,\Gamma_-)$, a rational function $p(a, q)$ such that
for a $t$-translated extension $\Gamma'=(\Gamma_+',\Gamma_-')$ of $\Gamma$ to
be the principal graph a subfactor with index $(q+q^{-1})^2$, we must have
$p_\Gamma(q^t, q) \leq 0$.  Now, the index of a subfactor with principal graph
$\Gamma'$ must be at least $\|\Gamma'_\pm\|^2\geq \|\Gamma_\pm\|^2$.  Thus, if
$p_\Gamma$ is positive for all $q > q_0$ where $\| \Gamma \| = q_0 +
q_0^{-1}$, then there are no possible subfactors with principal graph a
translated extension of $\Gamma$.  In this case, we say the result rules out
the weed $\Gamma$.

In each of these results, as long as the annular multiplicities match some
pattern, we obtain an inequality involving $q, a=q^{2t}$, and ratios of the
relative dimensions of certain vertices on the graph
$\Gamma=(\Gamma_+,\Gamma_-)$, called the \emph{relative branch factors}.
In fortunate circumstances, all the relative dimensions can be computed as
functions of $a$ and $q$, so we can easily write down the inequality
$p_\Gamma(a,q)\leq 0$. In less fortunate circumstances, there are undetermined
relative dimensions; sometimes, nevertheless, we can compute these relative
dimensions from a doubly one-by-one connection entry. See Sections
\ref{sec:11WithDoubleOneByOne} and \ref{sec:10WithDoubleOneByOne} for such
examples. This often gives these unknown dimensions as functions of $a,q$
which are no longer rational, and it is usually more work to apply
the inequality.

We now recall the three results on branch factors.
We assume that $\Gamma=(\Gamma_+,\Gamma_-)$ has an initial triple point at depth $n-1$:
$$
\begin{tikzpicture}[baseline=-.1cm]
	\draw[fill] (-2,0) circle (0.05);
	\node at (-2,-.3) {\scriptsize{$0$}};	
	\draw (-2.,0.) -- (-1.,0.);
	\draw[fill] (-1,0) circle (0.05);
	\node at (-1,-.3) {\scriptsize{$1$}};
	\node at (-.5,0) {$\cdots$};
	\draw[fill] (0,0) circle (0.05);
	\node at (0,-.3) {\scriptsize{$n-2$}};
	\draw (0.,0.) -- (1.,0.);
	\draw[fill] (1.,0.) circle (0.05);
	\node at (1,-.3) {\scriptsize{$n-1$}};
	\draw (1.,0.) -- (2.,-0.5);
	\draw (1.,0.) -- (2.,0.5);
	\draw[fill=white] (2.,-0.5) circle (0.05);
	\node at (2.3,.5) {$Q$};
	\draw[fill=white] (2.,0.5) circle (0.05);
	\node at (2.3,-.5) {$P$};
	\node at (2,-.8) {\scriptsize{$n$}};	
\end{tikzpicture}$$
We label the projections at depth $n$ of $\Gamma_+$ by $P$ and $Q$.
The new low-weight rotational eigenvector perpendicular to $\cT\cL\cJ_{n,+}$ is given by $A=rP-Q$, where $r=\Tr(Q)/\Tr(P)$ is the branch factor.
The rotational eigenvalue is denoted $\omega_A$.

\begin{remark}
If $(\Gamma_+,\Gamma_-)$ is $k$-supertransitive, then the $t$-translated graph has $n=t+k+1$.
\end{remark}

\begin{thm}[{\cite[Theorem 5.1.11]{MR2972458}}]
Suppose $(\Gamma_+,\Gamma_-)$ is a $(n-2)$-translated extension of 
$$
\left(\bigraph{bwd1v1p1v1x0p0x1vduals1v1x2v},\bigraph{bwd1v1p1v1x0p1x0vduals1v1x2v}\right).
$$
Then
\begin{equation}
\omega_A+\omega_A^{-1}=(r + r^{-1} -2)[n][n+2]-2.
\label{eq:QTEquation}
\end{equation}
\end{thm}

\begin{thm}[{\cite[Theorem 3]{MR3198588}}]
\label{thm:SnyderTripleSingle}
Suppose $(\Gamma_+,\Gamma_-)$ has an initial triple point at depth $n-1$ where $n$ is even, and $\Gamma_-$ has a univalent vertex at depth $n$.
Then Equation \eqref{eq:QTEquation} holds.
\end{thm}

\begin{cor}[Triple-single branch factor inequality]
\label{cor:10Inequality}
If $(\Gamma_+,\Gamma_-)$ has an initial triple point at depth $n-1$ where $n$ is even, and $\Gamma_-$ has a univalent vertex at depth $n$, then
$$
-4 \leq \omega_A+\omega_A^{-1}-2 = (r + r^{-1} -2)[n][n+2]-4 \leq 0.
$$
\end{cor}

The main result of \cite{MR3311757} is strictly stronger than those of \cite{MR2972458,MR3198588}, but we include only what we need for the annular multiplicity $*11$ weeds in this article.
For the following theorem, $\Gamma_\pm$ are both $(n-2)$-translated extensions of 
$$
\bigraph{gbg1v1p1v1x0p0x1p0x1v},
$$
and $P$ is chosen to be the bivalent vertex at depth $n$ of $\Gamma_+$ (regardless of $\Tr(P)$ and $\Tr(Q)$).
When $n$ is even, we choose $\check{P}$ to be the trivalent vertex at depth $n$ of $\Gamma_-$, and when $n$ is odd, we choose $\check{P}$ to be the bivalent vertex at depth $n$ of $\Gamma_-$

Again, $A=rP-Q$, and $\check{A}=\check{r}\check{P}-\check{Q}$ where
$r=\Tr(Q)/\Tr(P)$ and $\check{r}=\Tr(\check{Q})/\Tr(\check{P})$. We must have
that the one click rotation of $A$ is equal to
$\frac{\sqrt{r}}{\sqrt{\check{r}}}\sigma_A\check{A}$ where $\sigma_A$ is some
$2n$-th root of unity with $\sigma_A^2=\omega_A$.

\begin{thm}[{\cite[Theorem 3.10]{MR3311757}}]
\label{thm:Magic11Obstructions}
Suppose $\Gamma_\pm$ are both $(n-2)$-translated extensions of 
$
\bigraph{gbg1v1p1v1x0p0x1p0x1v}.
$
\begin{enumerate}[label=(\arabic*)]
\item
If $n$ is even, $(\Gamma_+,\Gamma_-)$ is a translated extension of
$
\left(\bigraph{bwd1v1p1v1x0p0x1p0x1vduals1v1x2v},\bigraph{bwd1v1p1v1x0p1x0p0x1vduals1v1x2v}\right)
$,
and
$$
\sigma_A+\sigma_A^{-1}
=
\frac{\sqrt{\check{r}}}{\sqrt{r}} [n+2]
-
\frac{\sqrt{r}}{\sqrt{\check{r}}} [n].
$$
\item
If $n$ is odd, $(\Gamma_+,\Gamma_-)$ is a translated extension of
$
\left(\bigraph{bwd1v1v1p1v1x0p0x1p0x1vduals1v1v2x1x3},\bigraph{bwd1v1v1p1v1x0p0x1p0x1vduals1v1v2x1x3}\right)
$,
and 
$$
\sigma_A+\sigma_A^{-1}
= 
\frac{[n+2]}{r}
-
r[n].
$$
\end{enumerate}
\end{thm}

\begin{remark}
The two formulas in Theorem \ref{thm:Magic11Obstructions} are really the same
using an alternate definition of the branch factor. If we let $s_\pm$ be the
trace of the trivalent vertex of $\Gamma_\pm$ divided by the trace of the
bivalent vertex of $\Gamma_\pm$, then both formulas can be written as
\begin{equation}
\sigma_A+\sigma_A^{-1}
=
\frac{[n+2]}{\sqrt{s_+s_-}}
-
\sqrt{s_+s_-} [n].
\label{eq:Magic11Formula}
\end{equation}
Indeed, we always have $s_+=r$; when $n$ is even $s_-=1/\check{r}$, and when $n$ is odd $s_-=s_+=r$.
\end{remark}

Squaring both sides of Equation \eqref{eq:Magic11Formula}, we have
\begin{equation}
\frac{[n+2]^2}{s_+s_-} - 2[n+2][n]+s_+s_-[n]^2=\omega_A+\omega_A^{-1}+2 \in [0,4],
\label{eq:Magic11FormulaNoSquareRoots}
\end{equation}
so subtracting 4, we get the following:

\begin{cor}[Non-univalent $*11$ branch factor inequality]
\label{cor:11Inequality}
$$
-4 \leq \omega_A+\omega_A^{-1}-2= \frac{[n+2]^2}{s_+s_-} - 2[n+2][n]+s_+s_-[n]^2 -4 \leq 0.
$$
\end{cor}

Corollary \ref{cor:11Inequality} was used in \cite[Theorem 3.17]{MR3311757} to
eliminate weed (2) in part \ref{case:11} of Theorem \ref{thm:Enumerate},
which has annular multiplicities $*11$ (see Section \ref{sec:11}). Corollary
\ref{cor:10Inequality} was used in \cite[Section 4]{MR2902285} to eliminate
two weeds in part \ref{case:10} of Theorem \ref{thm:Enumerate}, which have
annular multiplicities $*10$ (see Section \ref{sec:10}).

These weeds were eliminated by determining the relative dimensions of the
vertices as functions of $q$ and $a=q^{2t}$. Next, the relative branch factors
were computed, which are the expressions for $r,\check{r}$ as functions of $a$
and $q$. Finally, we use Corollaries \ref{cor:10Inequality} and
\ref{cor:11Inequality} to write $\omega_A+\omega_A^{-1}-2=p_\Gamma(a,q)$, and
we show $p_\Gamma$ is always positive for $a\geq q^{2t_0}$ and $q> q_0$ where
$q_0$ is determined by the graph norm of a $2t_0$-translate of the weed in
question. This eliminates all $2t$-translates of extensions of the weed for
$t\geq t_0$. For the $*11$ weed eliminated in \cite[Theorem 3.17]{MR3311757},
$t_0=0$, which eliminates all translated extensions. However, for the two
$*10$ weeds eliminated in \cite[Section 4]{MR2902285}, $t_0=1,2$ respectively,
and additional arguments were supplied for these small translates.

The following trick was used in \cite{MR3311757} to verify positivity for
certain polynomials $p(a,q)$.

\begin{defn}\label{defn:ObviouslyPositive}
Given a polynomial $p$ in variables $a$ and $q$ 
$$p(a,q) = \sum_{i=0}^k p_i(q) a^i$$ 
we say that $p$ is \emph{obviously positive for $q > q_0$} 
if each single variable polynomial 
$$\sum_{i=j}^k p_i(q) \text{ for $j=0,\ldots,k$}$$ 
has positive leading coefficient, and the largest root of any of them is at most $q_0$. 

We say a rational function in $a,q$ is obviously positive for $q>q_0$ if each irreducible
factor, of either the numerator or denominator, is obviously positive.
\end{defn}

This condition is sufficiently straightforward to
check that we make such claims without explicit proofs. 
The following lemma is `obvious'.

\begin{lem}\label{lem:ObviouslyPositive}
If $p$ is obviously positive for $q \geq q_0$, then $p(a,q)$ is positive for all $q > q_0$ and $a \geq 1$. 
\end{lem}

In the subsequent subsections, we suppress most of the calculations, which are
straightforward. These calculations are performed in the Mathematica notebook
{\tt Weeds.nb}, bundled with the {\tt arXiv} source.

\subsection{Ruling out a particular \texorpdfstring{$*11$}{*11} weed}
\label{sec:11WithDoubleOneByOne}

We begin by eliminating a particularly difficult $*11$ weed: number (5) of
part \ref{case:11} of Theorem \ref{thm:Enumerate}.

\begin{thm}
\label{thm:Hard11weed}
There is no subfactor whose principal graphs are a translated extension of
$$
\cD=
(\cD_+,\cD_-)=
\left(
\bigraph{bwd1v1v1v1p1v1x0p0x1p0x1v1x0x0p1x0x0p0x1x0p0x0x1v1x0x0x0p0x0x1x0p0x1x0x0p0x0x0x1vduals1v1v1x2v1x3x2x4v}
\,,\,
\bigraph{bwd1v1v1v1p1v1x0p1x0p0x1v1x0x0p0x1x0p0x0x1p0x0x1v1x0x0x0p1x0x0x0p0x0x1x0p0x1x0x0vduals1v1v1x2v1x3x2x4v}
\right).
$$
\end{thm}

We would like to eliminate this weed using the technique outlined in the
previous section, but when solving for the relative dimensions, we have one
unknown parameter:  $\dim(V^d_{5+2t,3})$. Luckily, as in Section
\ref{sec:DoublyOneByOne}, the loop
$(V^p_{2t+4,2},V^p_{2t+5,2},V^d_{2t+6,3},V^d_{2t+5,3})$ outlined in blue below
in $\cO(\cD)$ between depths $2t+3$ and $2t+7$ corresponds to a doubly one-by-one connection entry.
$$
\begin{tikzpicture}[baseline, scale = .95]
	\node at (-2.3,1.5) {$\cD_+\,\Bigg\{$};
	\node at (-2.3,-.5) {$\cD_-\,\Bigg\{$};
%
	\coordinate (aa1) at (1,2);
	\coordinate (aa2) at (2,2);
	\coordinate (aa3) at (6,2);
	\coordinate (aa4) at (7,2);
	\coordinate (aa5) at (8,2);
	\coordinate (aa6) at (9,2);
%
	\coordinate (ab1) at (0,1);
	\coordinate (ab2) at (3,1);
	\coordinate (ab3) at (4,1);
	\coordinate (ab4) at (5,1);
	\coordinate (ab5) at (10,1);
	\coordinate (ab6) at (11,1);
	\coordinate (ab7) at (12,1);
	\coordinate (ab8) at (13,1);
%
	\coordinate (bb1) at (1,0);
	\coordinate (bb2) at (2,0);
	\coordinate (bb3) at (6,0);
	\coordinate (bb4) at (7,0);
	\coordinate (bb5) at (8,0);
	\coordinate (bb6) at (9,0);
%
	\coordinate (ba1) at (0,-1);
	\coordinate (ba2) at (3,-1);
	\coordinate (ba3) at (4,-1);
	\coordinate (ba4) at (5,-1);
	\coordinate (ba5) at (10,-1);
	\coordinate (ba6) at (11,-1);
	\coordinate (ba7) at (12,-1);
	\coordinate (ba8) at (13,-1);
%
	\coordinate (aa1d) at (1,-2);
	\coordinate (aa2d) at (2,-2);
	\coordinate (aa3d) at (6,-2);
	\coordinate (aa4d) at (7,-2);
	\coordinate (aa5d) at (8,-2);
	\coordinate (aa6d) at (9,-2);
%
		\draw (aa1)--(ab1);
		\draw (aa2)--(ab1);
		\draw (aa1)--(ab2);
		\draw[blue] (aa2)--(ab3);
		\draw (aa2)--(ab4);
		\draw (aa3)--(ab2);
		\draw (aa4)--(ab2);
		\draw (aa5)--(ab3);
		\draw (aa6)--(ab4);
		\draw (aa3)--(ab5);
		\draw (aa4)--(ab7);
		\draw (aa5)--(ab6);
		\draw (aa6)--(ab8);
%
		\draw (ab1)--(bb1);
		\draw (ab1)--(bb2);
		\draw (ab2)--(bb1);
		\draw (ab3)--(bb1);
		\draw (ab4)--(bb2);
		\draw (ab2)--(bb3);
		\draw[blue] (ab3)--(bb5);
		\draw (ab4)--(bb4);
		\draw (ab4)--(bb6);
		\draw (ab5)--(bb3);
		\draw (ab6)--(bb3);
		\draw (ab7)--(bb4);
		\draw (ab8)--(bb5);
%
		\draw (bb1)--(ba1);
		\draw (bb2)--(ba1);
		\draw (bb1)--(ba2);
		\draw (bb1)--(ba3);
		\draw (bb2)--(ba4);
		\draw (bb3)--(ba2);
		\draw (bb4)--(ba3);
		\draw[blue] (bb5)--(ba4);
		\draw (bb6)--(ba4);
		\draw (bb3)--(ba5);
		\draw (bb3)--(ba6);
		\draw (bb4)--(ba8);
		\draw (bb5)--(ba7);
%
		\draw (ba1)--(aa1d);
		\draw (ba1)--(aa2d);
		\draw (ba2)--(aa1d);
		\draw (ba3)--(aa2d);
		\draw[blue] (ba4)--(aa2d);
		\draw (ba2)--(aa3d);
		\draw (ba3)--(aa4d);
		\draw (ba2)--(aa5d);
		\draw (ba4)--(aa6d);
		\draw (ba5)--(aa3d);
		\draw (ba6)--(aa4d);
		\draw (ba7)--(aa5d);
		\draw (ba8)--(aa6d);
%
	\node at (-1,2) {$\sb{A}{\sf{Mod}}_A$};
	\filldraw (aa1) circle (.8mm);
	\filldraw[blue] (aa2) circle (.8mm);
	\filldraw (aa3) circle (.8mm);
	\filldraw (aa4) circle (.8mm);
	\filldraw (aa5) circle (.8mm);
	\filldraw (aa6) circle (.8mm);
%
	\node at (-1,1) {$\sb{A}{\sf{Mod}}_B$};
	\filldraw (ab1) circle (.8mm); 
	\filldraw (ab2) circle (.8mm);
	\filldraw[blue] (ab3) circle (.8mm);
	\filldraw (ab4) circle (.8mm);
	\filldraw (ab5) circle (.8mm);
	\filldraw (ab6) circle (.8mm);
	\filldraw (ab7) circle (.8mm);
	\filldraw (ab8) circle (.8mm);
%
	\node at (-1,0) {$\sb{B}{\sf{Mod}}_B$};
	\filldraw (bb1) circle (.8mm);
	\filldraw (bb2) circle (.8mm);
	\filldraw (bb3) circle (.8mm);
	\filldraw (bb4) circle (.8mm);
	\filldraw[blue] (bb5) circle (.8mm);
	\filldraw (bb6) circle (.8mm);
%
	\node at (-1,-1) {$\sb{B}{\sf{Mod}}_A$};
	\filldraw (ba1) circle (.8mm);
	\filldraw (ba2) circle (.8mm);
	\filldraw (ba3) circle (.8mm);
	\filldraw[blue] (ba4) circle (.8mm);
	\filldraw (ba5) circle (.8mm);
	\filldraw (ba6) circle (.8mm);
	\filldraw (ba7) circle (.8mm);
	\filldraw (ba8) circle (.8mm);
%
	\node at (-1,-2) {$\sb{A}{\sf{Mod}}_A$};
	\filldraw (aa1d) circle (.8mm);
	\filldraw[blue] (aa2d) circle (.8mm);
	\filldraw (aa3d) circle (.8mm);
	\filldraw (aa4d) circle (.8mm);
	\filldraw (aa5d) circle (.8mm);
	\filldraw (aa6d) circle (.8mm);
\end{tikzpicture}
$$
This doubly one-by-one entry gives us the following formula for our undetermined relative dimension:
$$
\dim(V^d_{5+2t,3}) 
= 
\frac{
1+K-2 q^2+3 q^4+q^8+a^2 \left(-q^8-3 q^{12}+2 q^{14}-q^{16}\right)
}
{
2 a (-1+q) q^3 (1+q) \left(1+4 q^4+q^8\right)
}
$$
where $K>0$ such that
\begin{align*}
K^2 & = 
  a^4 \left(q^{32}{+}4 q^{30}{+}14 q^{28}{+}24 q^{26}{+}27 q^{24}{+}20 q^{22}{+}10 q^{20}{+}4 q^{18}{+}q^{16}\right)
  \\ & \quad 
  {+}a^2 \left({-}2 q^{24}{-}8 q^{22}{-}20 q^{20}{-}44 q^{18}{-}62 q^{16}{-}44 q^{14}{-}20 q^{12}{-}8 q^{10}{-}2 q^8\right)
  \\ & \quad 
  {+}q^{16}{+}4 q^{14}{+}10 q^{12}{+}20 q^{10}{+}27 q^8{+}24 q^6{+}14 q^4{+}4 q^2{+}1.
\end{align*}

\begin{remark}
\label{rem:11doublyOneByOneNegativeDimensions}
There is another solution for $\dim(V^d_{5+2t,3})$ with $-K$ instead
of $+K$; this is always negative for $a,q \geq 1$, which is impossible.
\end{remark}

Now applying the technique of the previous section, Corollary
\ref{cor:11Inequality} tells us the following inequality must be satisfied:

\begin{equation}
\label{eq:11inequality}
\frac{
4 (a q^5 + 1)^2 (a q^5 - 1)^2 (\alpha(a,q)-K \beta(a,q)) (\alpha(-a,q)-K \beta(-a,q)) }
{
	a^2 (q+1)^2 (q-1)^2 q^{10} \prod_{i=1}^4 (\gamma_i(a,q) - K \delta_i(a,q)) 
}
\leq 0
\end{equation}

where
\begin{align*}
\alpha(a,q) & =
  a^4 \left(q^{42}{+}q^{40}{+}4 q^{38}{-}7 q^{36}{-}5 q^{34}{-}13 q^{32}{+}q^{30}{+}8 q^{28}{+}17 q^{26}{+}16 q^{24}{+}8 q^{22}{+}4 q^{20}{+}q^{18}\right)
  \\ & \quad 
  {+} a^3 \left({-}2 q^{36}{-}4 q^{34}{-}14 q^{32}{-}14 q^{30}{-}16 q^{28}{+}10 q^{26}{+}38 q^{24}{+}28 q^{22}{+}26 q^{20}{+}14 q^{18}{+}4 q^{16}{+}2 q^{14}\right)
  \\ & \quad 
  {+} a^2 \left({-}q^{34}{-}2 q^{32}{-}8 q^{30}{-}17 q^{28}{-}17 q^{26}{-}14 q^{24}{-}25 q^{22}{+}25 q^{20}{+}14 q^{18}{+}17 q^{16}{+}17 q^{14}{+}8 q^{12}{+}2 q^{10}{+}q^8\right)
  \\ & \quad 
  {+} a \left({-}2 q^{28}{-}4 q^{26}{-}14 q^{24}{-}26 q^{22}{-}28 q^{20}{-}38 q^{18}{-}10 q^{16}{+}16 q^{14}{+}14 q^{12}{+}14 q^{10}{+}4 q^8{+}2 q^6\right)
  \\ & \quad 
  {-} q^{24}{-}4 q^{22}{-}8 q^{20}{-}16 q^{18}{-}17 q^{16}{-}8 q^{14}{-}q^{12}{+}13 q^{10}{+}5 q^8{+}7 q^6{-}4 q^4{-}q^2{-}1 \displaybreak[1] \\
\beta(a,q) & = 
  a^2 \left({-}q^{26}{+}q^{24}{-}q^{22}{+}6 q^{20}{-}3 q^{18}{+}6 q^{16}{+}q^{14}{+}2 q^{12}{+}q^{10}\right)
  \\ & \quad 
  {+}a \left(2 q^{20}{+}8 q^{16}{+}2 q^{14}{+}2 q^{12}{+}8 q^{10}{+}2 q^6\right)
  \\ & \quad 
  {+}q^{16}{+}2 q^{14}{+}q^{12}{+}6 q^{10}{-}3 q^8{+}6 q^6{-}q^4{+}q^2{-}1 \displaybreak[1] \\
\gamma_1(a,q) & = a^2 \left(q^{16}{+}2 q^{14}{+}5 q^{12}{+}q^8\right) {-}q^8{-}5 q^4{-}2 q^2 {-}1  \\
\gamma_2(a,q) & = a^2 \left(2 q^{20}{+}5 q^{18}{+}12 q^{16}{+}9 q^{14}{+}4 q^{12}{+}q^{10}\right) {-}q^{10}{-}4 q^8{-}9 q^6{-}12 q^4{-}5 q^2{-}2 \\
\gamma_3(a,q) & = a^2 \left(2 q^{20}{+}q^{18}{+}6 q^{16}{+}3 q^{14}{+}2 q^{12}{+}q^{10}\right) {-}q^{10}{-}2 q^8{-}3 q^6{-}6 q^4{-}q^2{-}2 \\
\gamma_4(a,q) & =
  a^2 \left(2 q^{24}{+}4 q^{22}{+}9 q^{20}{+}14 q^{18}{+}17 q^{16}{+}10 q^{14}{+}5 q^{12}{+}2 q^{10}\right)
  \\ & \quad 
  {-}2 q^{14}{-}5 q^{12}{-}10 q^{10}{-}17 q^8{-}14 q^6{-}9 q^4{-}4 q^2{-}2 \\
\delta_1(a,q) & = {-}1 \\
\delta_2(a,q) & = {-}2 q^4{-}q^2{-}2 \\
\delta_3(a,q) & = q^2 \\
\delta_4(a,q) & = 2 q^6{+}q^4{+}2 q^2
\end{align*}

\begin{lem}
\label{lem:11Weed6Positivitiy}
When $a\geq 1$ and $q\geq 1.65$, every factor in both the numerator and the
denominator of Inequality \eqref{eq:11inequality} is positive. 
Hence for $a\geq 1$ and $q\geq 1.65$, Inequality \eqref{eq:11inequality} does not hold.
\end{lem}
\begin{proof}
This calculation is straightforward, and performed in {\tt{Weeds.nb}}.
\end{proof}

\begin{proof}[Proof of Theorem \ref{thm:Hard11weed}]
If such a subfactor existed, then by Corollary \ref{cor:11Inequality}, Inequality \eqref{eq:11inequality} must hold.
We note that the $q$ for $(\cD_+,\cD_-)$ must be larger than $1.65$ by looking at the graph norm.
But by Lemma \ref{lem:11Weed6Positivitiy},
Inequality \eqref{eq:11inequality} is never satisfied for the relevant range of $a$ and $q$, a contradiction.
\end{proof}

\subsection{Ruling out the remaining \texorpdfstring{$*11$}{*11} weeds}
\label{sec:11}

Recall that the weed
$$
\left(\bigraph{bwd1v1v1p1v1x0p0x1p1x0v0x0x1p0x1x0p1x0x0p0x1x0v0x0x0x1p0x1x0x0p1x0x0x0p0x0x1x0p1x0x0x0v1x0x0x0x0p0x0x0x0x1p0x1x0x0x0p0x0x0x1x0p0x0x0x1x0vduals1v1v1x3x2v1x3x2x5x4v},
 \bigraph{bwd1v1v1p1v1x0p0x1p1x0v0x0x1p0x1x0p0x1x0p1x0x0v0x0x1x0p0x1x0x0p1x0x0x0p0x0x0x1p1x0x0x0v0x0x0x0x1p1x0x0x0x0p0x0x0x1x0p0x1x0x0x0p0x0x0x1x0vduals1v1v1x3x2v1x3x2x5x4v}\right)
$$
was eliminated in \cite[Theorem 3.17]{MR3311757}. The argument there is a
straightforward application of the method outlined in Section
\ref{sec:BranchFactorInequalities} by showing that each factor in the
numerator and each factor in the denominator of the relative branch factor are
obviously positive (see Lemma \ref{lem:ObviouslyPositive}) for $a\geq 1$ and
$q\geq 1.6789$, which is a lower bound for the $q$ for the above weed. Hence
no subfactor exists with principal graphs a translated extension, since
Corollary \ref{cor:11Inequality} would not hold.

The argument for each weed in the following theorem is identical to
\cite[Theorem 3.17]{MR3311757}, mutatis mutandis. The interested reader can
view the necessary calculations in {\tt Weeds.nb}.

\begin{thm}
No subfactor has principal graphs a translated extension of any of
\begin{align*}
&\left(\bigraph{bwd1v1v1p1v0x1p1x0p0x1v1x0x0p0x0x1p0x1x0p0x1x0v0x0x1x0p1x0x0x0p0x1x0x0v0x1x0p0x0x1p1x0x0vduals1v1v1x3x2v1x3x2v}, 
\bigraph{bwd1v1v1p1v0x1p1x0p0x1v0x1x0p0x0x1p1x0x0p0x1x0v1x0x0x0p0x1x0x0p0x0x1x0v0x0x1p1x0x0p0x1x0vduals1v1v1x3x2v1x3x2v}\right)
\\
&\left(\bigraph{bwd1v1v1v1p1v1x0p1x0p0x1v0x0x1p0x0x1p1x0x0vduals1v1v1x2v1x3x2}, 
\bigraph{bwd1v1v1v1p1v1x0p0x1p1x0v0x0x1p0x1x0p1x0x0vduals1v1v1x2v1x3x2}\right)
\\
&\left(\bigraph{bwd1v1v1v1p1v1x0p0x1p0x1v0x0x1p1x0x0p1x0x0p0x1x0v0x0x1x0p1x0x0x0p0x0x0x1vduals1v1v1x2v1x2x4x3v}, \bigraph{bwd1v1v1v1p1v1x0p1x0p0x1v1x0x0p0x0x1p0x1x0p0x0x1v0x0x0x1p0x0x1x0p1x0x0x0vduals1v1v1x2v1x2x4x3v}\right)
\\
&\left(\bigraph{bwd1v1v1v1p1v1x0p1x0p0x1v0x1x0p0x0x1p0x0x1p1x0x0v1x0x0x0p0x0x1x0p0x0x0x1p1x0x0x0v0x0x1x0p0x1x0x0p1x0x0x0vduals1v1v1x2v1x2x4x3v1x3x2}, 
\bigraph{bwd1v1v1v1p1v1x0p0x1p1x0v0x0x1p0x1x0p0x1x0p1x0x0v0x0x0x1p0x0x1x0p1x0x0x0p0x1x0x0v1x0x0x0p0x0x1x0p0x1x0x0p0x0x0x1p1x0x0x0vduals1v1v1x2v1x2x4x3v1x3x2x5x4}\right)
\end{align*}
\end{thm}

Now the previous theorem, \cite[Theorem 3.17]{MR3311757}, and Theorem
\ref{thm:Hard11weed} rule out all the $*11$ weeds in part \ref{case:11} of
Theorem \ref{thm:Enumerate}.
\qed

\subsection{Ruling out two particular \texorpdfstring{$*10$}{*10} weeds with doubly one-by-ones}
\label{sec:10WithDoubleOneByOne}

\subsubsection{A truncation of \texorpdfstring{$\cF$ from \cite{MR2902285}}{F}}

In \cite[Section 4.5]{MR2902285}, the problematic weed
$$
\cF=
\left(\bigraph{bwd1v1v1v1p1v1x0p0x1v1x0p1x0p0x1p0x1v0x1x0x0p0x0x1x0p0x0x0x1v1x0x0p0x1x0p0x0x1p1x0x0v1x0x0x0p0x1x0x0p0x0x1x0vduals1v1v1x2v1x3x2x4v1x2x4x3v}, \bigraph{bwd1v1v1v1p1v1x0p1x0v1x0p0x1v0x1p1x0p0x1v1x0x0p0x1x0v0x1p1x0p0x1vduals1v1v1x2v1x2v2x1v}\right)
$$
was ruled out using both branch factor inequalities and doubly one-by-one connection entries.
 
First, the relative dimensions can be computed in terms of $q$ and $a=q^{2t}$,
where $2t$ is the translation, which gives a formula for the relative branch
factor. Then Branch Factor Inequality \eqref{eq:QTEquation} shows that $t\leq
1$. Now for each of $t=0$ and $t=1$, we know $\omega_A$ is a $(2t+4)$-th root of
unity. This allows us to solve directly for $q$ in each case.

Second, there is a doubly one-by-one connection entry, which allows us to
exactly solve for $q$ in the cases $t=0,1$. It turns out that these values of
$q$ are incompatible with the values of $q$ obtained from Branch Factor
Inequality \eqref{eq:QTEquation} for any allowed $\omega_A$.  Hence no
translated extension of $\cF$ can be the principal graphs of a subfactor.

Working a bit harder, we can use the recent advances
\cite{MR3210716,MR3335120} to rule out the truncation $\cF'$ of $\cF$ by two
depths.

\begin{thm}
\label{thm:NoFPrime}
There are no subfactors whose principal graphs are a translated extension of
$$
\cF'=
\left(\bigraph{bwd1v1v1v1p1v1x0p0x1v1x0p1x0p0x1p0x1v0x1x0x0p0x0x1x0p0x0x0x1vduals1v1v1x2v1x3x2x4v}, \bigraph{bwd1v1v1v1p1v1x0p1x0v1x0p0x1v0x1p1x0p0x1vduals1v1v1x2v1x2v}\right).
$$
(This is case \ref{case:10}(7) from Theorem \ref{thm:Enumerate}.) 
\end{thm}

For $\cF'$, we can no longer solve for all the relative dimensions; we have one unknown parameter: $\dim(V_{2t+5,2}^{p})$.
By the proof of \cite[Proposition 4.17]{MR2902285}, we still have the same doubly-one-by-one connection entry, corresponding to the loop
$(V^p_{2t+6,2}, V^p_{2t+5,1},V^d_{2t+4,},V^d_{2t+5,2})$.
This yields the following equation for our unknown parameter:
$$
\dim(V_{2t+5,2}^{p})
=
\frac{
-2-q^2-K q^2+q^6+a^2 \left(-q^{10}+q^{14}+2 q^{16}\right)
}{
2 a (-1+q) q^5 (1+q) \left(1+q^2\right)^2
}
$$
where $K>0$ is the positive square root of
\begin{align*}
K^2 
&=
5{+}4 q^2{+}6 q^4{+}4 q^6{+}q^8{+}a^2 \left({-}2 q^8{-}8 q^{10}{-}20 q^{12}{-}8 q^{14}{-}2 q^{16}\right){+}a^4 \left(q^{16}{+}4 q^{18}{+}6 q^{20}{+}4 q^{22}{+}5 q^{24}\right).
\end{align*}

\begin{remark}
As in Remark \ref{rem:11doublyOneByOneNegativeDimensions} in Section
\ref{sec:11WithDoubleOneByOne}, we get two solutions for
$\dim(V_{2t+5,2}^{p})$, where the other solution has $+K$ instead of $-K$.
This time, both formulas give positive values for $\dim(V_{2t+5,2}^{p})$, but
this second solution with $+K$ gives a negative value for
$$
\dim(V_{2t+5,1}^{p}) 
=
\frac{a^2 q^{10}-a q^5 \dim(V_{2t+5,2}^{p})  +1}{a q^5}
$$
when $q,a\geq 1$.
\end{remark}

Applying the relative branch factor technique, Corollary \ref{cor:10Inequality} 
tells us that the following inequality must be satisfied:
\begin{equation}
\label{eq:FPrime10inequality}
\frac{
4 \left(-1+a^2 q^{10}\right)^2 
\left(\alpha(a,q)-q^2 K\right)
\left(\alpha(-a,q)-q^2 K\right)
}{
a^2 (-1+q)^2 q^{10} (1+q)^2 
\left(\gamma(a,q)+K\right)
\left(\delta(a,q) - q^2 K\right)
}
\leq 0
\end{equation}
where
\begin{align*}
\alpha(a,q)
&=
-1-q^4+a \left(q^4+2 q^6-2 q^{10}-q^{12}\right)+a^2 \left(q^{12}+q^{16}\right)
\\
\gamma(a,q)
&=
-3-2 q^2-q^4+a^2 \left(q^8+2 q^{10}+3 q^{12}\right)
\\
\delta(a,q)
&=
-2-3 q^2-4 q^4-q^6+a^2 \left(q^{10}+4 q^{12}+3 q^{14}+2 q^{16}\right)
\end{align*}

It is straightforward to calculate that every factor in both the numerator and the denominator of Inequality \eqref{eq:FPrime10inequality} is positive for $q>1.567$ and $a\geq q^6$ (see {\tt Weeds.nb} for more details).
This immediately gives us the following lemma.

\begin{lem}
\label{lem:FPrimeSmallT}
Any subfactor with principal graphs a $2t$-translate of an extension of $\cF'$ must have $t\leq 2$.
\end{lem}
\begin{proof}
Suppose we have such a subfactor with $t\geq 3$. Then by Corollary
\ref{cor:10Inequality}, Inequality \eqref{eq:FPrime10inequality} must hold. We
note that the $q$ for $\cF'$ translated by 6 is larger than 1.567. Now all the
factors in Equation \eqref{eq:FPrime10inequality} are positive when $a\geq
q^6$ and $q>1.567$. This shows that Inequality \eqref{eq:FPrime10inequality}
is never satisfied for the relevant range of $a\geq q^6$ and $q> 1.567$, a
contradiction. Hence $t\leq 2$.
\end{proof}

We now deal with the cases that $t\leq 2$ by hand as in \cite[Proposition
4.16]{MR2902285}, noting that there are only finitely many possibilities for
the rotational eigenvalue $\omega_A$ in Corollary \ref{cor:10Inequality}.

\begin{prop}
\label{prop:FPrimeTIsZero}
Any subfactor with principal graphs a translated extension of $\cF'$ either has
\begin{enumerate}[label=(\arabic*)]
\item $t=0$ and index $(q+q^{-1})^2 = 3+\sqrt{5}$, or
\item $t=1,2$ and index $(q+q^{-1})^2\leq 5$.
\end{enumerate}
\end{prop}
\begin{proof}
Using the notation $D=\dim(V^p_{2t+5,1})$
Equation \eqref{eq:QTEquation} simplifies to
\begin{equation}
\label{eq:FPrimeQTEquation}
\frac{f(a,q,\omega)}
{a^2 q^{10} \left(-1+a^2 q^8+a D \left(-q^3+q^5\right)\right) \left(-1+a^2 q^{12}+a D \left(q^5-q^7\right)\right)}
=0
\end{equation}
where
\begin{align*}
f(a,q,\omega)
=&
\left(a^2 q^8-1\right) \left(a^2 q^{12}-1\right) \left(a^2 q^{10}-\omega\right) \left(a^2 q^{10}-\omega^{-1} \right)
\\&
+a D q^5 \left(-1+a q^5 \left(D-a q^5\right)\right) \left(4+4 a^4 q^{20}+a^2 q^8 \left(\left(-1+q^2\right)^2 \left(\omega+\omega^{-1} \right)-2 \left(1+q^2\right)^2\right)\right)
\end{align*}
When $t=0,1,2$, $a=1,q^2,q^4$ respectively.
We note that we must have $q>1.558$, which comes from solving $\|\cF'\|=q+q^{-1}$.
Recall that if $\omega$ is a primitive $k$-th root of unity, then $2k | (2t+4)$ by \cite[Theorem 5.1.11]{MR2972458}.
Using the fact that Equation \eqref{eq:FPrimeQTEquation} is symmetric in $\omega$ and $\omega^{-1}$, it remains to look at the cases 
$$
(t,\omega)\in \{(0,1),(0,-1),(2,1),(2,e^{2\pi i/3}),(4,1),(4,-1),(4, i)\}.
$$
Now plugging each of these in to Equation \eqref{eq:FPrimeQTEquation} using $a=q^{2t}$, we can solve for $q$, and the claim follows. 
For example, when $\omega=-1$, for each $t\leq 2$, we solve Equation \eqref{eq:FPrimeQTEquation} to see $q<1.554$, which is too small.
Thus in the case $t=0$, we must have $\omega=1$, and solving for $q$ shows that the index is exactly $3+\sqrt{5}$.
Again, calculations can be viewed in {\tt Weeds.nb}.
\end{proof}

By the classification of subfactors with index at most 5
\cite{MR2914056,MR2902285,MR2993924,MR2902286,MR3335120}, we can rule out
$t=1$ and $t=2$. However, we need a different way to deal with $t=0$.  The
recent 3-supertransitive $*10$ obstruction of \cite{MR3210716} does the trick.

\begin{thm*}[\cite{MR3210716}]
\label{thm:ScottHexagon}
Suppose a subfactor has principal graphs $(\Gamma_+,\Gamma_-)$ an extension of
$$
\left(\bigraph{bwd1v1v1v1p1v1x0p0x1vduals1v1v1x2v},\bigraph{bwd1v1v1v1p1v1x0p1x0vduals1v1v1x2v}\right).
$$
If there is no vertex of $\Gamma_+$ at depth $6$ which connects to both vertices of $\Gamma_+$ at depth $5$, then 
$$
(\Gamma_+,\Gamma_-)=
\left(\bigraph{bwd1v1v1v1p1v1x0p0x1v1x0p0x1duals1v1v1x2v2x1},\bigraph{bwd1v1v1v1p1v1x0p1x0duals1v1v1x2}\right).
$$
\end{thm*}

\begin{proof}[Proof of Theorem \ref{thm:NoFPrime}]
Suppose we had a subfactor whose principal graph is a $2t$-translate of an
extension of $\cF'$. By Lemma \ref{lem:FPrimeSmallT}, $t\leq 2$. By
Proposition \ref{prop:FPrimeTIsZero} and the classification of subfactors to
index 5, we must have that $t=0$, but by \cite{MR3210716}, we have $t>0$, a
contradiction.
\end{proof}

\subsubsection{Another \texorpdfstring{$*10$}{*10} weed with undetermined relative dimensions}

We now tackle another difficult $*10$ weed. We are not able to completely
eliminate it at this time, but we can show any potential translated extension
must occur at index $\sim 5.3234$.

\begin{thm}
\label{thm:NoG}
Any subfactor with principal graphs a $2t$-translated extension of 
$$
\cG=\left(
\bigraph{bwd1v1v1v1p1v1x0p0x1v1x0p0x1p0x1p1x0v1x0x0x0p0x1x0x0p0x0x1x0p0x0x0x1v1x0x0x0p0x0x1x0p0x0x0x1p0x0x1x0p1x0x0x0p0x0x0x1p0x1x0x0vduals1v1v1x2v1x2x4x3v1x2x3x5x4x7x6},
\bigraph{bwd1v1v1v1p1v1x0p1x0v1x0p0x1v1x0p0x1p1x0p0x1v1x0x0x0p0x0x0x1p0x0x1x0vduals1v1v1x2v1x2v1x3x2}
\right)
$$
must have $t=1$, $\omega=1$, and $q$ is the root of $1-2 x^2-2 x^4-2 x^6-2 x^8-2 x^{10}+x^{12}$ which is about $1.72882$.
In this case, the principal graph is an extension of 
$$
\left(
\bigraph{bwd1v1v1v1v1v1p1v1x0p0x1v1x0p0x1p0x1p1x0v1x0x0x0p0x1x0x0p0x0x1x0p0x0x0x1v1x0x0x0p0x0x1x0p0x0x0x1p0x0x1x0p1x0x0x0p0x0x0x1p0x1x0x0vduals1v1v1v1x2v1x2x4x3v1x2x3x5x4x7x6},
\bigraph{bwd1v1v1v1v1v1p1v1x0p1x0v1x0p0x1v1x0p0x1p1x0p0x1v1x0x0x0p0x0x0x1p0x0x1x0vduals1v1v1v1x2v1x2v1x3x2}
\right),
$$
and the index is the root of $-4+15 x-8 x^2+x^3$ which is about $5.3234$.

(This graph is case \ref{case:10}(8) from Theorem \ref{thm:Enumerate}.)
\end{thm}

\begin{cor}
There is no finite depth subfactor with principal graphs a translated extension of $\cG$.
\end{cor}
\begin{proof}
The root of $-4+15 x-8 x^2+x^3$ which is about $5.3234$ is not a cyclotomic integer.
\end{proof}

In fact, this weed can be completely ruled out up to index $5\frac{1}{4}$ only
using enumeration! Of course this lemma is sufficient for the main result of
this article, but we can get the much stronger result of Theorem \ref{thm:NoG}
with some more work.

\begin{lem}
\label{lem:EnumerateDifficult10}
A subfactor with index at most $5\frac{1}{4}$ cannot have principal graphs a translated extension of $\cG$.
\end{lem}
\begin{proof}
See the section titled \emph{The \ref{case:10}(8) weed from Theorem 3.1} in the {\tt Mathematica} notebook {\tt enumerator.nb}.
\end{proof}

Together with the quadratic tangles technique and the main result of
\cite{MR3210716} (Theorem \ref{thm:ScottHexagon}), Lemma
\ref{lem:EnumerateDifficult10} actually reduces $\cG$ to the single case left
in Theorem \ref{thm:NoG}. Again, a quadratic tangles argument shows us that
$t\leq 2$, and Theorem \ref{thm:ScottHexagon} says $t>0$. However, the upper
bounds for $q$ when $t=2$ or when $t=1$ and $\omega=e^{\pm 2\pi i/3}$ are
between $5$ and $5\frac{1}{4}$, so Lemma \ref{lem:EnumerateDifficult10} does
the remaining work. While we could just ignore $\cG$ for now due to index
considerations and Lemma \ref{lem:EnumerateDifficult10}, we will perform the
extra analysis required to reduce $\cG$ to the one remaining case for future
enumeration considerations.

As in the previous subsection, there are undetermined relative dimensions, in
particular $\dim(V^{p}_{6,4})$.  (In this case, there is actually more than
one undetermined dimension, but we need only determine one of them!) Luckily,
there is a doubly-one-by-one entry of the connection corresponding to the loop
$(V^p_{6,3},V^p_{5,2},V^d_{4,1},V^d_{5,1})$.

This gives us the following formula for one of our undetermined dimensions:
$$
\dim(V^{p}_{6,4})=
\frac{-2+4 q^2+q^4+3 q^6+q^8+K \left(-1+q^2-q^4\right)+a^2 \left(-q^8-3 q^{10}-q^{12}-4 q^{14}+2 q^{16}\right)}{2 a (-1 + q) q^4 (1 + q) (1 + q^2)^3}
$$
where $K>0$ is the positive square root of
$$
K^2 = 
4+4 q^4+4 q^6+q^8+a^2 \left(-4 q^{10}-18 q^{12}-4 q^{14}\right)+a^4 \left(q^{16}+4 q^{18}+4 q^{20}+4 q^{24}\right).
$$

\begin{remark}
As in Remark \ref{rem:11doublyOneByOneNegativeDimensions}, the other solution 
for $\dim(V^p_{6,4})$ with $-K$ instead of $K$ is impossible, since it is 
always negative.
\end{remark}

Again, applying the relative branch factor technique, Corollary 
\ref{cor:10Inequality} tells us that the following inequality must be satisfied:

\begin{equation}
\label{eq:G10inequality}
\frac{4 \left(-1+a^2 q^{10}\right)^2 
\left( 
\alpha(a,q)-q^2K
\right)
\left( 
\alpha(-a,q)-q^2K
\right)
}{
a^2 (-1+q)^2 q^{10} (1+q)^2 
\left( 
\gamma(a,q)+K
\right)
\left( 
\delta(a,q)-q^2K
\right)
}
\leq 0
\end{equation}
where 
\begin{align*}
\alpha(a,q) &=
-1+q^2-q^4+a \left(-q^4-2 q^6+2 q^{10}+q^{12}\right)+a^2 \left(q^{12}-q^{14}+q^{16}\right)
\\
\gamma(a,q) &=
-4-2 q^2-q^4+a^2 \left(q^8+2 q^{10}+4 q^{12}\right)
\\
\delta(a,q) &=
-2-2 q^2-4 q^4-q^6+a^2 \left(q^{10}+4 q^{12}+2 q^{14}+2 q^{16}\right)
\end{align*}

\begin{proof}[Proof of Theorem \ref{thm:NoG}]
Similar to the weed $\cF'$, all the factors in Inequality
\eqref{eq:G10inequality} are positive for $q>1.69684$ and $a\geq q^4$. Hence
Inequality \eqref{eq:G10inequality} is never satisfied for $q$ this large. For
$q\leq 1.69684$, we note that $(1.69684+1.69684^{-1})^2 =5.22658<
5\frac{1}{4}$. Thus Lemma \ref{lem:EnumerateDifficult10} actually eliminates
all $2t$-translated extensions of $\cG$ with $t\geq 2$!

Again, the case $t=0$ is ruled out by \cite{MR3210716}, so the only remaining
case is $t=1$. Similar to the proof of Proposition \ref{prop:FPrimeTIsZero},
we need only check $\omega=1,e^{\pm2\pi i/3}$ by \cite[Theorem
5.1.11]{MR2972458}. When $\omega=e^{\pm2\pi i/3}$, we calculate that
$(q+q^{-1})^2\approx 5.24994<5\frac{1}{4}$, so Lemma
\ref{lem:EnumerateDifficult10} eliminates this case as well. However, when
$\omega=1$, we have $(q+q^{-1})^2$ is the root of $-4+15 x-8 x^2+x^3$ which is
approximately $5.3234$. We cannot yet rule out this case, but its index is too
large for our current goal of $5\frac{1}{4}$.
\end{proof}

\subsection{Ruling out the remaining \texorpdfstring{$*10$}{*10} weeds}
\label{sec:10}

Suppose we have a weed with annular multiplicities $*10$. When the relative
branch factor $r$ can be determined in terms of $a$ and $q$, and $r(a,q)$ is not
identically 1, then we can typically use Corollary \ref{cor:10Inequality} to
bound the translation $t<t_0$. For the remaining small values of $t$, as in
Proposition \ref{prop:FPrimeTIsZero}, we can find a new smaller upper bound
for the index. Re-running the enumerator to this smaller index is often
possible, which allows us to completely rule out the weed.

This method was applied to the weed $\cC$ in  \cite[Theorem 4.10]{MR2902285}, 
which appears in part \ref{case:10} of Theorem \ref{thm:Enumerate}:
$$
\cC=
\left(
\bigraph{bwd1v1v1v1p1v1x0p1x0v1x0v1p1vduals1v1v1x2v1v}, 
\bigraph{bwd1v1v1v1p1v1x0p0x1v1x0p1x0p0x1v1x0x0p0x0x1vduals1v1v1x2v1x3x2v}
\right).
$$
A similar argument rules out 3 of the remaining $*10$ weeds.

\begin{thm}
\label{thm:OnlyS4S5}
The only subfactor principal graph which is a translated extension of 
$$
\left(
\bigraph{bwd1v1v1v1p1v1x0p0x1v1x0p1x1v1x0vduals1v1v1x2v1x2v}, 
\bigraph{bwd1v1v1v1p1v1x0p1x0v1x0p1x0p0x1v1x0x0vduals1v1v1x2v1x2x3v}
\right)
$$
is the $S_4\subset S_5$ subfactor principal graph.
\end{thm}
\begin{proof}
One computes that for the 2-translate, the relative branch factor is given by
$$
r(a,q)=
\frac{q^2 \left(-2-q^2-q^4+a^2 q^{12}+a^2 q^{14}+2 a^2 q^{16}\right)}{\left(1+q^4\right) \left(-1+a^2 q^{16}\right)}
$$
so by Corollary \ref{cor:10Inequality}, we must have
$$
\frac{\left(a^2 q^{14}-1\right)^2  \left(1-2 q^2-q^6+a (-q^6-2 q^{10}+q^{12})\right) \left(-1+2 q^2+q^6+a (-q^6-2 q^{10}+q^{12})\right)}{a^2 (-1+q)^2 q^{14} (1+q)^2 \left(1+q^4\right) \left(-2-q^2-q^4+a^2 (q^{12}+ q^{14}+2 q^{16})\right)}\leq 0.
$$
However, a simple calculation shows that every factor in both the numerator
and the denominator above is obviously positive (see {\tt Weeds.nb} for more
details). This rules out all $2t$-translated extensions of the above weed when
$t\geq 1$.

We now consider the case $t=0$. Notice that this weed is actually stable, so
that any extension must end with $A_{\text{finite}}$ tails. By
\cite[Proposition 1.17]{MR3306607}, these tails must not be longer than the
initial arm of length $3$, so there are only two possible extensions. The
trivial extension is not possible for a number of reasons. For example, the
index is not a cyclotomic integer, and there is a vertex dimension between 1
and 2 which is not of the form $2\cos(\pi/k)$ for $k\geq 3$. The only other
possibility is the $S_4\subset S_5$ principal graph.
\end{proof}

\begin{thm}
\label{thm:NoOtherHexagonWeeds}
No subfactor has principal graphs a translated extension of either of
\begin{align*}
&\left(\bigraph{bwd1v1v1v1p1v1x0p0x1v1x0p1x1p0x1v1x0x0vduals1v1v1x2v1x2x3v}, 
\bigraph{bwd1v1v1v1p1v1x0p1x0v1x0p1x0p0x1p0x1v1x0x0x0vduals1v1v1x2v1x2x4x3v}
\right)
\\
&\left(\bigraph{bwd1v1v1v1p1v1x0p0x1v1x0p1x1p0x1v1x0x0vduals1v1v1x2v1x2x3v}, 
\bigraph{bwd1v1v1v1p1v1x0p1x0v1x0p1x0p0x1p0x1v1x0x0x0vduals1v1v1x2v1x2x3x4v}\right)
\end{align*}
\end{thm}
\begin{proof}
In fact, these two weeds can be ruled out simultaneously. Each one has the
same relative branch factor. Similar to the proof of Theorem
\ref{thm:OnlyS4S5}, we can rule out all $2t$-translated extensions for $t\geq
2$ by showing that Corollary \ref{cor:10Inequality} cannot hold. The
interested reader can view this calculation in {\tt Weeds.nb}.

We must now consider the case that $t\leq 1$. Again, these weeds are stable,
so any extension must end with $A_{\text{finite}}$ tails, and by
\cite[Proposition 1.17]{MR3306607}, the tail cannot be longer than the initial
arm. A simple calculation shows that the norm squared is only a cyclotomic
integer when $t=1$ and we extend stably by 3. At this point, we look at the
common graph for both graph pairs:
$$
\cX=\bigraph{bwd1v1v1v1v1v1p1v1x0p0x1v1x0p1x1p0x1v1x0x0v1v1v1duals1v1v1v1x2v1x2x3v1v1}\,.
$$
This graph cannot be the principal graph of a subfactor by Theorem \ref{thm:NoH} in the next section.
\end{proof}

Now Theorems \ref{thm:NoB}, \ref{thm:NoAMP} (with \cite{1502.00035}),
\ref{thm:r=1}, \ref{thm:NoFPrime}, \ref{thm:NoG}, \ref{thm:OnlyS4S5},
\ref{thm:NoOtherHexagonWeeds}, and \cite[Theorem 4.10]{MR2902285} together
rule out all the $*10$ weeds in part \ref{case:10} of Theorem
\ref{thm:Enumerate}.

\subsection{Ruling out a graph with formal codegrees}
\label{sec:FormalCodegrees}

In this section, we use formal codegrees to rule out the graph
$$
\cX=
\begin{tikzpicture}[baseline=-.1cm]
\draw[fill] (0,0) circle (0.05) node [below] {\scriptsize{$1$}};
\draw (0.,0.) -- (1.,0.);
\draw[fill] (1.,0.) circle (0.05);
\draw (1.,0.) -- (2.,0.);
\draw[fill] (2.,0.) circle (0.05) node [below] {\scriptsize{$\jw{2}$}};
\draw (2.,0.) -- (3.,0.);
\draw[fill] (3.,0.) circle (0.05);
\draw (3.,0.) -- (4.,0.);
\draw[fill] (4.,0.) circle (0.05) node [below] {\scriptsize{$\jw{4}$}};
\draw (4.,0.) -- (5.,0.);
\draw[fill] (5.,0.) circle (0.05);
\draw (5.,0.) -- (6.,-0.25);
\draw (5.,0.) -- (6.,0.25);
\draw[fill] (6.,-0.25) circle (0.05) node [below] {\scriptsize{$P$}};
\draw[fill] (6.,0.25) circle (0.05) node [above] {\scriptsize{$Q$}};
\draw (6.,-0.25) -- (7.,-0.25);
\draw (6.,0.25) -- (7.,0.25);
\draw[fill] (7.,-0.25) circle (0.05);
\draw[fill] (7.,0.25) circle (0.05);
\draw (7.,-0.25) -- (8.,-0.5);
\draw (7.,-0.25) -- (8.,0.);
\draw (7.,0.25) -- (8.,0.);
\draw (7.,0.25) -- (8.,0.5);
\draw[fill] (8.,-0.5) circle (0.05) node [right] {\scriptsize{$R$}};
\draw[fill] (8.,0.) circle (0.05) node [right] {\scriptsize{$gQ$}};
\draw[fill] (8.,0.5) circle (0.05) node [above] {\scriptsize{$g\jw{4}$}};
\draw (8.,0.5) -- (9.,0.);
\draw[fill] (9.,0.) circle (0.05);
\draw (9.,0.) -- (10.,0.);
\draw[fill] (10.,0.) circle (0.05) node [below] {\scriptsize{$g\jw{2}$}};
\draw (10.,0.) -- (11.,0.);
\draw[fill] (11.,0.) circle (0.05);
\draw (11.,0.) -- (12.,0.);
\draw[fill] (12.,0.) circle (0.05) node [below] {\scriptsize{$g$}};
\draw[red, thick] (0.,0.) -- +(0,0.166667) ;
\draw[red, thick] (2.,0.) -- +(0,0.166667) ;
\draw[red, thick] (4.,0.) -- +(0,0.166667) ;
\draw[red, thick] (6.,-0.25) -- +(0,0.166667) ;
\draw[red, thick] (6.,0.25) -- +(0,0.166667) ;
\draw[red, thick] (8.,-0.5) -- +(0,0.166667) ;
\draw[red, thick] (8.,0.) -- +(0,0.166667) ;
\draw[red, thick] (8.,0.5) -- +(0,0.166667) ;
\draw[red, thick] (10.,0.) -- +(0,0.166667) ;
\draw[red, thick] (12.,0.) -- +(0,0.166667) ;
\end{tikzpicture}\,.
$$
(We performed a graph isomorphism to make the fusion matrices in Appendix 
\ref{sec:FusionMatrices} easier to interpret.)

\begin{defn}
Let $\cC$ be a fusion category over $\bbC$, and let $K_0(\cC)$ be its
(Grothendieck) fusion ring. A dimension function on $K_0(\cC)$ is a ring
homomorphism $K_0(\cC) \to \bbC$. (Note that a dimension function gives a
character on $K_0(\cC)\otimes_\bbZ \bbC$, which is semi-simple
\cite[1.2(a)]{MR933415} and hence a multi-matrix algebra.)

Given a dimension function $\dim: K_0(\cC) \to \bbC$, its \emph{formal codegree} is
$$
f_\dim = \sum_{X\in \Irr(\cC)} |\dim(X)|^2,
$$
where $\Irr(\cC)$ is the set of isomorphism classes of simple objects of $\cC$.
\end{defn}

\begin{remark}
\label{ex:SimultaneousEigenvectors}
As $K_0(\cC)\otimes_\bbZ \bbC$ is a multi-matrix algebra, we can represent 
a fusion ring as a collection of matrices $L_X$ for $X\in\Irr(\cC)$ acting 
in the left regular representation, in the basis corresponding to $\Irr(\cC)$.
Suppose there is a common eigenvector $v$ for all the fusion matrices $L_X$, 
i.e., $L_X v = \lambda_X v$ for some $\lambda_X\in\bbC$ for all $X\in \Irr(\cC)$.
We normalize $v$ so that $v_1=1$. Letting $\varphi:\bbC v \to \bbC$ by 
$\lambda v \mapsto \lambda$, we get a dimension function by the formula
$$
\dim(X)=\varphi(L_Xv) = \lambda_X,
$$
which is easily seen to be a ring homomorphism.
Denoting the basis for $\Irr(\cC)$ by $\{e_X\}$ with dual basis $\{e_X^*\}$, we have
$$
\lambda_X = \lambda_X v_1 = \lambda_X e_1^*(v) = e_1^*(\lambda_X v) = e_1^*(L_X v) = e_{X^*}^*(v) = v_{X^*}.
$$
In particular, if all objects of $\Irr(\cC)$ are self-dual, then $v_X=\lambda_X$ for all $X\in\Irr(\cC)$.
\end{remark}

\begin{example}
\label{ex:FPdim}
For every fusion ring, there is unique simultaneous eigenvector for the left
fusion matrices $L_X$ such that every entry is strictly positive, normalized
so that the entry corresponding to 1 is equal to 1. This gives rise to the
Frobenius-Perron dimension function $\FPdim$. The Frobenius-Perron dimensions
of the even vertices of $\cX$ are given lexicographically by depth and height
(bottom to top) by
$$
\left(1,2+\sqrt{5},6+3 \sqrt{5},10+4 \sqrt{5},9+4 \sqrt{5},4+2 \sqrt{5},9+4 \sqrt{5},6+3 \sqrt{5},2+\sqrt{5},1\right).
$$
\end{example}

\begin{remark}
There are also formal codegrees for arbitrary irreducible representations of 
the fusion ring over $\bbC$ \cite{MR2576705}, but we only need the 
one-dimensional case here.
\end{remark}

There are strong number theoretic properties of formal codegrees of fusion
categories \cite{MR2576705,1309.4822}. To eliminate the graph $\cX$, we need
the following result.

\begin{thm*}[{\cite[Corollary 2.15]{1309.4822}}]
Let $\cC$ be a spherical
fusion category. Then every formal codegree of $K_0(\cC)$ belongs to the
number field generated by the dimensions of the objects of $\cC$.

In particular, if $\cC$ is pseudo-unitary, then the formal codegrees belong to
the number field generated by the Frobenius-Perron dimensions of the objects
of $\cC$. 
\end{thm*}

\begin{example}
\label{ex:NumberFieldOfH}
Using the dimension function $\FPdim$ from Example \ref{ex:FPdim}, we see that
the number field generated by the Frobenius-Perron dimensions of the even
vertices of $\cX$ is $\bbQ(\sqrt{5})$.
\end{example}

We now find a dimension function on $K_0(\frac{1}{2}\cX)$, the fusion ring of
the even part of $\cX$ whose formal codegree does not belong to
$\bbQ(\sqrt{5})$.

We note that since all even vertices of $\cX$ are self-dual, we must have that
$K_0(\frac{1}{2}\cX)$ is commutative. Also, the vertex $g$ at depth 12 has
dimension 1, so we must have that tensoring with $g$ gives us a
$\bbZ/2\bbZ$-symmetry on the vertices.   Thus giving the fusion rules amongst
the vertices $1,\jw{2},\jw{4},P,Q,R$, all other fusion rules can be determined
by commutativity together with $g^2 = 1$. Thus we have the following lemma:

\begin{lem}
The fusion ring of the even half of $\cX$ is completely determined by the fusion 
matrices given in Appendix \ref{sec:FusionMatrices}.
\end{lem}
\begin{proof}
To determine $K_0(\frac{1}{2}\cX)$, we use the {\tt FusionAtlas} function 
{\tt FindFusionRules}, and we keep the only solution with non-negative entries.
This calculation is performed in {\tt Weeds.nb}.
\end{proof}

\begin{lem}
\label{lem:DimensionFunction}
The map $\dim: K_0(\frac{1}{2}\cX)\to \bbR$ by
$$
\left(1,\jw{2},\jw{4},P,Q,R,gP,g\jw{4},g\jw{2},g\right)\mapsto \left(1,1+\sqrt{2},1+\sqrt{2},1,0,0,-1,-1-\sqrt{2},-1-\sqrt{2},-1\right)
$$
defines a dimension function on $K_0(\frac{1}{2}\cX)$.
\end{lem}
\begin{proof}
The right hand side gives a simultaneous eigenvector for all the fusion matrices 
listed in Appendix \ref{sec:FusionMatrices}.
We are finished by the discussion in Remark \ref{ex:SimultaneousEigenvectors}.
\end{proof}

\begin{thm}
\label{thm:NoH}
There is no subfactor whose principal graph is $\cX$.
\end{thm}
\begin{proof}
We see the formal codegree of the dimension function in Lemma
\ref{lem:DimensionFunction} is given by
$f_{\dim}=16+8\sqrt{2}\notin\bbQ(\sqrt{5})$, which is the number field
generated by the Frobenius-Perron dimensions of the vertices of $\cX$ by
Example \ref{ex:NumberFieldOfH}. Thus by \cite[Corollary 2.15]{1309.4822},
$K_0(\frac{1}{2}\cX)$ is not categorifiable, and thus $\cX$ is not the
principal graph of a subfactor.
\end{proof}

\section{Ruling out 4-spokes}
\label{sec:4-spokes}
A 4-spoke, called a 4-star in \cite{1304.5907,MR3335120}, is a simply laced
graph with a single central vertex with valence 4 and with all other vertices
having valence at most 2. We denote a 4-spoke by $S(a,b,c,d)$, which has arms
with $a,b,c,d$ edges respectively connected to the central 4-valent vertex. If
a 4-spoke is component of the principal graph of a subfactor, then the distinguished vertex
marked by $\star$ must be on the end of the longest arm, cf.
\cite[Proposition 1.17]{MR3306607}.

In \cite{1304.5907}, Schou gave a complete list of 4-spokes $\Gamma$ such that
$(\Gamma,\Gamma)$ has a biunitary connection. This is a necessary condition
for $(\Gamma,\Gamma)$ to be the principal graph pair of a subfactor.

\begin{thm}[{\cite[p. 41]{1304.5907}}]
If $\Gamma=S(a,b,c,d)$ is a 4-spoke such that $(\Gamma,\Gamma)$ has a biunitary 
connection, then $\Gamma$ must be $S(1,2,2,5)$ or one of:
\begin{itemize}
\item
$S(j,j,k,k)$ for $1\leq j\leq k$
\item
$S(j,j+1,j+1,j+m)$ for $1\leq j$ and $1\leq m\leq 3$
\item
$S(j,j+1,j+2,j+m)$ for $1\leq j$ and $2\leq m\leq 4$
\item
$S(j,j+2,j+2,j+2)$ for $1\leq j$
\end{itemize}
\end{thm}

\begin{lem}
Of the 4-spokes with biunitary connections, only the following have index 
in $(5, 5 \frac{1}{4}]$:
\begin{itemize}
\item
$S(2,2,k,k)$ for $k \geq 3$ and $S(3,3,3,3)$,
\item
$S(2,3,3,3)$, $S(2,3,3,4)$, $S(2,3,3,5)$,
\item
$S(2,3,4,4)$, $S(2,3,4,5)$, $S(2,3,4,6)$,
\item
$S(2,4,4,4)$
\end{itemize}
\end{lem}
\begin{proof}
Recall that if $\Gamma$ is a subgraph of $\Gamma'$, then $\|\Gamma\|\leq
\|\Gamma'\|$. Since $\|S(3,3,3,4)\|^2 > 5.25$, this gives an upper bound on
each of the 4 families. For the lower bound of the first family, we note that
$\|S(1,1,k,k)\|<5$ for all $k\in\bbN$.  In fact, these 4-spokes were treated
in \cite{MR2993924}. For the lower bound of the second two families, we note
that the norm of $S(1,2,3,5)$ is less than 5. Finally, for the fourth family,
we note that the norm of $S(1,3,3,3)$ is exactly 5.

These computations are performed in the Mathematica notebook {\tt{4Spokes.nb}}.
\end{proof}

\begin{lem}
Of the 4-spokes with biunitary connections and index in $(5, 5 \frac{1}{4}]$,
only $S(3,3,3,3)$, and $S(2,4,4,4)$ have norm squared which is a cyclotomic
integer (and indeed, these both have index $3+\sqrt{5}$).
\end{lem}
\begin{proof}
The only difficult case is the family $S(2,2,k,k)$, but this is readily treated by the same argument as in \cite[Section 4]{MR2993924}. We see these graphs have the same norm $c_k$ as the graphs
$$
H_k = 
\begin{tikzpicture}[baseline]
\draw[fill] (0,0) circle (0.05);
\draw (0.,0.) -- (1.,0.);
\draw[fill] (1.,0.) circle (0.05);
\draw[dashed] (1,0) -- (2,0);
\draw[fill] (2.,0.) circle (0.05);
\draw [decorate,decoration={brace,amplitude=5pt},yshift=-5pt]
   (2,0)  -- (0,0) 
   node [black,midway,below=4pt,xshift=-2pt] {\footnotesize $k-1$ edges};
\draw (2.,0.) -- (3.,-0.5);
\draw (2.,0.) -- (3.,0.5);
\draw[fill] (3.,-0.5) circle (0.05);
\draw[fill] (3.,0.5) circle (0.05);
\draw (3.,-0.5) -- (4.,0.);
\draw (3.,0.5) -- (4.,0.);
\draw[fill] (4.,0.) circle (0.05);
\draw (4,0) -- (5,0);
\draw[fill] (5.,0.) circle (0.05);
\end{tikzpicture}\,.
$$
The field $\mathbb{Q}(c_k^2)$ is not cyclotomic for any $k \geq 3$; the
argument in Section \ref{sec:vines} shows that the adjacency matrix of $H_k$
has a multiplicity free eigenvalue $\lambda_k$ with $\bbQ(\lambda_k^2)$ not
cyclotomic for all $k \geq 3$. Now the characteristic polynomials $P_k$ for
the adjacency matrix of $H_k$ satisfy 
$$
P_k(t + t^{-1})\left(t-t^{-1}\right) 
= 
t^k A(t) - t^{-k} A(t^{-1})
$$ 
with $A(t) = t^7-t^5-4 t^3-3 t-t^{-1}$.
This polynomial has just two real roots with magnitude grater than 1, namely
the square roots of the real root of $\mu^3-2\mu^2-2\mu-1$, and so by Remark
10.1.7 of \cite{MR2786219}, the polynomial $P_k(x)$ is $S(x^2)$ times a
product of cyclotomic polynomials where $S$ is a Salem polynomial. Thus
$\lambda_k$ must be Galois conjugate to $c_k$, giving the result.

The interested reader can view this calculation in the {\tt{ Salem 4-spoke}}
section of the Mathematica notebook {\tt{4Spokes.nb}}.
\end{proof}

Finally, we need to analyze the possible dual data for $S(3,3,3,3)$ and
$S(2,4,4,4)$. The function {\tt FindGraphPartners} in the {\tt FusionAtlas}
package computes all possible graphs pairs with dual data containing a
specified single graph without dual data (by applying the graph enumerator one
depth at a time, discarding all branches which do not agree with the specified
graph up to the current depth). We obtain

\begin{lem}
\label{lem:FindGraphPartners}
A subfactor principal graph containing an $S(3,3,3,3)$ or $S(2,4,4,4)$ must be amongst
\begin{align*}
& \left(\bigraph{bwd1v1v1v1p1p1v1x0x0p0x1x0p0x0x1v1x0x0p0x1x0p0x0x1duals1v1v2x1x3v2x1x3}, \bigraph{bwd1v1v1v1p1p1v1x0x0p0x1x0p0x0x1v1x0x0p0x1x0p0x0x1duals1v1v2x1x3v2x1x3} \right) \\
& \left(\bigraph{bwd1v1v1v1p1p1v1x0x0p0x1x0p0x0x1v1x0x0p0x1x0p0x0x1duals1v1v1x2x3v1x2x3}, \bigraph{bwd1v1v1v1p1p1v1x0x0p0x1x0p0x0x1v1x0x0p0x1x0p0x0x1duals1v1v1x2x3v1x2x3} \right) \\
& \left(\bigraph{bwd1v1v1v1p1p1v1x0x0p0x1x0p0x1x0v1x0x0duals1v1v1x2x3v1}, \bigraph{bwd1v1v1v1p1p1v1x0x0p0x1x0p0x0x1v1x0x0p0x1x0p0x0x1duals1v1v1x2x3v1x3x2} \right) \\
& \left(\bigraph{bwd1v1v1v1v1p1p1v1x0x0p0x1x0p0x0x1v0x1x0p0x0x1v1x0p0x1duals1v1v1v1x3x2v2x1}, \bigraph{bwd1v1v1v1v1p1p1v1x0x0p0x1x0p0x0x1v0x1x0p0x0x1v1x0p0x1duals1v1v1v1x3x2v2x1} \right) \\
& \left(\bigraph{bwd1v1v1v1v1p1p1v1x0x0p0x1x0p0x0x1v1x0x0p0x1x0v1x0p0x1duals1v1v1v1x2x3v1x2}, \bigraph{bwd1v1v1v1v1p1p1v1x0x0p0x1x0p0x0x1v1x0x0p0x1x0v1x0p0x1duals1v1v1v1x2x3v1x2} \right).
\end{align*}
\end{lem}
Here we don't care about the third case, as one of the graphs is not a 4-spoke
(in fact, this graph is the principal graph of the $3^{\bbZ/4\bbZ}$ subfactor,
which appears in Section \ref{sec:vines}). We easily rule out the fifth case,
as the univalent vertices would form a group of invertible bimodules of order
three, with all objects involutions. The first case is ruled out by the
following.

\begin{lem}
The first graph in Lemma \ref{lem:FindGraphPartners} cannot be the principal graph of a subfactor, because it does not satisfy the conditions
of
\cite [Theorem 4.5]{MR3311757}.
\end{lem}
\begin{proof}
Let $P$ be the self-dual vertex at depth 4 of $\Gamma_+$, and let $P'$ be the
vertex at depth 5 of $\Gamma_+$ connected to $P$. Then $\overline{P'}$ is only
connected to the self-dual vertex at depth 4 of $\Gamma_-$. By \cite [Theorem
4.5]{MR3311757}, we must have that $\delta\leq 2$, a contradiction.
\end{proof}

Summarizing this section, we have
\begin{thm}
The only subfactor planar algebras with principal graphs both 4-spokes, and
index  in the interval $(5,5\frac{1}{4}]$, are Izumi's $3^{\bbZ/2\bbZ\times
\bbZ/2\bbZ}$ planar algebra \cite{IzumiUnpublished,MR3314808} and the 4442 planar algebra
\cite{MR3314808,1406.3401}. In each case, the principal graph is realized by
just a single planar algebra.
\end{thm}

This finishes our treatment of case \ref{case:4-spoke} from Theorem \ref{thm:Enumerate}, and gives two of the subfactor standard invariants described in Theorem \ref{thm:Main}. 

\section{Cyclotomicity of vines}
\label{sec:vines}

We now perform the analysis of \cite{MR2786219,MR2902286} to determine which
translations of the vines enumerated in Theorem \ref{thm:Enumerate} may be principal graphs of subfactors.

We say a graph is \emph{cyclotomic} if for every multiplicity free eigenvalue
$\lambda$ of the adjacency matrix, the quantity $\lambda^2$ is a cyclotomic
integer. (This is stronger than the requirement that the square of the graph
norm is a cyclotomic integer, and is necessary for a graph to be the principal
graph of a subfactor, by
\cite{MR1266785,MR2183279}.)

For each vine $(\Gamma_+, \Gamma_-)$, we have a bound $N(\Gamma)=\min \{ N(\Gamma_+), N(\Gamma_-) \}$, where 
$N(\Gamma_\pm)$ is a bound on the total number of vertices a translation of
$\Gamma_\pm$ may have and still potentially be cyclotomic.  It is calculated
according to the results of \cite{MR2786219}, using the algorithm described in
\cite{MR2902286}.\footnote{Since the publication of \cite{MR2902286}, we
discovered a potentially significant error in the code used in that paper, in
particular in the {\tt BoundR1} function. While fixing this error was
essential for the following calculations, fortuitously it did not change any
of the claims of the original paper.} 

We then look at each of the finitely many translates remaining, and check whether each has cyclotomic index.
There are in every case very few translations which may have cyclotomic
index, and all are significantly smaller than the bounds given by $N(\Gamma)$. We rule
out all the other translations by explicitly finding a witness prime, modulo
which the minimal polynomial factors into irreducible factors with different
degrees. If we fail to find such a prime amongst the first 500 primes we say
that the index may be cyclotomic. Although in this case we don't certify
cyclotomicity, in practice these exceptions are always cyclotomic (and are
either ruled out by easy obstructions or realized as principal graphs of
subfactors).

Certain optimizations are necessary to efficiently find all the minimal
polynomials up to the bound. Observe that the minimal polynomial of the index
is a factor of the characteristic polynomial of $A^t A$, with $A$ the
adjacency matrix of the graph. In fact, in practice we see that the set of
irreducible factors of the quotient of this characteristic polynomial by the
minimal polynomial of the index is periodic in the translation, although we do
not know a proof. Nevertheless, this gives an efficient practical method for finding the
minimal polynomials; we compute the first few minimal polynomials directly,
observe the factors appearing in the quotient, and then for the tail we merely
remove these factors from the easily computed characteristic polynomial, and
verify the irreducibility of what remains.

Finally, we take each of the cyclotomic translates and run a few
simple tests on the graphs, allowing us to rule out most of them as principal
graphs of subfactors. All these computations are detailed in Appendix \ref{appendix:cyclotomicity}, and give the
following result.

\begin{thm}
\label{thm:survivors}
The only possible principal graphs arising from the vines enumerated in Theorem \ref{thm:Enumerate},
with index in the interval $(4, 21/4]$ are the following.
{
\setlength\columnsep{-2cm}
\begin{multicols}{2}
\renewcommand{\bigraph}[1]{{\hspace{-3pt}\begin{array}{c}%
  \raisebox{-2.5pt}{\includegraphics[height=4mm]{\pathtographs \hashlookup{#1}}}%
\end{array}\hspace{-3pt}}}
\begin{enumerate}[label=(\arabic*)]
\item \label{bwd1v1v1p1v1x0p0x1p0x1v1x1x0p0x0x1duals1v1v3x2x1,bwd1v1v1p1v1x0p0x1p0x1v1x1x0p0x0x1duals1v1v3x2x1} \(\left(\bigraph{bwd1v1v1p1v1x0p0x1p0x1v1x1x0p0x0x1duals1v1v3x2x1},\bigraph{bwd1v1v1p1v1x0p0x1p0x1v1x1x0p0x0x1duals1v1v3x2x1}\right) \)
\item \label{bwd1v1v1p1v1x0p1x0p0x1p0x1v0x1x1x0duals1v1v1x3x2x4,bwd1v1v1p1v1x1v1v1duals1v1v1v1} \(\left(\bigraph{bwd1v1v1p1v1x0p1x0p0x1p0x1v0x1x1x0duals1v1v1x3x2x4},\bigraph{bwd1v1v1p1v1x1v1v1duals1v1v1v1}\right) \)
\item \label{bwd1v1v1p1v1x0p1x0p0x1p0x1v0x1x1x0duals1v1v4x2x3x1,bwd1v1v1p1v1x1v1v1duals1v1v1v1} \(\left(\bigraph{bwd1v1v1p1v1x0p1x0p0x1p0x1v0x1x1x0duals1v1v4x2x3x1},\bigraph{bwd1v1v1p1v1x1v1v1duals1v1v1v1}\right) \)
\item \label{bwd1v1v1p1v1x1v1v1duals1v1v1v1,bwd1v1v1p1v1x1v1v1duals1v1v1v1} \(\left(\bigraph{bwd1v1v1p1v1x1v1v1duals1v1v1v1},\bigraph{bwd1v1v1p1v1x1v1v1duals1v1v1v1}\right) \)
\item \label{bwd1v1v1v1p1v1x0p0x1v1x0p0x1duals1v1v1x2v2x1,bwd1v1v1v1p1v1x0p1x0duals1v1v1x2} \(\left(\bigraph{bwd1v1v1v1p1v1x0p0x1v1x0p0x1duals1v1v1x2v2x1},\bigraph{bwd1v1v1v1p1v1x0p1x0duals1v1v1x2}\right) \)
\item \label{bwd1v1v1v1p1p1v0x1x0p0x1x0duals1v1v1x2x3,bwd1v1v1v1p1p1v1x0x0p1x0x0duals1v1v1x2x3} \(\left(\bigraph{bwd1v1v1v1p1p1v0x1x0p0x1x0duals1v1v1x2x3},\bigraph{bwd1v1v1v1p1p1v1x0x0p1x0x0duals1v1v1x2x3}\right) \)
\item \label{bwd1v1v1v1p1p1v0x1x0p0x1x0duals1v1v1x2x3,bwd1v1v1v1p1p1v1x0x0p0x0x1v1x0p0x1duals1v1v1x2x3v2x1} \(\left(\bigraph{bwd1v1v1v1p1p1v0x1x0p0x1x0duals1v1v1x2x3},\bigraph{bwd1v1v1v1p1p1v1x0x0p0x0x1v1x0p0x1duals1v1v1x2x3v2x1}\right) \)
\item \label{bwd1v1v1v1p1p1v1x0x0p1x0x0duals1v1v1x3x2,bwd1v1v1v1p1p1v1x0x0p1x0x0duals1v1v1x3x2} \(\left(\bigraph{bwd1v1v1v1p1p1v1x0x0p1x0x0duals1v1v1x3x2},\bigraph{bwd1v1v1v1p1p1v1x0x0p1x0x0duals1v1v1x3x2}\right) \)
\item \label{bwd1v1v1v1p1p1v0x1x0p0x0x1v1x0p0x1duals1v1v1x3x2v1x2,bwd1v1v1v1p1p1v1x0x0p1x0x0duals1v1v1x3x2} \(\left(\bigraph{bwd1v1v1v1p1p1v0x1x0p0x0x1v1x0p0x1duals1v1v1x3x2v1x2},\bigraph{bwd1v1v1v1p1p1v1x0x0p1x0x0duals1v1v1x3x2}\right) \)
\item \label{bwd1v1v1v1p1p1v0x1x0p0x1x0p0x0x1v0x0x1duals1v1v1x2x3v1,bwd1v1v1v1p1p1v1x0x0p0x1x0p0x0x1v1x0x0p0x1x0p0x0x1duals1v1v1x2x3v2x1x3} \(\left(\bigraph{bwd1v1v1v1p1p1v0x1x0p0x1x0p0x0x1v0x0x1duals1v1v1x2x3v1},\bigraph{bwd1v1v1v1p1p1v1x0x0p0x1x0p0x0x1v1x0x0p0x1x0p0x0x1duals1v1v1x2x3v2x1x3}\right) \)
\item \label{bwd1v1v1v1v1p1v1x1p0x1duals1v1v1v1x2,bwd1v1v1v1v1p1v1x1p0x1duals1v1v1v1x2} \(\left(\bigraph{bwd1v1v1v1v1p1v1x1p0x1duals1v1v1v1x2},\bigraph{bwd1v1v1v1v1p1v1x1p0x1duals1v1v1v1x2}\right) \)
\item \label{bwd1v1v1v1v1v1p1v1x0p0x1v1x0p0x1p0x1v1x0x0v1duals1v1v1v1x2v3x2x1v1,bwd1v1v1v1v1v1p1v1x0p1x0v0x1v1duals1v1v1v1x2v1} \(\left(\bigraph{bwd1v1v1v1v1v1p1v1x0p0x1v1x0p0x1p0x1v1x0x0v1duals1v1v1v1x2v3x2x1v1},\bigraph{bwd1v1v1v1v1v1p1v1x0p1x0v0x1v1duals1v1v1v1x2v1}\right) \)
\item \label{bwd1v1v1v1v1v1p1p1v0x1x0p0x0x1v1x1duals1v1v1v1x3x2v1,bwd1v1v1v1v1v1p1p1v1x0x0p1x0x0v1x0p0x1duals1v1v1v1x3x2v2x1} \(\left(\bigraph{bwd1v1v1v1v1v1p1p1v0x1x0p0x0x1v1x1duals1v1v1v1x3x2v1},\bigraph{bwd1v1v1v1v1v1p1p1v1x0x0p1x0x0v1x0p0x1duals1v1v1v1x3x2v2x1}\right) \)
\item \label{bwd1v1v1v1v1v1p1p1v0x1x0p0x1x0v1x0p0x1duals1v1v1v1x2x3v1x2,bwd1v1v1v1v1v1p1p1v1x0x0p0x0x1v1x1duals1v1v1v1x2x3v1} \(\left(\bigraph{bwd1v1v1v1v1v1p1p1v0x1x0p0x1x0v1x0p0x1duals1v1v1v1x2x3v1x2},\bigraph{bwd1v1v1v1v1v1p1p1v1x0x0p0x0x1v1x1duals1v1v1v1x2x3v1}\right) \)
\item \label{bwd1v1v1v1v1v1v1v1p1v1x0p0x1v1x0p0x1duals1v1v1v1v1x2v2x1,bwd1v1v1v1v1v1v1v1p1v1x0p1x0duals1v1v1v1v1x2} \(\left(\bigraph{bwd1v1v1v1v1v1v1v1p1v1x0p0x1v1x0p0x1duals1v1v1v1v1x2v2x1},\bigraph{bwd1v1v1v1v1v1v1v1p1v1x0p1x0duals1v1v1v1v1x2}\right) \)
\end{enumerate}
\end{multicols}
}
\end{thm}

We note that
for some of these principal graphs there is already a complete classification of subfactors realizing them. We summarize these here. The notation `2d' in the `$\#$ of
subfactors' column indicates that there are 2 non-isomorphic subfactors which
are dual to each other.

\begin{center}
\begin{tabular}{c|c|c|c}
principal graph
&
name
&
$\#$ of subfactors
&
citation
\\\hline
\ref{bwd1v1v1p1v1x0p0x1p0x1v1x1x0p0x0x1duals1v1v3x2x1,bwd1v1v1p1v1x0p0x1p0x1v1x1x0p0x0x1duals1v1v3x2x1}
&
$U_{\exp(2\pi i / 14)}(\mathfrak{su}_3)$
&
1
&
\cite{MR936086,MR1470857,MR3254427}
\\
\ref{bwd1v1v1p1v1x0p1x0p0x1p0x1v0x1x1x0duals1v1v1x3x2x4,bwd1v1v1p1v1x1v1v1duals1v1v1v1}
&
2D2
&
2d
&
\cite{1406.3401}
\\
\ref{bwd1v1v1v1p1v1x0p0x1v1x0p0x1duals1v1v1x2v2x1,bwd1v1v1v1p1v1x0p1x0duals1v1v1x2}
&
the Haagerup
&
2d
&
\cite{MR1686551}
\\
\ref{bwd1v1v1v1p1p1v0x1x0p0x1x0duals1v1v1x2x3,bwd1v1v1v1p1p1v1x0x0p0x0x1v1x0p0x1duals1v1v1x2x3v2x1}
&
$A_4 \subset A_5$
&
2d
&
\cite{MR3335120}
\\
\ref{bwd1v1v1v1p1p1v0x1x0p0x1x0p0x0x1v0x0x1duals1v1v1x2x3v1,bwd1v1v1v1p1p1v1x0x0p0x1x0p0x0x1v1x0x0p0x1x0p0x0x1duals1v1v1x2x3v2x1x3}
&
$3^{\bbZ/4\bbZ}$
&
2d
&
\cite{IzumiUnpublished,1308.5197}
\\
\ref{bwd1v1v1v1v1v1p1v1x0p0x1v1x0p0x1p0x1v1x0x0v1duals1v1v1v1x2v3x2x1v1,bwd1v1v1v1v1v1p1v1x0p1x0v0x1v1duals1v1v1v1x2v1}
&
the Asaeda-Haagerup
&
2d
&
\cite{MR1686551}
\\
\ref{bwd1v1v1v1v1v1v1v1p1v1x0p0x1v1x0p0x1duals1v1v1v1v1x2v2x1,bwd1v1v1v1v1v1v1v1p1v1x0p1x0duals1v1v1v1v1x2}
&
the extended Haagerup
&
2d
&
\cite{MR2979509}
\end{tabular}
\end{center}

Further, some of the graphs in Theorem \ref{thm:survivors} have already been
ruled out as principal graphs of subfactors in other papers.

\begin{thm*}[\cite{1406.3401}]
There is a unique subfactor with principal graph $\Gamma_+ = \bigraph{bwd1v1v1p1v1x1v1v1duals1v1v1v1}$, and it must have dual graph $\Gamma_- = \bigraph{bwd1v1v1p1v1x0p1x0p0x1p0x1v0x1x1x0duals1v1v1x3x2x4}$. 
Thus the graph pairs
numbered \ref{bwd1v1v1p1v1x0p1x0p0x1p0x1v0x1x1x0duals1v1v4x2x3x1,bwd1v1v1p1v1x1v1v1duals1v1v1v1} and \ref{bwd1v1v1p1v1x1v1v1duals1v1v1v1,bwd1v1v1p1v1x1v1v1duals1v1v1v1} above, are not principal graphs of subfactors.
\end{thm*}

\begin{thm}[{\cite[Lemma 3.10, Lemma 3.11, Proposition 3.12]
{MR3335120}}]
There are no subfactors with the principal graphs numbered
\ref{bwd1v1v1v1p1p1v0x1x0p0x1x0duals1v1v1x2x3,bwd1v1v1v1p1p1v1x0x0p1x0x0duals1v1v1x2x3},
\ref{bwd1v1v1v1p1p1v1x0x0p1x0x0duals1v1v1x3x2,bwd1v1v1v1p1p1v1x0x0p1x0x0duals1v1v1x3x2}, or
\ref{bwd1v1v1v1p1p1v0x1x0p0x0x1v1x0p0x1duals1v1v1x3x2v1x2,bwd1v1v1v1p1p1v1x0x0p1x0x0duals1v1v1x3x2}.
\end{thm}

Finally a calculation based on checking associativity of the fusion ring
concludes our treatment of these vines.
\begin{lem}
The remaining three graphs, numbered
\ref{bwd1v1v1v1v1p1v1x1p0x1duals1v1v1v1x2,bwd1v1v1v1v1p1v1x1p0x1duals1v1v1v1x2},
\ref{bwd1v1v1v1v1v1p1p1v0x1x0p0x0x1v1x1duals1v1v1v1x3x2v1,bwd1v1v1v1v1v1p1p1v1x0x0p1x0x0v1x0p0x1duals1v1v1v1x3x2v2x1}, and \ref{bwd1v1v1v1v1v1p1p1v0x1x0p0x1x0v1x0p0x1duals1v1v1v1x2x3v1x2,bwd1v1v1v1v1v1p1p1v1x0x0p0x0x1v1x1duals1v1v1v1x2x3v1} above, cannot be the principal graphs of subfactors, as there are no compatible fusion rings.
\end{lem}

\appendix
\section{Appendices}

\subsection{Cyclotomicity bounds}
\label{appendix:cyclotomicity}
We now display the cyclotomicity bounds for the vines discussed in Section
\ref{sec:vines}, along with the potentially cyclotomic translates and the
results of running simple tests on these. This table, and the computations underlying it, are constructed in the {\tt Mathematica} notebook {\tt processing-vines.nb} available with the {\tt arXiv} sources of this article.

The first column shows the vine $(\Gamma_+, \Gamma_-)$.  
In the second column, we give the upper bound $\max N(\Gamma_\pm)$ on the number of vertices appearing in a cyclotomic translate. The third column, named `CT' for
`cyclotomic translates', shows those translations up to that bound which may
be cyclotomic.  The fourth column, named `Obstr.' for `obstructions',
indicates if a simple obstruction can rule out each of the potentially
cyclotomic translates.

These obstructions are labelled as follows, with the most elementary ones coming first:
\begin{itemize}
\item[(a)] Some bimodule has dimension less than 1. (The index of a subfactor is bounded below by 1.)
\item[(b)] Some bimodule has a dimension which is not an algebraic integer. (The index of a finite depth subfactor is an eigenvalue of an integer matrix.)
\item[(c)] Some bimodule with dimension less than 2 has dimension not of the form $2 \cos(\pi / n)$ for $n \geq 3$, which is impossible by \cite{MR0696688}.
\item[(d)] Some low weight space would have negative dimension, as computed according to \cite[p. 33]{MR1929335}.
\item[(e)] The global dimension of the even part (that is, the sum of the squares of the dimensions of vertices at even depths) is not an Ostrik $d$-number, which is a necessary condition by \cite{MR2576705}.
\end{itemize}
Any potentially cyclotomic translate which is not ruled out by one of 
these obstructions is marked with a `?'.

In the above, we use the fact that the dimension of any bimodule is the square
root of the index of the associated reduced subfactor \cite{MR934296}, and if
the principal graph is finite depth then all the associated reduced subfactors
are finite depth.

\begin{longtable}{llll}
 vine& $N(\Gamma)$& CT& Obstr. \\ \hline
$\left(\bigraph{bwd1v1v1p1p1duals1v1}, \bigraph{bwd1v1v1p1p1duals1v1} \right)$ & 76 & \{0\} & \{\text{c}\} \\
$\left(\bigraph{bwd1v1v1v1p1v1x0p0x1v1x0p0x1duals1v1v1x2v2x1}, \bigraph{bwd1v1v1v1p1v1x0p1x0duals1v1v1x2} \right)$ & 87 & \{0,4\} & \{?,?\} \\
$\left(\bigraph{bwd1v1v1v1p1p1duals1v1v1x2x3}, \bigraph{bwd1v1v1v1p1p1duals1v1v1x2x3} \right)$ & 76 & \{0\} & \{\text{d}\} \\
$\left(\bigraph{bwd1v1v1v1p1p1duals1v1v1x2x3}, \bigraph{bwd1v1v1v1p1p1duals1v1v1x3x2} \right)$ & 76 & \{0\} & \{\text{d}\} \\
$\left(\bigraph{bwd1v1v1v1p1p1duals1v1v1x3x2}, \bigraph{bwd1v1v1v1p1p1duals1v1v1x3x2} \right)$ & 76 & \{0\} & \{\text{d}\} \\
$\left(\bigraph{bwd1v1v1v1p1v1x0p0x1v1x0p0x1p0x1v1x0x0v1duals1v1v1x2v3x2x1v1}, \bigraph{bwd1v1v1v1p1v1x0p1x0v0x1v1duals1v1v1x2v1} \right)$ & 96 & \{2\} & \{?\} \\
$\left(\bigraph{bwd1v1v1v1p1v1x0p0x1v1x1duals1v1v1x2v1}, \bigraph{bwd1v1v1v1p1v1x0p1x0v1x0p0x1duals1v1v1x2v1x2} \right)$ & 70 & \{0,2\} & \{\text{e},\text{e}\} \\
$\left(\bigraph{bwd1v1v1p1v0x1p1x0p0x1v1x0x0p0x1x0v0x1duals1v1v1x3x2v1}, \bigraph{bwd1v1v1p1v1x0p0x1p0x1v1x0x0p0x1x0v1x0duals1v1v3x2x1v1} \right)$ & 123 & \{0,2\} & \{\text{a},\text{e}\} \\
$\left(\bigraph{bwd1v1v1v1p1v1x0p1x0p0x1v1x0x0p0x0x1v0x1duals1v1v1x2v2x1}, \bigraph{bwd1v1v1v1p1v1x0p0x1p1x0v1x0x0p0x1x0v0x1duals1v1v1x2v2x1} \right)$ & 123 & \{\} & \{\} \\
$\left(\bigraph{bwd1v1v1p1v1x1duals1v1v1}, \bigraph{bwd1v1v1p1v1x1duals1v1v1} \right)$ & 75 & \{0\} & \{\text{b}\} \\
$\left(\bigraph{bwd1v1v1p1v1x0p1x0p0x1p0x1duals1v1v3x2x1x4}, \bigraph{bwd1v1v1p1v1x1duals1v1v1} \right)$ & 75 & \{0\} & \{\text{b}\} \\
$\left(\bigraph{bwd1v1v1v1p1v1x1duals1v1v1x2}, \bigraph{bwd1v1v1v1p1v1x1duals1v1v1x2} \right)$ & 75 & \{0\} & \{\text{e}\} \\
$\left(\bigraph{bwd1v1v1v1p1v1x1duals1v1v2x1}, \bigraph{bwd1v1v1v1p1v1x1duals1v1v2x1} \right)$ & 75 & \{0\} & \{\text{e}\} \\
$\left(\bigraph{bwd1v1v1v1p1v1x0p0x1v1x1p0x1duals1v1v1x2v1x2}, \bigraph{bwd1v1v1v1p1v1x0p1x0v1x0p0x1p0x1duals1v1v1x2v1x2x3} \right)$ & 122 & \{0\} & \{\text{e}\} \\
$\left(\bigraph{bwd1v1v1v1p1v1x0p0x1v1x1p0x1duals1v1v1x2v1x2}, \bigraph{bwd1v1v1v1p1v1x0p1x0v1x0p0x1p0x1duals1v1v1x2v1x3x2} \right)$ & 122 & \{0\} & \{\text{e}\} \\
$\left(\bigraph{bwd1v1v1v1p1v1x0p1x0p0x1v1x0x0p0x0x1p0x1x0p0x0x1duals1v1v1x2v2x1x4x3}, \bigraph{bwd1v1v1v1p1v1x0p1x0p1x0duals1v1v1x2} \right)$ & 101 & \{\} & \{\} \\
$\left(\bigraph{bwd1v1v1p1p1v0x0x1p0x0x1duals1v1v1x2}, \bigraph{bwd1v1v1p1p1v0x0x1p0x0x1duals1v1v1x2} \right)$ & 119 & \{\} & \{\} \\
$\left(\bigraph{bwd1v1v1p1p1v0x0x1p0x0x1duals1v1v1x2}, \bigraph{bwd1v1v1p1p1v0x0x1p0x0x1duals1v1v2x1} \right)$ & 119 & \{\} & \{\} \\
$\left(\bigraph{bwd1v1v1p1p1v0x0x1p0x0x1duals1v1v2x1}, \bigraph{bwd1v1v1p1p1v0x0x1p0x0x1duals1v1v2x1} \right)$ & 119 & \{\} & \{\} \\
$\left(\bigraph{bwd1v1v1v1p1v1x0p0x1v1x1v1v1duals1v1v1x2v1v1}, \bigraph{bwd1v1v1v1p1v1x0p1x0v1x0p0x1v1x1duals1v1v1x2v1x2} \right)$ & 104 & \{\} & \{\} \\
$\left(\bigraph{bwd1v1v1v1p1v1x0p0x1v1x0p0x1p0x1p0x1v1x0x0x0p1x0x0x0v1x0p0x1duals1v1v1x2v3x2x1x4v1x2}, \bigraph{bwd1v1v1v1p1v1x0p1x0v0x1p0x1v1x0p0x1duals1v1v1x2v1x2} \right)$ & 121 & \{\} & \{\} \\
$\left(\bigraph{bwd1v1v1v1p1v1x0p0x1v1x0p0x1p0x1p0x1v1x0x0x0p1x0x0x0v1x0p0x1duals1v1v1x2v4x3x2x1v1x2}, \bigraph{bwd1v1v1v1p1v1x0p1x0v0x1p0x1v1x0p0x1duals1v1v1x2v1x2} \right)$ & 121 & \{\} & \{\} \\
$\left(\bigraph{bwd1v1v1v1p1v1x0p0x1v0x1p1x0p0x1p0x1v0x1x0x0p0x1x0x0v1x0p0x1duals1v1v1x2v1x4x3x2v2x1}, \bigraph{bwd1v1v1v1p1v1x0p1x0v0x1p0x1v1x0p0x1duals1v1v1x2v2x1} \right)$ & 121 & \{\} & \{\} \\
$\left(\bigraph{bwd1v1v1v1p1v1x0p0x1v0x1p1x0p0x1p0x1v0x1x0x0p0x1x0x0v1x0p0x1duals1v1v1x2v4x3x2x1v2x1}, \bigraph{bwd1v1v1v1p1v1x0p1x0v0x1p0x1v1x0p0x1duals1v1v1x2v2x1} \right)$ & 121 & \{\} & \{\} \\
$\left(\bigraph{bwd1v1v1v1p1v1x0p0x1v0x1p1x0p0x1v1x0x0p1x0x0p0x1x0v1x0x0p0x1x0p0x0x1p0x0x1p0x0x1duals1v1v1x2v1x3x2v3x5x1x4x2}, \bigraph{bwd1v1v1v1p1v1x0p1x0v0x1v1p1p1duals1v1v1x2v1} \right)$ & 225 & \{\} & \{\} \\
$\left(\bigraph{bwd1v1v1v1p1v1x0p0x1v1x1p0x1v0x1duals1v1v1x2v1x2}, \bigraph{bwd1v1v1v1p1v1x0p1x0v1x0p0x1p0x1v0x1x0duals1v1v1x2v1x2x3} \right)$ & 131 & \{\} & \{\} \\
$\left(\bigraph{bwd1v1v1v1p1p1v0x1x0p0x1x0duals1v1v1x2x3}, \bigraph{bwd1v1v1v1p1p1v1x0x0p1x0x0duals1v1v1x2x3} \right)$ & 119 & \{0\} & \{?\} \\
$\left(\bigraph{bwd1v1v1v1p1p1v0x1x0p0x1x0duals1v1v1x2x3}, \bigraph{bwd1v1v1v1p1p1v1x0x0p0x0x1v1x0p0x1duals1v1v1x2x3v2x1} \right)$ & 119 & \{0\} & \{?\} \\
$\left(\bigraph{bwd1v1v1v1p1p1v1x0x0p1x0x0duals1v1v1x3x2}, \bigraph{bwd1v1v1v1p1p1v1x0x0p1x0x0duals1v1v1x3x2} \right)$ & 119 & \{0\} & \{?\} \\
$\left(\bigraph{bwd1v1v1v1p1p1v0x1x0p0x0x1v1x0p0x1duals1v1v1x3x2v1x2}, \bigraph{bwd1v1v1v1p1p1v1x0x0p1x0x0duals1v1v1x3x2} \right)$ & 119 & \{0\} & \{?\} \\
$\left(\bigraph{bwd1v1v1p1v0x1p1x0p0x1v1x0x0p1x0x0p0x1x0v0x1x0p0x0x1p0x0x1v0x0x1duals1v1v1x3x2v3x2x1}, \bigraph{bwd1v1v1p1v1x0p0x1p0x1v1x0x0p0x1x0p0x1x0v1x0x0p1x0x0p0x1x0v0x1x0duals1v1v3x2x1v1x3x2} \right)$ & 200 & \{0,2\} & \{\text{a},\text{e}\} \\
$\left(\bigraph{bwd1v1v1v1p1v1x0p1x0p0x1v1x0x0p0x0x1p1x0x0v0x1x0p0x1x0p0x0x1v1x0x0duals1v1v1x2v2x1x3v1}, \bigraph{bwd1v1v1v1p1v1x0p0x1p1x0v1x0x0p1x0x0p0x1x0v1x0x0p0x0x1p0x0x1v0x0x1duals1v1v1x2v1x3x2v1} \right)$ & 200 & \{\} & \{\} \\
$\left(\bigraph{bwd1v1v1p1v1x0p0x1p0x1v1x1x0p0x0x1duals1v1v3x2x1}, \bigraph{bwd1v1v1p1v1x0p0x1p0x1v1x1x0p0x0x1duals1v1v3x2x1} \right)$ & 113 & \{0\} & \{?\} \\
$\left(\bigraph{bwd1v1v1v1p1v1x0p0x1v0x1p1x0p0x1v1x0x0p1x0x0p0x1x0v1x0x1p0x1x0p0x0x1duals1v1v1x2v1x3x2v1x3x2}, \bigraph{bwd1v1v1v1p1v1x0p1x0v0x1v1p1p1v1x0x0duals1v1v1x2v1v1} \right)$ & 216 & \{\} & \{\} \\
$\left(\bigraph{bwd1v1v1v1p1v1x0p1x0p0x1v1x0x1p0x1x0duals1v1v1x2v1x2}, \bigraph{bwd1v1v1v1p1v1x0p0x1p1x0v1x1x0p0x0x1duals1v1v1x2v1x2} \right)$ & 113 & \{\} & \{\} \\
$\left(\bigraph{bwd1v1v1v1p1v1x0p0x1v1x1v1v1p1duals1v1v1x2v1v1x2}, \bigraph{bwd1v1v1v1p1v1x0p1x0v1x0p0x1v1x1v1duals1v1v1x2v1x2v1} \right)$ & 120 & \{\} & \{\} \\
$\left(\bigraph{bwd1v1v1v1p1v1x0p0x1v1x1v1v1p1duals1v1v1x2v1v2x1}, \bigraph{bwd1v1v1v1p1v1x0p1x0v1x0p0x1v1x1v1duals1v1v1x2v1x2v1} \right)$ & 120 & \{\} & \{\} \\
$\left(\bigraph{bwd1v1v1p1v0x1p1x0p0x1v1x0x0p1x0x0p0x1x0v0x1x1duals1v1v1x3x2v1}, \bigraph{bwd1v1v1p1v1x0p0x1p0x1v1x0x0p0x1x0p0x1x0v1x1x0duals1v1v3x2x1v1} \right)$ & 147 & \{\} & \{\} \\
$\left(\bigraph{bwd1v1v1v1p1v1x0p0x1v0x1p1x0p0x1v1x0x0p1x0x0p0x1x0v1x0x0p0x0x1p1x0x0p0x0x1p0x1x0p0x0x1v0x0x0x1x0x0v1duals1v1v1x2v1x3x2v1x2x4x3x6x5v1}, \bigraph{bwd1v1v1v1p1v1x0p1x0v0x1v1p1p1v1x0x0v1duals1v1v1x2v1v1} \right)$ & 261 & \{\} & \{\} \\
$\left(\bigraph{bwd1v1v1v1p1v1x0p0x1v1x0p1x1p0x1duals1v1v1x2v1x2x3}, \bigraph{bwd1v1v1v1p1v1x0p1x0v1x0p1x0p0x1p0x1duals1v1v1x2v1x2x3x4} \right)$ & 111 & \{\} & \{\} \\
$\left(\bigraph{bwd1v1v1v1p1v1x0p0x1v1x0p1x1p0x1duals1v1v1x2v1x2x3}, \bigraph{bwd1v1v1v1p1v1x0p1x0v1x0p1x0p0x1p0x1duals1v1v1x2v1x2x4x3} \right)$ & 111 & \{\} & \{\} \\
$\left(\bigraph{bwd1v1v1v1p1v1x0p0x1v1x0p1x1p0x1duals1v1v1x2v1x2x3}, \bigraph{bwd1v1v1v1p1v1x0p1x0v1x0p1x0p0x1p0x1duals1v1v1x2v2x1x4x3} \right)$ & 111 & \{\} & \{\} \\
$\left(\bigraph{bwd1v1v1v1p1v1x0p0x1v1x0p1x1p0x1duals1v1v1x2v3x2x1}, \bigraph{bwd1v1v1v1p1v1x0p1x0v1x1duals1v1v1x2v1} \right)$ & 109 & \{\} & \{\} \\
$\left(\bigraph{bwd1v1v1v1p1v1x0p0x1v1x0p0x1p1x0p1x0p0x1p0x1duals1v1v1x2v1x3x2x5x4x6}, \bigraph{bwd1v1v1v1p1v1x0p1x0v1x1duals1v1v1x2v1} \right)$ & 109 & \{\} & \{\} \\
$\left(\bigraph{bwd1v1v1v1p1v1x0p0x1v1x0p0x1p0x1p0x1v1x0x0x0p1x0x0x0p0x1x0x0v1x0x0p0x1x0p0x1x0p0x0x1duals1v1v1x2v3x2x1x4v1x2x4x3}, \bigraph{bwd1v1v1v1p1v1x0p1x0v0x1p0x1v1x0p0x1p0x1duals1v1v1x2v1x2} \right)$ & 137 & \{\} & \{\} \\
$\left(\bigraph{bwd1v1v1p1p1v0x1x0p0x1x0p0x0x1duals1v1v1x2x3}, \bigraph{bwd1v1v1p1p1v0x1x0p0x1x0p0x0x1duals1v1v1x2x3} \right)$ & 110 & \{\} & \{\} \\
$\left(\bigraph{bwd1v1v1p1p1v1x0x0p0x0x1p0x0x1duals1v1v1x2x3}, \bigraph{bwd1v1v1p1p1v1x0x0p0x0x1p0x0x1duals1v1v1x3x2} \right)$ & 110 & \{\} & \{\} \\
$\left(\bigraph{bwd1v1v1p1p1v0x1x0p0x0x1p0x0x1duals1v1v1x3x2}, \bigraph{bwd1v1v1p1p1v0x1x0p0x0x1p0x0x1duals1v1v1x3x2} \right)$ & 110 & \{\} & \{\} \\
$\left(\bigraph{bwd1v1v1p1v1x1v1duals1v1v1}, \bigraph{bwd1v1v1p1v1x1v1duals1v1v1} \right)$ & 100 & \{\} & \{\} \\
$\left(\bigraph{bwd1v1v1v1p1v1x0p1x0p0x1v1x0x0p1x0x0p0x0x1v1x0x1duals1v1v1x2v1x3x2}, \bigraph{bwd1v1v1v1p1v1x0p0x1p1x0v1x0x0p1x0x0p0x1x0v1x0x1duals1v1v1x2v1x3x2} \right)$ & 147 & \{\} & \{\} \\
$\left(\bigraph{bwd1v1v1v1p1v1x0p0x1v1x1p0x1v0x1p0x1duals1v1v1x2v1x2}, \bigraph{bwd1v1v1v1p1v1x0p1x0v1x0p0x1p0x1v0x1x0p0x1x0duals1v1v1x2v1x2x3} \right)$ & 120 & \{\} & \{\} \\
$\left(\bigraph{bwd1v1v1v1p1v1x0p0x1v1x1p0x1v0x1p0x1duals1v1v1x2v1x2}, \bigraph{bwd1v1v1v1p1v1x0p1x0v1x0p0x1p0x1v0x1x0p0x0x1v1x0p0x1duals1v1v1x2v1x2x3v2x1} \right)$ & 120 & \{\} & \{\} \\
$\left(\bigraph{bwd1v1v1v1p1v1x0p0x1v1x1p0x1v0x1p0x1duals1v1v1x2v1x2}, \bigraph{bwd1v1v1v1p1v1x0p1x0v1x0p0x1p0x1v0x1x0p0x0x1v1x0p0x1duals1v1v1x2v1x3x2v1x2} \right)$ & 120 & \{\} & \{\} \\
$\left(\bigraph{bwd1v1v1v1p1p1v0x1x0p0x1x0v1x0v1duals1v1v1x2x3v1}, \bigraph{bwd1v1v1v1p1p1v1x0x0p0x0x1v1x0p1x0p0x1v0x0x1v1duals1v1v1x2x3v1x3x2v1} \right)$ & 119 & \{\} & \{\} \\
$\left(\bigraph{bwd1v1v1v1p1v1x0p1x0p0x1v1x0x0p1x0x0p0x0x1p0x1x0p0x0x1v0x0x1x0x0v1duals1v1v1x2v1x3x2x5x4v1}, \bigraph{bwd1v1v1v1p1v1x0p1x0p1x0v1x0x0v1duals1v1v1x2v1} \right)$ & 155 & \{\} & \{\} \\
$\left(\bigraph{bwd1v1v1v1p1v1x0p0x1v1x1v1v1p1v1x0duals1v1v1x2v1v1x2}, \bigraph{bwd1v1v1v1p1v1x0p1x0v1x0p0x1v1x1v1v1duals1v1v1x2v1x2v1} \right)$ & 192 & \{\} & \{\} \\
$\left(\bigraph{bwd1v1v1v1p1v1x0p0x1v1x0p0x1p1x0p0x1v0x1x0x0p1x0x0x0p0x0x1x0p0x0x0x1v0x1x0x0p1x0x0x0p1x0x0x0p0x0x1x0p0x0x0x1p0x0x1x0v0x0x0x0x1x0p0x0x0x0x0x1p1x0x0x0x0x0p0x1x0x0x0x0v0x1x0x0p1x1x0x0p0x0x1x1p0x0x0x1duals1v1v1x2v1x3x2x4v3x2x1x5x4x6v4x3x2x1}, \bigraph{bwd1v1v1v1p1v1x0p1x0v1x0p0x1v1x0p1x0p0x1p0x1v1x0x0x0p0x0x1x0v1x0p1x0p0x1p0x1v0x1x0x1duals1v1v1x2v1x2v2x1v1} \right)$ & 348 & \{\} & \{\} \\
$\left(\bigraph{bwd1v1v1v1p1v1x1v1duals1v1v1x2v1}, \bigraph{bwd1v1v1v1p1v1x1v1duals1v1v1x2v1} \right)$ & 100 & \{\} & \{\} \\
$\left(\bigraph{bwd1v1v1v1p1v1x1v1duals1v1v2x1v1}, \bigraph{bwd1v1v1v1p1v1x1v1duals1v1v2x1v1} \right)$ & 100 & \{\} & \{\} \\
$\left(\bigraph{bwd1v1v1v1p1v1x0p1x0p0x1v1x0x0p0x1x0p1x0x0p0x0x1v0x1x0x1duals1v1v1x2v1x2x4x3}, \bigraph{bwd1v1v1v1p1v1x0p0x1p1x0v1x1x0v1duals1v1v1x2v1} \right)$ & 111 & \{\} & \{\} \\
$\left(\bigraph{bwd1v1v1p1v1x0p0x1p0x1v1x0x0p1x0x0p0x1x0p0x0x1v0x1x0x0p0x1x0x0p0x0x1x1duals1v1v3x2x1v1x2x3}, \bigraph{bwd1v1v1p1v1x0p0x1p0x1v1x0x0p0x1x0p1x0x0p0x0x1v0x1x0x1p0x0x1x0p0x0x1x0duals1v1v3x2x1v1x2x3} \right)$ & 244 & \{\} & \{\} \\
$\left(\bigraph{bwd1v1v1p1v1x0p0x1p0x1v1x0x0p1x0x0p0x1x0p0x0x1v0x1x0x0p0x1x0x0p0x0x1x1duals1v1v3x2x1v1x2x3}, \bigraph{bwd1v1v1p1v1x0p0x1p0x1v1x0x0p0x1x0p1x0x0p0x0x1v0x1x0x1p0x0x1x0p0x0x1x0duals1v1v3x2x1v1x3x2} \right)$ & 244 & \{\} & \{\} \\
$\left(\bigraph{bwd1v1v1p1v1x0p0x1p0x1v1x0x0p1x0x0p0x1x0p0x0x1v0x1x0x0p0x1x0x0p0x0x1x1duals1v1v3x2x1v2x1x3}, \bigraph{bwd1v1v1p1v1x0p0x1p0x1v1x0x0p0x1x0p1x0x0p0x0x1v0x1x0x1p0x0x1x0p0x0x1x0duals1v1v3x2x1v1x3x2} \right)$ & 244 & \{\} & \{\} \\
$\left(\bigraph{bwd1v1v1p1v1x1p0x1duals1v1v1x2}, \bigraph{bwd1v1v1p1v1x1p0x1duals1v1v1x2} \right)$ & 109 & \{2\} & \{?\} \\
$\left(\bigraph{bwd1v1v1v1p1p1v0x1x0p0x1x0p0x0x1duals1v1v1x2x3}, \bigraph{bwd1v1v1v1p1p1v1x0x0p1x0x0p0x0x1duals1v1v1x2x3} \right)$ & 110 & \{\} & \{\} \\
$\left(\bigraph{bwd1v1v1v1p1p1v0x1x0p0x1x0p0x0x1duals1v1v1x2x3}, \bigraph{bwd1v1v1v1p1p1v1x0x0p0x1x0p0x0x1v1x0x0p0x1x0duals1v1v1x2x3v2x1} \right)$ & 110 & \{\} & \{\} \\
$\left(\bigraph{bwd1v1v1v1p1v1x0p0x1v1x0p0x1p0x1p0x1v1x0x0x0p1x0x0x0p0x0x1x0v1x0x0p0x1x1duals1v1v1x2v4x2x3x1v1x2}, \bigraph{bwd1v1v1v1p1v1x0p1x0v0x1p0x1v1x0p0x1p0x1v0x0x1duals1v1v1x2v1x2v1} \right)$ & 218 & \{\} & \{\} \\
$\left(\bigraph{bwd1v1v1v1p1v1x0p0x1v0x1p1x0p0x1v1x0x0p1x0x0p0x1x0v1x0x0p0x0x1p1x0x0p0x0x1p0x1x0p0x0x1v0x0x0x1x0x0p1x0x0x0x0x0v1x0p1x0p0x1duals1v1v1x2v1x3x2v1x2x4x3x6x5v1x3x2}, \bigraph{bwd1v1v1v1p1v1x0p1x0v0x1v1p1p1v1x0x0v1p1duals1v1v1x2v1v1} \right)$ & 196 & \{\} & \{\} \\
$\left(\bigraph{bwd1v1v1v1p1v1x0p0x1v1x0p1x1p0x1v1x0x0duals1v1v1x2v1x2x3}, \bigraph{bwd1v1v1v1p1v1x0p1x0v1x0p1x0p0x1p0x1v1x0x0x0duals1v1v1x2v1x2x4x3} \right)$ & 120 & \{\} & \{\} \\
$\left(\bigraph{bwd1v1v1v1p1v1x0p0x1v1x0p1x1p0x1v0x0x1duals1v1v1x2v1x2x3}, \bigraph{bwd1v1v1v1p1v1x0p1x0v1x0p1x0p0x1p0x1v0x0x1x0duals1v1v1x2v1x2x3x4} \right)$ & 120 & \{\} & \{\} \\
$\left(\bigraph{bwd1v1v1v1p1v1x0p0x1v1x0p1x0p0x1p0x1v0x0x1x0p1x0x0x0p0x1x0x0p0x0x0x1v1x0x0x0p0x0x1x0p0x1x0x0p0x0x0x1p1x0x0x0p0x0x1x0v0x0x0x0x0x1p0x0x0x1x0x0p0x0x0x0x1x0p0x0x1x0x0x0v0x1x0x0p0x0x0x1p0x1x0x0p1x0x0x0p0x0x1x0p1x0x1x0p0x0x0x1v0x1x0x0x0x0x0p0x1x0x0x0x0x0p0x0x1x0x0x0x0p0x0x1x0x0x0x0v1x0x0x0p0x0x1x0p0x1x0x0p0x0x0x1duals1v1v1x2v1x3x2x4v3x4x1x2x5x6v7x4x5x2x3x6x1v1x2x3x4}, \bigraph{bwd1v1v1v1p1v1x0p1x0v1x0p0x1v1x0p1x0p0x1p0x1v1x0x0x0p0x0x1x0v1x0p1x0p0x1p0x1v1x0x0x0p1x0x0x0p0x0x1x0p0x0x1x0v1x0x0x0p0x1x0x0p0x0x1x0p0x0x0x1duals1v1v1x2v1x2v2x1v1x2x3x4} \right)$ & 345 & \{\} & \{\} \\
$\left(\bigraph{bwd1v1v1v1p1v1x0p0x1v1x0p1x0p0x1p0x1v0x0x1x0p1x0x0x0p0x1x0x0p0x0x0x1v1x0x0x0p1x0x0x0p0x1x0x0p0x0x1x0p0x0x1x0p0x0x0x1v0x0x0x0x1x0p0x1x0x0x0x0p0x0x0x0x0x1p0x0x1x0x0x0v1x1x0x0p0x0x1x0p0x1x0x0p1x0x0x0p0x0x0x1p0x0x1x0p0x0x0x1v0x0x0x0x1x0x0p0x0x0x0x1x0x0p0x1x0x0x0x0x0p0x1x0x0x0x0x0v1x0x0x0p0x1x0x0p0x0x1x0p0x0x0x1duals1v1v1x2v1x3x2x4v3x2x1x6x5x4v1x3x2x5x4x7x6v1x2x4x3}, \bigraph{bwd1v1v1v1p1v1x0p1x0v1x0p0x1v1x0p1x0p0x1p0x1v1x0x0x0p0x0x1x0v1x0p0x1p1x0p0x1v1x0x0x0p1x0x0x0p0x1x0x0p0x1x0x0v1x0x0x0p0x1x0x0p0x0x1x0p0x0x0x1duals1v1v1x2v1x2v2x1v1x2x4x3} \right)$ & 345 & \{\} & \{\} \\
$\left(\bigraph{bwd1v1v1v1p1v1x0p0x1v1x0p0x1p1x0p0x1v0x1x0x0p1x0x0x0p0x0x1x0p0x0x0x1v0x1x0x0p1x0x0x0p0x1x0x0p0x0x0x1p0x0x1x0p1x0x0x0v0x0x1x0x0x0p0x0x0x1x0x0p0x0x0x0x0x1p0x0x0x0x1x0v0x0x0x1p1x0x0x0p0x1x0x0p1x0x1x0p0x0x0x1p0x1x0x0p0x0x1x0v0x0x0x0x1x0x0p0x0x0x0x1x0x0p0x0x0x0x0x1x0p0x0x0x0x0x1x0v1x0x0x0p0x1x0x0p0x0x1x0p0x0x0x1duals1v1v1x2v2x1x3x4v4x5x3x1x2x6v3x5x1x4x2x7x6v2x1x4x3}, \bigraph{bwd1v1v1v1p1v1x0p1x0v1x0p0x1v1x0p0x1p1x0p0x1v1x0x0x0p0x1x0x0v1x0p1x0p0x1p0x1v1x0x0x0p1x0x0x0p0x0x1x0p0x0x1x0v0x1x0x0p1x0x0x0p0x0x1x0p0x0x0x1duals1v1v1x2v1x2v2x1v2x1x4x3} \right)$ & 345 & \{\} & \{\} \\
$\left(\bigraph{bwd1v1v1v1p1v1x0p0x1v0x1p1x0p1x0p0x1v1x0x0x0p0x1x0x0p0x0x1x0p0x0x0x1v0x1x0x0p1x0x0x0p1x0x0x0p0x1x0x0p0x0x1x0p0x0x0x1v0x0x0x1x0x0p0x1x0x0x0x0p0x0x0x0x0x1p0x0x0x0x1x0v1x0x0x0p0x1x0x0p1x0x0x0p0x1x0x0p0x0x1x0p0x1x0x0p1x0x0x0p0x0x0x1p0x0x1x0p0x0x0x1v0x0x0x0x1x0x0x1x0x0v1duals1v1v1x2v2x1x3x4v6x2x5x4x3x1v2x1x4x3x6x5x8x7x10x9v1}, \bigraph{bwd1v1v1v1p1v1x0p1x0v1x0p0x1v1x0p0x1p1x0p0x1v1x0x0x0p0x1x0x0v1x0p0x1p1x0p0x1v1x1x0x0v1duals1v1v1x2v1x2v2x1v1} \right)$ & 337 & \{\} & \{\} \\
$\left(\bigraph{bwd1v1v1v1p1v1x0p0x1v1x1p0x1v0x1p0x1v1x0v1duals1v1v1x2v1x2v1}, \bigraph{bwd1v1v1v1p1v1x0p1x0v1x0p0x1p0x1v0x1x0p0x0x1v1x0p1x0p0x1v0x0x1v1duals1v1v1x2v1x2x3v1x3x2v1} \right)$ & 174 & \{\} & \{\} \\
$\left(\bigraph{bwd1v1v1v1p1v1x0p0x1v1x0p0x1p0x1p0x1v1x0x0x0p1x0x0x0p0x1x0x0v1x0x0p0x1x0p0x0x1p0x1x0p0x0x1v0x0x0x1x0v1duals1v1v1x2v4x2x3x1v1x2x3x5x4v1}, \bigraph{bwd1v1v1v1p1v1x0p1x0v0x1p0x1v1x0p0x1p0x1v0x0x1v1duals1v1v1x2v1x2v1} \right)$ & 137 & \{\} & \{\} \\
$\left(\bigraph{bwd1v1v1v1p1p1v0x1x0p0x1x0v1x0v1p1duals1v1v1x2x3v1}, \bigraph{bwd1v1v1v1p1p1v1x0x0p0x0x1v1x0p1x0p0x1v1x0x0p0x0x1v1x0p0x1p0x1duals1v1v1x2x3v1x3x2v3x2x1} \right)$ & 180 & \{\} & \{\} \\
$\left(\bigraph{bwd1v1v1v1p1v1x0p1x0p0x1v1x0x0p0x0x1p0x1x0p0x0x1v1x0x1x0p0x0x0x1p0x0x0x1duals1v1v1x2v4x2x3x1}, \bigraph{bwd1v1v1v1p1v1x0p0x1p1x0v1x0x0p0x1x0p0x0x1p0x1x0v1x0x1x0p0x0x0x1p0x0x0x1duals1v1v1x2v4x2x3x1} \right)$ & 244 & \{\} & \{\} \\
$\left(\bigraph{bwd1v1v1p1v0x1p1x0p0x1v1x0x0p1x0x0p0x1x0v0x1x0p0x0x1p0x1x0p0x0x1v0x1x0x0p0x1x0x0p0x0x1x0v0x1x0duals1v1v1x3x2v2x1x3x4v1}, \bigraph{bwd1v1v1p1v1x0p0x1p0x1v1x0x0p0x1x0p0x1x0v1x0x0p1x0x0p0x1x0p0x1x0v0x1x0x0p0x0x1x0p0x1x0x0v0x0x1duals1v1v3x2x1v1x4x3x2v1} \right)$ & 205 & \{\} & \{\} \\
$\left(\bigraph{bwd1v1v1v1p1v1x0p0x1v0x1p1x0p0x1v1x0x0p1x0x0p0x1x0v1x0x0p0x0x1p1x0x0p0x0x1p0x1x0p0x0x1v1x0x0x0x0x0p0x0x0x1x0x0v1x1duals1v1v1x2v1x3x2v1x2x4x3x6x5v1}, \bigraph{bwd1v1v1v1p1v1x0p1x0v0x1v1p1p1v1x0x0v1p1v1x0duals1v1v1x2v1v1v1} \right)$ & 434 & \{\} & \{\} \\
$\left(\bigraph{bwd1v1v1v1p1p1v0x1x0p0x0x1v1x1duals1v1v1x3x2v1}, \bigraph{bwd1v1v1v1p1p1v1x0x0p1x0x0v1x0p0x1duals1v1v1x3x2v2x1} \right)$ & 120 & \{2\} & \{?\} \\
$\left(\bigraph{bwd1v1v1v1p1p1v0x1x0p0x1x0v1x0p0x1duals1v1v1x2x3v1x2}, \bigraph{bwd1v1v1v1p1p1v1x0x0p0x0x1v1x1duals1v1v1x2x3v1} \right)$ & 120 & \{2\} & \{?\} \\
$\left(\bigraph{bwd1v1v1v1p1v1x0p0x1v1x0p1x0p0x1p0x1v0x0x1x0p1x0x0x0p0x1x0x0p0x0x0x1v1x0x0x0p0x0x1x0p0x1x0x0p0x0x1x0p0x0x0x1p1x0x0x0v0x0x0x0x1x0p0x1x0x0x0x0p0x0x1x0x0x0p0x0x0x0x0x1v0x1x0x0p1x1x0x0p0x0x1x1p0x0x0x1v1x0x0x1duals1v1v1x2v1x3x2x4v3x2x1x5x4x6v4x3x2x1}, \bigraph{bwd1v1v1v1p1v1x0p1x0v1x0p0x1v1x0p1x0p0x1p0x1v1x0x0x0p0x0x1x0v1x0p1x0p0x1p0x1v0x1x0x1v1v1duals1v1v1x2v1x2v2x1v1v1} \right)$ & 337 & \{\} & \{\} \\
$\left(\bigraph{bwd1v1v1v1p1v1x0p0x1v0x1p1x0p0x1v1x0x0p1x0x0p0x1x0v1x0x0p1x0x0p0x0x1p0x1x0p0x0x1p0x0x1v1x0x0x0x0x0p0x0x0x0x1x0v1x0p1x0p0x1p0x1v0x0x0x1v1duals1v1v1x2v1x3x2v1x5x4x3x2x6v4x2x3x1v1}, \bigraph{bwd1v1v1v1p1v1x0p1x0v0x1v1p1p1v1x0x0v1p1v1x0v1duals1v1v1x2v1v1v1} \right)$ & 214 & \{\} & \{\} \\
$\left(\bigraph{bwd1v1v1v1p1v1x1p0x1duals1v1v1x2}, \bigraph{bwd1v1v1v1p1v1x1p0x1duals1v1v1x2} \right)$ & 109 & \{\} & \{\} \\
$\left(\bigraph{bwd1v1v1v1p1p1v0x1x0p0x1x0v1x0v1p1v1x0duals1v1v1x2x3v1v1}, \bigraph{bwd1v1v1v1p1p1v1x0x0p0x0x1v1x0p1x0p0x1v1x0x0p0x0x1v1x1duals1v1v1x2x3v1x3x2v1} \right)$ & 162 & \{\} & \{\} \\
$\left(\bigraph{bwd1v1v1v1p1v1x0p0x1v1x1p0x1v1x0v1duals1v1v1x2v1x2v1}, \bigraph{bwd1v1v1v1p1v1x0p1x0v1x0p0x1p0x1v1x1x0duals1v1v1x2v1x2x3} \right)$ & 147 & \{0\} & \{\text{e}\} \\
$\left(\bigraph{bwd1v1v1v1p1v1x0p1x0p0x1v1x0x0p0x0x1p1x0x0v0x1x0p0x1x0p0x0x1p0x0x1v1x0x0x0p1x0x0x0p0x0x1x0v0x1x0duals1v1v1x2v2x1x3v1x3x2}, \bigraph{bwd1v1v1v1p1v1x0p0x1p1x0v1x0x0p1x0x0p0x1x0v1x0x0p0x0x1p1x0x0p0x0x1v0x0x1x0p0x0x0x1p0x0x0x1v0x0x1duals1v1v1x2v1x3x2v3x2x1} \right)$ & 205 & \{\} & \{\} \\
$\left(\bigraph{bwd1v1v1p1p1p1duals1v1}, \bigraph{bwd1v1v1p1p1p1duals1v1} \right)$ & 124 & \{0\} & \{\text{c}\} \\
$\left(\bigraph{bwd1v1v1p1v1x0p1x0p0x1p0x1v0x1x1x0duals1v1v1x3x2x4}, \bigraph{bwd1v1v1p1v1x1v1v1duals1v1v1v1} \right)$ & 109 & \{0\} & \{?\} \\
$\left(\bigraph{bwd1v1v1p1v1x0p1x0p0x1p0x1v0x1x1x0duals1v1v4x2x3x1}, \bigraph{bwd1v1v1p1v1x1v1v1duals1v1v1v1} \right)$ & 109 & \{0\} & \{?\} \\
$\left(\bigraph{bwd1v1v1v1p1p1v0x1x0p0x1x0p0x0x1v0x0x1duals1v1v1x2x3v1}, \bigraph{bwd1v1v1v1p1p1v1x0x0p0x1x0p0x0x1v1x0x0p0x1x0p0x0x1duals1v1v1x2x3v2x1x3} \right)$ & 137 & \{0\} & \{?\} \\
$\left(\bigraph{bwd1v1v1p1v1x1v1v1duals1v1v1v1}, \bigraph{bwd1v1v1p1v1x1v1v1duals1v1v1v1} \right)$ & 109 & \{0\} & \{?\} \\
$\left(\bigraph{bwd1v1v1p1p1v0x0x1p0x0x1v1x0p0x1v1x0p0x1duals1v1v1x2v2x1}, \bigraph{bwd1v1v1p1p1v0x0x1p0x0x1v1x0p1x0duals1v1v1x2} \right)$ & 171 & \{0\} & \{\text{a}\} \\
$\left(\bigraph{bwd1v1v1p1p1v0x0x1p0x0x1v1x0p0x1v1x0p0x1duals1v1v2x1v1x2}, \bigraph{bwd1v1v1p1p1v0x0x1p0x0x1v1x0p1x0duals1v1v1x2} \right)$ & 171 & \{0\} & \{\text{a}\} \\
$\left(\bigraph{bwd1v1v1v1p1v1x0p0x1v1x0p0x1p1x0p0x1v0x1x0x0p1x0x0x0p0x0x1x0p0x0x0x1v1x0x0x0p1x0x1x0p0x1x0x0p0x1x0x1duals1v1v1x2v2x1x3x4v1x2x3x4}, \bigraph{bwd1v1v1v1p1v1x0p1x0v1x0p0x1v1x0p0x1p1x0p0x1v1x1x0x0p0x0x1x0p0x0x0x1duals1v1v1x2v1x2v1x2x3} \right)$ & 298 & \{0\} & \{\text{b}\} \\
$\left(\bigraph{bwd1v1v1v1p1v1x0p1x0p0x1v1x0x0p0x0x1p0x0x1p0x1x0p0x0x1v1x0x0x0x0p0x0x0x1x0v1x1duals1v1v1x2v3x2x1x5x4v1}, \bigraph{bwd1v1v1v1p1v1x0p1x0p1x0v0x0x1v1p1duals1v1v1x2v1} \right)$ & 144 & \{0\} & \{\text{c}\} \\
$\left(\bigraph{bwd1v1v1v1p1v1x0p1x0p0x1v0x0x1p1x0x0p0x1x0p0x0x1p0x0x1v0x1x0x0x0p0x0x1x0x0v1x0p1x0p0x1p0x1duals1v1v1x2v3x5x1x4x2v4x2x3x1}, \bigraph{bwd1v1v1v1p1v1x0p1x0p1x0v0x0x1v1p1duals1v1v1x2v1} \right)$ & 144 & \{0\} & \{\text{a}\} \\
$\left(\bigraph{bwd1v1v1v1p1v1x0p1x0p0x1v1x0x0p1x0x0p0x0x1p0x1x0p0x0x1v0x0x1x0x0p1x0x0x0x0v1x0p1x0p0x1duals1v1v1x2v1x3x2x5x4v1x3x2}, \bigraph{bwd1v1v1v1p1v1x0p1x0p1x0v1x0x0v1p1duals1v1v1x2v1} \right)$ & 144 & \{0\} & \{\text{a}\} \\
$\left(\bigraph{bwd1v1v1v1p1p1v0x1x0p0x1x0v1x0v1p1v1x0v1duals1v1v1x2x3v1v1}, \bigraph{bwd1v1v1v1p1p1v1x0x0p0x0x1v1x0p1x0p0x1v1x0x0p0x0x1v1x0p1x0p0x1p0x1v0x0x1x0v1duals1v1v1x2x3v1x3x2v3x2x1x4v1} \right)$ & 243 & \{0\} & \{\text{a}\} \\
$\left(\bigraph{bwd1v1v1v1p1v1x0p0x1v0x1p1x0p0x1v1x0x0p1x0x0p0x1x0v1x0x1p0x1x0p0x1x0p0x0x1v0x0x0x1v1duals1v1v1x2v1x3x2v1x2x4x3v1}, \bigraph{bwd1v1v1v1p1v1x0p1x0v0x1v1p1p1v1x0x0p0x1x0v0x1duals1v1v1x2v1v1x2} \right)$ & 234 & \{0\} & \{\text{a}\} \\
$\left(\bigraph{bwd1v1v1v1p1v1x0p0x1v1x1p0x1v0x1p0x1v1x0v1p1duals1v1v1x2v1x2v1}, \bigraph{bwd1v1v1v1p1v1x0p1x0v1x0p0x1p0x1v0x1x0p0x0x1v1x0p1x0p0x1v1x0x0p0x0x1v1x0p0x1p0x1duals1v1v1x2v1x2x3v1x3x2v3x2x1} \right)$ & 267 & \{\} & \{\} \\
$\left(\bigraph{bwd1v1v1v1p1v1x0p0x1v0x1p1x0p0x1v1x0x0p1x0x0p0x1x0v1x0x0p0x1x0p1x0x0p0x0x1p0x0x1p0x0x1v0x0x1x0x0x0p0x0x0x0x0x1v1x0p1x0p0x1p0x1v0x1x0x0p0x0x0x1v1x0p0x1p0x1duals1v1v1x2v1x3x2v6x4x3x2x5x1v4x2x3x1v3x2x1}, \bigraph{bwd1v1v1v1p1v1x0p1x0v0x1v1p1p1v1x0x0v1p1v1x0v1p1duals1v1v1x2v1v1v1} \right)$ & 275 & \{\} & \{\} \\
$\left(\bigraph{bwd1v1v1v1p1v1x0p0x1v1x0p0x1p0x1p0x1v1x0x0x0p1x0x0x0p0x0x0x1v1x0x0p0x1x0p0x0x1p0x1x0p0x0x1v0x0x1x0x0p0x0x0x1x0v1x0p0x1p0x1duals1v1v1x2v2x1x3x4v1x2x3x5x4v3x2x1}, \bigraph{bwd1v1v1v1p1v1x0p1x0v0x1p0x1v1x0p0x1p0x1v0x0x1v1p1duals1v1v1x2v1x2v1} \right)$ & 307 & \{\} & \{\} \\
$\left(\bigraph{bwd1v1v1v1p1v1x0p0x1v1x1p0x1v0x1p0x1v1x0p0x1duals1v1v1x2v1x2v1x2}, \bigraph{bwd1v1v1v1p1v1x0p1x0v1x0p0x1p0x1v0x1x0p0x0x1v1x1duals1v1v1x2v1x2x3v1} \right)$ & 165 & \{0\} & \{\text{e}\} \\
$\left(\bigraph{bwd1v1v1v1p1v1x0p0x1v1x1p0x1v0x1p0x1v1x0p0x1duals1v1v1x2v1x2v2x1}, \bigraph{bwd1v1v1v1p1v1x0p1x0v1x0p0x1p0x1v0x1x0p0x0x1v1x1duals1v1v1x2v1x3x2v1} \right)$ & 165 & \{0\} & \{\text{e}\} \\
\end{longtable}

\subsection{The fusion ring of \texorpdfstring{$\cX$}{X}}
\label{sec:FusionMatrices}

Below, we give the fusion matrices $L_X$ for tensoring on the left with $X\in
\{\jw {2},\jw{4},P,Q,R,g\}$ for the fusion ring $K_0(\frac{1}{2}\cX)$ from Section \ref{sec:FormalCodegrees} in the ordered basis
$$
B=\left(1,\jw{2},\jw{4},P,Q,R,gP,g\jw{4},g\jw{2},g\right).
$$
This means the $(i,j)$-th entry of $L_X$ is the coefficient of $X_i$ in $X\otimes X_j$, where $X_k$ denotes the $k$-th element of $B$.
For each object $X$, the fusion matrix for $gX$ is $L_{gX}=L_gL_X=L_XL_g$, which is obtained from $L_X$ by permuting the columns with the permutation $\left(10, 9, 8,7,5,6,4,3,2,1\right)$.
\begin{align*}
L_{\jw{2}}
&=
\left(
\begin{smallmatrix}
 0 & 1 & 0 & 0 & 0 & 0 & 0 & 0 & 0 & 0 \\
 1 & 1 & 1 & 0 & 0 & 0 & 0 & 0 & 0 & 0 \\
 0 & 1 & 1 & 1 & 1 & 0 & 0 & 0 & 0 & 0 \\
 0 & 0 & 1 & 1 & 1 & 1 & 1 & 0 & 0 & 0 \\
 0 & 0 & 1 & 1 & 1 & 0 & 1 & 1 & 0 & 0 \\
 0 & 0 & 0 & 1 & 0 & 0 & 1 & 0 & 0 & 0 \\
 0 & 0 & 0 & 1 & 1 & 1 & 1 & 1 & 0 & 0 \\
 0 & 0 & 0 & 0 & 1 & 0 & 1 & 1 & 1 & 0 \\
 0 & 0 & 0 & 0 & 0 & 0 & 0 & 1 & 1 & 1 \\
 0 & 0 & 0 & 0 & 0 & 0 & 0 & 0 & 1 & 0
\end{smallmatrix}
\right)
&
L_{\jw{4}}
&=
\left(
\begin{smallmatrix}
 0 & 0 & 1 & 0 & 0 & 0 & 0 & 0 & 0 & 0 \\
 0 & 1 & 1 & 1 & 1 & 0 & 0 & 0 & 0 & 0 \\
 1 & 1 & 2 & 2 & 2 & 1 & 2 & 1 & 0 & 0 \\
 0 & 1 & 2 & 3 & 3 & 1 & 3 & 2 & 0 & 0 \\
 0 & 1 & 2 & 3 & 3 & 2 & 3 & 2 & 1 & 0 \\
 0 & 0 & 1 & 1 & 2 & 1 & 1 & 1 & 0 & 0 \\
 0 & 0 & 2 & 3 & 3 & 1 & 3 & 2 & 1 & 0 \\
 0 & 0 & 1 & 2 & 2 & 1 & 2 & 2 & 1 & 1 \\
 0 & 0 & 0 & 0 & 1 & 0 & 1 & 1 & 1 & 0 \\
 0 & 0 & 0 & 0 & 0 & 0 & 0 & 1 & 0 & 0
\end{smallmatrix}
\right)
&
L_P
&=
\left(
\begin{smallmatrix}
 0 & 0 & 0 & 1 & 0 & 0 & 0 & 0 & 0 & 0 \\
 0 & 0 & 1 & 1 & 1 & 1 & 1 & 0 & 0 & 0 \\
 0 & 1 & 2 & 3 & 3 & 1 & 3 & 2 & 0 & 0 \\
 1 & 1 & 3 & 4 & 4 & 2 & 4 & 3 & 1 & 0 \\
 0 & 1 & 3 & 4 & 5 & 2 & 4 & 3 & 1 & 0 \\
 0 & 1 & 1 & 2 & 2 & 1 & 2 & 1 & 1 & 0 \\
 0 & 1 & 3 & 4 & 4 & 2 & 4 & 3 & 1 & 1 \\
 0 & 0 & 2 & 3 & 3 & 1 & 3 & 2 & 1 & 0 \\
 0 & 0 & 0 & 1 & 1 & 1 & 1 & 1 & 0 & 0 \\
 0 & 0 & 0 & 0 & 0 & 0 & 1 & 0 & 0 & 0
\end{smallmatrix}
\right)
\displaybreak[1]
\\
L_Q
&=
\left(
\begin{smallmatrix}
 0 & 0 & 0 & 0 & 1 & 0 & 0 & 0 & 0 & 0 \\
 0 & 0 & 1 & 1 & 1 & 0 & 1 & 1 & 0 & 0 \\
 0 & 1 & 2 & 3 & 3 & 2 & 3 & 2 & 1 & 0 \\
 0 & 1 & 3 & 4 & 5 & 2 & 4 & 3 & 1 & 0 \\
 1 & 1 & 3 & 5 & 4 & 2 & 5 & 3 & 1 & 1 \\
 0 & 0 & 2 & 2 & 2 & 0 & 2 & 2 & 0 & 0 \\
 0 & 1 & 3 & 4 & 5 & 2 & 4 & 3 & 1 & 0 \\
 0 & 1 & 2 & 3 & 3 & 2 & 3 & 2 & 1 & 0 \\
 0 & 0 & 1 & 1 & 1 & 0 & 1 & 1 & 0 & 0 \\
 0 & 0 & 0 & 0 & 1 & 0 & 0 & 0 & 0 & 0
\end{smallmatrix}
\right)
&
L_R 
&=
\left(
\begin{smallmatrix}
 0 & 0 & 0 & 0 & 0 & 1 & 0 & 0 & 0 & 0 \\
 0 & 0 & 0 & 1 & 0 & 0 & 1 & 0 & 0 & 0 \\
 0 & 0 & 1 & 1 & 2 & 1 & 1 & 1 & 0 & 0 \\
 0 & 1 & 1 & 2 & 2 & 1 & 2 & 1 & 1 & 0 \\
 0 & 0 & 2 & 2 & 2 & 0 & 2 & 2 & 0 & 0 \\
 1 & 0 & 1 & 1 & 0 & 1 & 1 & 1 & 0 & 1 \\
 0 & 1 & 1 & 2 & 2 & 1 & 2 & 1 & 1 & 0 \\
 0 & 0 & 1 & 1 & 2 & 1 & 1 & 1 & 0 & 0 \\
 0 & 0 & 0 & 1 & 0 & 0 & 1 & 0 & 0 & 0 \\
 0 & 0 & 0 & 0 & 0 & 1 & 0 & 0 & 0 & 0
\end{smallmatrix}
\right)
&
L_g
&=
\left(
\begin{smallmatrix}
 0 & 0 & 0 & 0 & 0 & 0 & 0 & 0 & 0 & 1 \\
 0 & 0 & 0 & 0 & 0 & 0 & 0 & 0 & 1 & 0 \\
 0 & 0 & 0 & 0 & 0 & 0 & 0 & 1 & 0 & 0 \\
 0 & 0 & 0 & 0 & 0 & 0 & 1 & 0 & 0 & 0 \\
 0 & 0 & 0 & 0 & 1 & 0 & 0 & 0 & 0 & 0 \\
 0 & 0 & 0 & 0 & 0 & 1 & 0 & 0 & 0 & 0 \\
 0 & 0 & 0 & 1 & 0 & 0 & 0 & 0 & 0 & 0 \\
 0 & 0 & 1 & 0 & 0 & 0 & 0 & 0 & 0 & 0 \\
 0 & 1 & 0 & 0 & 0 & 0 & 0 & 0 & 0 & 0 \\
 1 & 0 & 0 & 0 & 0 & 0 & 0 & 0 & 0 & 0
\end{smallmatrix}
\right)
\end{align*}

\begin{landscape}
\subsection{Subfactors with index in \texorpdfstring{$(4, 5\frac{1}{4}]$}{(4,5.25])}}
\label{appendix:SubfactorLandscape}

Combining the results of this paper with the previous results on the classification of small index subfactors, we obtain the following complete list.
The non Temperley-Lieb-Jones irreducible subfactor planar algebras with index in $(4, 5 \frac{1}{4}]$ are:

\newcommand{\dittotikz}{%
    \tikz{
        \draw [line width=0.12ex] (-0.2ex,0) -- +(0,0.8ex)
            (0.2ex,0) -- +(0,0.8ex);
        \draw [line width=0.08ex] (-0.6ex,0.4ex) -- +(-1.5em,0)
            (0.6ex,0.4ex) -- +(1.5em,0);
    }%
}

\begin{center}
\hspace*{-1cm}
\begin{tabular}{c|c|c|c|c}
index
&
principal graph
&
name
&
$\#$ of subfactors
&
citations
\\\hline
$\frac{1}{2}(5+\sqrt{13})$
&
$\smallbigraph{bwd1v1v1v1p1v1x0p0x1v1x0p0x1duals1v1v1x2v2x1}$
&
Haagerup
&
2d
&
\cite{MR1686551}
\\
$\sim 4.37720$
&
$\smallbigraph{bwd1v1v1v1v1v1v1v1p1v1x0p0x1v1x0p0x1duals1v1v1v1v1x2v2x1}$
&
extended Haagerup
&
2d
&
\cite{MR2979509}
\\
$\frac{1}{2}(5+\sqrt{17})$
&
$\smallbigraph{bwd1v1v1v1v1v1p1v1x0p0x1v1x0p0x1p0x1v1x0x0v1duals1v1v1v1x2v2x1x3v1}$
&
Asaeda-Haagerup
&
2d
&
\cite{MR1686551}
\\
$3+\sqrt{3}$
&
$\smallbigraph{bwd1v1v1v1p1p1v1x0x0v1duals1v1v1x2x3v1}$
&
3311
&
2d
&
existence \cite{MR999799}, 
classification \cite{MR1355948,MR2993924}
\\
$\frac{1}{2}(5+\sqrt{21})$
&
$\smallbigraph{bwd1v1v1p1p1v1x0x0p0x1x0duals1v1v2x1}$
&
2221
&
2c
&
existence \cite{MR1832764}, classification \cite{1102.2052} 
\\
\hline
5
&
$\smallbigraph{bwd1v1p1p1p1duals1v2x1x4x3}$
&
$\bbZ/5\bbZ$
&
1
&
classification \cite{MR1491121}
\\
5
&
$\smallbigraph{bwd1v1p1v1x1v1duals1v1x2v1}$
&
$\bbZ/2\bbZ\subset D_{10}$
&
1
&
\dittotikz
\\
5
&
$\smallbigraph{bwd1v1v1p1p1v1x0x0p0x1x0p0x0x1duals1v1v1x3x2}$
&
$\Integer/4\Integer \subset \Integer/5\Integer \rtimes \Integer/4\Integer$
&
1
&
\dittotikz
\\
5
&
$\smallbigraph{bwd1v1v1v1p1v1x0p0x1v1x1p0x1v0x1v1duals1v1v1x2v1x2v1}$
&
$S_4 \subset S_5$
&
2d
&
classification \cite{MR3335120}
\\
5
&
$\smallbigraph{bwd1v1v1v1p1p1v1x0x0p0x1x0v1x0p0x1duals1v1v1x2x3v2x1}$
&
$A_4 \subset A_5$
&
2d
&
\dittotikz
\\
\hline
$\sim 5.04892$
&
$\smallbigraph{bwd1v1p1v1x0p1x1duals1v1x2}$
&
$\mathfrak{su}(2)_5$
&
1
&
existence \cite{MR936086}, classification \cite{MR3254427}
\\
$\sim 5.04892$
&
$\smallbigraph{bwd1v1v1p1v1x0p0x1p0x1v1x1x0p0x0x1duals1v1v3x2x1}$
&
$\mathfrak{su}(3)_4$
&
1
&
\dittotikz
\\
\hline
$3+\sqrt{5}$
&
$\smallbigraph{bwd1v1p1p1v1x1x0duals1v1x2x3}$
&
$A_3\otimes A_4 = (A_3* A_4)/\sim_1$
&
1
&
classification \cite{MR3345186,1308.5723}
\\
$3+\sqrt{5}$
&
$\smallbigraph{bwd1v1p1v1x0p1x0v1x0p1x0v1x1duals1v1x2v1x2}$
&
$(A_3* A_4)/\sim_2$
&
2d
&
existence \cite{BischHaagerup}, classification \cite{MR3345186,1308.5723}
\\
$3+\sqrt{5}$
&
$\smallbigraph{bwd1v1p1v1x0p1x0v1x0v1p1v1x0p1x0v1x1duals1v1x2v1v1x2}$
&
$(A_3* A_4)/\sim_3$
&
2d
&
existence \cite{1308.5723} (due to Izumi), classification \cite{MR3345186,1308.5723}
\\
$3+\sqrt{5}$
&
$\smallbigraph{bwd1v1p1v1x0p1x0v1x0v1p1v1x0v1p1duals1v1x2v1v1}\cdots$
&
$A_3 * A_4$
&
2d
&
\cite{MR1437496}
\\
$3+\sqrt{5}$
&
$\smallbigraph{bwd1v1v1p1v1x1v1v1duals1v1v1v1}$
&
2D2
&
2d
&
existence \cite{IzumiUnpublished,1406.3401}, classification \cite{1406.3401}
\\
$3+\sqrt{5}$
&
$\smallbigraph{bwd1v1v1v1p1p1v1x0x0p0x1x0p0x0x1v1x0x0p0x1x0p0x0x1duals1v1v1x2x3v1x3x2}$
&
$3^{\bbZ/4\bbZ}$
&
2d
&
existence \cite{IzumiUnpublished,1308.5197}, classification \cite{IzumiUnpublished}
\\
$3+\sqrt{5}$
&
$\smallbigraph{bwd1v1v1v1p1p1v1x0x0p0x1x0p0x0x1v1x0x0p0x1x0p0x0x1duals1v1v1x2x3v1x2x3}$
&
$3^{\bbZ/2\bbZ\times \bbZ/2\bbZ}$
&
1
&
existence \cite{IzumiUnpublished,MR3314808}, classification \cite{IzumiUnpublished}
\\
$3+\sqrt{5}$
&
$\smallbigraph{bwd1v1v1v1v1p1p1v1x0x0p0x1x0p0x0x1v1x0x0p0x1x0v1x0p0x1duals1v1v1v2x1x3v2x1}$
&
4442
&
1
&
existence \cite{MR3314808,IzumiUnpublished}, classification \cite{1406.3401}
\\
\end{tabular}
\end{center}
\thispagestyle{empty}
\newpage
\thispagestyle{empty}
\subsection{The map of subfactors}
\label{appendix:MapOfSubfactors}
\hspace*{-2cm}
\includegraphics[scale=1]{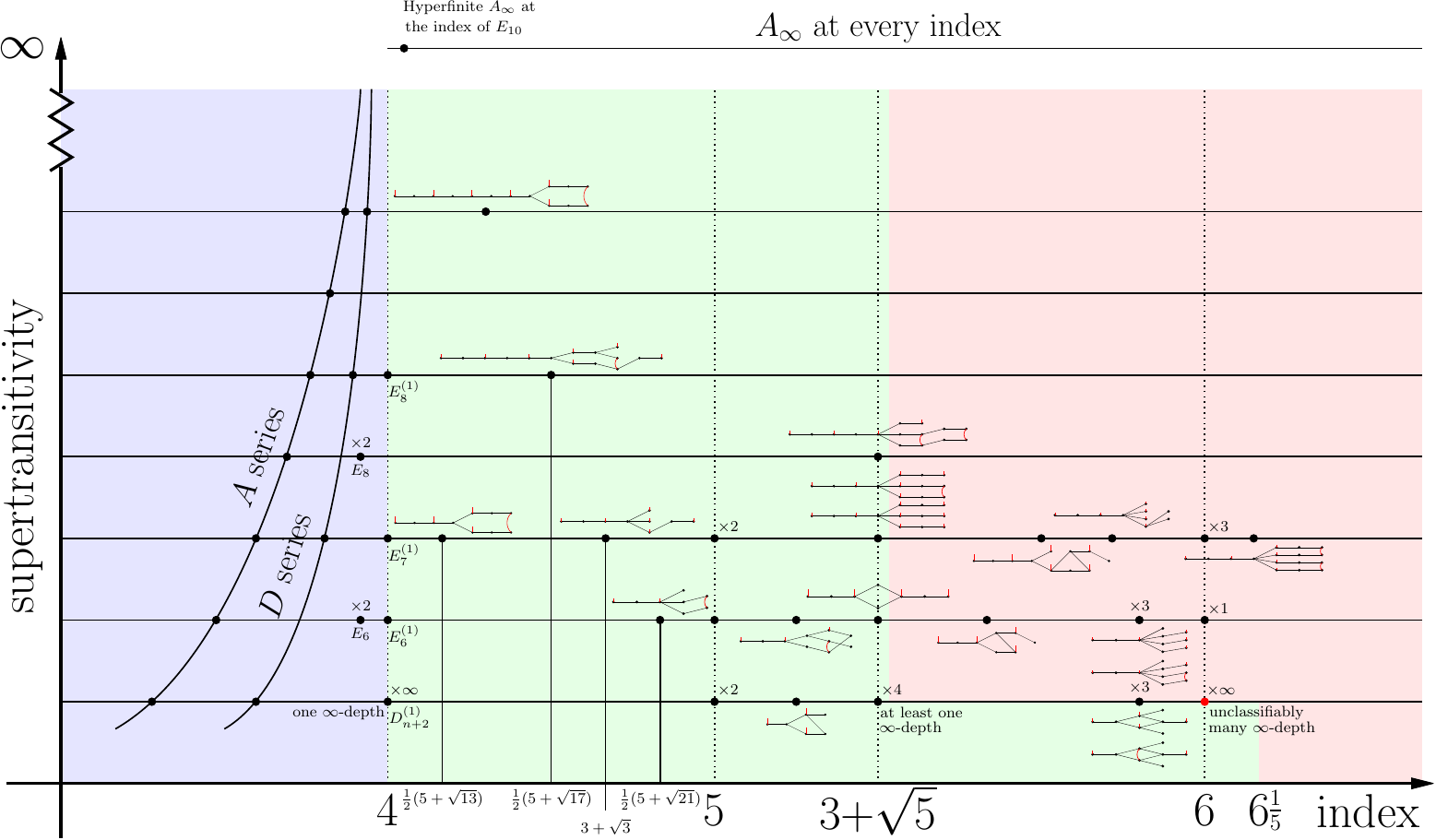}
\end{landscape}
 
\subsection*{Acknowledgements}
 
The authors would like to thank Frank Calegari, Vaughan F.R. Jones, Brendan
McKay, and Noah Snyder for helpful conversations.
Narjess Afzaly was supported by ARC Discovery Project DP0986827, `Structure 
enumeration, applications and analysis'. She would like to thank her supervisor, 
Brendan McKay, for his helpful guidance and suggestions on developing an 
efficient program that generates principal graph pairs.
Kim Morrison was supported by a Discovery Early Career Research Award
DE120100232 and a Discovery Project `Subfactors and symmetries' DP140100732
from the Australian Research Council. David Penneys was partially supported by
the Natural Sciences and Engineering Research Council of Canada, an AMS
Simons travel grant, and NSF DMS grant 1500387. Kim Morrison and David
Penneys were supported by DOD-DARPA grant HR0011-12-1-0009. Kim Morrison and
David Penneys would like to thank the Banff International Research Station for
hosting the 2014 workshop on Subfactors and Fusion Categories. Kim Morrison
would like to thank the Erwin Schr\"odinger Institute and its 2014 programme
on ``Modern Trends in Topological Quantum Field Theory'' for their
hospitality. 

\renewcommand*{\bibfont}{\small}
\setlength{\bibitemsep}{0pt}
\raggedright
\printbibliography

\end{document}